\documentclass{article}

\newif\ifshownavigationpage
\newif\ifshowreminders
\newif\ifshownotationindex
\newif\ifshowtheoremlinks
\newif\ifshowtheoremtree
\newif\ifshowtheoremlist
\newif\ifshowequationlist

\newif\ifshowcomments
\newif\ifhighlight 
\newif\ifelaborate

\newif\ifshowaddressedcomments
\newif\ifshowrvin
\newif\ifshowrvout

\showrvintrue
\showrvouttrue

\showcommentstrue
\usepackage{amsthm, amsmath, amssymb}
\usepackage{thmtools}
\usepackage{soul}
\usepackage{authblk}
\usepackage{framed}
\usepackage{mathrsfs}
\usepackage{hyperref}
\hypersetup{
    colorlinks,
    linkcolor={red!50!black},
    citecolor={blue!50!black},
    urlcolor={blue!80!black}
}

\usepackage[dvipsnames]{xcolor}
\usepackage{graphicx, multicol}
\usepackage{marvosym, wasysym}
\usepackage{fancyhdr}
\usepackage{stackengine}
\usepackage{algorithm}
\usepackage{bm}
\usepackage[noend]{algpseudocode}
\usepackage{bbm}
\usepackage{verbatim}
\usepackage{mdframed}
\usepackage{needspace}
\makeatletter
\renewcommand{\ALG@beginalgorithmic}{\scriptsize}
\makeatother
\usepackage{xcolor}
\allowdisplaybreaks

\DeclareFontFamily{U}{mathx}{}
\DeclareFontShape{U}{mathx}{m}{n}{ <-> mathx10 }{}
\DeclareSymbolFont{mathx}{U}{mathx}{m}{n}
\DeclareFontSubstitution{U}{mathx}{m}{n}

\renewcommand{\dj}[1]{\bm{d}_{J_1}^{_{#1}}}

\ifhighlight
\else
    
\fi

\newcommand{\rvoutopacity}{20}
\ifshowrvin

\else

\fi

\ifshowrvout
    \newcommand{\rvout}[1]{{\color{red!\rvoutopacity}{#1}} }
    \newcommand{\chrout}[1]{{\color{blue!\rvoutopacity}{#1}} }    
    \newcommand{\rvoutm}[1]{{\color{black!\rvoutopacity}{\ifmmode\text{\sout{\ensuremath{\displaystyle#1}}}\else\sout{#1}\fi}} } % rvout in math mode
\else
    \newcommand{\rvout}[1]{}
    \newcommand{\chrout}[1]{}
\fi

\ifshowcomments
    \newcommand{\summ}[1]{{\color{blue}[summary: #1]} } % summary
    \newcommand{\chr}[1]{{\color{PineGreen}[CR: #1]} } % comments by Chang-Han Rhee
    \newcommand{\xw}[1]{{\color{RoyalBlue}[XW: #1]} } % comments by Xingyu Wang
    \ifshowaddressedcomments
        \newcommand{\chra}[1]{{\color{PineGreen}\sout{[CR: #1]}} } % addressed comments by C.-H.R. 
        \newcommand{\xwa}[1]{{\color{RoyalBlue}\sout{[XW: #1]}} } % addressed comments by X.W.
    \else
        \newcommand{\chra}[1]{} % addressed comments by C.-H.R. 
        \newcommand{\xwa}[1]{} % addressed comments by X.W.
    \fi
\else
    \newcommand{\summ}[1]{} % summary
    \newcommand{\chr}[1]{} % comments by Chang-Han Rhee
    \newcommand{\chra}[1]{} % comments by C.-H.R. that are already addressed
    \newcommand{\xw}[1]{} % comments by Xingyu Wang
    \newcommand{\xwa}[1]{} % comments by X.W. that are already addressed 
\fi

\usepackage[shortlabels]{enumitem}

\newlist{thmdependence}{itemize}{10}
\setlist[thmdependence]{nosep,label=-}

\newcommand{\thmtreenode}[5]{\item[#1] \linkdest{location, thm tree #3} {#2}~\ref{#3} \linktopf{#3} \thmsum{#4}{#5}}
\newcommand{\thmtreenodewopf}[5]{\item[#1] \linkdest{location, thm tree #3} {#2}~\ref{#3} \thmsum{#4}{#5}}
\newcommand{\thmtreeref}[2]{\item[\elsewhere] {{\hyperlink{location, thm tree #2}{\color{gray}#1}}}~\ref{#2}\thmsum{0.5}{}}

\ifshowtheoremlinks
    \newcommand{\linksinthm}[1]{\emph{\linkdest{location, #1}\linktopf{#1} \linktothmtree{location, thm tree #1} }}
    \newcommand{\linksinthmwopf}[1]{\emph{\linkdest{location, #1} \linktothmtree{location, thm tree #1} }}
    \newcommand{\linksinpf}[1]{\linkdest{location, proof of #1}\linktothm{#1} \linktothmtree{location, thm tree #1} }
\else
    \newcommand{\linksinthm}[1]{}
    \newcommand{\linksinthmwopf}[1]{}
    \newcommand{\linksinpf}[1]{}
\fi

\ifshownotationindex
    \newcommand{\notationdef}[2]{\linkdest{location, notation definition of #1}\hyperlink{location, notation index of #1}{#2}}
    \newcommand{\notationidx}[2]{\linkdest{location, notation index of #1}\hyperlink{location, notation definition of #1}{#2}}
\else
    \newcommand{\notationdef}[2]{#2}
\fi
    
\newcommand{\linktopf}[1]{\hyperlink{location, proof of #1}{\pflinksymbol}}
\newcommand{\linktothm}[1]{\hyperlink{location, #1}{\thmlinksymbol}}
\newcommand{\linktothmtree}[1]{\hyperlink{#1}{\thmtreelinksymbol}}
\newcommand{\thmlinksymbol}{{\tiny [Theorem]}}
\newcommand{\pflinksymbol}{{\tiny [Proof]}}
\newcommand{\thmtreelinksymbol}{{\tiny [ThmTree]}}

\newcommand{\complete}{{\color{black}\checkmark}}

\newcommand{\issue}{{\color{red}\checkmark}}
\newcommand{\elsewhere}{}

\newcommand{\thmsum}[2]{\quad{\color{gray}\begin{minipage}[t]{#1\linewidth}{#2}\vspace{0.5\baselineskip}\end{minipage}}}

\makeatletter
\newcommand{\linkdest}[1]{%
    \Hy@raisedlink{\raisebox{5mm}[0pt][0pt]{\hypertarget{#1}{}}}%
}
\makeatother

\newcommand{\elaborateopacity}{50}
\newcommand{\elaboratecolor}{RawSienna}
\ifelaborate
    \newcommand{\elaborate}[1]{{\color{\elaboratecolor!\elaborateopacity}{
    \begin{framed}%[beforebreak=\par\noindent] 
    \noindent {\footnotesize[Elaboration]}
    #1 
    \end{framed}
    }}\noindent}
\else
    \newcommand{\elaborate}[1]{}
\fi
\usepackage[margin=1.2in]{geometry}

\newtheorem{theorem}{Theorem}
\newtheorem{lemma}[theorem]{Lemma}
\newtheorem{corollary}[theorem]{Corollary}
\newtheorem{proposition}[theorem]{Proposition}

\newtheorem{definition}[theorem]{Definition}

\newtheorem{assumption}{Assumption}
\newtheorem{remark}{Remark}
\newtheorem{condition}{Condition}

\newtheorem*{theorem-nonumber}{Theorem}
\newtheorem*{condition-nonumber}{Condition}
\newtheorem*{proposition-nonumber}{Proposition}

\usepackage{mathtools}
\DeclarePairedDelimiter{\ceil}{\lceil}{\rceil}
\DeclarePairedDelimiter\floor{\lfloor}{\rfloor}

\newcommand{\D}{\mathbb D}

\newcommand{\cmt}[1]{#1} % print note in box
\renewcommand{\cmt}[1]{} % omit note
\renewcommand{\P}{\mathbf{P}}

\newcommand{\E}{\mathbf{E}}
\newcommand{\RV}{\mathcal{RV}}
\newcommand{\R}{\mathbb{R}}
\newcommand{\Z}{\mathbb{Z}}
\renewcommand{\S}{\mathbb{S}}
\newcommand{\C}{\mathbb{C}}
\newcommand{\M}{\mathbb{M}}
\newcommand{\I}{\mathbbm{I}}
\renewcommand{\complement}{c}

\newcommand{\lo}{\mathit{o}}
\newcommand{\bo}{\mathcal{O}}

\usepackage[normalem]{ulem}

\usepackage{multirow}
\usepackage{longtable}

\usepackage{array}
\def\delequal{\mathrel{\ensurestackMath{\stackon[1pt]{=}{\scriptscriptstyle\Delta}}}}
\def\distequal{\mathrel{\ensurestackMath{\stackon[1pt]{=}{\scriptstyle d}}}}
\newcommand{\norm}[1]{\left\lVert#1\right\rVert}
\usepackage{mathtools}

\algrenewcommand\algorithmicrequire{\textbf{Require:}}
\algrenewcommand\algorithmicensure{\textbf{Postcondition:}}
\usepackage{subfigure}
\usepackage{float}
\usepackage{subfiles}
\title{First-Exit Time Analysis for Truncated Heavy-Tailed Dynamical Systems}

\DeclareMathAccent{\widecheck}{0}{mathx}{"71}

\author[1]{Xingyu Wang} 
\author[2]{Chang-Han Rhee}
\affil[1]{Quantitative Economics, University of Amsterdam\\
    Amsterdam, 1018 WB, NL}
\affil[2]{Industrial Engineering and Management Sciences, Northwestern University\\
    Evanston, IL, 60613, USA}
\begin{document}
%%%%%%%%%%%%%%%%
\maketitle

\begin{abstract}
\noindent
In this paper, we study the first-exit time of stochastic difference equation $X^\eta_{j+1}(x) = X^\eta_{j}(x) + \eta a\big( X^\eta_{j}(x)\big) + \eta \sigma\big( X^\eta_{j}(x)\big)Z_{j+1}$ and its truncated variant $X^{\eta|b}_{j+1}(x) =  X^{\eta|b}_{j}( x) + \varphi_b\big(\eta  a\big(  X^{\eta|b}_{j}( x)\big) + \eta \sigma\big(  X^{\eta|b}_{j}( x)\big) Z_{j+1}\big)$, where $\varphi_b(x) = (x/\norm{x})\min\{\norm{x}, b\}$
and the law of the noise $Z_t$ is multivariate regularly varying.
The truncation operator $\varphi_b(\cdot)$ is often introduced as a modulation mechanism in heavy-tailed systems, such as stochastic gradient descent algorithms in deep learning.  
By developing a framework that connects large deviations with metastability, we leverage the locally uniform sample-path large deviations for both processes in \cite{wang2024largedeviationsmetastabilityanalysis} to obtain precise characterizations of the joint distributions of the first exit times and exit locations. 
The resulting limit theorem unveils a discrete hierarchy of phase transitions (i.e., exit times) as the truncation threshold $b$ varies, and manifests the catastrophe principle, whereby key events or metastable behaviors in heavy-tailed systems are driven by catastrophic behavior in a few components while the rest of the system behaves nominally. 
These developments lead to a comprehensive heavy-tailed counterpart of the classical Freidlin-Wentzell theory.
\end{abstract}

\counterwithin{equation}{section}
\counterwithin{lemma}{section}
\counterwithin{corollary}{section}
\counterwithin{theorem}{section}
\counterwithin{definition}{section}
\counterwithin{proposition}{section}
\counterwithin{figure}{section}
\counterwithin{table}{section}

\tableofcontents

\section{Introduction}
Metastability analysis in stochastic dynamical systems has a rich history in probability theory and related fields. 
Since the foundational works of Kramers and Eyring \cite{eyring1935chemical, kramers1940brownian, glasstone1941theory}, which analyzed phase transitions in stochastic dynamical systems in the context of chemical reaction-rate theory, extensive theoretical advancements have been made. 
One of the most notable breakthroughs is the now-classical Freidlin-Wentzell theory \cite{freidlin1970onsmall, freidlin1973some, freidlin1984random}, which introduced large deviations machinery to the analysis of exit times and global behaviors of small random perturbations of dynamical systems. 
Further extensions of this approach in the context of statistical physics were pioneered in \cite{cassandro1984metastable} and described in detail in \cite{Olivieri_Vares_2005}. 
One of the key advantages of this approach---often called \emph{the pathwise approach}---is its ability to describe in detail the scenarios that lead to phase transitions. 
In particular, the large deviations formalism at the sample-path level enables precise identification of the most likely paths out of the metastability sets. 
This ensures that, asymptotically, whenever the dynamical system escapes from the metastability set, the escape routes almost always closely resemble these most likely paths.
However, the sample-path-level large deviations are typically available only in the form of logarithmic asymptotics, and hence, the asymptotic scale of the exit time can be determined only up to its exponential rate, requiring different approaches to identify the prefactor.
Another breakthrough is \emph{the potential-theoretic approach} initiated in \cite{bovier2001metastability, bovier2004metastability, bovier2005metastability} and later summarized in \cite{bovier2016metastability}. 
Instead of relying on large deviations machinery, this approach leverages potential-theoretic tools:
the scale of exit times for Markov processes can be expressed in terms of capacity, which, in turn, can be bounded using variational principles. 
The key advantage of this approach, compared to the pathwise approach, is that it is often possible to find test functions that tightly bound the capacity of the Markov chains, thereby yielding \emph{precise} asymptotics---rather than merely logarithmic asymptotics as in the pathwise approach---of the scales of exit times. 
Although the potential theoretic approach does not provide as much information (such as the most likely paths) as the pathwise approach beyond the asymptotics of the exit times, its sharpness has inspired extensive research activity. 
The early works in the potential theoretic approach were focused on reversible Markov processes. 
However, recent developments have extended the scope of the approach to enable the analysis of non-reversible Markov processes; see, for example, \cite{slowik2012note, landim2014metastability, gaudilliere2014dirichlet, lee2022non}.

While these developments provide powerful means to understand metastability of light-tailed systems, 
heavy-tailed systems exhibit a fundamentally different metastable behaviors and call for a different set of technical tools for successful analysis.
For example, early foundational works in heavy-tailed context \cite{imkeller2006first, imkeller2008levy, pavlyukevich2008metastable, imkeller2010first} 
proved that the exit times of the stochastic processes driven by heavy-tailed noises scale polynomially with respect to the scaling parameter.  
These papers also reveal that the exit events are almost always driven by a single disproportionately large jump, while the rest of the system's behavior remains close to its nominal behavior. 
By nominal behavior, we refer to the functional law of large numbers limit.
This is in sharp contrast to the light-tailed cases, where the exit times scale exponentially, and the exit events are driven by smooth tilting of the entire system from its nominal behavior.
One can view this as a manifestation of \emph{the principle of a single big jump}, a well-known folklore in extreme value theory. 
% For stochastic processes with independent increments over a finite time horizon, \cite{hult2005functional} systematically characterized the principle of a single big jump with an early formulation of heavy-tailed sample-path large deviations.

However, across machine learning, finance, operations research, and other disciplines,
many impactful key events or metastable behaviors
cannot be driven by a single big jump (e.g.\ \cite{Albrecher_Chen_Vatamidou_Zwart_2020,tankov2003financial,foss2006heavy,doi:10.1287/moor.1120.0539,wang2022eliminating}).
For example, in the context of deep learning, 
stochastic gradient descent (SGD) and its variants are the powerhouse behind the 
training of deep neural networks (DNNs),
and SGD's ability to find flat local minima are believed to be associated with the generalization performance of DNNs on test data (see, e.g., \cite{keskar2017on}).
Nevertheless, the commonly used gradient clipping technique (see, e.g., \cite{Engstrom2020Implementation, pascanu2013difficulty})
modifies the step size of SGD to prevent drastic changes of model weights in a single iteration,
which prevents the escape from a wide, flat local minima using a single big jump in SGD.
In such cases, we observe
the more general \emph{catastrophe principle}, where key events or metastable behaviors are driven by multiple big jumps under heavy tails.
% Stochastic gradient descent (SGD) and its variants are the methods of choice in training deep neural networks (DNNs). 
% Heavy-tailed SGDs have attracted significant attention in the recent past because of their ability to escape local minima with a single big jump, enabling them to explore non-convex energy landscapes within realistic training time horizons. 
% Such ability is widely believed to have fundamental connection to DNNs' remarkable generalization performance on test data.
% However, the pure form of SGD is rarely employed in practice.
% In particular, when the gradient noise appears to exhibit heavy-tailed behavior causing SGD to occasionally attempt to travel a long distance in a single step, the step size is truncated at a threshold. 
% This is a common practice known as gradient clipping; see, e.g., \cite{Engstrom2020Implementation, merity2018regularizing, graves2013generating, pascanu2013difficulty,zhang2020why}.
% With gradient clipping, the exit event from a large attraction field cannot be solely driven by a single big jump. 
% In general---as we rigorously confirm in this paper---when a single big jump is insufficient to cause the rare event of interest, it is driven by the minimal number of big jumps required to trigger it, while the rest of the system remains close to its nominal dynamics. 
% This portrayal provides a more complete picture than the principle of a single big jump and is referred to as \emph{the catastrophe principle}. 
% We will use the term catastrophe principle throughout this paper. 

A rigorous mathematical characterization of the catastrophe principle for L\'evy processes and random walks was established in the form of heavy-tailed sample-path large deviations \cite{rhee2019sample}, leveraging the $\M$-convergence theory originally introduced in \cite{lindskog2014regularly}. 
More recently, \cite{wang2024largedeviationsmetastabilityanalysis} considered a uniform extension of the $\M$-convergence theory and characterized the catastrophe principle for a much broader class of processes, including SGD-type stochastic difference equations under multivariate regularly varying dynamics.
The results in \cite{rhee2019sample, wang2024largedeviationsmetastabilityanalysis} can be viewed as the heavy-tailed counterpart of the Mogulskii's theorem \cite{lynch1987large, mogulskii1993large}
and Freidlin-Wentzell theory \cite{freidlin1984random},
and such a formulation of large deviations
provides precise asymptotics for heavy-tailed processes, in contrast to the logarithmic asymptotics of the classical large deviation principle.
This raises the hope that, for heavy-tailed dynamical systems, it may be possible to simultaneously obtain \emph{both} detailed descriptions of the scenarios leading to phase transitions (as in the pathwise approach \cite{freidlin1984random, Olivieri_Vares_2005}) and sharp asymptotics for the exit time (as in the potential-theoretic approach \cite{bovier2016metastability}).
% Successfully implementing this strategy for practical systems requires establishing strong enough sample-path large deviations and developing tools to translate these results into exit-time analyses tailored for heavy-tailed dynamical systems with transition dynamics potentially modulated by truncation. 
In this paper, we successfully implemented this strategy and obtained metastability analyses for 
heavy-tailed dynamical systems with transition dynamics potentially modulated by truncation. 
Specifically, the main contributions of this paper are as follows:
\begin{itemize}

    \item 
        We establish a scaling limit of the exit-time and exit-location for stochastic difference equations. 
        We accomplish this by developing a machinery for local stability analysis of general (heavy-tailed) Markov processes.
        Central to the development is the concept of asymptotic atoms, where the process recurrently enters and asymptotically regenerates.
        Leveraging the recent developments of locally uniform version of sample-path large deviations in \cite{wang2024largedeviationsmetastabilityanalysis} over these asymptotic atoms, 
        we derive sharp asymptotics for the joint distribution of (scaled) exit-times and exit-locations for heavy-tailed processes, as detailed in Theorem~\ref{theorem: first exit time, unclipped} and Corollary~\ref{corollary: first exit time, untruncated case}.
        Notably, the scaling rate parameter reflects an intricate interplay between the truncation threshold and the geometry of the drift, which is a feature absent in both the principle of a single big jump regime (heavy-tailed systems without truncation) and the conspiracy principle regime (light-tailed systems).

\end{itemize}

Below, we give a more detailed account of related literature and overview of the main results.
Let $(\bm Z_i)_{i \geq 1}$ be a sequence of independent copies of a random vector $\bm Z$ in $\R^d$ such that $\E \bm Z = \bm 0$.
% That is, there exists some slowly varying function $\phi$ such that $\P(\norm{\bm Z_1} > x) = \phi(x)x^{-\alpha}$.
For any $\eta > 0$ and $\bm x \in \R^m$, let $\big(\bm X^\eta_t(\bm x)\big)_{t \geq 0}$ be the solution to
% \begin{align}
%     \bm X^\eta_{j+1}(\bm x) & = \bm X^\eta_{j}(\bm x) + \eta \bm a\big( \bm X^\eta_{j}(\bm x)\big) + \eta \bm \sigma\big( \bm X^\eta_{j}(\bm x)\big)\bm Z_{j+1}\quad \forall j \geq 0
%     \label{intro, def for X eta j}
% \end{align}
\begin{align}
    \bm X^\eta_0(\bm x) = \bm x;\qquad
    \bm X^\eta_t(\bm x) = \bm X^\eta_{t - 1}(\bm x) +  \notationdef{notation-eta}{\eta} \bm a\big(\bm X^\eta_{t - 1}(\bm x)\big) + \eta\bm \sigma\big(\bm X^\eta_{t - 1}(\bm x)\big)\bm Z_t,\quad \forall t \geq 1.
     \label{def: X eta b j x, unclipped SGD}
\end{align}
Here, the vector-valued function 
$\notationdef{a}{\bm a}: \mathbb{R}^m \to \mathbb{R}^m$
and the matrix-valued function 
 $\notationdef{sigma}{\bm \sigma}:\mathbb{R}^m\to \mathbb{R}^{m\times d}$
are the drift coefficient and diffusion coefficients, and $\eta$ denotes the step length parameter.
The equivalent scalar form is
% Recall that we interpret all vectors as column vectors.
\begin{align}
    X^\eta_{t,i}(\bm x)
    =
    X^\eta_{t - 1,i}(\bm x)
    +
    \eta a_i\big(\bm X^\eta_{t - 1}(\bm x)\big)
    +
    \eta \sum_{j \in [d]} \sigma_{i,j}\big(\bm X^\eta_{t - 1}(\bm x)\big)Z_{t,j},
    \quad 
    \forall t \geq 1,\ i \in [m],    
    \nonumber
\end{align}
with
$
\bm a(\cdot) = \big(a_1(\cdot),\ldots,a_m(\cdot)\big)^\top,
$
$
\bm\sigma(\cdot) = \big(\sigma_{i,j}(\ldots)\big)_{i \in [m], j \in [d]},
$
and
$
\bm X^\eta_t(\bm x) = \big( X^\eta_{t,1}(\bm x),\ldots,X^\eta_{t,m}(\bm x) \big)^\top,
$
$
\bm Z_t = (Z_{t,1},\ldots,Z_{t,d})^\top.
$
% Dynamics of the form \eqref{def: X eta b j x, unclipped SGD}, as well as their extensions, 
% have been central to the study of applied probability and operations research, 
% as they commonly arise in key applications such as stochastic approximation (e.g., \cite{kushner2003stochastic}), 
% stochastic optimization (e.g., \cite{lan2020first}), 
% and Markov Chain Monte Carlo (e.g., \cite{10.1111/rssb.12183}).
Of particular interest is a truncated variant of \eqref{def: X eta b j x, unclipped SGD} under the operator
\begin{align}
    \notationdef{notation-truncation-operator-level-b}{\varphi_b}(\bm w) 
    \delequal{} 
     (b \wedge \norm{\bm w}) \cdot \frac{\bm w}{\norm{\bm w}}
    \ \ \ \forall \bm w \neq \bm 0,
    \qquad
    {\varphi_b}(\bm 0) \delequal \bm 0.
    \label{defTruncationClippingOperator}
\end{align}
Here, $u\wedge v = \min\{u,v\}$ and $u \vee v = \max\{u,v\}$, 
and throughout this paper we interpret $\norm{\cdot}$ over $\R^d$ as the $L_2$ norm.
That is, $\varphi_b$ is the projection operator onto the closed $L_2$ ball centered at the origin with radius $b$.
Given $\eta,b \in (0,\infty)$ and $\bm x \in \R^m$, let
\begin{align}
    \bm X^{\eta|b}_0(\bm x) = \bm x;\qquad
    {\bm X^{\eta|b}_t(\bm x)} = \bm X^{\eta|b}_{t - 1}(\bm x) +  \varphi_b\Big(\eta \bm a\big(\bm X^{\eta|b}_{t - 1}(\bm x)\big) + \eta \bm \sigma\big(\bm X^{\eta|b}_{t - 1}(\bm x)\big)\bm Z_t\Big),\quad \forall t \geq 1.
    \label{def: X eta b j x, clipped SGD}
\end{align}
% under the initial condition $$.
That is, 
$\bm X^{\eta|b}_t(\bm x)$ is a variation of $\bm X^{\eta}_t(\bm x)$ where the update at each step is bounded by $b$.
In the context of modern machine learning, 
 \eqref{def: X eta b j x, unclipped SGD} and \eqref{def: X eta b j x, clipped SGD}
can be viewed as models for the dynamics of SGD and the variant under gradient clipping (\cite{Engstrom2020Implementation, pascanu2013difficulty,zhang2020why,pmlr-v202-koloskova23a}), respectively.
In particular, if $\bm a$ is the negative gradient of the training loss $U$, then the argument of $\varphi_b(\cdot)$ in \eqref{def: X eta b j x, clipped SGD}, 
$\eta \bm a\big(\bm X_t^\eta(\bm x)\big) + \eta \bm \sigma \big(\bm X_t^\eta(\bm x)\big) \bm Z_{t+1} = -\eta \big(\nabla U\big(\bm X_t^\eta(\bm x)\big) - \bm \sigma \big(\bm X_t^\eta(\bm x)\big) \bm Z_{t+1} \big)$, represents the state-dependent stochastic gradient of $U$ at $\bm X_t^\eta(\bm x)$, scaled by the negative learning rate $-\eta$, which corresponds to the one-step displacement of SGD.

% The first exit time problem finds applications in numerous contexts, including chemical reactions \cite{kramers1940brownian}, physics \cite{chechkin2005barrier,chechkin2007barrier}, extreme climate events \cite{penland2008modelling}, mathematical finance \cite{scalas2000fractional}, and queueing systems \cite{shwartz1995large}.
As noted earlier, two modern approaches to the analysis of exit times for light-tailed stochastic dynamical systems are
% include the original Eyring-Kramers formula \cite{eyring1935chemical, kramers1940brownian, glasstone1941theory}, 
the Freidlin-Wentzell theory (or pathwise approach) detailed in the monographs \cite{freidlin1984random, Olivieri_Vares_2005} and the potential theoretic approach summarized in the monograph \cite{bovier2016metastability}.
% A classical result in this literature is the Eyring-Kramers law \cite{eyring1935chemical,kramers1940brownian} (see also \cite{freidlin1984random}), which characterizes the exit times of Brownian particles.
% See also \cite{Olivieri_Vares_2005} for the connections between large deviations, metastability, and statistical physics models.
Despite their success in the light-tailed contexts, neither the pathwise approach nor the potential theoretic approach readily extends to heavy-tailed contexts. 
% due to fundamentally different characteristics of the heavy-tailed dynamical systems. 
In particular, for truncated heavy-tailed dynamics such as $\bm X_{t}^{\eta|b}(\bm x)$, the explicit formula for the stationary distribution is rarely available, and its generator lacks the simplicity of the Brownian case, making the adaptation of potential theoretic approach to our context challenging.
Meanwhile, the pathwise approach hinges on the large deviation principles at the sample-path level.
Historically, however, the heavy-tailed large deviations at the sample-path level have been unavailable and considered to be out of reach until recently.
See, for instance, 
\cite{hult2005functional} for a functional level characterization of the principle of a single big jump,
and 
\cite{rhee2019sample,bernhard2020heavy,wang2024largedeviationsmetastabilityanalysis} for sample-path-level characterizations of the catastrophe principle. 
See also for an alternative formulation of the catastrophe principle in terms of the extended LDP in
\cite{borovkov2010large} providing log asymptotics.

Successful results in metastability analyses for heavy-tailed systems are relatively recent. 
For one-dimensional L\'evy driven SDEs, \cite{imkeller2006first, pavlyukevich2008metastable} proved that the exit times from metastability sets scale at a polynomial rate and the prefactor of the of the scale depend on the width of the potential wells rather than the height of the potential barrier.
Similar results have been established in more general settings,
such as the multi-dimensional analog in \cite{imkeller2010first},
% These results were extended to the multi-dimensional settings in \cite{imkeller2010first},
exit times for a global attractor instead of a stable point \cite{hogele2014exit}
(see also \cite{doi:10.1142/S0219493715500197} for its application in characterizing the limiting Markov chain of hyperbolic dynamical systems driven by heavy-tailed perturbations),
exit times under multiplicative noises in $\R^d$ \cite{pavlyukevich2011first},
extensions to infinite-dimensional spaces \cite{debussche2013dynamics},
and the (discretized) stochastic difference equations driven by $\alpha$-stable noises \cite{NEURIPS2019_a97da629}, to name a few.
% \cite{imkeller2009exponential} paints interesting picture of the hierarchy in the asymptotics of the first exit times.
Such metastability analyses were applied in \cite{simsekli2019tail} to study the generaliztion performance of DNNs trained by SGDs with heavy-tailed dynamics and, more recently, in \cite{JMLR:v25:21-1343} to analyze the sample efficiency of policy gradient algorithms in reinforcement learning.
% See also \cite{imkeller2009exponential,imkeller2010hierarchy} for the summary of the hierarchy in the asymptotics of the first exit times for heavy-tailed dynamics.
We note that these results focus on events associated with the principle of a single big jump.

In contrast, this paper develops a systematic tool for analyzing the exit times and locations, even in cases where the principle of a single big jump fails to account for the exit events, and more complex patterns arise during the exit process.
The process $\bm X_t^{\eta|b}(\bm x)$ exemplifies such a scenario, as the truncation operator $\varphi_b(\cdot)$ may prevent exits driven by a single big jump. 
We reveal phase transitions in the first exit times of $\bm X^{\eta|b}_t(\bm x)$, which depend on a notion of the ``discretized widths'' of the attraction fields.
% even when they are driven by multiple big jump events as in the case of $\bm X^{\eta|b}_j(\bm x)$. 
% Indeed, we characterize the asymptotics of the joint law of the first exit time and the exit location for heavy-tailed processes.
% In essence, under truncation threshold $b > 0$, it requires a minimum of $J^*_b = \ceil{ |s_\text{left}|\vee s_\text{right}/b }$ jumps for the truncated dynamics $X^{\eta|b}_j(x)$
% to exit from $I = (s_\text{left},s_\text{right})$ when initialized at the origin.
Specifically,
consider \eqref{def: X eta b j x, clipped SGD}
% and \eqref{intro, def for Y eta t} 
with drift coefficients $\bm a(\cdot) = - \nabla U(\cdot)$ 
for some potential function $U \in \mathcal{C}^1(\R^m)$.
Without loss of generality, let $I \subseteq \R^m$ be some open and bounded set containing the origin.
Suppose that the entire domain $I$ falls within the attraction field of the origin,
and the gradient field $-\nabla U(\cdot)$ is locally contractive around the origin
(see Section~\ref{subsec: first exit time, results, SGD} for details).
% in the following sense:
% for the ODE path $d\bm y_t(\bm x)/dt = -\nabla U(\bm y_t(\bm x))$ with initial condition $\bm y_0(\bm x) =\bm  x$,
% it holds that $\lim_{t \to \infty}\bm y_t(\bm x) = \bm 0$ for all $\bm x \in I$.
In other words,
when initialized within $I$, the deterministic gradient flow
$d\bm y_t(\bm x)/dt = -\nabla U(\bm y_t(\bm x))$ (under the initial condition $\bm y_0(\bm x) =\bm  x$)
will be attracted to and remain trapped near the origin.
However, due to the presence of random perturbations, 
% although 
$\bm X^{\eta|b}_t(\bm x)$ 
% tend to be attracted to the origin, they 
will eventually escape from $I$ after a sufficiently long time.
Of particular interest are the asymptotic scale of the first exit times as $\eta \downarrow 0$.
Theorem \ref{theorem: first exit time, unclipped} proves that
the joint law of the first exit time $\tau^{\eta|b}(\bm x) = \min\{t \geq 0:\ \bm X^{\eta|b}_t(\bm x) \notin I \}$
and the exit location $\bm X^{\eta|b}_{\tau}(\bm x) \triangleq \bm X^{\eta|b}_{\tau^{\eta|b}(\bm x)}(\bm x)$
admits the limit (uniformly for all $\bm x$ bounded away from $I^\complement$):
\begin{align}
    \Big(\lambda^I_b(\eta) \cdot  \tau^{\eta|b}(\bm x),\ \bm X^{\eta|b}_{\tau}(\bm x)\Big)
    \Rightarrow 
    (E,V_b)\qquad\text{ as }\eta \downarrow 0
    \label{result, intro, first exit time}
\end{align}
with some (deterministic) time-scaling function $\lambda^I_b(\eta)$.
Here, $E$ is an exponential random variable with the rate parameter 1, 
$V_b$ is some random element independent of $E$ and supported on $I^\complement$,
and the scaling function 
$\lambda_b^I(\eta)$ is regularly varying with index $-[1 + \mathcal J^I_b(\alpha-1)]$ as $\eta \downarrow 0$,
where $\mathcal J^I_b$ is the aforementioned discretized width of domain $I$ relative to the truncation threshold $b$.
The precise definition of $\mathcal J_b^I$ is provided in \eqref{def: first exit time, J *} in Section~\ref{subsec: first exit time, results, SGD}. 
However, we note that in the special case $b= \infty$, one can immediately verify that $\mathcal J_b^I = 1$, regardless of the geometry of $U$. 
Consequently, \eqref{result, intro, first exit time} reduces to the principle of a single big jump, as expected. 
When the drift is contractive so that $\nabla U(\bm x) \cdot \bm x \geq 0$ for all $\bm x \in I$, it is also straightforward to see that $\mathcal J_b^I = \lceil r / b \rceil$ where $r$ is the distance between $\bm 0$ and $I^c$, and hence, $\mathcal J_b^I$ is indeed precisely the discretized width of the attraction field $I$ relative to $b$. 
In particular, note that the drift is contractive within any attraction field in the one-dimensional cases.  
However, in general multi-dimensional spaces, $\mathcal J_b^I$ reflects a much more intricate interplay between the geometry of the drift $\bm a(\cdot)$ (or the potential $U(\cdot)$) and the truncation threshold $b$.  
\begin{figure}[t]
\vskip 0.2in
\begin{center}
\begin{tabular}{cc}
\footnotesize{(i)} & \footnotesize{(ii)} 
% & \footnotesize{(iii)} & \footnotesize{(iv)}
\\
\includegraphics[width=0.35\textwidth]{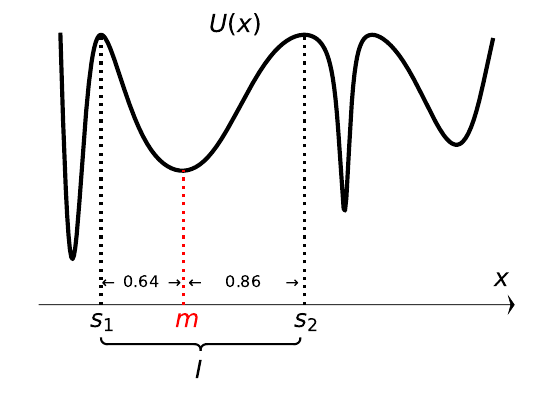} \qquad\qquad   & \qquad\qquad
\includegraphics[width=0.35\textwidth]{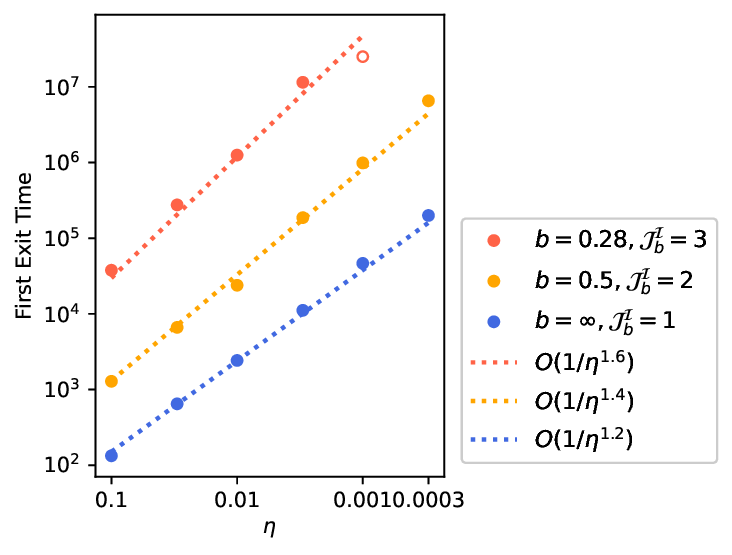} 
\end{tabular}
\caption{ 
Numerical examples of the metastability analysis.
\textbf{(i)} The univariate potential $U(\cdot)$ defined in \eqref{aeq: potential U, first exit time}.
\textbf{(ii)} First exit times $\tau^{\eta|b}(m)$ from $I$ under different truncation thresholds $b$ and scale parameters $\eta$.
Dashed lines are predictions from our results in Section~\ref{sec: first exit time simple version}, whereas
the dots are the exit times estimated using 20 samples.
Non-solid dot represents an underestimation due to early termination after long runtime ($5 \times 10^7$ steps). 
% It can be observed that the predictions and estimates align well.
% \textbf{(iii)} Histograms of locations within the potential potential $U(\cdot)$ visited by $X^{\eta|b}_t(x)$.
% Note that in (b), the sharp minima are almost completely eliminated from the trajectory of the SGD.
% \textbf{(iv)} Sample paths of $X^{\eta|b}_t(x)$.
% Dashed lines in (iii) and (iv) are added as references for the locations of local minima.
% Driven by truncated heavy-tailed noise, $X^{\eta|b}_t(x)$ almost completely avoids the sharp minima of $U(\cdot)$ in (b).
}
\label{fig: exit time}
\end{center}
\vskip -0.2in
\end{figure}
% All the associated definitions and the formal statement of \eqref{result, intro, first exit time} can be found in Section~\ref{subsec: first exit time, results, SGD} and Theorem~\ref{theorem: first exit time, unclipped}.
Figure~\ref{fig: exit time} illustrates the key role of the relative width $\mathcal J_b^{I}$ in one dimension.

Figure~\ref{fig: exit time} illustrates the catastrophe principle in the first exit behaviors of $\bm X^{\eta|b}_t(\bm x)$ by considering a univariate example.
Specifically, consider a potential function $U:\R \to \R$ depicted in 
Figure~\ref{fig: exit time} (i), where $I = (s_1,s_2)$ is the attraction field of the local minimum $m$.
Since $m$ is closer to the left boundary $s_1$, the minimal number of steps required to exit $I$ when starting from $m$ is $\mathcal J^{I}_b = \ceil{ |s - m_1|/b }$ where $b \in (0,\infty)$.  
In the untruncated case \eqref{def: X eta b j x, unclipped SGD} (i.e., with $b = \infty$), 
we simply have $\mathcal J^\mathcal{I}_\infty = 1$.
Figure~\ref{fig: exit time} (ii) illustrates the discrete structure of phase transitions in \eqref{result, intro, first exit time}, 
where
the first exit time $\tau^{\eta|b}(\bm x)$ is (roughly) of order
$
1/\eta^{ 1 + \mathcal J^I_b \cdot (\alpha - 1)  }
$
for small $\eta$, with $\alpha = 1.2$ being the index of $\bm Z_i$'s regular variation.
This means that the order of the first exit time $\tau^{\eta|b}(\bm x)$ does not vary continuously with respect to the truncation threshold $b$.
Instead, it exhibits a discrete dependence on $b$ through the integer-valued quantity $\mathcal J^I_b$.
Consequently, the wider the domain $I$, the \emph{asymptotically longer} the exit time $\tau^{\eta|b}(\bm x)$ will be.
See Section~\ref{subsec: numerical examples} for details of the numerical experiments presented in Figure~\ref{fig: exit time}.
% In the companion paper \cite{wangSGDpaper2}, we build on these phase transitions in exit times to reveal an intriguing global behavior of $\bm X_j^{\eta|b}$ over a multi-well potential: the distribution of $\bm X_j^{\eta|b}$'s sample path closely resembles a Markov chain that \textbf{completely avoids narrow local minima}; see Figure~\ref{fig: exit time} (iii) and (iv).
% More importantly, we demonstrate in \cite{wangSGDpaper2} that
% such global dynamics under truncated heavy tails are intimately related to the generalization performance of deep neural networks.

% generated under some probability measure $C_b(\cdot)$ supported on $I^\complement$, and $E$ and $V_b$ are independent with each other.
% $C^I_b$ is a normalization constant,
% and
% the scale function $\lambda(\eta) = \eta^{-1} \cdot \P(|Z_1| > \eta^{-1})$ is roughly decaying at a polynomial rate $\eta^{\alpha - 1}$ for small $\eta$
% with $\alpha > 1$ being the heavy-tailed index for noises $Z_i$.
% See Section \ref{subsec: first exit time, results, SGD} for precise definitions.
% Notably, Theorem \ref{theorem: first exit time, unclipped} presents a stronger result where the asymptotics in \eqref{result, intro, first exit time} hold uniformly for initial values $x$ over any compact set within $I$.
% Theorem~\ref{theorem: first exit time, unclipped} also obtains the first exit time analysis for $\bm X^\eta_j(\bm x)$ by considering an arbitrarily large truncation threshold $b \approx \infty$.

Our approach to the metastability analysis hinges on the concept of asymptotic atoms,
a general machinery we develop in Section \ref{subsec: framework, first exit time analysis}.
Asymptotic atoms are nested regions of recurrence at which the process asymptotically regenerates upon each visit.
Our locally uniform sample-path large deviations then prove to be the right tool in this framework,
empowering us to characterize the behavior of the stochastic processes uniformly for all initial values over the asymptotic atoms.
% We also point out that our results make weaker assumptions than the previous results, allowing for non-constant diffusion coefficient $\bm \sigma(\cdot)$
% and eliminating the need for regularity conditions such as $U \in \mathcal C^3(\R^m)$ and non-degeneracy for the Hessian of $U(\cdot)$ at the boundary of $I$.
It should be noted that
\cite{imkeller2009exponential} also investigated the exit events driven by multiple truncated jumps. 
However, in their context, the mechanism through which multiple jumps arise is due to a different tail behavior of the increment distribution that is lighter than any polynomial rate---more precisely, a Weibull tail---and it is fundamentally different from that of the regularly varying case.
(See also \cite{imkeller2010hierarchy} for the summary of the hierarchy in the asymptotics of the first exit times for heavy-tailed dynamics.)
% case where the L\'evy measure $\nu$ decays exponentially fast with speed $\nu\big((-\infty,-u]\cup[u,\infty)\big) \approx \exp(-u^\alpha)$.
% Along with the aforementioned results \cite{imkeller2006first,pavlyukevich2008metastable,imkeller2010first} for regularly varying SDEs, \cite{imkeller2009exponential} paints interesting picture of the hierarchy in the asymptotics of the first exit times. 
% See \cite{imkeller2010hierarchy} for the summary of such hierarchy.
Our results complement the picture and provide a missing piece of the puzzle by unveiling the phase transitions in the exit times under truncated regularly varying perturbations.
% In particular, we characterize a discrete structure of phase transitions in \eqref{result, intro, first exit time}, where
% we find that the first exit time $\tau^{\eta|b}(\bm x)$ is (roughly) of order
% $
% 1/\eta^{ 1 + \mathcal J^I_b \cdot (\alpha - 1)  }
% $
% for small $\eta$.
% This means that the order of the first exit time $\tau^{\eta|b}(\bm x)$ does not vary continuously with $b$;
% rather, it exhibits a discrete dependence on $b$
% through $\mathcal J^I_b$.
%, i.e., the minimum number of jumps  required for the exit.
% This phase transition phenomenon further exemplifies the catastrophe principle under regularly varying noises, as the ``cost'' function $J^*_b$ dictates not only the most likely cause (i.e., through $J^*_b$ large noises) but also the rarity of the exit (i.e., occurring roughly once every $1/\eta^{ 1 + J^*_b \cdot (\alpha - 1)  }$ steps).

A culmination of metastability analysis is the sharp characterization of the global dynamics of heavy-tailed processes, often established in the form of process-level convergence of scaled processes to simpler ones, such as Markov jump processes on a discrete state space; see, for example, \cite{Olivieri_Vares_2005, bovier2016metastability, imkeller2008levy, pavlyukevich2008metastable, lee2022non-b, rezakhanlou2023scaling}. 
For the SDE counterpart of the unregulated processes such as \eqref{def: X eta b j x, unclipped SGD}, it is well known 
that the exit events are governed by the principle of a single big jump,
and
the scaling limit is a Markov jump process with a state space consisting of the local minima of the potential function; see e.g., \cite{doi:10.1142/S0219493715500197, imkeller2008levy, pavlyukevich2008metastable}. 
In a companion paper \cite{wangSGDpaper2}, 
we demonstrate that the framework developed in this paper is strong enough to extend the above-mentioned results to the systems \emph{not} governed by the principle of a single big jump over a multi-well potential, by identifying scaling limits and characterizing their global behavior at the process level. 
In particular, the scaling limit for the truncated heavy-tailed dynamics is a Markov jump process that \emph{only visits the widest minima of the potential}. 
This is in sharp contrast to the untruncated cases \cite{doi:10.1142/S0219493715500197, imkeller2008levy, pavlyukevich2008metastable} where the limiting Markov jump process visits all the local minima with certain fractions.
% These findings systematically characterize a curious phenomena that the truncated heavy-tailed processes avoid narrow local minima altogether in the limit.
As a result, the fraction of time such processes spend in narrow attraction fields converges to zero as the scaling parameter (often called learning rate in the machine learning literature) tends to zero.
Precise characterization of such phenomena is of fundamental importance in understanding and further leveraging the curious effectiveness of the stochastic gradient descent (SGD) algorithms in training deep neural networks.
\\

Some of the the metastability analysis in Section~\ref{sec: first exit time simple version} of this paper have been presented in a preliminary univariate form at a conference \cite{wang2022eliminating}.
The main focus of \cite{wang2022eliminating} was the connection between the metastability analysis of stochastic gradient descent (SGD) and its generalization performance in the context of training deep neural networks. 
Compared to the brute force approach in \cite{wang2022eliminating},
the current paper provides a systematic framework to characterize the global dynamics for significantly more general class of heavy-tailed dynamical systems.
We also note that first-exit analysis of the form \eqref{result, intro, first exit time}
can be extended to stochastic difference equations akin to \eqref{def: X eta b j x, unclipped SGD} and \eqref{def: X eta b j x, clipped SGD} but under different scaling w.r.t.\ $\eta$,
or L\'evy-driven stochastic differential equations.
We state those results in Sections~\ref{sec: appendix, SGD, general scaling, results} and \ref{sec: appendix, SDE reults} in the Appendix.
\\
% We also note that (i) by sending the truncation threshold $b$ to $\infty$ in \eqref{intro, result, metastability}, we recover the global dynamics of $X^\eta_j(x)$ in Theorem \ref{theorem: metastability, unclipped}; 
% (ii) metastability analysis can be conducted analogously for stochastic differential equation $Y^\eta_t(x)$ and its corresponding truncated dynamics; see Section~\ref{subsec: metastability, SDE}.

The rest of the paper is organized as follows. 
Section~\ref{sec: preliminaries} reviews the uniform sample path large deviations developed in \cite{wang2024largedeviationsmetastabilityanalysis}.
Section~\ref{sec: main results} carries out the first-exit analysis and presents the main results of this paper,
with numerical examples collected in Section~\ref{subsec: numerical examples}.
% Section \ref{subsec: implications in ML, metastability} discusses the implication of the results in machine learning.
Sections~\ref{subsec: Exit time analysis framework}, \ref{sec: proof of proposition: first exit time}, and \ref{sec: proof for untruncated case} provides the proofs.
% Section~\ref{sec: LD of SGD, proof} and Section~\ref{sec: first exit time simple version, proof}
% % , and Section~\ref{sec: metastability, proof} 
% provide the proofs of 
% Sections~\ref{sec: M convergence, asymptotic equivalence}, \ref{subsec: LD, SGD}, and \ref{sec: first exit time simple version}.
Results for stochastic difference equations under more general scaling regimes are presented in Appendix~\ref{sec: appendix, SGD, general scaling, results}.
Results for SDEs driven by L\'evy processes with regularly varying increments are collected in Appendix~\ref{sec: appendix, SDE reults}.

\section{Preliminaries}
\label{sec: preliminaries}

In Section~\ref{subsec: notations}, we collect frequently used notations.
In Section~\ref{subsec: LD, SGD}, we review the sample-path large deviations developed in \cite{wang2024largedeviationsmetastabilityanalysis} for heavy-tailed dynamical systems.

\subsection{Notations}
\label{subsec: notations}

Let $\notationdef{set-for-integers-below-n}{[n]} \delequal \{1,2,\cdots,n\}$ for any positive integer $n$.
Let $\notationdef{notation-non-negative-numbers}{\mathbb{N}} = \{0,1,2,\cdots\}$ be the set of non-negative integers.
% and 
% $\notationdef{notation-non-negative-numbers-and-zero}{\mathbb{Z}_+} = \{0,1,2,\cdots\}$
% be the set of non-negative integers.
Let $(\mathbb{S},\bm{d})$ be a metric space with 
$\notationdef{notation-borel-sigma-algebra}{\mathscr{S}_\mathbb{S}}$
being the corresponding Borel $\sigma$-algebra.
For any $E\subseteq \mathbb{S}$,
let
$\notationdef{notation-interior-of-set-E}{E^\circ}$ and $\notationdef{notation-closure-of-set-E}{E^-}$ be the interior and closure of $E$, respectively.
For any $r > 0$, 
let
$\notationdef{notation-epsilon-enlargement-of-set-E}{E^r} \delequal 
\{ y \in \mathbb{S}:\ \bm{d}(E,y)\leq r\}$ be the $r$-enlargement of a set $E$.
Here for any set $A \subseteq \mathbb{S}$ and any $x \in \mathbb{S}$,
we define $\bm{d}(A,x) \delequal \inf\{\bm{d}(y,x):\ y \in A\}$.
Also, let
$
\notationdef{notation-epsilon-shrinkage-of-set-E}{E_{r}} \delequal
% \{x \in E: \bm{d}(x,y) < r \Longrightarrow y \in E\}
((E^c)^r)^\complement
$
be the $r$-shrinkage
of $E$.
Note that for any $E$, the enlargement $E^r$ of $E$ is closed, and the shrinkage $E_r$ of $E$ is open.
We say that set $A \subseteq \mathbb{S}$ is bounded away from another set $B \subseteq \mathbb{S}$
if $\inf_{ x\in A,y\in B }\bm{d}(x,y) > 0$.
For any Borel measure $\mu$ on $(\mathbb{S},\mathscr{S}_{\mathbb{S}})$,
let the support of $\mu$
(denoted as $\notationdef{notation-support-of-mu}{\text{supp}(\mu)}$)
% $\delequal \bigcap_{E:\ E\text{ closed},\ \mu(E) = 1}E $ 
be the smallest closed set $C$ such that $\mu(\S \setminus C)= 0$.
% be the smallest closed set $E$ such that $\mu(E) = 1$.
For any function $g: \mathbb{S} \to \R$, let
$
\notationdef{notation-support-of-function-g}{\text{supp}(g)} \delequal \big(\{ x \in \mathbb{S}:\ g(x) \neq 0 \}\big)^-.
$
Given two sequences of positive real numbers $(x_n)_{n \geq 1}$ and $(y_n)_{n \geq 1}$, 
we say that $x_n = \bo(y_n)$ (as $n \to \infty$) if there exists some $C \in [0,\infty)$ such that $x_n \leq C y_n\ \forall n\geq 1$.
Besides, we say that $x_n = \lo(y_n)$ if $\lim_{n \rightarrow \infty} x_n/y_n = 0$.

\subsection{Heavy-Tailed Large Deviations}
\label{subsec: LD, SGD}

We review the sample path large deviations for the truncated dynamics $\bm X^{\eta|b}_t(\bm x)$ in \eqref{def: X eta b j x, clipped SGD}
under i.i.d.\ $\bm Z_t$'s whose heavy tails are characterized via multivariate regular variation.
A measurable function $\phi:(0,\infty) \to (0,\infty)$ is regularly varying as $x \rightarrow\infty$ with index $\beta$ (denoted as $\phi(x) \in \RV_\beta(x)$ as $x \to \infty$) if $\lim_{x \rightarrow \infty}\phi(tx)/\phi(x) = t^\beta$ for all $t>0$. 
See also 
\cite{bingham1989regular, resnick2007heavy, foss2011introduction}
for a more detailed treatment to regularly varying functions.
We also say that a measurable function $\phi(\eta)$
is regularly varying as $\eta \downarrow 0$ with index $\beta$ 
if $\lim_{\eta \downarrow 0} \phi(t\eta)/\phi(\eta) = t^\beta$ for any $t > 0$,
denoted by
 $\phi(\eta) \in \notationdef{notation-RV-LDP}{\RV_{\beta}}(\eta)$ as $\eta \downarrow 0$.
Let
\begin{align}
    \notationdef{notation-H}{H(x)} \delequal \P(\norm{\bm Z} > x). \label{def: H, law of Z_j}
\end{align}
Given $\alpha > 0$, let $\notationdef{notation-measure-nu-alpha}{\nu_\alpha}$ be the Borel measure on $(0,\infty)$ with
\begin{align}
     \nu_\alpha[x,\infty) = x^{-\alpha}. \label{def: measure nu alpha}
\end{align}
Let $\notationdef{notation-R-d-unit-sphere}{\mathfrak N_d} \delequal \{\bm x \in \R^d:\ \norm{\bm x} = 1\}$ be the unit sphere of $\R^d$.
Let $\Phi: \R^d \setminus \{\bm 0\} \to [0,\infty) \times \mathfrak N_d$ be the polar transform, i.e.,
\begin{align}
    \notationdef{notation-Phi-polar-transform}{\Phi(\bm x)} \delequal 
    \Big(\norm{\bm x},\frac{\bm x}{\norm{\bm x}}\Big).
    % \begin{cases}
    %       &\text{ if }\bm x \neq 0,
    %      \\
    %      \big( 0, (1,0,0,\cdots,0)\big) & \text{ otherwise.}
    % \end{cases}
    \label{def: Phi, polar transform in Rm}
\end{align}
% Note that the origin is included in the domain of $\Phi$ simply to ease notations in the proofs, and $\Phi(\bm x)$ will not be applied at $\bm x = \bm 0$ in the analyses.
% Thus, $\Phi$ can be interpreted as the polar transform with domain extended to $\bm 0$.
Throughout this paper, we impose the following multivariate regular variation (MRV) assumption on $\bm Z$.
\begin{assumption}[Regularly Varying Noise]\label{assumption gradient noise heavy-tailed}
$\E \bm Z = \bm 0$. 
Besides, there exist some $\notationdef{alpha-noise-tail-index-LDP}{\alpha} > 1$ and 
a probability measure $\mathbf S(\cdot)$ on the unit sphere $\mathfrak N_d$ such that
\begin{itemize}
    \item 
        $H(x) \in \RV_{-\alpha}(x)$ as $x \to \infty$,
    \item 
        for the polar coordinates $(R,\bm \Theta) \delequal \Psi(\bm Z)$ (when $\bm Z \neq \bm 0$), we have (as $x \to \infty$)
        \begin{align}
            \frac{
                \P\Big( (x^{-1}R, \bm \Theta) \in\ \cdot\ \Big)
            }{
                H(x)
            }
            \xrightarrow{v}
            \nu_\alpha \times \mathbf S,
            % \qquad 
            % \text{in $\mathbb M\Big( 
            % \big([0,\infty) \times \mathfrak N_d \big)
            % \setminus
            % \big( \{0\} \times \mathfrak N_d \big)
            % \Big)$},
            \label{claim, Rd heavy tailed assumption}
        \end{align}
        where
        $\xrightarrow{v}$ denotes vague convergence on $(0,\infty] \times \mathfrak N_d$.
        % \textcolor{red}{state vague convergence of $M$-convergence?}

    % \item 
    %     the measure $\mathbf S(dx) = f_{\mathbf S}(x)dx$ admits a density over $\mathfrak N_d$, with $\inf_{ x \in \mathfrak N_d }f_{\mathbf S}(x) > 0$.
\end{itemize}
\end{assumption}
As noted in \cite{wang2024largedeviationsmetastabilityanalysis},
the classical vague convergence approach to the characterization of MRV (see, e.g.,  \cite{Resnick_2004, hult2006regular})
is equivalent to the $\M$-convergence approach adopted in \cite{wang2024largedeviationsmetastabilityanalysis}.

Throughout, we fix $m$ and $d$ to denote the dimensionality of $\bm X^{\eta|b}_t(\bm x)$ and $\bm Z_t$, respectively.
We consider the $L_2$ vector norm induced matrix norm
$
\norm{\textbf A} = \sup_{ \bm x \in \R^q:\ \norm{\bm x} = 1 }\norm{\textbf A\bm x}
$
for any $\textbf A \in \R^{p \times q}$,
and impose the following conditions on
the drift coefficient $
\bm a(\cdot) = \big(a_1(\cdot),\cdots,a_m(\cdot)\big)^\top
$ and
the
diffusion coefficient $
\bm\sigma(\cdot) = \big(\sigma_{i,j}(\cdot)\big)_{i \in [m], j \in [d]}
$
in \eqref{def: X eta b j x, clipped SGD}.

\begin{assumption}[Lipschitz Continuity]
\label{assumption: lipschitz continuity of drift and diffusion coefficients}
There exists $\notationdef{notation-Lipschitz-constant-L-LDP}{D} \in (0,\infty)$ such that
$$\norm{\bm \sigma(\bm x) - \bm \sigma(\bm y)} \vee \norm{\bm a(\bm x)-\bm a(\bm y)} \leq D\norm{\bm x - \bm y},\qquad \forall \bm x,\ \bm y \in \mathbb{R}^m.$$
\end{assumption}

We introduce a few notations required for presenting the precise asymptotics in the sample path large deviations results.
Let $(\notationdef{notation-D-0T-cadlag-space}{\mathbb{D}{[0,T]}},\notationdef{notation-D-J1}{\dj{[0,T]}})$
be the metric space where $\mathbb{D}[0,{T}] = \D\big([0,T],\R^m\big)$ is the space of all càdlàg functions with domain $[0,{T}]$ and codomain $\R^m$,
and $\dj{[0,T]}$ is the Skorodkhod $J_1$ metric
\begin{align}
    \dj{[0,T]}(x,y) \delequal 
\inf_{\lambda \in \Lambda_T} \sup_{t \in [0,T]}|\lambda(t) - t|
\vee \norm{ x(\lambda(t)) - y(t) }.
\label{def: J 1 metric on D 0 T}
\end{align}
Here,
$
\Lambda_T
$
is the set of all homeomorphism on $[0,T]$.
Throughout this paper, we fix some $m$ and $d$ and consider $\bm X^{\eta|b}_t(\bm x)$ taking values in $\R^m$ driven by $\bm Z_t$'s in $\R^d$.
% so we omit the subscript $m$ in $\D[0,T]$ and $\dj{[0,T]}$.
Given $A \subseteq \R$, 
let
$
\notationdef{order-k-time-on-[0,t]}{A^{k \uparrow}} \delequal 
\{
(t_1,\cdots,t_k) \in A^k:\ t_1 < t_2 < \cdots < t_k
\}
$
be the set of sequences of increasing real numbers on $A$ with length $k$.
For any $b$, $T \in (0,\infty)$ and $k \in \mathbb{N}$,
% any $\R^{m \times p}$-matrix-valued mapping $\textbf{V}(\cdot) = \big(V_{i,j}(\cdot)\big)_{i \in [m], j \in [p]}: \R^m \to \R^{m \times p}$ for some $p \geq 1$,
define the mapping 
$\notationdef{notation-h-k-t-bar-mapping-truncation-level-b-LDP}{\bar h^{(k)|b}_{[0,T]}}: \mathbb{R}^m\times \mathbb{R}^{d \times k} \times \R^{m \times k} \times (0,{T}]^{k\uparrow} \to \mathbb{D}{[0,T]}$ as follows.
Given
$\bm x \in \mathbb{R}^m$,
$\textbf{W} = (\bm w_1,\cdots,\bm w_k) \in \mathbb{R}^{d \times k}$, 
$\textbf V = (\bm v_1,\cdots,\bm v_k) \in \R^{m \times k}$,
and $\bm{t} = (t_1,\cdots,t_k)\in (0,T]^{k\uparrow}$, let $\bar h^{(k)|b}_{[0,T]}(\bm x,\textbf{W},\textbf V,\bm{t})$ be the solution $\xi$ to
\begin{align}
    \xi_0 & = \bm x; \label{def: perturb ode mapping h k b, 1}
    \\
    \frac{d\xi_s}{d s} & = \bm a(\xi_s)\ \ \ \forall s \in [0,{T}],\ s \neq t_1,t_2,\cdots,t_k; \label{def: perturb ode mapping h k b, 2}
    \\
    \xi_s & = \xi_{s-} + \bm v_j + \varphi_b\big( \bm \sigma(\xi_{s-} + \bm v_j)\bm w_j\big)\ \ \ \text{ if }s = t_j\text{ for some }j\in[k] \label{def: perturb ode mapping h k b, 3}
\end{align}
% That is, $\bar h^{(k)|b}_{[0,T]}(\bm x, \textbf W,\textbf V,\bm{t})$ returns a perturbed ODE path where the impact of the jumps $\bm w_j$ are modulated by $\bm \sigma(\cdot)$ and then truncated under $b$.
Similarly, define the mapping
$
\notationdef{notation-h-k-b-t-mapping-LDP}{h^{(k)|b}_{[0,T]}}:\R^m \times \R^{d \times k} \times (0,T]^{k \uparrow} \to \D[0,T]
$
by
\begin{align}
    h^{(k)|b}_{[0,T]}(\bm x, \bm W, \bm t)
    \delequal 
    \bar h^{(k)|b}_{[0,T]}\big(\bm x,\bm W, (\bm 0,\cdots,\bm 0), \bm t\big).
    \label{def: perturb ode mapping h k b, 4}
\end{align}
In other words,  $h^{(k)|b}_{[0,T]}(\bm x,\textbf{W},\bm{t})$ produces an ODE path perturbed by jumps $\bm w_1,\cdots,\bm w_k$ (with their sizes modulated by $\bm \sigma(\cdot)$ and then truncated under threshold $b$) at times $t_1,\cdots,t_k$,
and the mapping $\bar h^{(k)|b}_{[0,T]}$ further includes perturbations $\bm v_j$'s right before each jump.
For $k = 0$,
we adopt the convention that $\xi = \bar h^{(0)|b}_{[0,T]}(\bm x)$ is the solution to the ODE
${d\xi_s}/{d s} = \bm a( \xi_s)\ \forall s \in [0,T]$
with  $\xi_0 = \bm x$.
For each $r > 0$ and $\bm x \in \R^m$, let $\notationdef{notation-ball-r-x}{\bar B_r(\bm x)} \delequal \{ \bm{y}\in\R^m:\ \norm{\bm y - \bm x} \leq r \}$
be the closed ball with radius $r$ centered at $\bm x$.
Given $b,T \in (0,\infty)$, $\epsilon \geq 0$, $A \subseteq \R^m$ and $k \in \mathbb N$,
let
\begin{align}
    \notationdef{notation-D-A-k-t-truncation-b-LDP}{\mathbb{D}_{A}^{(k)|b}{[0,T]}(\epsilon) } \delequal \bar h^{(k)|b}_{[0,T]}\Big( A,\ \mathbb{R}^{m \times k},\ \big(\bar B_\epsilon(\bm 0)\big)^k,\ (0,T]^{k\uparrow} \Big)
    \label{def: l * tilde jump number for function g, clipped SGD}
\end{align}
% \chr{
% $$ 
% h^{(k)|b}_{[0,T]}\Big( A ,\, \mathbb{R}^{m \times k} ,\, \big(\bar B_\epsilon(\bm 0)\big)^k ,\, (0,T]^{k\uparrow} \Big)
% $$
% }%
be the set that contains all the ODE path with $k$ jumps by time $T$, 
i.e., the image of the mapping $\bar h^{(k)|b}_{[0,T]}$ defined in \eqref{def: perturb ode mapping h k b, 1}--\eqref{def: perturb ode mapping h k b, 3}, 
under small perturbations $\norm{\bm v_j} \leq \epsilon$ for all $j \in [k]$.
% By our definition of $\bar h^{(0)|b}_{[0,T]}$ above, $\D^{(0)|b}_A[0,T](\epsilon)$ simply contains all ODE paths associated with vector field $\bm a(\cdot)$ and initial values over $A$.
For $k = -1$, we adopt the convention that
$\mathbb{D}_{A}^{(-1)|b}[0,T](\epsilon) \delequal \emptyset$.
% Also, note the monotonicity property that 
% $
% \mathbb{D}^{(k)|b}_A{[0,T]}(\epsilon) \subseteq \mathbb{D}^{(k)|b}_A{[0,T]}(\epsilon^\prime)
% $
% for any $0 \leq \epsilon < \epsilon^\prime$
% and $k \geq -1$.
For any $t > 0$, let 
$\notationdef{notation-lebesgue-measure-restricted}{\mathcal{L}_t}$ be the Lebesgue measure restricted on $(0,t)$ and $\notationdef{notation-lebesgue-measure-on-ordered-[0,t]}{\mathcal{L}^{k\uparrow}_t}$ be the Lebesgue measure restricted on $(0,t)^{k \uparrow}$.
Given $\bm x \in \mathbb{R}^m$, $k \in \mathbb N$ and $b,T \in (0,\infty)$,
define the Borel measure
\begin{align}
 \notationdef{notation-measure-C-k-t-truncation-b-LDP}{{\mathbf{C}}^{(k)|b}_{{[0,T]}}(\ \cdot \ ;\bm x)} \delequal &
   \int \mathbbm{I}\Big\{ h^{(k)|b}_{{[0,T]}}\big( \bm x,\textbf W,\bm t  \big) \in\ \cdot\  \Big\} 
   \big((\nu_\alpha \times \mathbf S)\circ \Phi\big)^k(d \textbf W) \times\mathcal{L}^{k\uparrow}_{ {T} }(d\bm t),
   \label{def: measure mu k b t}
\end{align}
where
$\mathbf S$ is the probability measure on the unit sphere $\mathfrak N_d$ characterized in Assumption~\ref{assumption gradient noise heavy-tailed},
$\nu_\alpha$ is specified in \eqref{def: measure nu alpha},
$(\nu_\alpha \times \mathbf S) \circ \Phi$
is the composition of the product measure $\nu_\alpha \times \mathbf S$ with the polar transform $\Phi$, 
i.e.,
\begin{align}
    \big((\nu_\alpha \times \mathbf S)\circ \Phi\big)(B) \delequal 
    (\nu_\alpha \times \mathbf S)\big( \Phi(B) \big),
    \qquad \forall \text{Borel set }B \subseteq \R^d\setminus\{\bm 0\},
    \label{def, nu alpha times S composition polar transform}
\end{align}
and $\big((\nu_\alpha \times \mathbf S)\circ \Phi\big)^k$ is the $k$-fold product of $(\nu_\alpha \times \mathbf S)\circ \Phi$.
% In other words, 
% $(\nu_\alpha \times \mathbf S)^k$ is the $k$-fold measure of $\nu_\alpha \times \mathbf S$.
% In other words, for $\textbf W = (\bm w_1,\cdots,\bm w_k) \in \R^{d \times k}$ with $\bm w_j \neq 0\ \forall j \in [k]$, we have
% $
% \big((\nu_\alpha \times \mathbf S)\circ \Phi\big)^k(d\textbf W) 
%     = 
% \bigtimes_{j \in [k]}\big((\nu_\alpha \times \mathbf S)\circ \Phi\big)(d \bm w_j).
% $
% Note that for any $\bm x \in A$, the measure ${{\mathbf{C}}^{(k)|b}_{{[0,T]}}(\ \cdot \ ;\bm x)}$ is supported on $\D^{(k)|b}_A[0,T](0)$.
Next, define the rate function
\begin{align}
    \notationdef{notation-lambda-scale-function}{\lambda(\eta)} \delequal \eta^{-1}H(\eta^{-1})
    \label{def lambda scale function}
\end{align}
with $H(x) = \P(\norm{\bm Z} > x)$ in \eqref{def: H, law of Z_j}.
By Assumption~\ref{assumption gradient noise heavy-tailed}, $\lambda(\eta) \in \RV_{\alpha - 1}(\eta)$ as $\eta \downarrow 0$.
We write $\lambda^k(\eta) = \big(\lambda(\eta)\big)^k$.
For any $T,\ \eta,\ b \in (0,\infty)$, and $\bm x \in \R^m$,
let
$$
 \notationdef{notation-scaled-X-eta-mu-truncation-b-LDP}{{\bm{X}}^{\eta|b}_{[0,T]}(\bm x)} \delequal \{ \bm X^{\eta|b}_{ \floor{ t/\eta } }(\bm x):\ t \in [0,T] \}
$$
be the time-scaled version of the sample path of $\bm X_t^{\eta|b}(\bm x)$ embedded in $\D[0,T]$,
with $\notationdef{floor-operator}{\floor{t}}\delequal \max\{n \in \mathbb{Z}:\ n \leq t\}$
and
$\notationdef{ceil-operator}{\ceil{t}} \delequal \min\{n \in \Z:\ n \geq t\}$.
The next result is established in \cite{wang2024largedeviationsmetastabilityanalysis},
and our extension from domain $[0,1]$ to $[0,T]$ follows directly from the change of variable $\eta^\prime = \eta/T$.

% Now, we are ready to state Theorem~\ref{corollary: LDP 2},
% which establishes the uniform $\mathbb M$-convergence for the law of $\bm X^{\eta|b}_{[0,T]}(\bm x)$ to $\mathbf C^{(k)|b}_{[0,T]}(\ \cdot\ ;\bm x)$ and a uniform version of the sample path large deviations for $\bm X^{\eta|b}_{[0,T]}(\bm x)$.

\begin{theorem}[Theorem 2.5 of \cite{wang2024largedeviationsmetastabilityanalysis}]
\label{corollary: LDP 2}
\linksinthm{corollary: LDP 2}
Let Assumptions \ref{assumption gradient noise heavy-tailed} and \ref{assumption: lipschitz continuity of drift and diffusion coefficients} hold.
Let $k \in \mathbb N$, $b,T,\epsilon \in (0,\infty)$, and  $A\subset \R^m$ be compact.
% $$
%  \lambda^{-k}(\eta) \P\big( \bm{X}^{\eta|b}_{[0,T]}(\bm x) \in\ \cdot\ \big) 
%     \rightarrow 
%     \mathbf{C}^{(k)|b}_{[0,T]}   (\ \cdot\ ; \bm x)
%     % \text{ in } \mathbb{M}\big( \mathbb{D}[0,T]\setminus \mathbb{D}^{(k-1)|b}_{A}[0,T](\epsilon) \big),
%     % \text{ uniformly in $\bm x$ on $A$} 
% $$
% in
% $
% \mathbb{M}\big( \mathbb{D}[0,T]\setminus \mathbb{D}^{(k-1)|b}_{A}[0,T](\epsilon) \big)
% $
% uniformly in {$\bm x$ on $A$}
% as $\eta \downarrow 0$.
Given
$B \in \mathscr{S}_{\mathbb{D}[0,T]}$
that is 
bounded away from ${ \mathbb{D}}_{A}^{(k - 1)|b}[0,T](\epsilon)$ for some (and hence all) $\epsilon > 0$ small enough,
% Let Assumptions \ref{assumption gradient noise heavy-tailed}, \ref{assumption: lipschitz continuity of drift and diffusion coefficients}, and \ref{assumption: nondegeneracy of diffusion coefficients} hold.
% For any $b,T > 0$,
% any compact $A\subseteq \mathbb{R}$,
% any
% $B \in \mathscr{S}_{\mathbb{D}[0,T]}$ and $k \geq 1$,
% if
% $B$ is bounded away from ${ \mathbb{D}}_{A}^{(k - 1)|b}[0,T]$, then
\begin{equation}\label{claim, uniform sample path LD, corollary: LDP 2}
    \begin{split}
        \inf_{\bm x \in A}
    \mathbf{C}^{(k)|b}_{[0,T]}\big( B^\circ; \bm x \big)
& \leq  \liminf_{\eta \downarrow 0}\frac{ \inf_{\bm x \in A}\P\big({\bm{X}}^{\eta|b}_{[0,T]}(\bm x) \in B \big) }{  \lambda^k(\eta)  } 
\\
   & \leq  \limsup_{\eta \downarrow 0}\frac{ \sup_{\bm x \in A}\P\big({\bm{X}}^{\eta|b}_{[0,T]}(\bm x) \in B \big) }{  \lambda^k(\eta)  } 
    \leq 
    \sup_{\bm x \in A}
    \mathbf{C}^{(k)|b}_{[0,T]}\big( B^-; \bm x \big)
    <
    \infty.
    \end{split}
\end{equation}
\end{theorem}

\section{Main Results}
\label{sec: main results}
This section presents the main results of this paper.
Section~\ref{sec: first exit time simple version} carries out the metastability analysis.
Section~\ref{subsec: numerical examples} presents details for numerical example in Figure~\ref{fig: exit time}.

\subsection{Metastability Analysis}
\label{sec: first exit time simple version}
This section analyzes the metastability of $\bm X_t^\eta(\bm x)$ and $\bm X_t^{\eta|b}(\bm x)$
defined in \eqref{def: X eta b j x, unclipped SGD} and \eqref{def: X eta b j x, clipped SGD}.
Section~\ref{subsec: first exit time, results, SGD} establishes the scaling limits of their exit times.
Section~\ref{subsec: framework, first exit time analysis} introduces a framework that facilitates such analysis for general Markov chains.
The results for stochastic differential equations and/or under more general scaling regimes are collected in the Appendix.

\subsubsection{First Exit Times and Locations}
\label{subsec: first exit time, results, SGD}
In this section,
we analyze the first exit times and locations of $\bm X^\eta_t(\bm x)$ and $\bm X^{\eta|b}_t(\bm x)$
% , defined in \eqref{def: X eta b j x, unclipped SGD} and \eqref{def: X eta b j x, clipped SGD},
% where
% the heavy-tailed noises $(Z_j)_{j \geq 1}$ satisfy Assumption \ref{assumption gradient noise heavy-tailed}
% and $a(\cdot)$ and $\sigma(\cdot)$ satisfy Assumptions \ref{assumption: lipschitz continuity of drift and diffusion coefficients} and \ref{assumption: nondegeneracy of diffusion coefficients}.
% We study the first exit time of $X^\eta_j(x)$ and $X^{\eta|b}_j(x)$
from an attraction field of some potential with a unique local minimum at the origin.
Specifically, throughout Section \ref{subsec: first exit time, results, SGD}, we fix an open set $\notationdef{notation-exit-domain-I}{I} \subset \R^m$ 
that is bounded and contains the origin,
i.e., $\sup_{\bm x \in I}\norm{\bm x} < \infty$
and
$\bm 0 \in I$.
Let $\notationdef{notation-continuous-gradient-descent}{\bm{y}_t(\bm x)}$ be the solution of ODE
\begin{align}
\bm y_0(\bm x) = \bm x,\qquad
   \frac{d\bm{y}_t(\bm x)}{dt} = \bm a\big(\bm{y}_t(\bm x)\big) \ \ \forall t \geq 0.
   \label{def ODE path y t}
\end{align}
We impose the following assumption on the gradient field $\bm a:\ \R^m \to \R^m$.

\begin{assumption}
\label{assumption: shape of f, first exit analysis}
$\bm a(\bm 0) = \bm 0$.
The open set $I \subset \R^m$ contains the origin and is bounded, i.e.,
$\sup_{\bm x \in I}\norm{\bm x} < \infty$
and
$\bm 0 \in I$.
For all $\bm x \in I \setminus \{\bm 0\}$,
\begin{align*}
    \bm y_t(\bm x) \in I\ \ \forall t \geq 0,
    \qquad
    \lim_{t \to \infty}\bm y_t(\bm x) = \bm 0.
\end{align*}
Besides, it holds for all $\epsilon > 0$ small enough that
$
\bm a(\bm x)\bm x < 0\ \forall \bm x \in \bar B_\epsilon(\bm 0) \setminus \{\bm 0\}.
$

\end{assumption}

An immediate consequence of the condition $\lim_{t \to \infty}\bm y_t(\bm x) = \bm 0\ \forall \bm x \in I\setminus\{\bm 0\}$ is that $\bm a(\bm x) \neq \bm 0$ for all $\bm x \in I\setminus \{\bm 0\}$.
Of particular interest is the case where $\bm a(\cdot) = -\nabla U(\cdot)$
for some potential $U\in \mathcal{C}^1(\R^m)$
that has a unique local minimum at $\bm x = 0$ over the domain ${I}$.
In particular, Assumption~\ref{assumption: shape of f, first exit analysis} holds if $U$ is also locally $\mathcal C^2$ around the origin, and the Hessian of $U(\cdot)$ at the origin $\bm x = \bm 0$ is positive definite.
We note that Assumption~\ref{assumption: shape of f, first exit analysis} is a standard one in existing literature;
see e.g.\ \cite{doi:10.1142/S0219493711003413,imkeller2010first}.
% Assumption \ref{assumption: shape of f, first exit analysis}
% then implies that $U$ has a unique local minimum at $\bm x = 0$ over the domain ${I}$.
% Besides, the potential $U$ has a non-degenerate second-order derivative $U^{\prime\prime}(0) = - a^\prime(0) > 0$.
% It is worth noticing that Assumption \ref{assumption: shape of f, first exit analysis} is more flexible than the assumptions in other related works in the literature.
% For instance,
% \cite{pavlyukevich2008metastable,imkeller2006first} required the second-order derivative $U^{\prime\prime}(\cdot)$ to be non-degenerate at the origin as well as the boundary points of $I$, with an extra condition of $U \in \mathcal C^3$ over a wide enough compact set, and held the drift coefficient $\sigma(\cdot)$ as constant.
% In contrast, Assumption~\ref{assumption: shape of f, first exit analysis} is close to minimum assumption required for $I$ to be an attraction field associated with the origin. 
% we conduct a first exit time analysis with significantly relaxed assumptions.

Define 
\begin{align}
    \notationdef{notation-tau-eta-x-first-exit-time}{\tau^\eta(\bm x)} \delequal \min\big\{j \geq 0:\ \bm X^\eta_j(\bm x) \notin I\big\},
    \qquad
    \notationdef{notation-tau-eta-b-x-first-exit-time}{\tau^{\eta|b}(\bm x)} \delequal \min\big\{j \geq 0:\ \bm X^{\eta|b}_j(\bm x) \notin I \big\},
    \label{def: first exit time for heavy tailed SGD}
\end{align}
as the first exit time of $\bm X^\eta_j(\bm x)$ and $\bm X^{\eta|b}_j(\bm x)$ from $I$, respectively.
To facilitate the presentation of the main results, we introduce a few concepts.
Define the mapping 
$
\bar g^{(k)|b}: \R^m \times \R^{d \times k} \times \R^{m \times k} \times (0,\infty)^{k\uparrow} \to \R^m
$
as the location of the (perturbed) ODE with $k$ jumps at the last jump time:
\begin{align}
    \notationdef{notation-mapping-bar-g-k-b}{\bar g^{(k)|b}\big( \bm x, \textbf W, \textbf V, (t_1,\cdots,t_k)\big)}
    \delequal 
    \bar h^{(k)|b}_{ [0,t_k + 1] }
    \Big(
        \bm x,
        \textbf W,
        \textbf V,
        (t_1,\cdots,t_k)
    \Big)(t_k),
    \label{def: bar g k b mapping, metastability}
\end{align}
where $\bar h^{(k)|b}_{[0,T]}$ is the perturbed ODE mapping defined in \eqref{def: perturb ode mapping h k b, 1}--\eqref{def: perturb ode mapping h k b, 3}.
Note that the definition remains the same if, in \eqref{def: bar g k b mapping, metastability}, we use mapping $\bar h^{(k)|b}_{[0,T]}$
with any $T \in [t_k,\infty)$ instead of $\bar h^{(k)|b}_{[0,t_k + 1]}$.
We include a $+1$ offset only to extend the time range of the mapping and simplify some arguments in our proofs.
Besides,
define 
$
\widecheck{g}^{(k)|b}: \R^m \times \R^{d\times k} \times (0,\infty)^{k\uparrow} \to \R^m
$
by
\begin{align}
    \notationdef{notation-check-g-k-b}{\widecheck{g}^{(k)|b}(\bm x,\textbf W,\bm t)}
    \delequal 
    \bar g^{(k)|b}\big(\bm x, \textbf W, (\bm 0,\cdots,\bm 0), \bm t\big)
    =
    h^{(k)|b}_{[0,t_k+1]}(\bm x,\textbf W,\bm t)(t_{k}),
    \label{def: mapping check g k b, endpoint of path after the last jump, first exit analysis}
\end{align}
where $\bm t = (t_1,\ldots,t_k) \in (0,\infty)^{k\uparrow}$, and the mapping 
$h^{(k)|b}_{[0,T]}$ is defined in \eqref{def: perturb ode mapping h k b, 4}.
\elaborate{%%%%%%%%%%%%%%%%%%%%%%%%%%%%%%%%%%%%%%%
That is,
for any $T > t_k$ and $\xi =  h^{(k)|b}_{[0,T]}(x,w_1,\cdots,w_k,t_1,\cdots,t_k)$,
we set $\widecheck{g}^{(k)|b}(x,w_1,\cdots,w_k,t_1,\cdots,t_k) = \xi(t_k)$, i.e., the value of path $\xi$ right after the last jump.
}%%%%%%%%%%%%%%% end of \elaborate %%%%%%%%%%%%%%%
For $k = 0$,
we adopt the convention that $\bar{g}^{(0)|b}(\bm x) = \bm x$.
With mappings $\bar g^{(k)|b}$ defined, we are able to introduce 
(for any $k \geq 1$, $b> 0$, and $\epsilon \geq 0$)
\begin{align}
    \notationdef{notation-set-G-k-b-epsilon}{\mathcal G^{(k)|b}(\epsilon)} 
    & \delequal 
    \bigg\{
    \bar g^{(k - 1)|b}
    \Big( \bm v_1 + \varphi_b\big(\bm \sigma(\bm v_1)\bm w_1\big),
    (\bm w_2,\cdots, \bm w_k), (\bm v_2,\cdots,\bm v_k), \bm t 
    \Big):
    \nonumber
    \\ 
    & \qquad\qquad
    \textbf W = (\bm w_1,\cdots, \bm w_k) \in \R^{d\times k},
    \textbf V = (\bm v_1,\cdots, \bm v_k) \in \Big(\bar B_\epsilon(\bm 0)\Big)^k,
    \bm t \in (0,\infty)^{k - 1 \uparrow}
    \bigg\}
    \label{def: set G k b epsilon}
\end{align}
as the set covered by the $k^\text{th}$ jump of along ODE path initialized at the origin, with each jump modulated by $\bm \sigma(\cdot)$ and truncated under $b$ (and an $\epsilon$ perturbation right before each jump).
Here, the truncation operator $\varphi_b$ is defined in \eqref{defTruncationClippingOperator}, and $\bar B_{r}(\bm 0)$ is the closed ball with radius $r$ centered at the origin.
Under $\epsilon = 0$, we write
\begin{align*}
    \notationdef{notation-set-G-k-b}{\mathcal G^{(k)|b}}
    \delequal 
    \mathcal G^{(k)|b}(0)
    = 
    \bigg\{
    \widecheck{g}^{(k - 1)|b}
    \Big( \varphi_b\big(\bm \sigma(\bm 0)\bm w_1\big),
    (\bm w_2,\cdots, \bm w_k), \bm t 
    \Big):\ 
    \textbf W = (\bm w_1,\cdots,\bm w_k) \in \R^{d \times k},
    \bm t \in (0,\infty)^{k - 1 \uparrow}
    \bigg\}.
\end{align*}
Furthermore, as a convention for the case with $k = 0$, we set
\begin{align}
    \mathcal{G}^{(0)|b}(\epsilon) \delequal \bar B_{\epsilon}(\bm 0).
    \label{def: 0 jump coverage set, first exit times}
\end{align}
We note that $\mathcal G^{(k)|b}(\epsilon)$ is monotone in $\epsilon$, $k$, and $b$, in the sense that
$
\mathcal{G}^{(k)|b}(\epsilon) \subseteq \mathcal{G}^{(k)|b}(\epsilon^\prime) 
$
for all $0 \leq \epsilon \leq \epsilon^\prime$,
$
\mathcal{G}^{(k)|b}(\epsilon) \subseteq \mathcal{G}^{(k+1)|b}(\epsilon), 
$
and
$
\mathcal{G}^{(k)|b}(\epsilon) \subseteq \mathcal{G}^{(k)|b^\prime}(\epsilon) 
$
for all $0 < b \leq b^\prime$.

The intuition behind our metastability analysis (in particular, Theorem~\ref{theorem: first exit time, unclipped})
is as follows.
The characterization of the $k$-jump-coverage sets of form $\mathcal G^{(k)|b}$ reveals that, due to the truncation of  $\varphi_b(\cdot)$, the space reachable by ODE paths would expand as more jumps are added to the ODE path.
This leads to an intriguing phase transition for the law of the first exit times $\tau^{\eta|b}(\bm x)$ (as $\eta \downarrow 0$) in terms of the minimum number of jumps required for exit.
% i.e., the smallest $k$ such that $\mathcal G^{(k)|b} \cap I^c \neq \emptyset$.
More precisely, let
\begin{align}
    % \notationdef{notation-r-radius-of-exit-domain}{l} & \delequal \inf_{x \in I^\complement}|x| = |s_\text{left}| \wedge s_\text{right},\qquad 
    \notationdef{notation-J-*-first-exit-analysis}{\mathcal J^I_b} \delequal
    \min\big\{ k \geq 1:\ \mathcal G^{(k)|b} \cap I^\complement \neq \emptyset \big\}
    \label{def: first exit time, J *}
\end{align}
be the smallest $k$ such that, under truncation at level $b$, the $k$-jump-coverage sets can reach outside the attraction field $I$.
Theorem~\ref{theorem: first exit time, unclipped} reveals a discrete hierarchy that 
% the asymptotics of $\tau^{\eta|b}(\bm x)$ does not vary with the truncation level $b$ in a continuous fashion;
% instead, 
the order of the first exit time $\tau^{\eta|b}(\bm x)$ and the limiting law of the exit location $\bm X^{\eta|b}_{ \tau^{\eta|b}(\bm x) }(\bm x)$ are dictated by this ``discretized width'' metric $\mathcal J^I_b$ of the domain $I$, relative to the truncation threshold $b$.
Here, the limiting law is characterized by measures
% This allows us to define Borel measures (for each $k \geq 1$ and $b > 0$)
\begin{align}
    % \widecheck{ \mathbf C }^{(k)}(\cdot )
    % & \delequal 
    % \int \mathbbm{I}\Big\{ g^{(k-1)}\big( \sigma(0)\cdot w_k, w_1,\cdots,w_{k-1},\bm t \big) \in \ \cdot \  \Big\}
    % \nu^k_\alpha(d w_1,\cdots,dw_k) \times \mathcal{L}^{k-1\uparrow}_\infty(d\bm t),
    % \\
     \notationdef{notation-check-C-k-b}{\widecheck{ \mathbf C }^{(k)|b}(\ \cdot\ )}
    & \delequal 
    \int \mathbbm{I}\bigg\{ \widecheck{g}^{(k-1)|b}\Big( \varphi_b\big(\bm\sigma(\bm x)\bm w_1\big),(\bm w_2,\cdots,\bm w_k),\bm t \Big) \in \ \cdot \  \bigg\}
     \big((\nu_\alpha \times \mathbf S)\circ \Phi\big)^k(d \textbf W) \times \mathcal{L}^{k-1\uparrow}_\infty(d\bm t),
    \label{def: measure check C k b}
\end{align}
where 
$\alpha > 1$ is the heavy-tail index in Assumption~\ref{assumption gradient noise heavy-tailed},
$\textbf W = (\bm w_1, \bm w_2, \cdots, \bm w_k) \in \R^{d \times k}$,
$\big((\nu_\alpha \times \mathbf S)\circ \Phi\big)^k$ is the $k$-fold of $(\nu_\alpha \times \mathbf S)\circ \Phi$
defined in \eqref{def, nu alpha times S composition polar transform},
and
\notationdef{notation-measure-L-k-up-infty}{$\mathcal{L}^{k\uparrow}_\infty$}
is the Lebesgue measure restricted on $\{ (t_1,\cdots,t_k) \in (0,\infty)^k:\ 0 < t_1 < t_2 < \cdots < t_k \}$.
% Section~\ref{subsec: lemma for measure check C} collects useful properties of the mapping $\widecheck g^{(k)|b}$ and the measure ${\widecheck{ \mathbf C }^{(k)|b}}$.
% In case that $x = 0$, we write $\widecheck{ \mathbf C }^{(k)|b}(\cdot ) \delequal \widecheck{ \mathbf C }^{(k)|b}(\ \cdot\ ;0)$.
% and
% $\notationdef{notation-check-C-x=0}{\widecheck{ \mathbf C }(\cdot )} \delequal \widecheck{ \mathbf C }(\ \cdot\ ;0 )$.
% Also, let

Recall that $H(\cdot) = \P(\norm{\bm Z_1} > \cdot)$, $\lambda(\eta) = \eta^{-1}H(\eta^{-1})$,
and for any $k \geq 1$ we write $\lambda^k(\eta) = \big(\lambda(\eta)\big)^k$.
Recall that 
$
\notationdef{notation-I-epsilon-shrinkage}{I_\epsilon} = \{ \bm y:\ \norm{\bm x - \bm y} < \epsilon\ \Longrightarrow\ \bm x \in I \}
$
is the $\epsilon$-shrinkage of $I$.
As the main result of this section,
Theorem~\ref{theorem: first exit time, unclipped} provides sharp asymptotics for the joint law of first exit times and exit locations of  $\bm X^{\eta|b}_t(\bm x)$ and $\bm X^\eta_t(\bm x)$.
The proof of Theorem~\ref{theorem: first exit time, unclipped} 
is deferred to Section~\ref{sec: proof of proposition: first exit time},
and the proof
is based on a general framework that will be developed in Section~\ref{subsec: framework, first exit time analysis}.
% By applying the general framework developed in Secton \ref{subsec: framework, first exit time analysis}
% and verifying Condition \ref{condition E2} over the asymptotic atoms $A(\epsilon) = (-\epsilon,\epsilon)$,
% we obtain the following result regarding the first exit times and exit locations of $X^{\eta|b}_j(x)$ and $X^\eta_j(x)$ from $I$.
% The detailed proof is provided in Section \ref{sec: proof of proposition: first exit time}.
% % Now we are ready to state the main results regarding the first exit times $\tau^\eta(x)$ and $\tau^{\eta|b}(x)$.

\begin{theorem}
\label{theorem: first exit time, unclipped}
\linksinthm{theorem: first exit time, unclipped}
\textbf{(First Exit Times and Locations: Truncated Case)}
    Let Assumptions \ref{assumption gradient noise heavy-tailed}, \ref{assumption: lipschitz continuity of drift and diffusion coefficients}, and \ref{assumption: shape of f, first exit analysis} hold.
    Let $b > 0$.
        Suppose that 
        $\mathcal J^I_b < \infty$,
        $I^c$ is bounded away from $\mathcal G^{(\mathcal J^I_b - 1)|b}(\epsilon)$ for some (and hence all) $\epsilon > 0$ small enough,
        and
        $
        \widecheck{\mathbf C}^{( \mathcal J^I_b )|b}(\partial I) = 0.
        $
        Then 
        $\notationdef{notation-C-b-*}{C^I_b} \delequal  \widecheck{ \mathbf{C} }^{ (\mathcal{J}^I_b)|b }(I^\complement) < \infty$.
        Furthermore, if $C^I_b \in (0,\infty)$,
        then
        for any $\epsilon > 0$, $t \geq 0$, and measurable set $B \subseteq I^c$,
    \begin{align*}
    \limsup_{\eta\downarrow 0}\sup_{\bm x \in I_\epsilon}
    \P\bigg(
        C^I_b \eta\cdot \lambda^{ \mathcal J^I_b }(\eta)\tau^{\eta|b}(\bm x) > t;\ 
        \bm X^{\eta|b}_{ \tau^{\eta|b}(\bm x)}(\bm x) \in B
     \bigg)
     & \leq \frac{ \widecheck{\mathbf{C}}^{ (\mathcal J^I_b)|b }(B^-) }{ C^I_b }\cdot\exp(-t),
     \\
     \liminf_{\eta\downarrow 0}\inf_{\bm x \in I_\epsilon}
    \P\bigg(
        C^I_b \eta\cdot \lambda^{ \mathcal J^I_b }(\eta)\tau^{\eta|b}(\bm x) > t;\ 
        \bm X^{\eta|b}_{ \tau^{\eta|b}(\bm x)}(\bm x) \in B
     \bigg)
     & \geq \frac{ \widecheck{\mathbf{C}}^{ (\mathcal J^I_b)|b }(B^\circ) }{ C^I_b }\cdot\exp(-t).
    \end{align*}
    Otherwise, we have $C^I_b = 0$, and
    \begin{align*}
        \limsup_{\eta\downarrow 0}\sup_{\bm x \in I_\epsilon}
    \P\bigg(
        \eta\cdot \lambda^{ \mathcal J^I_b }(\eta)\tau^{\eta|b}(\bm x) \leq t
     \bigg) = 0
     \qquad
     \forall \epsilon > 0,\ t \geq 0.
    \end{align*}
\end{theorem}

\begin{remark}
\label{remark: regularity conditions for first exit times}
Regarding the regularity conditions in Theorem~\ref{theorem: first exit time, unclipped},
conditions of form 
$\widecheck{\mathbf C}^{( \mathcal J^I_b )|b}(\partial I) = 0$ are standard even for metastability analyses of untruncated dynamics;
see e.g.\ \cite{doi:10.1142/S0219493715500197,
hogele2014exit}.
% Besides, we note that theses conditions hold almost automatically in the non-degenerate one-dimensional settings:
% suppose that $m = d = 1$ (so $\bm Z_j$'s and $\bm X^{\eta|b}_j(\bm x)$'s are random variables in $\R^1$) and for the diffusion coefficient $\sigma: \R \to \R$ we have $\inf_{x \in I}\sigma(x) > 0$;
% then for (Lebesgue) almost every $b \in (0,\infty)$,
% $I^c$ is bounded away from $\mathcal G^{(\mathcal J^I_b - 1)|b}(\epsilon)$ (for small $\epsilon > 0$),
% $
%         \widecheck{\mathbf C}^{( \mathcal J^I_b )|b}(\partial I) = 0,
%         $
% and $C^I_b \in (0,\infty)$
% with $\mathcal J^I_b = \inf_{ x \notin I }\ceil{ |x|/b  }$.
% See Lemmas~\ref{lemma: measure check C J * b, continuity, first exit analysis, R1} and \ref{lemma: exit rate strictly positive, first exit analysis, R1} for more details.
\end{remark}
\begin{remark}
As will be confirmed in Corollary~\ref{corollary: first exit time, untruncated case} below, $\mathcal J_b^I = 1$ when $b= \infty$,  regardless of the geometry of $\bm a(\cdot)$. 
In this case, Theorem~\ref{theorem: first exit time, unclipped} reduces to the manifestation of the principle of a single big jump. 
For $b \neq \infty$ and a contractive drift---i.e., $a(\bm x) \cdot \bm x \leq 0$ for all $\bm x \in I$---note that $\mathcal J_b^I = \lceil r / b \rceil$, where $r \triangleq \inf\{ \|\bm x - \bm 0\|: \bm x \in I^c\}$. 
This is because gradient flow will not bring $\bm X_j^{\eta|b}(\bm x)$ closer to $I^c$, and hence, the most efficient way to escape from $I$ is through $\lceil r / b \rceil$ consecutive jumps in the direction where $I^c$ is closest.
In the general case, however, $\mathcal J_b^I$ is determined as the solution to the discrete optimization problem in \eqref{def: first exit time, J *}, where the geometry of $\bm a(\cdot)$---in particular, gradient flows and their distances from $I^c$---plays a more sophisticated role. 
\end{remark}

% In summary, by developing the machinery of uniform $\mathbb M$-convergence and asymptotic atoms, we provide a general framework that connects large deviations and first exit analysis.
% Applying this framework for the truncated heavy-tailed dynamics $\bm X^{\eta|b}_j(\bm x)$,
% we reveal an intriguing phase transition in terms of the truncation threshold $b$, where the discretized width $\mathcal J^I_b$ dictates the order of the first exit times and the limiting law of first exit locations.
We conclude this section by noting that the first exit analysis for untruncated process $\bm X_j^\eta(\bm x)$
(see e.g.\ \cite{imkeller2006first, pavlyukevich2008metastable, imkeller2008levy}
% , doi:10.1142/S0219493711003413
% ,imkeller2010first} 
for analogous results for continuous processes) 
follows directly from Theorem~\ref{theorem: first exit time, unclipped}.
Let
\begin{align}
    \notationdef{notation-check-C}{\widecheck{\mathbf C}(\ \cdot\ )} \delequal \int \mathbbm{I}\Big\{ \bm \sigma(\bm 0) \bm w \in\ \cdot\ \Big\}
    \big((\nu_\alpha \times \mathbf S)\circ \Phi\big)(d \bm w).
    \label{def: measure check C}
\end{align}
The asymptotic analysis for exit times and locations of the untruncated dynamics $\bm X^\eta_j(\bm x)$ follows from the result for $\bm X^{\eta|b}_j(\bm x)$ by sending $b$ to $\infty$,
and the limiting laws of the exit location $\bm X^\eta_{\tau^\eta(\bm x)}(\bm x)$ is characterized by $\widecheck{\mathbf C}(\ \cdot\ )$, as presented in Corollary~\ref{corollary: first exit time, untruncated case}.
The proof is straightforward and we collect it in Section~\ref{sec: proof for untruncated case}
for the sake of completeness.

\begin{corollary}\label{corollary: first exit time, untruncated case}
\linksinthm{corollary: first exit time, untruncated case}
\textbf{(First Exit Times and Locations: Untruncated Case)}
Let Assumptions \ref{assumption gradient noise heavy-tailed}, \ref{assumption: lipschitz continuity of drift and diffusion coefficients}, and \ref{assumption: shape of f, first exit analysis} hold.
        Suppose that $\widecheck{\mathbf C}(\partial I) = 0$ and $\norm{\bm \sigma(\bm 0)} > 0$.
        Then $\notationdef{notation-C-*-first-exit-time}{C^I_\infty} \delequal  \widecheck{ \mathbf{C} }(I^\complement) < \infty$.
        Furthermore,
        if $C^I_\infty > 0$,
        then
        for any $t \geq 0$, $\epsilon > 0$, and measurable set $B \subseteq I^c$,
    \begin{align*}
        \limsup_{\eta\downarrow 0}\sup_{\bm x \in I_\epsilon}
        \P\bigg(
            C^I_\infty H(\eta^{-1})\tau^\eta(\bm x) > t;\ \bm X^\eta_{ \tau^\eta(\bm x)}(\bm x) \in B
        \bigg)
        & \leq
        \frac{ \widecheck{\mathbf{C}}(B^-) }{ C^I_\infty }\cdot\exp(-t),
        \\ 
        \liminf_{\eta\downarrow 0}\inf_{\bm x \in I_\epsilon}
        \P\bigg(
            C^I_\infty H(\eta^{-1})\tau^\eta(\bm x) > t;\ \bm X^\eta_{ \tau^\eta(\bm x)}(\bm x) \in B
        \bigg)
        & \geq
        \frac{ \widecheck{\mathbf{C}}(B^\circ) }{ C^I_\infty }\cdot\exp(-t).
    \end{align*}
    Otherwise, we have $C^I_\infty = 0$, and 
    \begin{align*}
        \limsup_{\eta\downarrow 0}\sup_{\bm x \in I_\epsilon}
        \P\bigg(
            H(\eta^{-1})\tau^\eta(\bm x) \leq t
        \bigg) = 0
        \qquad
        \forall \epsilon > 0,\ t \geq 0.
    \end{align*}
\end{corollary}

\subsubsection{General Framework: Asymptotic Atoms}
\label{subsec: framework, first exit time analysis}
This section proposes a general framework that enables sharp characterization of exit times and exit locations of Markov chains.
The new heavy-tailed large deviations formulation introduced in Section~\ref{subsec: LD, SGD} is conducive to this framework.

Consider a general metric space $(\mathbb S,\bm d)$
and a family of $\mathbb S$-valued Markov chains $\big\{\{V_j^\eta(x): j\geq 0\}:\eta>0\big\}$ parameterized by $\eta$, where $x\in \mathbb S$ denotes the initial state and $j$ denotes the time index. 
We use
$
\bm V_{[0,T]}^\eta(x)\delequal \{V^\eta_{\lfloor t/\eta\rfloor}(x): t  \in [0,T]\}
$ 
to denote the scaled version of $\{V_j^\eta(x): j\geq 0\}$ as a  $\D[0,T]$-valued random element.
% We will consider $\tau^{\eta}_{I_\epsilon^c}(x) \delequal \min\{j\geq0: V_j^\eta(x) \notin I_\epsilon\}$, and $\tau^\eta_{J(\delta)}(x) \delequal \min\{j\geq0: V_j^\eta(x) \in J(\delta)\}$.
For a given set $E$, let $\tau^{\eta}_{E}(x) \delequal \min\{j\geq0: V_j^\eta(x) \in  E\}$ denote $\{V_j^\eta(s): j\geq 0\}$'s first hitting time of $E$.
We consider an asymptotic domain of attraction $I\subseteq \mathbb S$, within which $\bm V_{[0,T]}^\eta(x)$ typically (i.e., as $\eta\downarrow 0$) stays within $I$ throughout any fixed time horizon $[0,T]$ as far as the initial state $x$ is in $I$. 
% We will make these informal descriptions precise in Condition~\ref{condition E2}.
However, if one considers an infinite time horizon, $V_{\boldsymbol{\cdot}}^\eta(x)$ is typically bound to escape $I$ eventually due to the stochasticity. 
The goal of this section is to establish an asymptotic limit of the joint distribution of the exit time $\tau^{\eta}_{I^\complement}(x)$ and the exit location $V^\eta_{\tau^\eta_{I^\complement}(x)}(x)$.
Throughout this section, we will
denote 
\(
    V^\eta_{\tau^\eta_{I(\epsilon)^\complement}(x)}(x)
\)
and 
\(
    V^\eta_{\tau^\eta_{I^\complement}(x)}(x)
\)
with 
\(
    V^\eta_{\tau_\epsilon}(x)
\)
and
\(
    V^\eta_{\tau}(x)
\),
respectively, for notation simplicity. 

% Let $E(\epsilon,T) \delequal \{\xi \in \D[0,T]: \xi(t) \notin I(\epsilon) \text{ for some } t\in[0,T]\}$.

% \begin{condition}\label{condition E0}
% $\big\{\{V_j^\eta(x): j\geq 0\}:\eta>0\big\}$ possesses an asymptotic atom $\{A(\epsilon)\subseteq S: \epsilon>0\}$ in the following sense:
% For sufficiently large $T > 0$, sufficiently small $\epsilon>0$,
% \begin{align*}
% l(\epsilon,T)
% & \leq  \liminf_{\eta \downarrow 0}\frac{ \inf_{x \in A(\epsilon)}\P\big({\bm{V}}^{\eta}_{[0,T]}(x) \in E(0,T) \big) }{  \gamma(\eta)T/\eta } 
% \\
% & \leq  \limsup_{\eta \downarrow 0}\frac{ \sup_{x \in A(\epsilon)}\P\big({\bm{V}}^{\eta}_{[0,T]}(x) \in E(\epsilon,T) \big) }{  \gamma(\eta)T/\eta  } 
% \leq 
% u(\epsilon,T)
% \end{align*}
% where 
% % for each sufficiently small $\epsilon>0$ and sufficiently large $T>0$
% $u(\epsilon,T;\cdot):\D[0,T]\to \R_+$ 
% and 
% $l(\epsilon,T;\cdot):\D[0,T]\to \R_+$ 
% are set functions such that 
% $$
% \lim_{\epsilon\to0}\lim_{T\to\infty}\frac {l(\epsilon,T)}{T} 
% = \lim_{\epsilon\to0}\lim_{T\to\infty}\frac {u(\epsilon,T)}{T} 
% = 1
% $$
% \end{condition}

% \begin{condition}\label{condition E1}
% % There exists $C_0:\R_+ \to \R_+$ such that $C_0(\epsilon) \to \infty$ as $\epsilon\to 0$ and
% For all sufficiently small $\epsilon>0$, 
% \begin{align}
% \liminf_{\eta\downarrow 0} \frac{\ \inf_{x\in I\setminus I(\epsilon)} \P\big( \tau^{\eta}_{I^\complement}(x) \leq T/\eta\big)} {\gamma(\eta) T/\eta} \geq 1
% % \tag{E3}
% \label{eq: exit time condition easy to exit from around the boundary}
% \end{align}
% \end{condition}

We introduce the notion of asymptotic atoms to facilitate the analyses.
Let $\{I(\epsilon)\subseteq I: \epsilon>0\}$ and $\{A(\epsilon)\subseteq \mathbb S: \epsilon>0\}$ be collections of subsets of $I$ such that $\bigcup_{\epsilon>0}I(\epsilon) = I$ and $\bigcap_{\epsilon>0}A(\epsilon) \neq \emptyset$.
Let $C(\cdot)$ is a Borel measure on $\mathbb S\setminus I$ satisfying $C(\partial I) = 0$ that 
characterizes the (asymptotics limit of the) exit location of $V^\eta_{\cdot}(x)$.
Specifically, we consider two different cases for the location measure $C(\cdot)$:
\begin{enumerate}[$(i)$]
    \item
        $C(I^\complement) \in (0,\infty)$: by incorporating the normalizing constant $C(I^\complement)$ into the scale function $\gamma(\eta)$, we can assume w.l.o.g.\ that $C(\cdot)$ \emph{is a probability measure}, and $C(B)$ dictates the limiting probability that $\P(V^\eta_\tau(x) \in B)$ as shown in Theorem~\ref{thm: exit time analysis framework};

    \item 
        $C(I^\complement) = 0$: as a result, $C(B) = 0$ for any Borel set $B \subseteq I^\complement$, and it is equivalent to stating that $C(\cdot)$ \emph{is trivially zero}.
\end{enumerate}

\begin{definition} \label{def: asymptotic atom}
 $\big\{\{V_j^\eta(x): j\geq 0\}:\eta>0\big\}$ possesses an asymptotic atom $\{A(\epsilon)\subseteq \mathbb S: \epsilon>0\}$ associated with the domain $I$, location measure $C(\cdot)$, scale $\gamma:(0,\infty) \to (0,\infty)$, and covering $\{I(\epsilon)\subseteq I: \epsilon>0\}$ if the following holds: 
For each measurable set $B \subseteq \mathbb S$,
there exist $\delta_B:(0,\infty)\times(0,\infty)\to(0,\infty)$, $\epsilon_B>0$, and $T_B:(0,\infty) \to (0,\infty)$ such that 
% For any given $\epsilon>0$, there exist sets $A(\epsilon)$ and $I(\epsilon)$ such that $A(\epsilon) \subsetneq I(\epsilon) \subseteq I$ and
% \textcolor{red}{in \eqref{eq: exit time condition lower bound}, $\tau^\eta_{I^c}$ or $\tau^\eta_{I(\epsilon)^c}$}
\begin{align}
C(B^\circ) - \delta_B(\epsilon,T)
\leq\,
&
\liminf_{\eta\downarrow0} \frac{\ \inf_{x \in A(\epsilon)} \P\big(\tau^{\eta}_{I(\epsilon)^\complement}(x) \leq T/\eta;\; V_{\tau_{\epsilon}}^\eta(x)\in B\big)}{\gamma(\eta)T/\eta} 
\label{eq: exit time condition lower bound}
\\
\leq\,
&
\limsup_{\eta\downarrow0} \frac{\sup_{x \in A(\epsilon)} \P\big(\tau^{\eta}_{I(\epsilon)^\complement}(x) \leq T/\eta;\; V_{\tau_{\epsilon}}^\eta(x)\in B\big)}{\gamma(\eta)T/\eta} 
\leq C(B^-) + \delta_B(\epsilon,T)
\label{eq: exit time condition upper bound}
\\
&
\limsup_{\eta\downarrow 0} \frac{\sup_{x\in I(\epsilon)} \P\big( \tau^{\eta}_{(I(\epsilon)\setminus A(\epsilon))^\complement}(x) > T/\eta\big)} {\gamma(\eta) T/\eta} = 0
% \tag{E4}
\label{eq:E3}
\\[7pt]
&
\liminf_{\eta\downarrow 0} \ \inf_{x\in I(\epsilon)} \P\big(\tau^{\eta}_{A(\epsilon)}(x) \leq T/{\eta}\big) = 1
% \tag{E5}
\label{eq:E4}
\end{align}
for any $\epsilon \leq \epsilon_B$ and $T \geq T_B(\epsilon)$,
where $\gamma(\eta)/\eta \to 0$ as $\eta \downarrow 0$ and
$\delta_B$'s
are such that 
$$
\lim_{\epsilon\downarrow 0}\lim_{T\to\infty} {\delta_B(\epsilon,T)}
% = \lim_{\epsilon\to0}\lim_{T\to\infty} {u(\epsilon,T)}
= 0.
$$
\end{definition}

To see how Definition~\ref{def: asymptotic atom} asymptotically characterize the atoms in $V^\eta_{ \boldsymbol{\cdot}}(x)$ for the first exit analysis from domain $I$,
note that
the condition~\eqref{eq:E4} requires the process to efficiently return to the asymptotic atoms $A(\epsilon)$.
The conditions~\eqref{eq: exit time condition lower bound} and \eqref{eq: exit time condition upper bound} then state that,
upon hitting the asymptotic atoms $A(\epsilon)$,
the process almost regenerates in terms of the law of the exit time $\tau^\eta_{I(\epsilon)^\complement}(x)$ and exit locations $V^\eta_{\tau_\epsilon}(x)$.
Furthermore, the condition~\eqref{eq:E3} prevents the process $V^\eta_{ \boldsymbol{\cdot}}(x)$ from spending a long time without either returning to the asymptotic atoms $A(\epsilon)$ or exiting from $I(\epsilon)$, which covers the domain $I$ as $\epsilon$ tends to $0$.

The existence of an asymptotic atom is a sufficient condition for characterization of exit time and location asymptotics as in Theorem~\ref{theorem: first exit time, unclipped}. 
To minimize repetition, we refer to the existence of an asymptotic atom---with specific domain, location measure, scale, and covering---Condition~\ref{condition E2} throughout the paper.
\begin{condition}\label{condition E2}
A family $\big\{\{V_j^\eta(x): j\geq 0\}:\eta>0\big\}$ of Markov chains possesses an asymptotic atom $\{A(\epsilon)\subseteq \mathbb S: \epsilon>0\}$ associated with the domain $I$, location measure $C(\cdot)$, scale $\gamma:(0,\infty) \to (0,\infty)$, and covering $\{I(\epsilon)\subseteq I: \epsilon>0\}$.
\end{condition}

Recall that, right before Definition~\ref{def: asymptotic atom}, we state that for the location measure $C(\cdot)$ we consider two cases that $(i)$ $C(I^\complement) = 1$ (more generally, $C(\cdot)$ is a finite measure), and $(ii)$ $C(I^\complement) = 0$.
The following theorem is the key result of this section. See Section~\ref{subsec: Exit time analysis framework} for the proof of the theorem.

% Thanks to the previously developed uniform sample-path large deviations, 
% we can verify Condition \ref{condition E2} uniformly for all initial values over the asymptotic atoms $A(\epsilon) = (-\epsilon,\epsilon)$.

\begin{theorem}
\label{thm: exit time analysis framework} 
\linksinthm{thm: exit time analysis framework} 
If Condition~\ref{condition E2} holds, then the first exit time $\tau_{I^\complement}^\eta(x)$ scales as $1/\gamma(\eta)$, and the distribution of the location $V_\tau^\eta(x)$ at the first exit time converges to $C(\cdot)$. 
Moreover, the convergence is uniform over $I(\epsilon)$ for any $\epsilon>0$.
That is, 
\begin{enumerate}[$(i)$]
    \item 
        If $C(I^\complement) = 1$, then for each $\epsilon>0$, measurable $B\subseteq I^\complement$, and $t\geq 0$,
        \begin{align*}
            C(B^\circ) \cdot e^{-t}
            & \leq 
            \liminf_{\eta\downarrow 0} \inf_{x\in I(\epsilon)}
            \P\big(
                \gamma(\eta)\tau_{I^\complement}^{\eta}(x)>t,\,V_{\tau}^\eta(x)\in B  
            \big) 
            \\
            & \leq 
            \limsup_{\eta\downarrow 0} \sup_{x\in I(\epsilon)}
            \P\big(
                \gamma(\eta)\tau_{I^\complement}^{\eta}(x)>t,\,V_{\tau}^\eta(x)\in B  
            \big) 
            \leq C(B^-) \cdot e^{-t};
        \end{align*}

    \item 
        If $C(I^\complement) = 0$, then for each $\epsilon,t >0$,
        \begin{align*}
            \lim_{\eta\downarrow 0} \sup_{x\in I(\epsilon)}
            \P\big(
                \gamma(\eta)\tau_{I^\complement}^{\eta}(x) \leq t
            \big) = 0.
        \end{align*}
\end{enumerate}

\end{theorem}

\subsection{Numerical Examples}
\label{subsec: numerical examples}

In this section, we provide the details for numerical samples illustrated in Figure~\ref{fig: exit time}.
We consider iterates in $\R^1$:
\begin{align}
    \bm X^{\eta|b}_0(\bm x) = \bm x,\qquad
\bm X^{\eta|b}_j(\bm x) = \bm X^{\eta|b}_{j-1}(\bm x) + \varphi_b\big( -\eta U^\prime\big(\bm X^{\eta|b}_{j-1}(\bm x)\big) + \eta \bm Z_j\big)\quad \forall j \geq 1,
\label{def: R1 dynamics, numerical examples}
\end{align}
under the potential function
\begin{equation}\label{aeq: potential U, first exit time}
\begin{aligned}
    & U(x)= (x+1.6)(x+1.3)^2(x-0.2)^2(x-0.7)^2(x-1.6)\big(0.05|1.65-x|\big)^{0.6} \\
    & \ \ \cdot \Big( 1 + \frac{1}{ 0.01 + 4(x-0.5)^2  } \Big)\Big( 1 + \frac{1}{0.1 + 4(x+1.5)^2} \Big)\Big( 1 - \frac{1}{4}\exp( -5(x + 0.8)(x + 0.8)  ) \Big).
\end{aligned}
\end{equation}
See Fig~\ref{fig: exit time} (i) for an illustration.
Specifically, we consider the case where the law of $(Z_j)_{j \geq 1}$ is of form
\begin{align}
      c_{\text{pareto}} \cdot W_\alpha + c_{\text{normal}}\cdot N(0,1)
      \label{def: law of noises, LD experiment}
\end{align}
with  $c_\text{pareto} = 0.1$, $c_\text{normal} = 0$, and $\alpha = 1.2$,
and
focus on the local minimum $m = -0.66$ and its attraction field $I = (-1.3, 0.2)$.
We initialize the process at $\bm x = m$ and are interested in first exit times $\tau^{\eta|b}(\bm x)$  from $I$; see \eqref{def: first exit time for heavy tailed SGD}. 
In this case,
the index $\mathcal J^I_b$ defined in \eqref{def: first exit time, J *}
reduces to $\mathcal J^I_b = \ceil{ 0.64/b  }$ for any $b \in (0,\infty)$,
and the regularity conditions in Theorem~\ref{theorem: first exit time, unclipped}
hold for (Lebesgue) almost all $b > 0$; see Remark~\ref{remark: regularity conditions for first exit times}.
Therefore, for Lebesgue almost all $b > 0$,
the stopping times $\tau^{\eta|b}(\bm x)$ is roughly of order $1/\eta^{ 1 + \mathcal J^I_b(\alpha - 1) } = 1/\eta^{ 1 + \mathcal J^I_b * 0.2 }$ as $\eta \downarrow 0$.
This characterizes the phase transitions in the order of first exit times depending on the (discretized) relative width $\mathcal J^I_b$.
In case that $b = \infty$, we apply Corollary~\ref{corollary: first exit time, untruncated case} and obtain that 
the exit times $\tau^\eta(\bm x)$ in the untruncated case (see \eqref{def: first exit time for heavy tailed SGD})
is roughly of order 
$1/\eta^{\alpha } = 1/\eta^{ 1.2 }$ for small $\eta$.

We confirm these asymptotics through Monte-Carlo simulation and present the results in 
Fig~\ref{fig: exit time} (ii).
This is a log-log scale plot, where each dot represents an average of 20 samples, and the dashed lines indicate the asymptotics provided by our metastability analysis.
To prevent the experiment from running too long, a stopping criterion of $5\times 10^7$ steps 
is employed.
This stopping criterion was reached only in the case where $b = 0.28$ and $\eta = 0.001$, which is indicated in the plot by the only non-solid dot, highlighting that it is an underestimation.
The plot confirms the asymptotic law of first exit times established in our metastability analysis, as well as the phase transition in first exit times  w.r.t.\ $\mathcal J^I_b$.

\section{Proof of Theorem~\ref{thm: exit time analysis framework}}
\label{subsec: Exit time analysis framework}

Our proof of Theorem~\ref{thm: exit time analysis framework} hinges on the following proposition.

\begin{proposition}\label{prop: exit time analysis main proposition}
\linksinthm{prop: exit time analysis main proposition}
Suppose that Condition~\ref{condition E2} holds.
\begin{enumerate}[$(i)$]
    \item 
        If $C(\cdot)$ is a probability measure supported on $I^\complement$ (i.e., $C(I^\complement) = 1$),
        then for each measurable set $B\subseteq \mathbb S$ and $t\geq 0$, there exists $\delta_{t,B}(\epsilon)$ such that
        \begin{align*}
            C(B^\circ)\cdot e^{-t} - \delta_{t,B}(\epsilon)
            &
            \leq
            \liminf_{\eta\downarrow 0} \inf_{x\in A(\epsilon)}\P\big(\gamma(\eta) \tau_{I(\epsilon)^\complement}^\eta(x) > t;\; V^\eta_{\tau_\epsilon}(x) \in B\big)
            \\
            &
            \leq
            \limsup_{\eta\downarrow 0} \sup_{x\in A(\epsilon)}\P\big(\gamma(\eta) \tau_{I(\epsilon)^\complement}^\eta(x) > t;\; V^\eta_{\tau_\epsilon}(x) \in B\big)
            \leq C(B^-)\cdot e^{-t} + \delta_{t,B}(\epsilon)
        \end{align*}
        for all sufficiently small $\epsilon>0$, where $\delta_{t,B}(\epsilon) \to 0$ as $\epsilon \to 0$.

    \item 
        If $C(I^\complement) = 0$ (i.e., $C(\cdot)$ is trivially zero), then for each $t > 0$, there exists $\delta_{t}(\epsilon)$ such that
        \begin{align*}
            \limsup_{\eta \downarrow 0}
            \sup_{x\in A(\epsilon)}\P\big(\gamma(\eta) \tau_{I(\epsilon)^\complement}^\eta(x) \leq t\big)
            \leq \delta_{t}(\epsilon)
        \end{align*}
        for all $\epsilon > 0$ sufficiently small, where $\delta_t(\epsilon) \to 0$ as $\epsilon \to 0$.
\end{enumerate}

\end{proposition}

\begin{proof}
\linksinpf{prop: exit time analysis main proposition}
Fix some measurable $B \subseteq \S$ and $t \geq 0$.
Henceforth in the proof, given any choice of $0 < r < R$, we only consider
$\epsilon \in (0,\epsilon_B)$ and $T$ sufficiently large such that Condition~\ref{condition E2} holds with $T$ replaced with $\frac{1-r}{2}T$, $\frac{2-r}{2}T$, $rT$, and $RT$.
Let
\[
    % \rho^{(i+1)}_{r}(x) 
    % \rho^{\eta}_{i+1}(x) = 
    \rho^{\eta}_{i}(x)
    \delequal \inf\Big\{j\geq \rho^{\eta}_{i-1}(x) + \lfloor rT/\eta \rfloor: V_j^\eta(x) \in A(\epsilon)\Big\}
\]
where 
\(\rho_0^\eta(x) = 0.\)
One can interpret these as the $i$\textsuperscript{th} asymptotic regeneration times after cooling period $rT/\eta$.
We start with the following two observations: For any $0 < r < R$,
\begin{align}
    % &
    \P\Big(\tau_{I(\epsilon)^\complement}^\eta(y) \in \big(RT/\eta,\, \rho_1^\eta(y)\big]\Big)
    &
    \leq 
    \P\Big(\tau_{I(\epsilon)^\complement}^\eta(y) \wedge \rho_1^\eta(y) > RT/\eta\Big)
    \nonumber\\
    &
    \leq
    \P\Big(V_j^\eta(y)\in I(\epsilon)\setminus A(\epsilon)\quad \forall j \in \big[\lfloor r T/\eta \rfloor,\,RT/\eta \big] \Big)
    \nonumber\\
    &
    \leq
    \sup_{z\in I(\epsilon)\setminus A(\epsilon)}
        \P\Big(\tau^\eta_{(I(\epsilon)\setminus A(\epsilon))^c}(z) > \frac{R-r}{2} T/\eta \Big)
    \nonumber\\
    &
    =
    \gamma(\eta)T/\eta \cdot \lo(1),
    \label{eq: exit time analysis: almost nothing happens between T over eta and rho one}
\end{align}
where  the last equality is from \eqref{eq:E3} of Condition~\ref{condition E2}, and
\begin{align}
    &
    \sup_{y\in A(\epsilon)}
    \P\Big(
        V_{\tau_\epsilon}^\eta(y) \in B
        ;\;
        \tau_{I(\epsilon)^\complement}^\eta(y) \leq \rho_1^\eta(y)
    \Big)
    \nonumber\\
    &
    \leq 
    \sup_{y\in A(\epsilon)}
    \P\Big(
        V_{\tau_\epsilon}^\eta(y) \in B
        ;\;
        \tau_{I(\epsilon)^\complement}^\eta(y) \leq RT/\eta
    \Big)
    +
    \sup_{y\in A(\epsilon)}
    \P\Big(
        \tau_{I(\epsilon)^\complement}^\eta(y) \in \big(RT/\eta,\, \rho_1^\eta(y)\big]
    \Big)
    \nonumber\\
    &
    \leq 
    \sup_{y\in A(\epsilon)}
    \P\Big(
        V_{\tau_\epsilon}^\eta(y) \in B
        ;\;
        \tau_{I(\epsilon)^\complement}^\eta(y) \leq RT/\eta
    \Big)
    +
    \gamma(\eta)T/\eta \cdot \lo(1)
    \nonumber\\
    % &
    % =\big(u(\epsilon, T)\cdot C(B)+ \lo(1)\big) \cdot \gamma(\eta) RT/\eta + \frac{R-r}2\cdot \gamma(\eta)T/\eta \cdot \lo(1)
    % \nonumber\\
    &
    \leq\big(C(B^-)+ \delta_B(\epsilon,RT) + \lo(1)\big) \cdot \gamma(\eta) RT/\eta,
    \label{eq: exit time analysis: upper bound of the probability that tau is before rho}
\end{align}
where the second inequaility is from \eqref{eq: exit time analysis: almost nothing happens between T over eta and rho one} and the last equality is from \eqref{eq: exit time condition upper bound} of Condition~\ref{condition E2}.
% for sufficiently large $T$'s.

\medskip
\noindent
\textbf{Proof of Case $(i)$.}

We work with different choices of $R$ and $r$ for the lower and upper bounds.
For the lower bound, we work with $R>r>1$ and set
\(K = \left\lceil \frac{t/\gamma(\eta)}{T/\eta} \right\rceil\).
Note that for $\eta \in \big(0, (r-1)T\big)$,
we have
$\lfloor r T/\eta \rfloor \geq T/\eta$ and hence $\rho_K^\eta(x) \geq K \lfloor r T/\eta\rfloor \geq t/\gamma(\eta)$.
Note also that from the Markov property conditioning on $\mathcal F_{\rho_j^\eta(x)}$,
\begin{align}
    &
    \inf_{x \in A(\epsilon)}\P\big(\gamma(\eta) \tau_{I(\epsilon)^\complement}^\eta(x) > t;\; V^\eta_{\tau_\epsilon}(x) \in B\big)
    % \\
    % &
    % =
    %     \P\big(\tau_{I(\epsilon)^\complement}^\eta(x) > t/\gamma(\eta);\; V^\eta_{\tau_\epsilon}(x) \in B\big)
    \nonumber\\
    &
    \geq
    \inf_{x \in A(\epsilon)}
        \P(\tau_{I(\epsilon)^\complement}^\eta (x) > \rho_K^\eta(x);\; V^\eta_{\tau_\epsilon}(x) \in B)
    % \\
    % &
    =
    \inf_{x \in A(\epsilon)}
    \sum_{j=K}^\infty     
        \P\Big(
            \tau_{I(\epsilon)^\complement}^\eta(x) \in \big(\rho_j^\eta(x),\, \rho_{j+1}^\eta(x) \big];\; V^\eta_{\tau_\epsilon}(x) \in B
        \Big)         
    \nonumber\\
    &
    \geq
    \inf_{x \in A(\epsilon)}
    \sum_{j=K}^\infty     
        \P\Big(
            \tau_{I(\epsilon)^\complement}^\eta(x) \in \big(\rho_j^\eta(x),\, \rho_{j}^\eta(x) + T/\eta \big];\; V^\eta_{\tau_\epsilon}(x) \in B
        \Big)         
    \nonumber\\
    &
    \geq
    \inf_{x \in A(\epsilon)}
    \sum_{j=K}^\infty     
        \inf_{y\in A(\epsilon)}
        \P\Big(
            \tau_{I(\epsilon)^\complement}^\eta(y) \leq T/\eta;\; V^\eta_{\tau_\epsilon}(y) \in B
        \Big) 
        \cdot
        \P\Big(\tau_{I(\epsilon)^\complement}^\eta(x) > \rho_j^\eta(x)\Big).
    \nonumber\\
    &
    \geq
    \inf_{y\in A(\epsilon)}
    \P\Big(
        \tau_{I(\epsilon)^\complement}^\eta(y) \leq T/\eta;\; V^\eta_{\tau_\epsilon}(y) \in B
    \Big) 
    \cdot
    \sum_{j=K}^\infty \inf_{x \in A(\epsilon)}  
        \P\Big(\tau_{I(\epsilon)^\complement}^\eta(x) > \rho_j^\eta(x)\Big).
        \label{eq: exit time and location analysis with epsilon lower bound}
\end{align}
\elaborate{
\begin{align*}
    &
    \P\Big(
        \tau_{I(\epsilon)^\complement}^\eta(x) \in \big(\rho_j^\eta(x),\, \rho_j^\eta(x) + T/\eta\big];\; V^\eta_{\tau_\epsilon}(x) \in B
    \Big) 
    \\
    &
    =
    \E\bigg[
        \P\Big(
            \tau_{I(\epsilon)^\complement}^\eta(x) \in \big(\rho_j^\eta(x),\, \rho_j^\eta(x) + T/\eta\big];\; V^\eta_{\tau_\epsilon}(x) \in B
        \Big|
            \mathcal F_{\rho_j^\eta(x)}
        \Big) 
    \bigg]
    \\    
    &
    =
    \E\bigg[
        \P\Big(
            \tau_{I(\epsilon)^\complement}^\eta(x) \leq \rho_j^\eta(x) + T/\eta;\; V^\eta_{\tau_\epsilon}(x) \in B
        \Big|
            \mathcal F_{\rho_j^\eta(x)}
        \Big) 
        \cdot
        \I\big\{\tau_{I(\epsilon)^\complement}^\eta(x) > \rho_j^\eta(x)\big\}
    \bigg]
    \\
    &
    \geq
        \E\bigg[
            \inf_{y\in A(\epsilon)}
            \P\Big(
                \tau_{I(\epsilon)^\complement}^\eta(y) \leq T/\eta;\; V^\eta_{\tau_\epsilon}(y) \in B
            \Big) 
            \cdot
            \I\big\{\tau_{I(\epsilon)^\complement}^\eta(x) > \rho_j^\eta(x)\big\}
        \bigg]
    \\
    &
    =
    \inf_{y\in A(\epsilon)}
    \P\Big(
        \tau_{I(\epsilon)^\complement}^\eta(y) \leq T/\eta;\; V^\eta_{\tau_\epsilon}(y) \in B
    \Big) 
    \cdot
    \P\big(\tau_{I(\epsilon)^\complement}^\eta(x) > \rho_j^\eta(x)\big).
\end{align*}
}%
From the Markov property conditioning on $\mathcal F_{\rho_j^\eta(x)}$,
the second term can be bounded as follows:
\begin{align}
    &
    \sum_{j=K}^\infty \inf_{x \in A(\epsilon)}  
    \P\Big(\tau_{I(\epsilon)^\complement}^\eta(x) > \rho_j^\eta(x)\Big)
    \nonumber\\
    &
    \geq 
    \sum_{j=0}^\infty
    \bigg(
        \inf_{y \in A(\epsilon)}
            \P\Big(\tau_{I(\epsilon)^\complement}^\eta(y) > \rho_1^\eta(y)\Big)
    \bigg)^{K+j}
    % \\
    % &
    =
    \sum_{j=0}^\infty
        \bigg(
            1 -\sup_{y \in A(\epsilon)}
                \P\Big(\tau_{I(\epsilon)^\complement}^\eta(y) \leq \rho_1^\eta(y) \Big)
        \bigg)^{K+j}
    \nonumber\\
    &
    =
    \frac1{\sup_{y \in A(\epsilon)}\P\Big(\tau_{I(\epsilon)^\complement}^\eta(y) \leq \rho_1^\eta(y) \Big)}
    \cdot
    \bigg(
        1 - \sup_{y \in A(\epsilon)}\P\Big(\tau_{I(\epsilon)^\complement}^\eta(y) \leq \rho_1^\eta(y) \Big)
    \bigg)^{ \ceil{\frac{t/\gamma(\eta)}{T/\eta}}  }   
    \nonumber\\
    &
    \geq
    \frac1{\big(1+\delta_{\S}(\epsilon, RT) + \lo(1)\big) \cdot \gamma(\eta) RT/\eta}
    \cdot
    \bigg(
        1 - \big(1+\delta_{\S}(\epsilon, RT) + \lo(1)\big) \cdot \gamma(\eta) RT/\eta
    \bigg)^{\frac{t/\gamma(\eta)}{T/\eta} +1}.    
    \label{eq:exit time and location tail of tau upper bound}    
\end{align}
where the last inequality is from \eqref{eq: exit time analysis: upper bound of the probability that tau is before rho} with $B=\mathbb S$.
From \eqref{eq: exit time and location analysis with epsilon lower bound},
\eqref{eq:exit time and location tail of tau upper bound}, and    
\eqref{eq: exit time condition lower bound} of Condition~\ref{condition E2}, we have
\begin{align*}
    &
    \liminf_{\eta\downarrow 0} 
    \inf_{x\in A(\epsilon)}
        \P\big(\gamma(\eta) \tau_{I(\epsilon)^\complement}^\eta(x) > t;\; V^\eta_{\tau_\epsilon}(x) \in B\big) 
    \\
    &
    \geq
    \liminf_{\eta\downarrow 0}
    \frac{C(B^\circ) - \delta_B(\epsilon, T)+ \lo(1)}{\big(1+\delta_{\S}(\epsilon, RT)+ \lo(1)\big)\cdot R}
    \cdot 
        \bigg(
            1 - \big(1+\delta_{\S}(\epsilon, RT) + \lo(1)\big) \cdot \gamma(\eta) RT/\eta
        \bigg)^{\frac{ R\cdot t}{\gamma(\eta)RT/\eta} +1}.    
    \\
    &
    \geq
    \frac{C(B^\circ)-\delta_B(\epsilon,T)}{1+\delta_{\S}(\epsilon,RT)}\cdot\exp\Big(-\big(1+\delta_{\S}(\epsilon,RT)\big) \cdot R\cdot t\Big).
\end{align*}
By taking limit $T\to\infty$ and then considering an $R$ arbitrarily close to 1, it is straightforward to check that the desired lower bound holds.
% \begin{align*}
%     &
%     C(B^\circ)\cdot\exp(-t)
%     \leq 
%     \liminf_{\eta\downarrow 0} 
%         \P\big(\gamma(\eta) \tau_{I(\epsilon)^\complement}^\eta(x) > t;\; V^\eta_{\tau_\epsilon}(x) \in B\big).
% \end{align*}

Moving on to the upper bound, we set $R=1$ and fix an arbitrary $r\in (0,1)$.
% Let $\rho^\eta_{i+1}(x;\epsilon,t) \delequal \inf\{j\geq \rho_i(x; \epsilon,t) + \lfloor t/\eta \rfloor: V_j^\eta(x) \in A(\epsilon)\}$ where $\rho_0^\eta(x;\epsilon,t) = 0$.
Set
\(k = \left\lfloor \frac{t/\gamma(\eta)}{T/\eta} \right\rfloor\)
and note that
\begin{align*}
    % &
    \sup_{x \in A(\epsilon)}\P\big(\gamma(\eta) \tau_{I(\epsilon)^\complement}^\eta(x) > t;\; V^\eta_{\tau_\epsilon}(x) \in B\big)
    % \\
    &
    =\sup_{x \in A(\epsilon)}
        \P\big(\tau_{I(\epsilon)^\complement}^\eta(x) > t/\gamma(\eta);\; V^\eta_{\tau_\epsilon}(x) \in B\big)
    \\
    &
    =
        \underbrace{ \sup_{x \in A(\epsilon)}
            \P\big(\tau_{I(\epsilon)^\complement}^\eta(x) > t/\gamma(\eta)\geq \rho_k^\eta(x);\; V^\eta_{\tau_\epsilon}(x) \in B\big)
        }_{\mathrm{(I)}}
        \\&
        \quad
        +
        \underbrace{ \sup_{x \in A(\epsilon)}
            \P\big(\tau_{I(\epsilon)^\complement}^\eta(x) > t/\gamma(\eta);\; \rho_k^\eta(x) > t/\gamma(\eta);\; V^\eta_{\tau_\epsilon}(x) \in B\big)
        }_{\mathrm{(II)}}
\end{align*}
We first show that (II) vanishes as $\eta \to 0$. 
Our proof hinges on the following claim:
\begin{align*}
    \big\{
        \tau_{I(\epsilon)^\complement}^\eta (x) > t/ \gamma(\eta)
        ;\;
        % \quad\&\quad 
        \rho_k^\eta(x) > t/\gamma(\eta)
    \big\}
    % \quad
    % \implies 
    % \quad
    \ \subseteq\ 
    \bigcup_{j=1}^k\big\{\tau_{I(\epsilon)^\complement}^\eta(x) \wedge \rho_j^\eta(x) - \rho_{j-1}^\eta(x) \geq T/\eta\big\}
\end{align*}
Proof of the claim: 
Suppose that \(\tau_{I(\epsilon)^\complement}^\eta (x) > t/ \gamma(\eta)\) and \(\rho_k^\eta(x) > t/\gamma(\eta)\).
Let $k^* \delequal \max\{j\geq 1: \rho_j^\eta(x) \leq t/\gamma(\eta)\}$. 
Note that $k^* < k$.
We consider two cases separately: (i)
\(
\rho_{k^*}^\eta(x)/k^* > (t/\gamma(\eta) - T/\eta)/k^*
\)
and
(ii) 
\(
\rho_{k^*}^\eta(x) \leq t/\gamma(\eta) - T/\eta.
\)
In case of (i), since $\rho_{k^*}^\eta(x)/k^*$ is the average of \(\{\rho_j^\eta(x) - \rho_{j-1}^\eta(x): j=1,\ldots,k^*\}\), there exists $j^*\leq k^*$ such that
\[
    \rho_{j^*}^\eta(x) - \rho_{j^*-1}^\eta(x) > \frac{t/\gamma(\eta) - T/\eta}{k^*} \geq \frac{kT/\eta - T/\eta}{k-1} = T/\eta
\]
Note that since $\rho_{j^*}^\eta(x) \leq \rho_{k^*}^\eta(x) \leq t/\gamma(\eta) \leq \tau_{I(\epsilon)^\complement}^\eta(x)$, this proves the claim for case (i). 
For case (ii), note that
\[
    \rho_{k^*+1}^\eta(x) \wedge \tau_{I(\epsilon)^\complement}^\eta(x) - \rho_{k^*}^\eta(x)
    \geq t / \gamma(\eta)  - (t/\gamma(\eta) - T/\eta)   = T/\eta,
\]
which proves the claim.

Now, with the claim in hand, we have that
\begin{align*}
% & 
    \mathrm{\text{(II)}}
    % \\
    &
    \leq  
    \sum_{j=1}^k \sup_{x \in A(\epsilon)}
        \P\big(\tau_{I(\epsilon)^\complement}^\eta(x) \wedge \rho_j^\eta(x) - \rho_{j-1}^\eta (x) \geq T/\eta\big)
    \\
    &
    =
    \sum_{j=1}^k \sup_{x \in A(\epsilon)}
        \E\Big[
            \P\big(
                \tau_{I(\epsilon)^\complement}^\eta(x) \wedge \rho_j^\eta(x) - \rho_{j-1}^\eta (x) \geq T/\eta
            \big|
                \mathcal F_{\rho_{j-1}^\eta(x)}
            \big)
        \Big]
    \\
    &
    \leq
    \sum_{j=1}^k \sup_{y \in A(\epsilon)}
        \P\big(
            \tau_{I(\epsilon)^\complement}^\eta(y) \wedge \rho_1^\eta(y) \geq T/\eta        
        \big)
    \\
    &
    \leq 
    \frac{t}{\gamma(\eta)T/\eta}
    \cdot 
    \gamma(\eta) T/\eta \cdot \lo(1) 
    = \lo (1)
\end{align*}
for sufficiently large $T$'s, where the last inequality is from the definition of $k$ and \eqref{eq: exit time analysis: almost nothing happens between T over eta and rho one}.
We are now left with bounding (I) from above.
\begin{align*}
    \mathrm{(I)}
    &
    =  \sup_{x \in A(\epsilon)}
    \P
    \big(
        \tau_{I(\epsilon)^\complement}^\eta(x) > t/\gamma(\eta)\geq \rho_K^\eta(x);\; V^\eta_{\tau_\epsilon}(x) \in B
    \big)
    % \\&
    \leq  \sup_{x \in A(\epsilon)}
    \P
    \big(
        \tau_{I(\epsilon)^\complement}^\eta(x) 
        >
        \rho_K^\eta(x);\; V^\eta_{\tau_\epsilon}(x) \in B
    \big)
    \\&
    = 
    \sum_{j=k}^\infty \sup_{x \in A(\epsilon)}
        \P
        \Big(
            \tau_{I(\epsilon)^\complement}^\eta(x) \in \big(\rho_j^\eta(x),\,\rho_{j+1}^\eta(x)\big] ;\; 
            V^\eta_{\tau_\epsilon}(x) \in B
        \Big)
     % \textcolor{red}{\big[\rho_j^\eta(x),\,\rho_{j+1}^\eta(x)\big)
    % \text{ instead?}}
    \\
    &=
    \sum_{j=k}^\infty     \sup_{x \in A(\epsilon)}
    \E\bigg[
        \E\Big[
            \I\big\{V_{\tau_\epsilon}^\eta(x) \in B\big\}
            \cdot
            \I\big\{\tau_{I(\epsilon)^\complement}^\eta(x) \leq \rho_{j+1}^\eta(x)\big\}
        \Big|
            \mathcal F_{\rho_j^\eta(x)}
        \Big]
        \cdot
        \I\big\{\tau_{I(\epsilon)^\complement}^\eta(x) > \rho_j^\eta(x) \big\}
    \bigg]
    \\
    &\leq
    \sum_{j=k}^\infty     \sup_{x \in A(\epsilon)}
    \E\bigg[
        \sup_{y\in A(\epsilon)}
        \P\Big(
            V_{\tau_\epsilon}^\eta(y) \in B
            ;\;
            \tau_{I(\epsilon)^\complement}^\eta(y) \leq \rho_1^\eta(y)
        \Big)
        \cdot
        \I\big\{\tau_{I(\epsilon)^\complement}^\eta(x) > \rho_j^\eta(x) \big\}
    \bigg]    
    \\
    &=
    \sup_{y\in A(\epsilon)}
    \P\Big(
        V_{\tau_\epsilon}^\eta(y) \in B
        ;\;
        \tau_{I(\epsilon)^\complement}^\eta(y) \leq \rho_1^\eta(y)
    \Big)
    \cdot
    \sum_{j=k}^\infty     \sup_{x \in A(\epsilon)}
    \P\Big(\tau_{I(\epsilon)^\complement}^\eta(x) > \rho_j^\eta(x)\Big)
\end{align*}
The first term can be bounded via \eqref{eq: exit time analysis: upper bound of the probability that tau is before rho} with $R=1$:
\begin{align*}
    &
    \sup_{y\in A(\epsilon)}
    \P\Big(
        V_{\tau_\epsilon}^\eta(y) \in B
        ;\;
        \tau_{I(\epsilon)^\complement}^\eta(y) \leq \rho_1^\eta(y)
    \Big)
    \\
    &
    \leq
    \big(C(B^-)+\delta_B(\epsilon, T)+ \lo(1)\big) \cdot \gamma(\eta) T/\eta + \frac{1-r}2\cdot \gamma(\eta)T/\eta \cdot \lo(1)
\end{align*}
whereas the second term is bounded via \eqref{eq: exit time condition lower bound} of Condition~\ref{condition E2} as follows: 
\begin{align}
    &
    \sum_{j=k}^\infty \sup_{x \in A(\epsilon)}
    \P\Big(\tau_{I(\epsilon)^\complement}^\eta(x) > \rho_{j}^{\eta}(x)\Big)
    \nonumber\\
    &
    \leq 
    \sum_{j=0}^\infty
    \bigg(
        \sup_{y \in A(\epsilon)}
            \P\Big(\tau_{I(\epsilon)^\complement}^\eta(y) > \lfloor rT/\eta\rfloor\Big)
    \bigg)^{k+j}
    % \nonumber\\
    % &
    =
    \sum_{j=0}^\infty
        \bigg(
            1 -\inf_{y \in A(\epsilon)}
                \P\Big(\tau_{I(\epsilon)^\complement}^\eta(y) \leq rT/\eta \Big)
        \bigg)^{k+j}
    \nonumber\\
    &\leq
    \frac1{\inf_{y \in A(\epsilon)}\P\Big(\tau_{I(\epsilon)^\complement}^\eta(y) \leq rT/\eta \Big)}
    \cdot
    \bigg(
        1 - \inf_{y \in A(\epsilon)}\P\Big(\tau_{I(\epsilon)^\complement}^\eta(y) \leq rT/\eta \Big)
    \bigg)^{\frac{t/\gamma(\eta)}{T/\eta} -1}
    \nonumber\\
    &
    = 
    \frac{1}
    {r\cdot\big(1-\delta_B(\epsilon, rT)+\lo(1)\big)\cdot \gamma(\eta)T/\eta} 
    \cdot 
    \Big(
    1- r\cdot\big(1-\delta_B(\epsilon, rT)+\lo(1)\big)\cdot \gamma(\eta)T/\eta
    \Big)^{\frac{t}{\gamma(\eta)T/\eta}-1}
    \label{eq:exit time and location tail of tau lower bound}
    \nonumber
\end{align}
Therefore,
\begin{align*}
    \limsup_{\eta\downarrow 0}
    \sup_{x\in A(\epsilon)}
        \P\big(\gamma(\eta) \tau_{I(\epsilon)^\complement}^\eta(x) > t;\; V^\eta_{\tau_\epsilon}(x) \in B\big)
    &
    \leq
    \frac{C(B^-)+\delta_B(\epsilon,T)}{r\cdot (1-\delta_B(\epsilon,rT))} \cdot \exp\Big(-r \cdot\big(1-\delta_B(\epsilon, rT)\big)\cdot  t\Big).
\end{align*}
Again, taking $T\to\infty$ and considering $r$ arbitrarily close to 1, we can check that the desired upper bound holds. 

\medskip
\noindent
\textbf{Proof of Case $(ii)$.}

We work with $R=1$ and set
\(K = \left\lceil \frac{t/\gamma(\eta)}{T/\eta} \right\rceil\).
Again, for $\eta \in \big(0, (r-1)T\big)$,
we have
$\lfloor r T/\eta \rfloor \geq T/\eta$ and hence $\rho_K^\eta(x) \geq K \lfloor r T/\eta\rfloor \geq t/\gamma(\eta)$.
By the Markov property conditioning on $\mathcal F_{\rho_j^\eta(x)}$,
\begin{align*}
    &
    \sup_{x \in A(\epsilon)}\P\big(\gamma(\eta) \tau_{I(\epsilon)^\complement}^\eta(x) \leq t\big)
    \\ 
    & \leq 
    \sup_{x \in A(\epsilon)}\P\Big( \tau_{I(\epsilon)^\complement}^\eta(x) \leq \rho^\eta_K(x) \Big)
    =
    \sup_{x \in A(\epsilon)}\sum_{j = 1}^K\P
    \Big(\tau_{I(\epsilon)^\complement}^\eta(x) \in \big(\rho^\eta_{j-1}(x), \rho^\eta_j(x)\big] \Big)
    \\ 
    & \leq 
    \sum_{j = 1}^K
    \sup_{y \in A(\epsilon)}
    \bigg(
        1 -
        \P\Big(
        \tau_{I(\epsilon)^\complement}^\eta(y) \leq \rho_1^\eta(y)
    \Big)
    \bigg)^{j - 1}
    \cdot 
    \sup_{y \in A(\epsilon)}\P\Big(
        \tau_{I(\epsilon)^\complement}^\eta(y) \leq \rho_1^\eta(y)
    \Big)
    \\ 
    & \leq 
    K \cdot \sup_{y \in A(\epsilon)}\P\Big(
        \tau_{I(\epsilon)^\complement}^\eta(y) \leq \rho_1^\eta(y)
    \Big)
    \leq 
    K \cdot \big(\delta_{I^\complement}(\epsilon,T) + \lo(1)\big) \cdot \gamma(\eta) T/\eta
    \\ 
    &\qquad
    \text{by \eqref{eq: exit time analysis: upper bound of the probability that tau is before rho} (with $B = I^\complement$) and the running assumption of Case $(ii)$ that $C(\cdot)\equiv 0$}
    \\ 
    & \leq 
    \frac{2t/\gamma(\eta)}{T/\eta} \cdot 
     \big(\delta_{I^\complement}(\epsilon,T) + \lo(1)\big) \cdot \gamma(\eta) T/\eta
     \qquad
     \text{ for all $\eta$ small enough under }K = \ceil{\frac{t/\gamma(\eta)}{T/\eta}}
     \\
     & = 
     2t \cdot   \big(\delta_{I^\complement}(\epsilon,T) + \lo(1)\big).
\end{align*}
Lastly, by Condition~\ref{condition E2} (specifically, $\lim_{\epsilon \downarrow 0}\lim_{T \uparrow \infty}\delta_{I^\complement}(\epsilon,T) = 0$ in Definition~\ref{def: asymptotic atom}),
we verify the upper bounds in Case $(ii)$ and conclude the proof.
\end{proof}

Now, we are ready to prove Theorem~\ref{thm: exit time analysis framework}.
\begin{proof}[Proof of Theorem~\ref{thm: exit time analysis framework}]
\linksinpf{thm: exit time analysis framework} 
We focus on the proof of Case $(i)$ since the proof of Case $(ii)$ is almost identical, with the only key difference being that we apply part $(ii)$ of Proposition~\ref{prop: exit time analysis main proposition} instead of part $(i)$.

We first claim that for any $\epsilon, \epsilon' > 0$, $t\geq 0$, and measurable $B\subseteq \mathbb S$,
\begin{equation}\label{exit time and location analysis framework main theorem proof claim}
\begin{aligned}
C(B^\circ)\cdot e^{-t} - \delta_{t,B}(\epsilon)
&
\leq
\liminf_{\eta\downarrow 0} 
    \inf_{x\in I(\epsilon')}
        \P\Big(\gamma(\eta)\cdot\tau_{I(\epsilon)^\complement}^\eta (x) > t,\, V_{\tau_\epsilon}^\eta(x) \in B \Big)
\\
&
\leq
\limsup_{\eta\downarrow 0} 
    \sup_{x\in I(\epsilon')}
        \P\Big(\gamma(\eta)\cdot\tau_{I(\epsilon)^\complement}^\eta (x) > t,\, V_{\tau_\epsilon}^\eta(x) \in B \Big)
\leq 
C(B^-)\cdot e^{-t} + \delta_{t,B}(\epsilon)
\end{aligned}
\end{equation}
where $\delta_{t,B}(\epsilon)$ is characterized in part $(i)$ of Proposition \ref{prop: exit time analysis main proposition} such that $\delta_{t,B}(\epsilon) \to 0$ as $\epsilon \to 0$.
Now, note that for any measurable $B \subseteq I^\complement$,
\begin{align}
    &
    \P\Big(\gamma(\eta)\cdot \tau^\eta_{I^\complement}(x) > t,\, V_\tau^\eta (x) \in B\Big)
    \nonumber
    \\
    &
    =
    \underbrace{
        \P\Big(\gamma(\eta)\cdot \tau^\eta_{I^\complement}(x) > t,\, V_\tau^\eta (x) \in B,\, V^\eta_{\tau_\epsilon}(x) \in I\Big)
    }_{\mathrm{\text{(I)}}}
    +
    \underbrace{
        \P\Big(\gamma(\eta)\cdot \tau^\eta_{I^\complement}(x) > t,\, V_\tau^\eta (x) \in B,\, V^\eta_{\tau_\epsilon}(x) \notin I\Big)
    }_{\mathrm{\text{(II)}}}
    \nonumber
\end{align}
and since 
\[
    \mathrm{\text{(I)}}
    \leq 
    \P\Big(V^\eta_{\tau_\epsilon}(x) \in I\Big)
    \qquad\text{and}\qquad
    \mathrm{\text{(II)}}
    =\P\Big(\gamma(\eta)\cdot \tau^\eta_\epsilon(x) > t,\, V^\eta_{\tau_\epsilon}(x) \in B\setminus I\Big),
\]
we have that
\begin{align*}
    % &
    \liminf_{\eta\downarrow 0}
        \inf_{x\in I(\epsilon')}
            \P\Big(\gamma(\eta)\cdot \tau^\eta_{I^\complement}(x) > t,\, V_\tau^\eta (x) \in B\Big)    
    % \\
    &
    \geq
    \liminf_{\eta\downarrow 0}
        \inf_{x\in I(\epsilon')}
            \P\Big(\gamma(\eta)\cdot \tau^\eta_\epsilon(x) > t,\, V^\eta_{\tau_\epsilon}(x) \in B\setminus I\Big)
    \\
    &
    \geq
    C\big((B\setminus I)^\circ\big) \cdot e^{-t} - \delta_{t,B\setminus I}(\epsilon)
    \\[2pt]
    &
    =
    C(B^\circ)\cdot e^{-t} - \delta_{t,B\setminus I}(\epsilon)
\end{align*}
due to $B \subseteq I^\complement$, and
\begin{align*}
    &
    \limsup_{\eta\downarrow 0}
        \sup_{x\in I(\epsilon')}
            \P\Big(\gamma(\eta)\cdot \tau^\eta_{I^\complement}(x) > t,\, V_\tau^\eta (x) \in B\Big)
    \\
    &
    \leq
    \limsup_{\eta\downarrow 0}
        \sup_{x\in I(\epsilon')}
            \P\Big(\gamma(\eta)\cdot \tau^\eta_\epsilon(x) > t,\, V^\eta_{\tau_\epsilon}(x) \in B\setminus I\Big) 
    + 
    \limsup_{\eta\downarrow 0}
        \sup_{x\in I(\epsilon')}
            \P\Big(V^\eta_{\tau_\epsilon}(x) \in I\Big)
    \\
    &
    \leq
    C\big((B\setminus I)^-\big)\cdot e^{-t} + \delta_{t,B\setminus I}(\epsilon) + C (I^-) + \delta_{0,I}(\epsilon)
    \\[3pt]
    &
    = 
    C(B^-)\cdot e^{-t} + \delta_{t,B\setminus I}(\epsilon) + \delta_{0,I}(\epsilon).
\end{align*}
Taking $\epsilon\to0$, we arrive at the desired lower and upper bounds of the theorem.
Now we are left with the proof of the claim~\eqref{exit time and location analysis framework main theorem proof claim} is true. 
Note that for any $x \in I$,
\begin{align}
    &
    \P\Big(\gamma(\eta)\cdot \tau^\eta_\epsilon(x) > t,\, V^\eta_{\tau_\epsilon}(x) \in B\Big)
    \nonumber
    \\
    &
    =
    \E
    \bigg[
        \P\Big(\gamma(\eta)\cdot \tau^\eta_\epsilon(x) > t,\, V^\eta_{\tau_\epsilon}(x) \in B\Big| \mathcal F_{\tau_{A(\epsilon)}^\eta(x)}\Big)
        \cdot 
        \Big( \I\big\{\tau_{A(\epsilon)}^\eta (x) \leq T/\eta\big\} + \I\big\{\tau_{A(\epsilon)}^\eta (x) > T/\eta\big\}\Big)
    \bigg]
    \label{eq:exit time and location analysis framework: proof of the main theorem: decomposition in the proof of the claim}
\end{align}
Fix an arbitrary $s>0$, and note that from the Markov property,
\begin{align*}
    &
    \P\Big(\gamma(\eta)\cdot \tau^\eta_\epsilon(x) > t,\, V^\eta_{\tau_\epsilon}(x) \in B\Big)
    \\
    &
    \leq
    \E\bigg[
        \sup_{y \in A(\epsilon)}
            \P
            \Big(
                \tau^\eta_\epsilon(y) > t/\gamma(\eta) - T/\eta,\, V^\eta_{\tau_\epsilon}(y) \in B
            \Big)
            \cdot 
            \I\big\{\tau_{A(\epsilon)}^\eta (x) \leq T/\eta\big\}
    \bigg] 
    % \\
    % &
    % \quad
    + 
    \P\Big(\tau_{A(\epsilon)}^\eta (x) > T/\eta\Big)
    \\
    &
    \leq
    \sup_{y \in A(\epsilon)}
        \P
        \Big(
            \gamma(\eta)\cdot \tau^\eta_\epsilon(y) > t-s,\, V^\eta_{\tau_\epsilon}(y) \in B
        \Big)
    % \\
    % &
    % \quad
    + 
    \P\Big(\tau_{A(\epsilon)}^\eta (x) > T/\eta\Big)
\end{align*}
for sufficiently small $\eta$'s;
here, we applied $\gamma(\eta)/\eta \to 0$ as $\eta \downarrow 0$ in the last inequality.
In light of part $(i)$ of Proposition~\ref{prop: exit time analysis main proposition},
by taking  $T \to \infty$ we yield
\begin{align*}
    \limsup_{\eta\downarrow 0}
        \sup_{x\in I(\epsilon')}
            \P\Big(\gamma(\eta)\cdot \tau^\eta_\epsilon(x) > t,\, V^\eta_{\tau_\epsilon}(x) \in B\Big)
    &
    \leq
    C(B^-)\cdot e^{-(t-s)} + \delta_{t,B}(\epsilon)
\end{align*}
Considering an arbitrarily small $s>0$, we get the upper bound of the claim~\eqref{exit time and location analysis framework main theorem proof claim}. 
For the lower bound, again from \eqref{eq:exit time and location analysis framework: proof of the main theorem: decomposition in the proof of the claim} and the Markov property,
\begin{align*}
    &
    \liminf_{\eta\downarrow 0}
        \inf_{x\in I(\epsilon')}
            \P\Big(\gamma(\eta)\cdot \tau^\eta_\epsilon(x) > t,\, V^\eta_{\tau_\epsilon}(x) \in B\Big)
    \\
    &
    \geq 
    \liminf_{\eta\downarrow 0}
        \inf_{x\in I(\epsilon')}
            \E\bigg[
                \inf_{y \in A(\epsilon)}
                    \P
                    \Big(
                        \tau^\eta_\epsilon(y) > t/\gamma(\eta),\, V^\eta_{\tau_\epsilon}(y) \in B
                    \Big)
                \cdot 
                \I\big\{\tau_{A(\epsilon)}^\eta (x) \leq T/\eta\big\}
            \bigg] 
    \\
    &
    \geq 
    \liminf_{\eta\downarrow 0}       
        \inf_{y \in A(\epsilon)}
            \P
            \Big(
                \gamma(\eta)\cdot \tau^\eta_\epsilon(y) > t,\, V^\eta_{\tau_\epsilon}(y) \in B
            \Big)
        \cdot 
        \inf_{x\in I(\epsilon')}
            \P\big(\tau_{A(\epsilon)}^\eta (x) \leq T/\eta\big)
    \\
    &
    \geq
    C(B^\circ) - \delta_{t,B}(\epsilon),
\end{align*}
which is the desired lower bound of the claim~\eqref{exit time and location analysis framework main theorem proof claim}. 
This concludes the proof. 
\end{proof}

\section{Proof of Theorem~\ref{theorem: first exit time, unclipped}}
\label{sec: proof of proposition: first exit time}

In this section, we apply the framework developed in Section~\ref{subsec: framework, first exit time analysis}
and prove Theorem~\ref{theorem: first exit time, unclipped}.
Throughout this section,
we impose Assumptions \ref{assumption gradient noise heavy-tailed}, \ref{assumption: lipschitz continuity of drift and diffusion coefficients}, and \ref{assumption: shape of f, first exit analysis}.

We start by fixing a few constants.
Recall the definition of the discretized width metric $\mathcal J^I_b$ defined in \eqref{def: first exit time, J *}.
To prove Theorem~\ref{theorem: first exit time, unclipped},
in this section
we fix some $b > 0$ such that the conditions in  Theorem~\ref{theorem: first exit time, unclipped} hold.
This allows us to fix some $\widecheck \epsilon > 0$ small enough such that 
\begin{align}
    \bar B_{\widecheck{\epsilon}}(\bm 0) \subseteq I_{\widecheck{\epsilon}},
    \qquad
    \bm a(\bm x)\bm x < 0\ \forall \bm x \in \bar B_{\widecheck \epsilon}(\bm 0)\setminus\{\bm 0\},
    \quad
    \inf\Big\{ 
        \norm{\bm x - \bm y}:\ 
        \bm x \in I^\complement,\ \bm y \in \mathcal G^{( \mathcal J^I_b - 1 )|b}(\widecheck \epsilon)
    \Big\}
    > 0.
    \label{constant bar epsilon, first exit time analysis}
\end{align}
Here, $\bar B_r(\bm x) = \{ \bm x:\ \norm{\bm x} \leq r  \}$ is the closed ball with radius $r$ centered at $\bm x$.
An implication of the first condition in \eqref{constant bar epsilon, first exit time analysis} is the following positive invariance property under the gradient field $\bm a(\cdot)$:
for any $r \in (0,\widecheck \epsilon]$, 
\begin{align}
     \bm x \in \bar B_{ r }(\bm 0)
    \quad \Longrightarrow \quad 
    \bm y_t(\bm x) \in \bar B_{ r }(\bm 0)\ \forall t \geq 0,
    \label{property: contraction of the ODE around the origin}
\end{align}
with the ODE $\bm y_t(x)$ defined in \eqref{def ODE path y t}.
Next, for any $\epsilon \in (0,\widecheck \epsilon)$, let
\begin{align}
    \notationdef{notation-domain-check-I-epsilon}{\widecheck I(\epsilon)} \delequal 
    \Big\{
        \bm x \in I:\  \norm{ \bm y_{1/\epsilon}(\bm x) } < \widecheck \epsilon
    \Big\}
    \label{def: covering sets I epsilon, first exit time}
\end{align}
By Gronwall's inequality, it is easy to see that $\widecheck I(\epsilon)$ is an open set.
Meanwhile, by Assumption \ref{assumption: shape of f, first exit analysis},
given $\bm x \in I$ we must have $\bm x \in \widecheck I(\epsilon)$ for any $\epsilon > 0$ small enough.
As a result, the collection of open sets $\{ \widecheck I(\epsilon):\ \epsilon \in (0,\widecheck \epsilon) \}$
provides a covering for $I$:
\begin{align*}
    \bigcup_{\epsilon \in (0,\widecheck \epsilon)}\widecheck I(\epsilon) = I.
\end{align*}
% see Section~\ref{subsec: framework, first exit time analysis} and Condition~\ref{condition E2} for the definition of a covering for asymptotic atoms.
% We will also make use of the following property:
% given $0 < \epsilon^\prime < \epsilon$,
% \begin{align}
%     \inf\Big\{
%         \norm{\bm x - \bm z}:\ \bm x \in I(\epsilon),\ \bm z \notin I(\epsilon^\prime)
%     \Big\} > 0.
%     \label{property: bounded away, I epsilon, different epsilon}
% \end{align}
% To see why, it suffices to consider a straightforward proof by contradiction.
% Suppose that we have sequences $\bm x_n \in I(\epsilon)$ and $\bm z_n \notin I(\epsilon^\prime)$
% such that
% $
% \norm{\bm x_n - \bm z_n } \leq 1/n.
% $
% Then from the boundedness of $I$ (and hence $I(\epsilon)$),
% by picking a converging sub-sequence if necessary we can w.l.o.g.\ assume that $\bm x_n \to \bm x^*$ for some $\bm x^* \in \big(I(\epsilon)\big)^-$,
% and hence $\norm{\bm z_n - \bm x^*} \to 0$.
% However, due to $\bm z_n \notin I(\epsilon^\prime)$,
% we must have $\norm{\bm y_{1/\epsilon^\prime}(\bm z_n)} \geq \widecheck \epsilon$ for all $n$.
% Gronwall's inequality then implies that $\norm{\bm y_{1/\epsilon^\prime}(\bm x^*)} \geq \widecheck \epsilon$.
% By \eqref{property: contraction of the ODE around the origin},
% we yield the contradiction that $\norm{\bm y_{1/\epsilon}(\bm x^*)} \geq \widecheck \epsilon$, and hence $\bm x^* \notin I(\epsilon)$.
% This contradiction verifies claim \eqref{property: bounded away, I epsilon, different epsilon}.

Next, recall that we use 
$
{I_\epsilon} = \{ \bm y\in \R^n:\  \norm{\bm x - \bm y} < \epsilon\ \Longrightarrow\ \bm x \in I  \}
$
to denote the $\epsilon$-shrinkage of the set $I$.
Given any $\epsilon > 0$,
note that $I_\epsilon$ is an open set and, by definition, its closure $(I_\epsilon)^-$ is still bounded away from $I^\complement$,
i.e.,
$
\norm{\bm x-\bm y} \geq \epsilon
$
for all $\bm x\in (I_\epsilon)^-$, $\bm y \in I^\complement$.
Besides, the continuity of $\bm a(\cdot)$ (see Assumption~\ref{assumption: lipschitz continuity of drift and diffusion coefficients})
implies the continuity (w.r.t.\ $\bm x$) of $\bm y_t(\bm x)$'s hitting time to $\bar B_{\widecheck{\epsilon}}(\bm 0)$.
Then, by
the boundedness of set $I$ and hence $(I_\epsilon)^- \subseteq I$,
as well as property \eqref{property: contraction of the ODE around the origin},
we know that
given any $\epsilon > 0$, the claim
\begin{align*}
    \norm{\bm y_T(\bm x)} < \widecheck \epsilon\qquad \forall \bm x \in (I_\epsilon)^-
\end{align*}
holds for any $T > 0$ large enough.
This confirms that
given $\epsilon > 0$, it holds for any $\epsilon^\prime > 0$ small enough that 
\begin{align}
    (I_\epsilon)^- \subseteq \widecheck I(\epsilon^\prime).
    \label{property, I epsilon and I _ epsilon }
\end{align}

As a direct consequence of the discussion above,
we highlight another important property of the sets $\mathcal G^{(k)|b}(\epsilon)$ defined in \eqref{def: set G k b epsilon}.
For any $k \in \mathbb N$, $b > 0$, and $\epsilon \geq 0$, let
\begin{align}
    \notationdef{notation-extended-coverage-set-bar-G-k-b-epsilon}{\bar{\mathcal G}^{(k)|b}(\epsilon)}
    \delequal
    \Big\{
        \bm y_t(\bm x):\ \bm x \in \mathcal G^{(k)|b}(\epsilon),\ t \geq 0
    \Big\},
    \label{def: bar G k b epsilon, extended k jump coverage set}
\end{align}
where $\bm y_\cdot(\bm x)$ is the ODE defined in \eqref{def ODE path y t}.
First, due to \eqref{property, I epsilon and I _ epsilon } and the fact that $\mathcal G^{(\mathcal J^I_b - 1)|b}(\widecheck{\epsilon})$ is bounded away from $I^\complement$ (see \eqref{constant bar epsilon, first exit time analysis}),
given any $\epsilon \in (0,\widecheck{\epsilon}]$, it holds for all $\epsilon^\prime > 0$ small enough that 
$
\mathcal G^{(\mathcal J^I_b - 1)|b}(\epsilon) \subseteq \widecheck I(\epsilon^\prime).
$
Furthermore, we claim that $\bar{\mathcal G}^{(\mathcal J^I_b - 1)|b}(\widecheck\epsilon)$ is also bounded away from $I^\complement$, i.e.,
\begin{align}
    \inf\Big\{
        \norm{\bm x - \bm z}:\ \bm x \in \bar{\mathcal G}^{(\mathcal J^I_b - 1)|b}(\widecheck\epsilon),\ \bm z \in I^\complement
    \Big\}
    > 0.
    \label{property, bar G k b epsilon bounded away from I complement}
\end{align}
This can be argued with a proof by contradiction.
Suppose that there exist sequences $\bm x_n^\prime \in \bar{\mathcal G}^{(\mathcal J^I_b - 1)|b}(\widecheck\epsilon)$ and $\bm z_n \notin I$
such that $\norm{\bm x_n^\prime - \bm z_n} \leq 1/n$.
By definition of $\bar{\mathcal G}^{(\mathcal J^I_b - 1)|b}(\widecheck\epsilon)$,
there exist sequences $\bm x_n \in {\mathcal G}^{(\mathcal J^I_b - 1)|b}(\widecheck\epsilon)$ and $t_n \geq 0$ such that
$
\bm x^\prime_n = \bm y_{t_n}(\bm x_n)
$
for each $n\geq 1$.
Besides, we have shown that
$
\mathcal G^{(\mathcal J^I_b - 1)|b}(\widecheck \epsilon) \subseteq\widecheck I(\epsilon)
$
holds
for any $\epsilon > 0$ small enough.
Then under such $\epsilon>0$, by \eqref{property: contraction of the ODE around the origin}\eqref{def: covering sets I epsilon, first exit time},
we must have $t_n \leq 1/\epsilon$ for each $n$;
otherwise, we will have $\bm x^\prime_n = \bm y_{t_n}(\bm x_n) \in \bar B_{\widecheck{\epsilon}}(\bm 0)$, 
which prevents  $\norm{\bm x_n^\prime - \bm z_n} \leq 1/n$ to hold for the sequence $\bm z_n \in I^c$ due to 
$\bar B_{\widecheck{\epsilon}}(\bm 0) \subseteq I_{\widecheck\epsilon}$; see
\eqref{constant bar epsilon, first exit time analysis}.
Now, in light of the boundedness for the sequence $\bm x_n \in \bar{\mathcal G}^{(\mathcal J^I_b - 1)|b}(\widecheck\epsilon) \subset I$ (due to the boundedness of $I$) and the boundedness of the sequence $t_n$,
by picking a sub-sequence if necessary, we can w.l.o.g.\ assume that $\bm x_n \to \bm x_*$ for some $\bm x_* \in \big(\mathcal{G}^{(\mathcal J^I_b - 1)|b}(\widecheck\epsilon)\big)^- \subset I$
and $t_n \to t_*$ for some $t_* \in [0,\infty)$.
Let $\bm x^\prime_* \delequal \bm y_{t_*}(\bm x_*)$.
Since $\bm x_* \in I$,
by Assumption~\ref{assumption: shape of f, first exit analysis}
we must have $\bm x^\prime_* \in I$.
However, the continuity of the flow (specifically, using Gronwall's inequality) implies $\bm x^\prime_n \to \bm x^\prime_*$.
Under the condition $\norm{\bm x_n^\prime - \bm z_n} \leq 1/n$ and $\bm z_n \in I^\complement$ for each $n$,
we arrive at the contradiction that $\bm z_n \to \bm x^\prime_* \in I^\complement$ due to $I$ being an open set and $I^\complement$ being closed.
This verifies \eqref{property, bar G k b epsilon bounded away from I complement}.
Summarizing \eqref{constant bar epsilon, first exit time analysis}, \eqref{property: contraction of the ODE around the origin}, and \eqref{property, bar G k b epsilon bounded away from I complement},
we can fix some $\bar \epsilon > 0$ small enough such that the following claims hold:
\begin{align}
    &\bar B_{\bar{\epsilon}}(\bm 0) \subseteq I_{\bar\epsilon},
    \label{constant bar epsilon, new, 1, first exit time analysis}
    \\ 
    &r \in (0,\bar\epsilon],\ \bm x \in \bar B_r(\bm 0)
    \quad \Longrightarrow \quad 
    \bm y_t(\bm x) \in \bar B_r(\bm 0)\ \forall t \geq 0,
    \label{constant bar epsilon, new, 2, first exit time analysis}
    \\ 
     &\inf\Big\{
        \norm{\bm x - \bm z}:\ \bm x \in \bar{\mathcal G}^{(\mathcal J^I_b - 1)|b}(2\bar\epsilon),\ \bm z \notin I_{\bar\epsilon}
    \Big\}
     > \bar\epsilon.
    \label{constant bar epsilon, new, 3, first exit time analysis}
\end{align}

Moving on, let $\notationdef{notation-hitting-time-t-x-epsilon}{\bm t_{\bm x}(\epsilon)} \delequal \inf\Big\{ t \geq 0:\ \bm y_t(\bm x) \in \bar B_{\epsilon}(\bm 0) \Big\}$
be the hitting time to the closed ball $\bar B_{\epsilon}(\bm 0)$ for the ODE $\bm y_t(\bm x)$,
and let
\begin{align}
    \notationdef{notation-t-epsilon-ode-return-time}{ \bm{t}(\epsilon) } \delequal
    \sup\Big\{ \bm t_{\bm x}(\epsilon):\ \bm x \in (I_\epsilon)^-   \Big\}
    \label{def: t epsilon function, first exit analysis}
\end{align}
be the upper bound for the hitting times $\bm t_{\bm x}(\epsilon)$ over $\bm x \in (I_\epsilon)^-$.
By the continuity of $\bm a(\cdot)$ (and hence $\bm t_{\bm x}(\epsilon)$ w.r.t.\ $\epsilon$), 
the contraction of $\bm y_t(\bm x)$ around the origin (see Assumption~\ref{assumption: shape of f, first exit analysis} and its implication \eqref{constant bar epsilon, new, 2, first exit time analysis}), 
and the boundedness of $(I_\epsilon)^-$,
we have $\bm t(\epsilon) < \infty$ for any $\epsilon > 0$.
Besides, by the definition of $\bm{t}(\cdot)$, we have (for any $\epsilon \in (0,\bar\epsilon)$)
\begin{align}
    \bm{y}_{ t }(\bm x) \in \bar B_{\epsilon}(\bm 0) \qquad \forall \bm x \in (I_\epsilon)^-,\ t \geq \bm{t}(\epsilon).
    \label{property: t epsilon function, first exit analysis}
\end{align}
Furthermore, by repeating the arguments for \eqref{property, bar G k b epsilon bounded away from I complement},
one can show that (for any $\epsilon > 0$)
\begin{align}
    \inf\Big\{
        \norm{ \bm y_t(\bm x) - \bm z  }:\ 
        \bm x \in (I_\epsilon)^-,\ t \geq 0,\ \bm z \notin I
    \Big\} > 0.
    \label{property: coverage of I_epsilon, bounded away from I complement}
\end{align}
Specifically,
as a direct consequence of \eqref{property: coverage of I_epsilon, bounded away from I complement}, there is some $\bar c \in (0,1)$ such that
\begin{align}
    \Big\{
        \bm y_t(\bm x):\ \bm x \in (I_{\bar\epsilon})^-,\ t \geq 0
    \Big\}
    \subseteq I_{\bar c\bar\epsilon},
    \label{constant bar c for bar epsilon, first exit time}
\end{align}
where  $\bar\epsilon > 0$ is the constant fixed in \eqref{constant bar epsilon, new, 1, first exit time analysis}--\eqref{constant bar epsilon, new, 3, first exit time analysis}.

\subsection{Technical Lemmas for $\widecheck{ \mathbf C }^{(k)|b}$}

Recall that
$
{I_\epsilon} = \{ \bm y:\ \norm{\bm x - \bm y} < \epsilon\ \Longrightarrow\ \bm x \in I \}
$
is  the $\epsilon$-shrinkage of the domain $I$,
and that $I^-_\epsilon$ is the closure of $I_\epsilon$.
We first study the mapping ${\widecheck{g}^{(k)|b}}$ in \eqref{def: mapping check g k b, endpoint of path after the last jump, first exit analysis},
which is defined based on 
$\bar h^{(k)|b}_{[0,T]}$ and $h^{(k)|b}_{[0,T]}$ defined in
\eqref{def: perturb ode mapping h k b, 1}--\eqref{def: perturb ode mapping h k b, 4}.
Some of the results below will impose the following boundedness assumption on $\bm a$ and $\bm \sigma$.
\begin{assumption}[Boundedness]  
\label{assumption: boundedness of drift and diffusion coefficients}
There exists $\notationdef{notation-constant-C-boundedness-assumption}{C} > 0$ such that
\begin{align*}
    \norm{\bm a(\bm x)} \vee \norm{\bm \sigma(\bm x)} \leq  C,\qquad \forall \bm x \in \mathbb{R}^m.
\end{align*}
\end{assumption}

\begin{lemma}
\label{lemma: choose key parameters, first exit time analysis}
\linksinthm{lemma: choose key parameters, first exit time analysis}
    % Let $k \in \mathbb{Z}_+, b > 0$.
    % Let $\gamma,\bar\epsilon > 0$ be such that $\gamma > (k-1)b + 3\bar\epsilon$.
    % Suppose that $a(0) = 0$ and $a(x)x < 0\ \forall x \in (-\gamma,0)\cup(0,\gamma)$.
    % Under Assumption \ref{assumption: f and sigma, stationary distribution of SGD},
    Let Assumptions~\ref{assumption: lipschitz continuity of drift and diffusion coefficients}  and \ref{assumption: shape of f, first exit analysis}
    % and \ref{assumption: shape of f, first exit analysis} 
    hold.
    Let $\bar\epsilon > 0$ be the constant in \eqref{constant bar epsilon, new, 1, first exit time analysis}--\eqref{constant bar epsilon, new, 3, first exit time analysis}. 
    Let $C \in [1,\infty)$ be such that $\sup_{\bm x \in I^-}\norm{\bm a(\bm x)} \vee \norm{\bm \sigma(\bm x)} \leq C$.
    (Below, we adopt the convention that $t_0 = 0$.)
    \begin{enumerate}[(a)]

        % \item  
        % \textcolor{red}{??? do we need part (b) ???}
        % Suppose that $\mathcal J^I_b \geq 2$.
        % Given any $T > 0$ and 

        % then
        % it holds for all $T > 0$, $\bm x \in \bar B_{ b + \bar\epsilon }$, $\bm{w} = (w_1,\cdots,w_{\mathcal J^I_b-2})\in \R^{\mathcal J^I_b-2}$, and $\bm{t} = (t_1,\cdots,t_{\mathcal J^I_b-2})\in (0,T]^{ \mathcal J^I_b-2 \uparrow}$ that
        % \begin{align*}
        % \sup_{t \in [0,T]}|\xi(t)| \leq (\mathcal J^I_b - 1)b +\bar\epsilon < l - 2\Bar{\epsilon}\qquad\text{ where }
        % \xi = h^{(\mathcal J^I_b-2)|b}_{[0,T]}(x_0,\bm{w},\bm{t}).
        % \end{align*}

        \item Given any $T > 0$, the claim $\xi(t) \in I^-_{2\bar\epsilon}\ \forall t \in[0,T]$
        holds for all 
        $
        \xi \in \D^{( \mathcal J^I_b - 1)|b}_{ \bar B_{\bar\epsilon} }[0,T](\bar\epsilon).
        $

        \item 
            Let $\bar c \in (0,1)$ be the constant fixed in \eqref{constant bar c for bar epsilon, first exit time}.
        There exist $\Bar{\delta}>0$ and $\Bar{t} > 0$ such that the following claim holds:
        Given any $T > 0$ and $\bm x_0 \in \R^m$ with $\norm{\bm x_0} \leq \bar\epsilon$,
        if
        \begin{align}
            \xi(t) \notin I_{\bar c\bar\epsilon}\qquad\text{ for some }
            \xi = h^{(\mathcal J^I_b-1)|b}_{[0,T]}\Big( \bm x_0 + \varphi_b\big( \bm \sigma(\bm x_0)\bm w_0\big),\textbf W,(t_1,\cdots,t_{\mathcal J^I_b - 1})\Big),\ t \in [0,T],
            \label{claim, def of xi, lemma: choose key parameters, first exit time analysis}
        \end{align}
        where
        % $\bm w_0 \in \R^d$, 
        $\textbf W = (\bm w_1,\cdots,\bm w_{\mathcal J^I_b-1}) \in \R^{d \times \mathcal J^I_b-1}$, and 
        $(t_1,\cdots,t_{\mathcal J^I_b-1})\in (0,T]^{ \mathcal J^I_b-1 \uparrow}$,
        then
        \begin{enumerate}[(i)]
        \item 
            $\xi(t) \in I_{2\bar\epsilon}^-$ for all $t \in [0,t_{\mathcal J^I_b - 1})$;
            % $\sup_{t \in [0,t_{\mathcal J^I_b-1})}|\xi(t)| \leq (\mathcal J^I_b - 1)b +\bar\epsilon < l - 2\Bar{\epsilon}$;
            
        \item 
            $\xi(t_{\mathcal J^I_b-1})\notin I_{\bar\epsilon}$;

        \item 
            $\norm{\xi(t)} \geq \bar\epsilon$ for all $t \leq t_{\mathcal J^I_b - 1}$;

         \item 
            $t_{\mathcal J^I_b-1} < \Bar{t}$;

        \item 
            $\norm{\bm w_j} > \Bar{\delta}$ for all $j = 0,1,\cdots,\mathcal J^I_b - 1$.

        \end{enumerate}

        \item 
            Let $T > 0$, 
            $\bm x \in \R^m,\ \textbf W = (\bm w_1,\cdots, \bm w_{\mathcal J^I_b}) \in \R^{d \times  \mathcal J^I_b},\ (t_1,\cdots,t_{\mathcal J^I_b})\in (0,T]^{\mathcal J^I_b\uparrow}$, and $\epsilon \in (0,\bar\epsilon)$.
            Let
            \begin{align*}
                \xi & =  h^{ (\mathcal J^I_b)|b }_{[0,T]}\big(\bm x, \textbf W, (t_1,\cdots,t_{\mathcal J^I_b})\big),
                \\ 
                \widecheck{\xi}
                & =
                 h^{ (\mathcal J^I_b - 1)|b }_{[0,T]}
                \Big(
                    \varphi_b\big( \bm \sigma(\bm 0)\bm w_1\big), (\bm w_2,\cdots, \bm w_{ \mathcal J^I_b }), 
                    (t_2 - t_1,t_3 - t_1,\cdots,t_{\mathcal J^I_b} - t_1)
                \Big).
            \end{align*}
            If $\norm{\xi(t_1-)} < \epsilon$ 
            and
            $
            \norm{\bm w_j} \leq \epsilon^{-\frac{1}{2\mathcal J^I_b}}\ \forall j \in [\mathcal J^I_b],
            $
            then
        \begin{align*}
           \sup_{t \in [t_1,t_{\mathcal J^I_b}]} 
           \norm{\xi(t) - \widecheck\xi(t - t_1)} \leq \Big( 2\exp\big(D( t_{\mathcal J^I_b}  - t_1)\big) \cdot D\Big)^{\mathcal J^I_b + 1} \cdot  \epsilon,
        \end{align*}
        where
        $D \geq 1$ is the constant in Assumption~\ref{assumption: lipschitz continuity of drift and diffusion coefficients}.
        % ,
        % and $c \in (0,1]$ is the constant in \eqref{uniform nondegeneracy condition, first exit time analysis}.

        \item 
            Let $\bar c \in (0,1)$ be the constant fixed in \eqref{constant bar c for bar epsilon, first exit time}.
        Given $\Delta > 0$, there exists $\epsilon_0 = \epsilon_0(\Delta) \in (0,\bar\epsilon)$ such that
        the following claim holds:
        given 
        $T > 0$, 
        $\bm x \in \R^m$,
        $\textbf W = (\bm w_1,\cdots, \bm w_{\mathcal J^I_b}) \in \R^{d \times  \mathcal J^I_b},\ (t_1,\cdots,t_{\mathcal J^I_b})\in (0,T]^{\mathcal J^I_b\uparrow}$,
        if $\norm{\bm x} \leq \epsilon_0$
        and
        $
            \max_{j \in [\mathcal J^I_b] }\norm{\bm w_j} \leq \epsilon_0^{-\frac{1}{2\mathcal J^I_b}},
        $
        then        
        % for any $T > 0$, $x \in [-\epsilon_0$, $\epsilon_0]$, $\bm{w} = (w_1,\cdots,w_{\mathcal J^I_b}) \in \R^{\mathcal J^I_b}$, and $\bm{t} = (t_1,\cdots,t_{\mathcal J^I_b}) \in (0,T]^{{\mathcal J^I_b}\uparrow}$,
        \begin{align*}
        \xi(t) \notin I_{\bar c\bar \epsilon}\text{ or }\widecheck{\xi}(t)\notin I_{\bar c\bar \epsilon}
        \text{ for some }t \in [t_1,T - t_1]
        \qquad
        \Longrightarrow
        \qquad
        \sup_{t \in [t_1,t_{\mathcal J^I_b}]}
        \norm{\widecheck{\xi}(t - t_1) - \xi(t) } < \Delta,
        \end{align*}
        where $\xi$ and $\widecheck{\xi}$ are defined as in part $(c)$.
    \end{enumerate}
\end{lemma}

\begin{proof}
\linksinpf{lemma: choose key parameters, first exit time analysis}
Before the proof of the claims,
we highlight two facts.
First, Assumption~\ref{assumption: lipschitz continuity of drift and diffusion coefficients} and $I$ being a bounded set (so $I^-$ is compact) imply the existence of $C \in (0,\infty)$ such that $\sup_{\bm x \in I^-}\norm{\bm a(\bm x)} \vee \norm{\bm \sigma(\bm x)} \leq C$.
Without loss of generality, in the statement of Lemma~\ref{lemma: choose key parameters, first exit time analysis} we pick some $C \geq 1$.
Next,
one can see that the validity of all claims do not depend on the values of $\bm \sigma(\cdot)$ and $\bm a(\cdot)$ outside of $I^-$. 
% Take part $(a)$ as an example.
% Suppose that we can prove part $(a)$ under the stronger assumption that $\sup_{x \in \R}|a(x)| \wedge \sigma(x) \leq C$ for some $C \in [1,\infty)$
% and $\inf_{x \in \R}\sigma(x) \geq c$ for some $c \in (0,1]$.
% Then due to $\sup_{t \in [0,T]}|\xi(t)| < l = |s_\text{left}| \wedge s_\text{right}$ for $\xi = h^{(\mathcal J^I_b-2)|b}_{[0,T]}(x_0,\bm{w},\bm{t})$,
% we have $\xi(t) \in I^-$ for all $t \in [0,T]$.
% This implies that part $(a)$ is still valid even if we only have 
% $\sup_{x \in I^-}|a(x)| \wedge \sigma(x) \leq C$
% and $\inf_{x \in I^-}\sigma(x) \geq c$.
% The same applies to all the other claims.
Therefore, throughout this proof below we w.l.o.g.\ assume that
\begin{align}
    \norm{\bm a(\bm x)}\vee \norm{\bm \sigma(\bm x)} \leq C
    \qquad
    \forall \bm x \in \R^m.
    \label{constant C, boundedness of a and sigma, lemma: choose key parameters, first exit time analysis}
\end{align}
for some $C \in [1,\infty)$.
That is, we impose the boundedness condition in Assumption~\ref{assumption: boundedness of drift and diffusion coefficients}.

% \medskip
% $(a)$
% The proof hinges on the following observation. 
% For any $j \geq 0, T > 0, x_0 \in \R, \bm{w} = (w_1,\cdots,w_j) \in \R^j$ and $\bm t = (t_1,\cdots,t_j) \in (0,T]^{j\uparrow}$,
% let $\xi = h^{(j)|b}_{[0,T]}(x_0,\bm w,\bm t)$.
% The condition $a(x)x \leq 0$ implies that 
% \begin{align}
%     \frac{d |\xi(t)|}{dt} = -\big|a\big(\xi(t)\big)\big|\ \ \ \forall t \in [0,T]\setminus \{t_1,\cdots,t_j\}
%     \label{proof, observation on xi, lemma: choose key parameters, first exit time analysis}
% \end{align}
% Specifically, suppose that ${\mathcal J^I_b} \geq 2$.
% For all $T > 0, x_0 \in [-b - \Bar{\epsilon},b + \bar\epsilon], \bm{w} = (w_1,\cdots,w_{{\mathcal J^I_b}-2})\in \R^{{\mathcal J^I_b}-2}$ and $\bm{t} = (t_1,\cdots,t_{{\mathcal J^I_b}-2})\in (0,T]^{ {\mathcal J^I_b}-2 \uparrow}$,
% it holds for $\xi = h^{({\mathcal J^I_b}-2)|b}_{[0,T]}(x_0,\bm{w},\bm{t})$ that
% $
% d|\xi(t)|/dt \leq 0
% $
% for any $t \in [0,T]\setminus \{t_1,\cdots,t_{\mathcal{J}^I_b-2}\}$,
% thus leading to
% \begin{align*}
%     \sup_{t \in [0,T]}|\xi(t)| 
%     & \leq |\xi(0)| + \sum_{t \leq T}|\Delta \xi(t)| 
%     \\
%     &
%     \leq |\xi(0)| + ({\mathcal J^I_b}-2)b
%     \qquad \text{due to truncation operators $\varphi_b$ in $h^{({\mathcal J^I_b}-2)|b}_{[0,T]}$}
%     \\
%     & 
%     \leq b + \bar\epsilon + ({\mathcal J^I_b}-2)b 
%     \\
%     & = ({\mathcal J^I_b}-1)b + \bar\epsilon < l - 2\bar\epsilon
%     \qquad \text{due to \eqref{constant bar epsilon, first exit time analysis}}.
% \end{align*}
% This concludes the proof of part $(a)$.

\medskip
\noindent
$(a)$
Arbitrarily pick some $T > 0$ and $\xi \in \D^{( \mathcal J^I_b - 1)|b}_{ \bar B_{\bar\epsilon} }[0,T](\bar\epsilon)$.
To lighten notations, in the proof of part $(a)$ we write $k = J^I_b$.
 By the definition of $\D^{(k-1)|b}_A(\epsilon)$ in \eqref{def: l * tilde jump number for function g, clipped SGD},
there are some $\bm x$ with $\norm{\bm x} \leq \bar\epsilon$,
some $(\bm w_1,\cdots,\bm w_{k-1}) \in \R^{d \times k - 1}$,
some $(\bm v_1,\cdots,\bm v_{k-1}) \in \R^{m \times k-1}$ with $\max_{j \in [k-1]}\norm{\bm v_j} \leq \bar\epsilon$,
and $0 < t_1 < t_2 < \cdots < t_{k-1} < \infty$ such that
\begin{align*}
    \xi = \bar h^{(k-1)|b}_{[0,T]}\big(\bm x, (\bm w_1,\cdots,\bm w_{k-1}), (\bm v_1,\cdots,\bm v_{k-1}), (t_1,\cdots,t_{k-1})\big).
\end{align*}
Given any $t \in [0,T]$,
Let $j^* = j^*(t) = \max\{ j = 0,1,\cdots,k-1:\ t_j \leq t  \}$.
By definition of the mapping $\bar h^{(k-1)|b}_{[0,T]}$ in \eqref{def: perturb ode mapping h k b, 1}--\eqref{def: perturb ode mapping h k b, 3},
we have $\xi(t) = \bm y_{t - t_{j^*}}\big(\xi(t_{j^*})\big)$
where $\bm y_\cdot(\bm x)$ is the ODE under the vector field $\bm a(\cdot)$; see \eqref{def ODE path y t}.
By the definition of $\mathcal G^{(k)|b}(\epsilon)$ and $\bar{\mathcal G}^{(k)|b}$ in \eqref{def: set G k b epsilon}, \eqref{def: bar G k b epsilon, extended k jump coverage set},
we then yield $\xi(t) \in \bar{\mathcal G}^{(k-1)|b}(2\bar\epsilon)$.
However, 
by property~\eqref{constant bar epsilon, new, 3, first exit time analysis},
we must have
\begin{align}
    \bar{\mathcal G}^{(k-1)|b}(2\bar\epsilon) \subseteq I^-_{2\bar\epsilon} \subseteq I_{\bar\epsilon}.
    \label{proof, property, inclusion of bar mathcal G k b and I bar epsilon, lemma: choose key parameters, first exit time analysis}
\end{align}
and hence
$
\xi(t) \in \bar{\mathcal G}^{(k-1)|b}(2\bar\epsilon) \subseteq I_{2\bar\epsilon}.
$
This concludes the proof.

\medskip
\noindent
$(b)$
For simplicity, in the proof of part $(b)$ we write $k = \mathcal J^I_b$.
For claim $(i)$, note that due to $\norm{\bm x_0} \leq \bar\epsilon$,
we have 
$
\bm x_0 + \varphi_b(\bm \sigma(\bm x_0)\bm w_0) \in \mathcal G^{(1)|b}(2\bar\epsilon).
$
Moreover, for all $n = 0,1,\cdots,k - 2$ (recall our convention of $t_0 = 0$),
for the cadlag path $\xi$ defined in \eqref{claim, def of xi, lemma: choose key parameters, first exit time analysis}
we have
$
\xi(t_n) \in \mathcal{G}^{(n + 1)|b}(2\bar\epsilon) \subseteq \mathcal{G}^{(k - 1)|b}(2\bar\epsilon).
$
As a result, 
for all $t \in [0, t_{k-1})$ we have
$
\xi(t) \in \bar{\mathcal G}^{(k-1)|b}(2\bar\epsilon) \subseteq I_{2\bar\epsilon}
$
due to \eqref{proof, property, inclusion of bar mathcal G k b and I bar epsilon, lemma: choose key parameters, first exit time analysis}.
This verifies claim $(i)$.

For claim $(ii)$, we proceed with a proof by contradiction and suppose that $\xi(t_{k-1}) \in I_{\bar\epsilon}$.
By \eqref{constant bar c for bar epsilon, first exit time}, we then get 
$
\xi(t) = \bm y_{ t - t_{k-1} }\big(\xi(t_{k-1})\big) \in I_{\bar c\bar\epsilon}
$
for all $t \in [t_{k-1},T]$.
Together with claim $(i)$, we arrive at the contradiction that $\xi(t) \in I_{\bar c\bar\epsilon}$ for all $t \in [0,T]$.

For claim $(iii)$, the fact $\norm{\xi(t_{k-1})} \geq \bar\epsilon$ follows directly from claim~$(ii)$ and \eqref{constant bar epsilon, new, 1, first exit time analysis}.
For any $j = 1,\cdots,k-1$ and any $t \in [t_{j-1}, t_{j})$,
we proceed with a proof by contradiction and suppose that $\norm{\xi(t)} \leq \bar\epsilon$.
Then we have $\norm{\xi(t_j-)} \leq \bar\epsilon$ due to \eqref{constant bar epsilon, new, 2, first exit time analysis},
and hence
$
\xi(t_j) \in \mathcal G^{(1)|b}(2\bar\epsilon).
$
As a result, we arrive at the contradiction that 
$
\xi(t_{k-1}) \in \mathcal{G}^{(k-1)|b}(2\bar\epsilon) \subseteq I_{\bar\epsilon},
$
due to
$\mathcal{G}^{(k-1)|b}(2\bar\epsilon) \subseteq \mathcal{\bar G}^{(k-1)|b}(2\bar\epsilon)$
and \eqref{proof, property, inclusion of bar mathcal G k b and I bar epsilon, lemma: choose key parameters, first exit time analysis}.
This concludes the proof of claim~$(iii)$.

We prove claim~$(iv)$ for $\bar t \delequal k \cdot \bm t(\bar\epsilon/2)$ where $\bm t(\epsilon)$ is defined in \eqref{def: t epsilon function, first exit analysis}.
Consider the following proof by contradiction.
If $t_{k-1} \geq \bar t = (k - 1) \cdot \bm t(\bar\epsilon/2)$,
then there must be some $j = 1,2,\cdots,k-1$ such that $t_j - t_{j-1} \geq \bar t(\epsilon/2)$.
By claim~$(i)$,
we have $\xi(t_{j-1}) \in I^-_{2\bar\epsilon} \subseteq I_{\bar\epsilon/2}$.
Using the property~\eqref{property: t epsilon function, first exit analysis},
we yield $\xi(t_j-) = \lim_{t \uparrow t_j}\xi(t) \in \bar B_{ \bar\epsilon/2 }(\bm 0)$,
which implies $\norm{\xi(t)} < \epsilon$ for all $t$ less than but close enough to $t_{j}$ and contradicts claim~$(iii)$.
This concludes the proof of claim~$(iv)$.

Lastly, we prove claim~$(v)$ for $\bar\delta > 0$ small enough such that
\begin{align}
    \exp(D\bar t)\cdot C\bar\delta < \bar\epsilon,
    \qquad
    C\bar\delta < b,
    \nonumber
\end{align}
where $D\geq 1$ is the Lipschitz coefficient in Assumption~\ref{assumption: lipschitz continuity of drift and diffusion coefficients} and $C \geq 1$ is the constant in \eqref{constant C, boundedness of a and sigma, lemma: choose key parameters, first exit time analysis}.
Again, we consider a proof by contradiction.
Suppose that for the cadlag path $\xi$ in \eqref{claim, def of xi, lemma: choose key parameters, first exit time analysis}
there is some $j = 0,1,\cdots, k - 1$ such that $\norm{\bm w_j} < \bar\delta$.
First, we consider the case where $j \leq k-2$.
Then note that (for the proof of claim~$(v)$, we interpret $\xi(0-)$ as $\bm x_0$ while, by definition, $\xi(0) = \bm x_0 + \varphi_b\big( \bm \sigma(\bm x_0)\bm w_0\big)$),
we have 
\begin{align*}
    \xi(t_j) - \xi(t_j-) 
    =
    \varphi_b\big(\bm \sigma\big(\xi(t_j-)\bm w_j\big),
\end{align*}
and hence
$
\norm{\xi(t_j) - \xi(t_j-) } \leq C\bar\delta.
$
By Gronwall's inequality, we then get
\begin{align*}
    \norm{\bm y_{t - t_j}\big(\xi(t_j-)\big) - \xi(t) } \leq \exp\big(D(t - t_j)\big)\cdot C\bar\delta
    \qquad
    \forall t \in [t_j, t_{j+1}).
\end{align*}
Recall that we currently focus on the case where $j \leq k-2$.
By claim $(iv)$ and our choice of $\bar\delta$, we get $\exp\big(D(t - t_j)\big)\cdot C\bar\delta \leq \exp\big(D\bar t\big) \cdot C\bar\delta <\bar\epsilon$ in the display above.
This implies the existence of some $\xi^\prime \in \D^{(k-1)|b}_{ \bar B_{\bar\epsilon}(\bm 0) }(\bar\epsilon)$
such that 
$
\sup_{t \in [0,T]}\norm{\xi(t) - \xi^\prime(t)} < \bar\epsilon.
$
However, by results in part $(a)$,
we must have $\xi^\prime(t) \in I^-_{2\bar\epsilon}\ \forall t \in [0,T]$,
which leads to $\xi(t) \in I_{\epsilon}^-\ \forall t \in [0,T]$.
This contradicts the running assumption of part $(b)$
that $\xi(t) \notin I_{\bar c\bar\epsilon}$ for some $t \in [0,T]$,
and allows us to conclude the proof of claim~$(v)$ for the cases where $j \leq k-2$.
In case that $j = k-1$,
by claim~$(i)$ we have $\xi(t_{k-1} - ) = \lim_{t \uparrow t_{k-1}}\xi(t)\in I^-_{2\bar\epsilon}$.
Meanwhile, by definition of the mapping $\bar h^{(k-1)|b}_{[0,T]}$,
we have
$
\xi(t_{k-1})
 = 
\xi(t_{k-1} - ) + \varphi\Big(\bm \sigma\big(\xi(t_{k-1} - )\big)\bm w_{k-1} \Big).
$
By $\norm{\bm w_{k-1}} < \bar\delta$ and our choice of $\bar\delta$ above,
we have 
$
\norm{ \varphi\Big(\bm \sigma\big(\xi(t_{k-1} - )\big)\bm w_{k-1} \Big)} < \bar\epsilon
$
and hence
$
\xi(t_{k-1}) \in I_{\bar\epsilon}.
$
Due to the contradiction with claim~$(ii)$, we conclude the proof.

\medskip
\noindent
$(c)$
The proof is almost identical to that of Lemma~3.7 in \cite{wang2024largedeviationsmetastabilityanalysis} based on an inductive argument.
We omit the details to avoid repetition.

\medskip
\noindent
$(d)$
% Note that the statement is not affected by the values of $\xi(t)$ beyond $ t\in [0,t_{\mathcal J^I_b}]$ or the values of $\widecheck  \xi(t)$ outside of the domain $ t \in [0,t_{\mathcal J^I_b} - t_1]$.
% Therefore, without loss of generality we can set $T = t_{\mathcal J_b^*} + 1$.
Let $\bar t$ be the constant specified in part $(b)$.
We claim that: if $\xi(t) \notin I_{\bar c\bar \epsilon}\text{ or }\widecheck{\xi}(t)\notin I_{\bar c\bar \epsilon}$ for some $t \in [0,T]$, then
\begin{align}
        \sup_{t \in [t_1,t_{\mathcal J^I_b}]}
        \norm{\widecheck{\xi}(t - t_1) - \xi(t) }
        < 
        \underbrace{\Big( 
            2\exp\big(D\bar t\big)\cdot D
        \Big)^{\mathcal J^I_b + 1}}_{\delequal \rho^*} \cdot  \epsilon_0
            \qquad \forall \epsilon_0 \in (0,\bar\epsilon].
    \label{goal, part e, lemma: choose key parameters, first exit time analysis}
\end{align}
As a result, claims of part $(d)$ hold for any $\epsilon_0 \in (0,\bar\epsilon)$ small enough such that $\rho^*\epsilon_0 < \Delta$.
Now, it only remains to prove claim~\eqref{goal, part e, lemma: choose key parameters, first exit time analysis}.
Due to $\norm{\bm x} = \norm{\xi(0)} < \epsilon_0$ and \eqref{constant bar epsilon, new, 2, first exit time analysis},
we have 
$
\norm{\xi(t_1-)} \leq \epsilon_0.
$
This allows us to apply results in part $(c)$
and get (recall our choice of $T = t_{\mathcal J_b^*} + 1$)
\begin{align*}
        \sup_{t \in [t_1,t_{\mathcal J^I_b}]} 
           \norm{\xi(t) - \widecheck\xi(t - t_1)} \leq \Big( 2\exp\big(D( t_{\mathcal J^I_b}  - t_1)\big) \cdot D\Big)^{\mathcal J^I_b + 1} \cdot  \epsilon_0,
\end{align*}
Lastly, 
% let $\hat t_j = t_{j+1} - t_1$.
% For 
% \begin{align*}
% \hat{\xi} & = h^{({\mathcal J^I_b}-1)|b}_{[0, T - t_1] }\big( \varphi_b( \sigma(0)\cdot w_1\big),(w_2,\cdots,w_{\mathcal J^I_b}),(t_2 - t_1, t_3 - t_1,\cdots,t_{\mathcal J^I_b} - t_1)\big)
% \\
% & = h^{({\mathcal J^I_b}-1)|b}_{[0, T - t_1] }\big( \varphi_b( \sigma(0)\cdot w_1\big),(w_2,\cdots,w_{\mathcal J^I_b}),(\hat t_1, \hat t_2,\cdots,\hat t_{\mathcal J^I_b-1})\big),
% \end{align*}
if $\xi(t) \notin I_{\bar c\bar \epsilon}\text{ for some }{t \in [t_1,T]}$, 
then $t_{\mathcal{J}^I_b} - t_1 < \bar t$ by claim~$(iv)$ of part $(b)$.
Likewise, if $\widecheck{\xi}(t) \notin I_{\bar c\bar \epsilon}\text{ for some }{t \in [0,T]}$, 
then we get $t_{\mathcal{J}^I_b} < \bar t$.
In both cases, we get $t_{\mathcal J_b^*}-t_1 \leq \bar t$.
This concludes the proof.
\end{proof}

The next lemma studies the mass the measure $\widecheck{\mathbf C}^{(k)|b}$ charges on the boundary of the domain $I$.

\begin{lemma}
\label{lemma: exit rate strictly positive, first exit analysis}
\linksinthm{lemma: exit rate strictly positive, first exit analysis}
Under Assumptions \ref{assumption: lipschitz continuity of drift and diffusion coefficients} and \ref{assumption: shape of f, first exit analysis},
    $
    \widecheck{\mathbf C}^{(\mathcal{J}^I_b)|b}( I^\complement ) < \infty.
    $
\end{lemma}

\begin{proof}
\linksinpf{lemma: exit rate strictly positive, first exit analysis}
Let $\bar\epsilon > 0$ be such that the conditions in \eqref{constant bar epsilon, new, 1, first exit time analysis}--\eqref{constant bar epsilon, new, 3, first exit time analysis} hold.
Let $\bar t$ and $\bar \delta$ be the constants characterized in Lemma \ref{lemma: choose key parameters, first exit time analysis}.
Observe that
(we write $\textbf W = (\bm w_1,\cdots,\bm w_{\mathcal J^I_b})$)
\begin{align*}
    & 
    \widecheck{\mathbf C}^{(\mathcal{J}^I_b)|b}\big( I^\complement \big)
    \\
    % & \leq 
    % \widecheck{\mathbf C}^{(\mathcal{J}^I_b)|b}\big( \R \setminus  [-(l - \bar\epsilon),l-\bar\epsilon] \big)
    % \\
    & = 
    \int \mathbbm{I}\bigg\{
    \widecheck{g}^{ (\mathcal{J}^I_b - 1)|b }
    \big( \varphi_b(\bm \sigma(\bm 0) \bm w_1), (\bm w_2,\cdots,\bm w_{\mathcal{J}^I_b}), (t_1,\cdots,t_{\mathcal{J}^I_b-1})\big) \notin I
    \bigg\}
    \\ 
    &\qquad\qquad\qquad\qquad\qquad\qquad\qquad\qquad
    \big((\nu_\alpha \times \mathbf S)\circ \Phi\big)^{ \mathcal J^I_b }(d \textbf W) \times \mathcal{L}^{\mathcal{J}^I_b - 1 \uparrow}_\infty(dt_1,\cdots, dt_{\mathcal J^I_b - 1})
    \\
    & = 
    \int \mathbbm{I}\bigg\{
    h^{ (\mathcal{J}^I_b - 1)|b }_{[0,1 + t_{\mathcal{J}^I_b-1}]}
    \big( \varphi_b(\bm \sigma(\bm 0) \bm w_1), (\bm w_2,\cdots,\bm w_{\mathcal{J}^I_b}), (t_1,\cdots,t_{\mathcal{J}^I_b-1})\big)(t_{\mathcal{J}^I_b-1}) \notin I
    \bigg\}
    \\ 
    &\qquad\qquad\qquad\qquad\qquad\qquad\qquad\qquad
    \big((\nu_\alpha \times \mathbf S)\circ \Phi\big)^{ \mathcal J^I_b }(d \textbf W) \times \mathcal{L}^{\mathcal{J}^I_b - 1 \uparrow}_\infty(dt_1,\cdots, dt_{\mathcal J^I_b - 1})
    % \\
    % & \leq 
    % \int \mathbbm{I}\Big\{
    % \Big|{h}^{ (\mathcal{J}^I_b) }_{[2 + t_{\mathcal{J}^I_b-1}]}\big( 0, (w_1,\cdots,w_{\mathcal{J}^I_b}), (1, 1+ t_1,\cdots,1 + t_{\mathcal{J}^I_b-1})\big)(t)\Big| > \gamma - \bar\epsilon\text{ for some }t \in [0,2 + t_{\mathcal{J}^I_b-1}]
    % \Big\}
    % \\
    % &\qquad \qquad \qquad \qquad \qquad \qquad
    % \times \nu^{ \mathcal{J}^I_b }_\alpha(d\bm w) \times \mathcal{L}^{\mathcal{J}^I_b - 1 \uparrow}_\infty(d\bm t)
    % \\
    \\
    & \leq 
    \int \mathbbm{I}\Big\{ \norm{\bm w_j} > \bar\delta\ \forall j \in [\mathcal{J}^I_b];\ t_{\mathcal{J}^I_b-1} < \bar t\ \Big\}
    \big((\nu_\alpha \times \mathbf S)\circ \Phi\big)^{ \mathcal J^I_b }(d \textbf W) \times \mathcal{L}^{\mathcal{J}^I_b - 1 \uparrow}_\infty(dt_1,\cdots, dt_{\mathcal J^I_b - 1})
    \\ 
    &\qquad\qquad\qquad
    \text{by part $(b)$ of Lemma \ref{lemma: choose key parameters, first exit time analysis}}
    \\
    & \leq \bar t^{ \mathcal{J}^I_b - 1 } \big/ \bar\delta^{\alpha\mathcal{J}^I_b} < \infty.
\end{align*}
This concludes the proof.
\end{proof}

Recall that we use $E^-$ and $E^\circ$ to denote the closure and interior of any Borel set $E$.
In our analysis below, we make use of the following inequality in Lemma~\ref{lemma: limiting measure, with exit location B, first exit analysis}.
% We collect its proof in Section~\ref{subsec: lemma for measure check C}, together with the proofs of other useful properties regarding measures $\widecheck{ \mathbf C }^{(k)|b}$.
\begin{lemma}
\label{lemma: limiting measure, with exit location B, first exit analysis}
\linksinthm{lemma: limiting measure, with exit location B, first exit analysis}
% Let $\bar c \in (0,1)$ be the constant fixed in \eqref{constant bar c for bar epsilon, first exit time}.
Let $\bar t,\bar\delta \in (0,\infty)$ be the constants characterized in part $(b)$ of Lemma~\ref{lemma: choose key parameters, first exit time analysis}.
Given $\Delta \in (0, {\bar\epsilon})$,
there exists $\epsilon_0 = \epsilon_0(\Delta) > 0$ such that for any $\epsilon \in (0,\epsilon_0]$, $T \geq \bar t$,
and Borel measurable $B \subseteq (I_\epsilon)^c$,
\begin{align*}
    \inf_{\bm x:\ \norm{\bm x} \leq \epsilon  }\mathbf{C}^{ (\mathcal{J}^I_b)|b}_{[0,T]}
    \bigg( 
    \Big(\widecheck E(\epsilon,B,T)\Big)^\circ;\ \bm x
    \bigg)
    & \geq 
    (T-\bar t)\cdot \Big( \widecheck{\mathbf{C}}^{ (\mathcal{J}^I_b)|b }(B_{\Delta}) - \widecheck{\bm c}(\epsilon_0)\Big), 
    \\ 
    \sup_{\bm x:\ \norm{\bm x} \leq \epsilon}
    \mathbf{C}^{ (\mathcal{J}^I_b)|b}_{[0,T]}
    \bigg( 
    \Big(\widecheck E(\epsilon,B,T)\Big)^-;\ \bm x
    \bigg)
    & \leq 
    T\cdot \Big( \widecheck{\mathbf{C}}^{ (\mathcal{J}^I_b)|b }(B^{\Delta})+ \widecheck{\bm c}(\epsilon_0)\Big),
\end{align*}
where 
$\mathbf C^{(k)|b}_{[0,T]}$ is defined in \eqref{def: measure mu k b t},
$\widecheck{\mathbf C}^{(k)|b}$ is defined in \eqref{def: measure check C k b},
$\mathcal J_b^{I}$ is defined in \eqref{def: first exit time, J *},
\begin{align}
    \notationdef{notation-set-check-E-epsilon-B-T}{\widecheck{E}(\epsilon,B,T)} 
    & \delequal 
    \Big\{ \xi \in \mathbb{D}[0,T]:\ \exists t \leq T\ s.t.\ \xi_t \in B\text{ and }\xi_s \in I_\epsilon\ \forall s \in [0,t) \Big\},
    \label{def: set check E epsilon B T, measure check C k b}
    \\
    \notationdef{notation-error-function-check-c-epsilon}{\widecheck{\bm c}(\epsilon)} 
    & \delequal 
    \mathcal J^I_b \cdot (\bar t)^{ \mathcal J^I_b - 1 } \cdot (\bar\delta)^{ -\alpha \cdot (\mathcal J^I_b - 1) }
    \cdot 
    \epsilon^{ \frac{\alpha}{2\mathcal J^I_b}  }.
    \label{def: check c error function, first exit time proof}
\end{align}
\end{lemma}

\begin{proof}\linksinpf{lemma: limiting measure, with exit location B, first exit analysis}
% Let $\bar t$ be characterized as in Lemma~\ref{lemma: choose key parameters, first exit time analysis}.
Let $\bar c \in (0,1)$ be the constant fixed in \eqref{constant bar c for bar epsilon, first exit time}.
By part $(e)$ of Lemma~\ref{lemma: choose key parameters, first exit time analysis},
for the fixed $\Delta \in (0,\bar\epsilon)$, we are able to fix some $\epsilon_0 \in (0,\frac{\Delta}{2} \wedge \bar c\bar\epsilon)$ such that
the following claim holds:
given 
        $T > 0$, 
        $\bm x \in \R^m$,
        $\textbf W = (\bm w_1,\cdots, \bm w_{\mathcal J^I_b}) \in \R^{d \times  \mathcal J^I_b},\ (t_1,\cdots,t_{\mathcal J^I_b})\in (0,T]^{\mathcal J^I_b\uparrow}$,
        if $\norm{\bm x} \leq \epsilon_0$
        and
        $
            \max_{j \in [\mathcal J^I_b] }\norm{\bm w_j} \leq \epsilon_0^{-\frac{1}{2\mathcal J^I_b}},
        $
        then    
        \begin{align}
        \xi(t) \notin I_{\bar c\bar \epsilon}\text{ or }\widecheck{\xi}(t)\notin I_{\bar c\bar \epsilon}
        \text{ for some }t \in [t_1,T - t_1]
        \quad
        \Longrightarrow
        \quad
        \sup_{t \in [t_1,t_{\mathcal J^I_b}]}
        \norm{\widecheck{\xi}(t - t_1) - \xi(t) } < \Delta,
        \label{choice of epsilon 0, lemma: limiting measure, with exit location B, first exit analysis}
        \end{align}
        where
        \begin{align*}
                \xi & =  h^{ (\mathcal J^I_b)|b }_{[0,T]}\big(\bm x, \textbf W, (t_1,\cdots,t_{\mathcal J^I_b})\big),
                \\ 
                \widecheck{\xi}
                & =
                 h^{ (\mathcal J^I_b - 1)|b }_{[0,T]}
                \Big(
                    \varphi_b\big( \bm \sigma(\bm 0)\bm w_1\big), (\bm w_2,\cdots, \bm w_{ \mathcal J^I_b }), 
                    (t_2 - t_1,t_3 - t_1,\cdots,t_{\mathcal J^I_b} - t_1)
                \Big).
            \end{align*}
        % $\widecheck{g}^{ (\mathcal{J}^I_b - 1)|b }\big( \varphi_b(\sigma(0)\cdot w_1),w_2,\cdots,w_{\mathcal{J}^I_b},t_2 - t_1, t_3 - t_1,\cdots,t_{\mathcal{J}^I_b} - t_1 \big)$.

Henceforth in the proof, we fix some $\epsilon \in (0,\epsilon_0]$
and
$B \subseteq (I_{\epsilon})^c$.
Due to our choice of $\epsilon \leq \epsilon_0 < \bar c\bar\epsilon$,
we have
$B \subseteq (I_{\bar c \bar \epsilon})^c$.
To prove the lower bound,
let 
$$ 
\widetilde E =  
\Big\{ \xi \in \D[0,T]:\ \exists t \in [0,T]\text{ s.t. }\xi(t) \in B_{\Delta/2},\ \xi(s) \in I_{2\epsilon}\ \forall s \in [0,t)  \Big\}.
$$
For any $\xi \in \widetilde E$ and any $\xi^\prime$ with $\dj{[0,T]}(\xi,\xi^\prime) < \epsilon$,
due to $\epsilon \leq \epsilon_0 < \Delta/2$,
there must be some $t^\prime \in [0,T]$ such that 
$
\xi^\prime(t^\prime) \in B
$
and
$
\xi^\prime(s) \in I_\epsilon\ \forall s \in [0,t^\prime).
$
This implies that
$\xi^\prime \in \widecheck E(\epsilon,B,T)$,
and hence
\begin{align*}
\widetilde E 
% =
    % \big\{ \xi \in \D[0,T]:\ \exists t \in [0,T]\text{ s.t. }\xi(t) \in B_{\Delta/2},\ \xi(s) \in I_{2\epsilon}\ \forall s \in [0,t)\big\}
    \subseteq
    \Big(\widecheck E(\epsilon,B,T)\Big)_{\epsilon}
    \subseteq
    \Big(\widecheck E(\epsilon,B,T)\Big)^\circ.
\end{align*}
Therefore, for any $\bm x \in \R^m$ with $\norm{\bm x} \leq \epsilon \leq \epsilon_0$,
\begin{align}
            \mathbf{C}^{ (\mathcal{J}^I_b)|b}_{[0,T]}
    \bigg( 
    \Big(\widecheck E(\epsilon,B,T)\Big)^\circ;\ \bm x
    \bigg)
    % \\ 
    & \geq 
    \int \mathbbm{I}\Big\{ 
    % \exists t \in [0,T]\text{ s.t. }
    h^{ (\mathcal{J}^I_b)|b }_{[0,T]}(\bm x,\textbf W,\bm t)\in \widetilde E
    % \text{ and }
    %  h^{ (\mathcal{J}^I_b)|b }_{[0,T]}(x,\bm w,\bm t)(s) \in I_{2\epsilon}\ \forall s \in [0,t)
    \Big\}
    \big((\nu_\alpha \times \mathbf S)\circ \Phi\big)^{ \mathcal J^I_b }(d \textbf W)
    \times \mathcal{L}^{ \mathcal{J}^I_b \uparrow }_T(d\bm t)
    \nonumber
    \\
    &
    = 
    \int \widetilde \phi_B(t_1,\bm x) \mathcal L_T(dt_1),
    \label{proof: lower bound, intermediate, lemma: limiting measure, with exit location B, first exit analysis}
\end{align}
where $\mathcal L_T$ is the Lebesgue measure on $(0,T)$,
$\mathcal L_T^{k\uparrow}$ is the $k$-fold ofq Lebesgue measure restricted on 
$
\{ (t_1,\cdots,t_k) \in (0,T)^k:\ t_1 < t_2 < \cdots <t_k \},
$
and
\begin{align*}
\widetilde \phi_B(t_1,\bm x)
& = 
\int 
    \mathbbm{I}\bigg\{
        \exists t \in [0,T]\text{ s.t. }
        h^{ (\mathcal{J}^I_b)|b }_{[0,T]}
        \Big(\bm x,\textbf W,
        (t_1,t_1 + u_2, t_1 + u_3,\cdots,t_1 + u_{\mathcal{J}^I_b})\Big)(t) \in B_{\Delta/2}
    \\ 
&\qquad\qquad 
\text{ and }
h^{ (\mathcal{J}^I_b)|b }_{[0,T]}\Big(\bm x,\textbf W,
        (t_1,t_1 + u_2, t_1 + u_3,\cdots,t_1 + u_{\mathcal{J}^I_b})\Big)(s) \in I_{2\epsilon}\ \forall s \in [0,t)
\bigg\} 
\\ 
&\qquad\qquad\qquad\qquad\qquad\qquad
     \big((\nu_\alpha \times \mathbf S)\circ \Phi\big)^{ \mathcal J^I_b }(d \textbf W)
    \times \mathcal{L}^{ \mathcal{J}^I_b - 1 \uparrow }_{T-t_1}(du_2,\cdots,du_{\mathcal{J}^I_b}).
\end{align*}
Set
$
\bm x_0 = \lim_{t \uparrow t_1} \bm y_{t}(\bm x),
$
and note that
\begin{align*}
&  h^{ (\mathcal{J}^I_b)|b }_{[0,T]}\Big(\bm x,(\bm w_1,\cdots,\bm w_{\mathcal J^I_b}),
        (t_1,t_1 + u_2, t_1 + u_3,\cdots,t_1 + u_{\mathcal{J}^I_b})\Big)(t_1 + s)  
\\ 
& = 
h^{ (\mathcal{J}^I_b - 1)|b }_{[0,T - t_1]}\Big(\bm x_0 + \varphi_b(\bm \sigma(\bm x_0) \bm w_1),(\bm w_2,\cdots,\bm w_{\mathcal J^I_b}),
        (u_2, u_3,\cdots,u_{\mathcal{J}^I_b})\Big)(s)  \qquad \forall s \in [0,T - t_1].
\end{align*}
Therefore, for any $t_1 \in [0,T - \bar t]$ and $\bm x$ with $\norm{\bm x} \leq \epsilon$, 
by property \eqref{constant bar epsilon, new, 2, first exit time analysis}
we have $\norm{\bm x_0} \leq \epsilon \leq \epsilon_0 \leq \Delta/2$, and
\begin{align*}
& \widetilde \phi_B(t_1,x)
\\
& \geq \inf_{\bm x_0:\ \norm{\bm x_0} \leq \frac{\Delta}{2}}
\int 
    \mathbbm{I}\bigg\{
        \exists t \in [0,T - t_1]\text{ s.t. }
        h^{ (\mathcal{J}^I_b - 1)|b }_{[0,T - t_1]}
        \Big(\bm x_0 + \varphi_b(\bm \sigma(\bm x_0) \bm w_1),(\bm w_2,\cdots,\bm w_{\mathcal J^I_b}),
        (u_2, \cdots,u_{\mathcal{J}^I_b})\Big)(t)
        \in B_{\Delta/2}
    \\ 
&\qquad\qquad 
\text{ and }
h^{ (\mathcal{J}^I_b - 1)|b }_{[0,T - t_1]}
\Big(\bm x_0 + \varphi_b(\bm\sigma(\bm x_0) \bm w_1),(\bm w_2,\cdots,\bm w_{\mathcal J^I_b}),
        (u_2, \cdots,u_{\mathcal{J}^I_b})\Big)(s) \in I_{2\epsilon}\ \forall s \in [0,t)
\bigg\} 
\\ 
&\qquad\qquad\qquad\qquad\qquad\qquad\qquad\qquad\qquad\qquad
    \big((\nu_\alpha \times \mathbf S)\circ \Phi\big)^{ \mathcal J^I_b }(d \textbf W)
    \times \mathcal{L}^{ \mathcal{J}^I_b - 1 \uparrow }_{T-t_1}(du_2,\cdots,du_{\mathcal{J}^I_b})
\\ 
& = \inf_{\bm x_0:\ \norm{\bm x_0} \leq \frac{\Delta}{2}}
\int
\mathbbm{I}\bigg\{
    h^{ (\mathcal{J}^I_b - 1)|b }_{[0,T - t_1]}
    \Big(\bm x_0 + \varphi_b(\bm \sigma(\bm x_0) \bm w_1),(\bm w_2,\cdots,\bm w_{\mathcal J^I_b}),
        (u_2, \cdots,u_{\mathcal{J}^I_b})\Big)(u_{ \mathcal J^I_b }) \in B_{\Delta/2};\ 
        \min_{j \in [\mathcal J^I_b]}\norm{\bm w_j} > \bar\delta
\bigg\}
\\ 
&\qquad\qquad\qquad\qquad\qquad\qquad\qquad\qquad\qquad\qquad
    \big((\nu_\alpha \times \mathbf S)\circ \Phi\big)^{ \mathcal J^I_b }(d \textbf W)
    \times \mathcal{L}^{ \mathcal{J}^I_b - 1 \uparrow }_{T-t_1}(du_2,\cdots,du_{\mathcal{J}^I_b})
\\ 
&\qquad \qquad \text{by claims $(i)$, $(ii)$, and $(v)$ in part $(b)$ of Lemma~\ref{lemma: choose key parameters, first exit time analysis}}
\\ 
& \geq 
\inf_{\bm x_0:\ \norm{\bm x_0} \leq \frac{\Delta}{2}}
\int
\mathbbm{I}\bigg\{
    h^{ (\mathcal{J}^I_b - 1)|b }_{[0,T - t_1]}
    \Big(\bm x_0 + \varphi_b(\bm \sigma(\bm x_0) \bm w_1),(\bm w_2,\cdots,\bm w_{\mathcal J^I_b}),
        (u_2, \cdots,u_{\mathcal{J}^I_b})\Big)(u_{ \mathcal J^I_b }) \in B_{\Delta/2};
        \\ 
    &\qquad\qquad\qquad\qquad\qquad \qquad\qquad\qquad\qquad\qquad\qquad\qquad
        \min_{j \in [\mathcal J^I_b]}\norm{\bm w_j} > \bar\delta,\ 
         \max_{j \in [\mathcal J^I_b] }\norm{\bm w_j} \leq \epsilon_0^{-\frac{1}{2\mathcal J^I_b}}
\bigg\}
\\ 
&\qquad\qquad\qquad\qquad\qquad\qquad\qquad\qquad\qquad\qquad
    \big((\nu_\alpha \times \mathbf S)\circ \Phi\big)^{ \mathcal J^I_b }(d \textbf W)
    \times \mathcal{L}^{ \mathcal{J}^I_b - 1 \uparrow }_{T-t_1}(du_2,\cdots,du_{\mathcal{J}^I_b})
\\
& \geq 
\int
\mathbbm{I}\bigg\{
    h^{ (\mathcal{J}^I_b - 1)|b }_{[0,T - t_1]}
    \Big(\varphi_b(\bm \sigma(\bm 0) \bm w_1),(\bm w_2,\cdots,\bm w_{\mathcal J^I_b}),
        (u_2, \cdots,u_{\mathcal{J}^I_b})\Big)(u_{ \mathcal J^I_b }) \in B_{\Delta};
        \\ 
    &\qquad\qquad\qquad\qquad\qquad \qquad\qquad\qquad\qquad\qquad\qquad\qquad
        \min_{j \in [\mathcal J^I_b]}\norm{\bm w_j} > \bar\delta,\ 
         \max_{j \in [\mathcal J^I_b] }\norm{\bm w_j} \leq \epsilon_0^{-\frac{1}{2\mathcal J^I_b}}
\bigg\}
\\ 
&\qquad\qquad\qquad\qquad\qquad\qquad\qquad\qquad\qquad\qquad
    \big((\nu_\alpha \times \mathbf S)\circ \Phi\big)^{ \mathcal J^I_b }(d \textbf W)
    \times \mathcal{L}^{ \mathcal{J}^I_b - 1 \uparrow }_{T-t_1}(du_2,\cdots,du_{\mathcal{J}^I_b})
\\
&\qquad \qquad \text{by property \eqref{choice of epsilon 0, lemma: limiting measure, with exit location B, first exit analysis}}
\\
& = 
\int
\mathbbm{I}\bigg\{
    \widecheck{g}^{ (\mathcal{J}^I_b - 1)|b }_{[0,T - t_1]}
    \Big(\varphi_b(\bm \sigma(\bm 0) \bm w_1),(\bm w_2,\cdots,\bm w_{\mathcal J^I_b}),
        (u_2, \cdots,u_{\mathcal{J}^I_b})\Big) \in B_{\Delta};
        \\ 
    &\qquad\qquad\qquad\qquad\qquad \qquad\qquad\qquad\qquad\qquad\qquad\qquad
        \min_{j \in [\mathcal J^I_b]}\norm{\bm w_j} > \bar\delta,\ 
         \max_{j \in [\mathcal J^I_b] }\norm{\bm w_j} \leq \epsilon_0^{-\frac{1}{2\mathcal J^I_b}}
\bigg\}
\\ 
&\qquad\qquad\qquad\qquad\qquad\qquad\qquad\qquad\qquad\qquad
    \big((\nu_\alpha \times \mathbf S)\circ \Phi\big)^{ \mathcal J^I_b }(d \textbf W)
    \times \mathcal{L}^{ \mathcal{J}^I_b - 1 \uparrow }_{T-t_1}(du_2,\cdots,du_{\mathcal{J}^I_b})
\\
&\qquad\qquad \text{by the definition of $\widecheck g^{(k)|b}$ in \eqref{def: mapping check g k b, endpoint of path after the last jump, first exit analysis}}
\\ 
& = 
\int
\mathbbm{I}\bigg\{
    \widecheck{g}^{ (\mathcal{J}^I_b - 1)|b }_{[0,T - t_1]}
    \Big(\varphi_b(\bm \sigma(\bm 0) \bm w_1),(\bm w_2,\cdots,\bm w_{\mathcal J^I_b}),
        (u_2, \cdots,u_{\mathcal{J}^I_b})\Big) \in B_{\Delta};
        \\ 
    &\qquad\qquad\qquad\qquad\qquad \qquad\qquad\qquad\qquad\qquad\qquad\qquad
        \min_{j \in [\mathcal J^I_b]}\norm{\bm w_j} > \bar\delta,\ 
         \max_{j \in [\mathcal J^I_b] }\norm{\bm w_j} \leq \epsilon_0^{-\frac{1}{2\mathcal J^I_b}}
\bigg\}
\\ 
&\qquad\qquad\qquad\qquad\qquad\qquad\qquad\qquad\qquad\qquad
    \big((\nu_\alpha \times \mathbf S)\circ \Phi\big)^{ \mathcal J^I_b }(d \textbf W)
    \times \mathcal{L}^{ \mathcal{J}^I_b - 1 \uparrow }_{\bar t}(du_2,\cdots,du_{\mathcal{J}^I_b})
\\
&\qquad\qquad
\text{by claim $(v)$ in part $(b)$ of Lemma~\ref{lemma: choose key parameters, first exit time analysis}}
\\
& \geq 
\int
\mathbbm{I}\bigg\{
    \widecheck{g}^{ (\mathcal{J}^I_b - 1)|b }_{[0,T - t_1]}
    \Big(\varphi_b(\bm \sigma(\bm 0) \bm w_1),(\bm w_2,\cdots,\bm w_{\mathcal J^I_b}),
        (u_2, \cdots,u_{\mathcal{J}^I_b})\Big) \in B_{\Delta};\ \min_{j \in [\mathcal J^I_b]}\norm{\bm w_j} > \bar\delta\bigg\}
\\ 
&\qquad\qquad\qquad\qquad\qquad\qquad\qquad\qquad\qquad\qquad
    \big((\nu_\alpha \times \mathbf S)\circ \Phi\big)^{ \mathcal J^I_b }(d \textbf W)
    \times \mathcal{L}^{ \mathcal{J}^I_b - 1 \uparrow }_{\bar t}(du_2,\cdots,du_{\mathcal{J}^I_b})
\\ 
&\qquad - 
\int
\mathbbm{I}\bigg\{
    \min_{j \in [\mathcal J^I_b]}\norm{\bm w_j} > \bar\delta,\ \max_{j \in [\mathcal J^I_b]}\norm{\bm w_j} > \epsilon_0^{-\frac{1}{2\mathcal J^I_b}}
    \bigg\}
% \\ 
% &\qquad\qquad\qquad\qquad\qquad\qquad\qquad\qquad\qquad\qquad
    \big((\nu_\alpha \times \mathbf S)\circ \Phi\big)^{ \mathcal J^I_b }(d \textbf W)
    \times \mathcal{L}^{ \mathcal{J}^I_b - 1 \uparrow }_{\bar t}(du_2,\cdots,du_{\mathcal{J}^I_b}).
\end{align*}
We focus on the two integrals one the RHS of the last inequality in the display above.
It is easy to see that the latter is upper bounded by 
\begin{align*}
    \widecheck{\bm c}(\epsilon_0) =
    \mathcal J^I_b \cdot (\bar t)^{ \mathcal J^I_b - 1 } \cdot (\bar\delta)^{ -\alpha \cdot (\mathcal J^I_b - 1) }
    \cdot 
    \epsilon_0^{ \frac{\alpha}{2\mathcal J^I_b}  }.
\end{align*}
As for the former, using part $(b)$ of Lemma~\ref{lemma: choose key parameters, first exit time analysis} and the fact that $B_\Delta \subseteq B \subseteq (I_{\bar c \bar\epsilon})^\complement$ again,
we yield
\begin{align*}
    & \int
\mathbbm{I}\bigg\{
    \widecheck{g}^{ (\mathcal{J}^I_b - 1)|b }_{[0,T - t_1]}
    \Big(\varphi_b(\bm \sigma(\bm 0) \bm w_1),(\bm w_2,\cdots,\bm w_{\mathcal J^I_b}),
        (u_2, \cdots,u_{\mathcal{J}^I_b})\Big) \in B_{\Delta};\ \min_{j \in [\mathcal J^I_b]}\norm{\bm w_j} > \bar\delta\bigg\}
\\ 
&\qquad\qquad\qquad\qquad\qquad\qquad\qquad\qquad\qquad\qquad
    \big((\nu_\alpha \times \mathbf S)\circ \Phi\big)^{ \mathcal J^I_b }(d \textbf W)
    \times \mathcal{L}^{ \mathcal{J}^I_b - 1 \uparrow }_{\bar t}(du_2,\cdots,du_{\mathcal{J}^I_b})
\\ 
& = 
\int
\mathbbm{I}\bigg\{
    \widecheck{g}^{ (\mathcal{J}^I_b - 1)|b }_{[0,T - t_1]}
    \Big(\varphi_b(\bm \sigma(\bm 0) \bm w_1),(\bm w_2,\cdots,\bm w_{\mathcal J^I_b}),
        (u_2, \cdots,u_{\mathcal{J}^I_b})\Big) \in B_{\Delta}\}
\\
&\qquad\qquad\qquad\qquad\qquad\qquad\qquad\qquad\qquad\qquad
    \big((\nu_\alpha \times \mathbf S)\circ \Phi\big)^{ \mathcal J^I_b }(d \textbf W)
    \times \mathcal{L}^{ \mathcal{J}^I_b - 1 \uparrow }_{\infty}(du_2,\cdots,du_{\mathcal{J}^I_b})
\\ 
& = \widecheck{\mathbf C}^{ (\mathcal J^I_b)|b}(B_\Delta).
\end{align*}
In summary, for any $\bm x \in \R^m$ with $\norm{\bm x} \leq \epsilon$
and
$t_1 \in [0,T - \bar t]$ ,
we have shown that
$$
\widetilde \phi_B(t_1, \bm x) \geq \widecheck{\mathbf C}^{ (\mathcal J^I_b)|b}(B_\Delta) - 
\widecheck{\bm c}(\epsilon_0).
$$
Together with the trivial bound that 
$
\widetilde \phi_B(t_1,\bm x) \geq 0
$
for all $t_1 > T - \bar t$,
we have in \eqref{proof: lower bound, intermediate, lemma: limiting measure, with exit location B, first exit analysis}
 that
$$
\mathbf{C}^{ (\mathcal{J}^I_b)|b}_{[0,T]}\bigg( \Big(\widecheck E(\epsilon,B,T)\Big)^\circ;\ \bm x\bigg) 
\geq 
(T-\bar t) \cdot \Big( \widecheck{\mathbf C}^{ (\mathcal J^I_b)|b}(B_\Delta) - \widecheck{\bm c}(\epsilon_0)  \Big)
$$
for all $\bm x \in \R^m$ with $\norm{\bm x} \leq \epsilon$.
This concludes the proof of the lower bound.
The proof to the upper bound is almost identical, so we omit the details here.
\end{proof}

\subsection{Proof of Theorem~\ref{theorem: first exit time, unclipped} via Asymptotic Atoms}

% The key step in our first exit time analysis is to contextualize the framework developed in Section \ref{subsec: Exit time analysis framework} in the setup of Theorem \ref{theorem: first exit time, unclipped}.
% Specifically, we provide a few technical lemmas that verify Condition \ref{condition E2}.
To see how we apply the framework developed in Section \ref{subsec: framework, first exit time analysis},
let us specialize Condition \ref{condition E2}
to a setting
where $\S = \R$, $A(\epsilon) = B_\epsilon(\bm 0)  = \{ \bm x \in \R^m:\ \norm{\bm x} < \epsilon \}$,
and the covering
$I(\epsilon) = I_\epsilon$.
Let $V^\eta_j(x) = \bm X^{\eta|b}_j(\bm x)$.
Meanwhile, for
$
C^I_b = \widecheck{\mathbf C}^{(\mathcal{J}^I_b)|b}\big( I^\complement\big),
$
it is shown in Lemma~\ref{lemma: exit rate strictly positive, first exit analysis} that $C^I_b < \infty$.
Now, recall that $H(\cdot) = \P(\norm{\bm Z_1} > \cdot)$ and $\lambda(\eta) = \eta^{-1}H(\eta^{-1})$.
Recall that in Theorem~\ref{theorem: first exit time, unclipped},
we consider two cases: $(i)$ $C^I_b \in (0,\infty)$, and $(ii)$ $C^I_b = 0$.
We first discuss our choices in Case $(i)$.
When $C^I_b > 0$, we set the location measure and scale as 
\begin{align}
    C(\ \cdot\ ) \delequal { \widecheck{\mathbf C }^{ (\mathcal J^I_b)|b }(\ \cdot \ \setminus I )  }\big/{C^I_b},
    \qquad 
    \gamma(\eta) \delequal C^I_b \cdot \eta \cdot \big(\lambda(\eta)\big)^{\mathcal J^I_b}.
    \label{def: measure C and scale gamma when applying the exit time framework}
\end{align}
The regularity conditions in Theorem~\ref{theorem: first exit time, unclipped} dictate that
$
\widecheck{\mathbf C}^{(\mathcal J^I_b)|b}(\partial I) = 0,
$
and hence
$C(\partial I) = 0$.
Besides, note that $C(\cdot)$ is a probability measure and
$
\gamma(\eta)T/\eta =  C^I_b  T \cdot \big(\lambda(\eta)\big)^{\mathcal J^I_b}.
$
% Besides, this corresponds to Case $(i)$ for the location measure in the definition of asymptotic atoms; see the discussion before Definition~\ref{def: asymptotic atom}.

The application of the framework developed in Section \ref{subsec: framework, first exit time analysis}
(specifically, Theorem~\ref{thm: exit time analysis framework})
hinges on the verification of \eqref{eq: exit time condition lower bound}--\eqref{eq:E4}.
We start by verifying \eqref{eq: exit time condition lower bound} and \eqref{eq: exit time condition upper bound}.
First, given any Borel measurable $B \subseteq \R$, we specify the choice of function $\delta_B(\epsilon,T)$ in Condition \ref{condition E2}.
From the continuity of measures,
we get
$\lim_{\Delta \downarrow 0}\widecheck{\mathbf{C}}^{ (\mathcal{J}^I_b)|b }\Big( (B^{\Delta}\cap I^\complement ) \setminus (B^- \cap I^\complement)\Big) = 0$
and
$\lim_{\Delta \downarrow 0} \widecheck{\mathbf{C}}^{ (\mathcal{J}^I_b)|b }\Big( (B^\circ \cap I^\complement) \setminus (B_{\Delta} \cap I^\complement ) \Big) = 0$.
This allows us to fix a sequence $(\Delta^{(n)})_{n \geq 1}$
such that $\Delta^{(n+1)} \in (0,\Delta^{(n)}/2)$ and
\begin{align}
    \widecheck{\mathbf{C}}^{ (\mathcal{J}^I_b)|b }\Big( (B^{\Delta^{(n)}}\cap I^\complement ) \setminus (B^- \cap I^\complement)\Big) \vee 
   \widecheck{\mathbf{C}}^{ (\mathcal{J}^I_b)|b }\Big( (B^\circ \cap I^\complement) \setminus (B_{\Delta^{(n)}} \cap I^\complement ) \Big)
   \leq 1/2^n
   \label{choice of Delta n, first exit time proof}
    \end{align}
for each $n \geq 1$.
Next,
recall the definition of set $\widecheck E(\epsilon,B,T)$ in Lemma \ref{lemma: limiting measure, with exit location B, first exit analysis},
and
let $\widetilde B(\epsilon) \delequal B \setminus I_\epsilon$.
Using Lemma \ref{lemma: limiting measure, with exit location B, first exit analysis},
we are able to fix another sequence $(\epsilon^{(n)})_{n \geq 1}$
with
$
\epsilon^{(n+1)} \in (0,\epsilon^{(n)}/2)
$
and
$\epsilon^{(n)} \in (0,\bar\epsilon]\ \forall n \geq 1$,
such that
for any $n \geq 1$, $\epsilon \in (0,\epsilon^{(n)}]$,
% we have
% Then we are able to fix another sequence $(\epsilon^{(n)})_{n \geq 1}$ with $\epsilon^{(n)} \in (0,\bar\epsilon]\ \forall n$
% such that, for each $n \geq 1$, we have $\epsilon^{(n+1)} \leq \epsilon^{(n)}/2$
% and claims in Lemma \ref{lemma: limiting measure, with exit location B, first exit analysis} hold for $\Delta = \Delta^{(n)}$ and $\epsilon_0 = \epsilon^{(n)}$.
% Now, observe the following consequence of our choice of $\epsilon^{(n)},\Delta^{(n)}$.
% Using Lemma  \ref{lemma: limiting measure, with exit location B, first exit analysis},
% for all $\epsilon \in (0,\epsilon^{(n)}]$
% we have
\begin{align}
 \sup_{\bm x:\ \norm{\bm x} \leq \epsilon}
 \mathbf{C}^{ (\mathcal{J}^I_b)|b}_{[0,T]}
    \bigg(
    \Big(\widecheck E\big(\epsilon,\widetilde B(\epsilon),T\big)\Big)^-;\ \bm x
    \bigg)
    & \leq 
    T\cdot 
    \bigg( 
        \widecheck{\mathbf{C}}^{ (\mathcal{J}^I_b)|b }\Big( \big(B \setminus I_\epsilon\big)^{\Delta^{(n)}} \Big)
        + \widecheck{\bm c}(\epsilon^{(n)})
    \bigg),
    % +
    % (\bar t/\Bar{\delta}^\alpha)^{ \mathcal{J}^I_b},
    \label{measure C bound, 2, first exit time analysis}
    \\
    \inf_{\bm x:\ \norm{\bm x} \leq \epsilon}
    \mathbf{C}^{ (\mathcal{J}^I_b)|b}_{[0,T]}
    \bigg(
    \Big(\widecheck E\big(\epsilon,\widetilde B(\epsilon),T\big)\Big)^\circ;\ \bm x
    \bigg)
    & \geq 
    (T - \bar t)\cdot 
    \bigg( 
        \widecheck{\mathbf{C}}^{ (\mathcal{J}^I_b)|b }\Big( \big(B \setminus I_\epsilon\big)_{\Delta^{(n)}} \Big) - \widecheck{\bm c}(\epsilon^{(n)})
    \bigg).
    \label{measure C bound, 1, first exit time analysis}
\end{align}
% where the set $\widecheck E(\epsilon,B,T)$ is defined in Lemma \ref{lemma: limiting measure, with exit location B, first exit analysis}.
Besides, note that
given any $\epsilon \in (0,\epsilon^{(1)}]$,
there uniquely exists some $n = n_\epsilon \geq 1$ such that $\epsilon \in (\epsilon^{(n+1)},\epsilon^{(n)}]$.
This allows us to set (under $n = n_\epsilon$)
\begin{align}
            & \widecheck{\delta}_B(\epsilon,T) \label{def: bm delta B, first exit analysis}
        \\ 
    & =
    T \cdot \Bigg[\widecheck{ \mathbf C}^{ (\mathcal J^I_b)|b }\Big( (B^{ \Delta^{(n)} } \cap I^\complement) \symbol{92} (B^- \cap I^\complement) \Big)
    \nonumber
         \\
     &\qquad\qquad\qquad\qquad
    \vee 
     \widecheck{ \mathbf C}^{ (\mathcal J^I_b)|b }\Big( (B^\circ\cap I^\complement) \symbol{92} (B_{ \Delta^{(n)} } \cap I^\complement) \Big)
     \vee 
     \widecheck{ \mathbf C}^{ (\mathcal J^I_b)|b }\Big( (\partial I)^{ \epsilon + \Delta^{(n)} }\Big)
     \Bigg]
     \nonumber
     \\
     &\qquad
     +
     T \cdot \widecheck{\bm c}(\epsilon^{(n)})
     + \bar t \cdot  \widecheck{ \mathbf C}^{ (\mathcal J^I_b)|b }\big(B^\circ \setminus I\big),
     \nonumber
\end{align}
where $\widecheck{\bm c}(\cdot)$ is defined in \eqref{def: check c error function, first exit time proof}.
Also, let $$\delta_B(\epsilon,T) \delequal {\widecheck{\delta}_B(\epsilon,T)}/({ C^I_b \cdot T }).$$
By \eqref{choice of Delta n, first exit time proof}
and $\widecheck{\mathbf C}^{(\mathcal J^I_b)|b}( B \setminus I) \leq \widecheck{\mathbf C}^{(\mathcal J^I_b)|b}(I^\complement) < \infty$,
we get
\begin{align*}
   \lim_{T \to \infty}\delta_B(\epsilon,T)
   \leq 
   \big({C^I_b}\big)^{-1} \cdot \Big[ \widecheck{\bm c}(\epsilon^{(n)}) +  \frac{1}{2^{ n } } \vee \widecheck{ \mathbf C}^{ (\mathcal J^I_b)|b }\Big( (\partial I)^{ \epsilon + \Delta^{(n)} }\Big) \Big],
\end{align*}
where $n = n_\epsilon$ is the unique positive integer satisfying $\epsilon \in (\epsilon^{(n+1)},\epsilon^{(n)}]$.
Moreover, as $\epsilon \downarrow 0$ we get $n_\epsilon \to \infty$.
Since $\partial I$ is closed, we get $\cap_{r > 0}(\partial I)^r = \partial I$,
which implies $\lim_{r \downarrow 0}\widecheck{ \mathbf C}^{ (\mathcal J^I_b)|b }\Big( (\partial I)^{ r }\Big)
=
\widecheck{ \mathbf C}^{ (\mathcal J^I_b)|b }( \partial I) = 0.
$
Also, by definition of $\widecheck{\bm c}$ in \eqref{def: check c error function, first exit time proof},
we have $\lim_{\epsilon \downarrow 0}\widecheck{\bm c}(\epsilon) = 0$.
In summary, we have verified that $\lim_{\epsilon \downarrow 0}\lim_{T \to \infty}\delta_B(\epsilon,T) = 0$.

Next, in Case $(ii)$ (i.e., $C^I_b = 0$), we set 
\begin{align*}
    C(\cdot) \equiv 0,\qquad
    \gamma(\eta) \delequal \eta \cdot  \big(\lambda(\eta)\big)^{\mathcal J^I_b},
    \qquad
    \delta_B(\epsilon,T) \delequal \widecheck{\delta}_{B}(\epsilon,T)/T.
\end{align*}
The same calculations above verify that $\lim_{\epsilon \downarrow 0}\lim_{T \to \infty}\delta_B(\epsilon,T) = 0$.

% Note that $\lim_{T \to \infty}\bm \delta_B(\epsilon,T) \leq 1/(C^I_b \cdot 2^n)$
% where $n = n_\epsilon$ is the unique positive integer satisfying $\epsilon \in (\epsilon^{(n+1)},\epsilon^{(n)}]$,
% and hence $ \lim_{\epsilon \downarrow 0} \lim_{T \to \infty}\bm \delta_B(\epsilon,T) \leq \lim_{n \to \infty}  1/(C^I_b \cdot 2^n) = 0$.
Now, we are ready to verify conditions \eqref{eq: exit time condition lower bound} and \eqref{eq: exit time condition upper bound}.
Specifically,
we
introduce stopping time
\begin{align}
    \notationdef{notation-tau-eta-b-epsilon-exit-time}{\tau^{\eta|b}_\epsilon(\bm x)} & \delequal \min\big\{ j \geq 0:\ \bm X^{\eta|b}_j(\bm x) \notin I_\epsilon \big\}.
    \label{def: epsilon relaxed first exit time}
\end{align}

\begin{lemma}[Verifying conditions \eqref{eq: exit time condition lower bound} and \eqref{eq: exit time condition upper bound}]
\label{lemma: exit prob one cycle, with exit location B, first exit analysis}
\linksinthm{lemma: exit prob one cycle, with exit location B, first exit analysis}
% Let $\bar\epsilon$ be specified as in \eqref{constant bar epsilon, first exit time analysis}.
Let $\bar t$ be characterized as in Lemma~\ref{lemma: limiting measure, with exit location B, first exit analysis}.
Given any measurable $B \subseteq \R$,
any $\epsilon \in (0,\bar\epsilon]$ small enough,
and any $T > \bar t$,
% Given $\delta \in (0,1)$,
% there exists $\epsilon_0 = \epsilon_0(\delta) > 0$ such that 
% for any $\epsilon \in (0,\epsilon_0]$, any $T > \bar t$, and any measurable $B \subseteq (I_{\bar\epsilon/2})^c$,
\begin{align*}
 C(B^\circ) - \delta_B(\epsilon,T)
& \leq 
    \liminf_{\eta \downarrow 0}\inf_{\bm x:\ \norm{\bm x} \leq \epsilon }
    \frac{ 
        \P\Big(
            \tau^{\eta|b}_\epsilon(\bm x) \leq T/\eta;\ \bm X^{\eta|b}_{\tau^{\eta|b}_\epsilon(\bm x)}(\bm x) \in B
        \Big) }{\gamma(\eta) T/\eta }
    \\
    & \leq 
     \limsup_{\eta \downarrow 0}\sup_{\bm x:\ \norm{\bm x} \leq \epsilon}
    \frac{ \P\Big(
        \tau^{\eta|b}_\epsilon(\bm x) \leq T/\eta;\ \bm X^{\eta|b}_{\tau^{\eta|b}_\epsilon(\bm x)}(\bm x) \in B
    \Big) }{\gamma(\eta)T/\eta }
    \leq 
   C(B^-)
    +
    \delta_B(\epsilon,T).
\end{align*}
\end{lemma}

\begin{proof}
\linksinpf{lemma: exit prob one cycle, with exit location B, first exit analysis}
Recall that 
\begin{enumerate}[$(i)$]
    \item
        in case that $C^I_b \in (0,\infty)$, we have
        $
        \gamma(\eta)T/\eta =  C^I_b  T \cdot \big(\lambda(\eta)\big)^{\mathcal J^I_b},
        $
        $
        C(\cdot) = \widecheck{\mathbf C}^{ (\mathcal J^I_b)|b }(\ \cdot\ \setminus I)/C^I_b,
        $
        and
        $\delta_B(\epsilon,T) = {\widecheck{\delta}_B(\epsilon,T)}/({ C^I_b \cdot T })$;

    \item 
        in case that $C^I_b = 0$, we have
        $
        \gamma(\eta)T/\eta =  T \cdot \big(\lambda(\eta)\big)^{\mathcal J^I_b},
        $
        $
        C(\cdot) \equiv 0,
        $
        and
        $\delta_B(\epsilon,T) = {\widecheck{\delta}_B(\epsilon,T)}/T$.
\end{enumerate}
In both cases,
by rearranging the terms, it suffices to show that
\begin{align}
 \limsup_{\eta \downarrow 0}\sup_{\bm x:\ \norm{\bm x} \leq \epsilon}
    \frac{ \P\Big(\tau^{\eta|b}_\epsilon(\bm x) \leq T/\eta;\ \bm X^{\eta|b}_{\tau^{\eta|b}_\epsilon(\bm x)}(\bm x) \in B\Big) }{\big(\lambda(\eta)\big)^{\mathcal J^I_b} }
    & \leq 
   T \cdot \widecheck{ \mathbf C }^{ (\mathcal J^I_b)|b }(B^- \setminus I)
    +
   \widecheck{\delta}_B(\epsilon,T),
   \label{proof, goal upper bound, lemma: exit prob one cycle, with exit location B, first exit analysis}
   \\
    \liminf_{\eta \downarrow 0}\inf_{\bm x:\ \norm{\bm x} \leq \epsilon}
    \frac{ \P\Big(\tau^{\eta|b}_\epsilon(\bm x) \leq T/\eta;\ \bm X^{\eta|b}_{\tau^{\eta|b}_\epsilon(\bm x)}(\bm x) \in B\Big) }{\big(\lambda(\eta)\big)^{\mathcal J^I_b} }
    & \geq 
   T \cdot \widecheck{ \mathbf C }^{ (\mathcal J^I_b)|b }(B^\circ \setminus I ) - \widecheck{\delta}_B(\epsilon,T).
   \label{proof, goal lower bound, lemma: exit prob one cycle, with exit location B, first exit analysis}
\end{align}

Recall the definition of set $\widecheck E(\epsilon,\cdot,T)$ in \eqref{def: set check E epsilon B T, measure check C k b}.
Let $\widetilde B(\epsilon) \delequal B \setminus I_\epsilon$.
Note that
\begin{align*}
    \Big\{
    \tau^{\eta|b}_\epsilon(\bm x) \leq T/\eta;\ \bm X^{\eta|b}_{\tau^{\eta|b}_\epsilon(\bm x)}(\bm x) \in B
\Big\}
& =
\Big\{
    \tau^{\eta|b}_\epsilon(\bm x) \leq T/\eta;\ \bm X^{\eta|b}_{\tau^{\eta|b}_\epsilon(\bm x)}(\bm x) \in \widetilde B(\epsilon)
\Big\}
\\ 
& 
=
 \Big\{ 
    \bm{X}^{\eta|b}_{[0,T]}(\bm x) \in \widecheck E\big(\epsilon, \widetilde B(\epsilon) ,T\big)
\Big\}.
\end{align*}
For any $\epsilon \in (0,\bar \epsilon)$ and $\xi \in \widecheck E(\epsilon,\widetilde B(\epsilon),T)$, 
there exists $t \in [0,T]$ such that $\xi_t \notin I_\epsilon$.
On the other hand,
recall that we use $\bar B_\epsilon(\bm 0)$ to denote the closed ball with radius $\epsilon$ centered at the origin.
By part $(a)$ of Lemma~\ref{lemma: choose key parameters, first exit time analysis},
given $\epsilon \in (0,\bar\epsilon]$,
it holds for any $\xi \in \mathbb{D}^{ (\mathcal{J}^I_b - 1)|b }_{ \bar B_\epsilon(\bm 0) }[0,T](\epsilon)$
that $\xi_t \in (I_{2\bar\epsilon})^-\ \forall t \in [0,T]$.
Therefore, the claim
\begin{align}
    \dj{[0,T]}
    \bigg( \widecheck E\big(\epsilon,\widetilde B(\epsilon),T\big),\ 
    \mathbb{D}^{ (\mathcal{J}^I_b - 1)|b }_{ \bar B_\epsilon(\bm 0) }[0,T](\epsilon)
    \bigg) \geq \bar\epsilon
    \nonumber
\end{align}
holds for any $\epsilon \in (0,\bar\epsilon]$.
Next, recall the 
sequence $(\epsilon^{(n)})_{n \geq 1}$
specified in \eqref{measure C bound, 2, first exit time analysis}--\eqref{measure C bound, 1, first exit time analysis}.
For any $\epsilon > 0$ small enough we have $\epsilon \in (0,\epsilon^{(1)}]$,
and for such $\epsilon$ we set $n = n_\epsilon$ as the unique positive integer such that $\epsilon \in (\epsilon^{(n+1)},\epsilon^{(n)}]$.
% where the strictly decreasing positive real number sequence $(\epsilon^{(n)})_{n \geq 1}$
% is specified in \eqref{measure C bound, 2, first exit time analysis}--\eqref{measure C bound, 1, first exit time analysis}.
It then follows from Theorem~\ref{corollary: LDP 2} that
\begin{align}
     & \limsup_{\eta \downarrow 0}\sup_{ \bm x:\ \norm{\bm x} \leq \epsilon }
         \frac{ 
    \P\Big(
        \tau^{\eta|b}_\epsilon(\bm x) \leq T/\eta;\ \bm X^{\eta|b}_{\tau^{\eta|b}_\epsilon(\bm x)}(\bm x) \in B
    \Big)   
    }{ \big(\lambda(\eta)\big)^{ \mathcal{J}^I_b }} 
    \label{proof, ineq for upper bound, lemma: exit prob one cycle, with exit location B, first exit analysis}
    \\
    &  \leq 
    \sup_{ \bm x:\ \norm{\bm x} \leq \epsilon }
    \mathbf{C}^{ (\mathcal{J}^I_b)|b}_{[0,T]}\bigg( \Big(\widecheck E\big(\epsilon, \widetilde B(\epsilon) ,T\big)\Big)^-;\bm x\bigg)
    % \nonumber
    % \\ 
    % & 
    \leq
    T\cdot
    \bigg( 
        \widecheck{\mathbf{C}}^{ (\mathcal{J}^I_b)|b }\Big( (B \setminus I_\epsilon)^{\Delta^{(n)}} \Big) + \widecheck{\bm c}(\epsilon^{(n)})
    \bigg),
    \nonumber
\end{align}
where we applied property \eqref{measure C bound, 2, first exit time analysis} in the last inequality.
Furthermore,
\begin{align*}
    & \widecheck{\mathbf{C}}^{ (\mathcal{J}^I_b)|b }\Big( (B \setminus I_\epsilon)^{\Delta^{(n)}} \Big)
    \\
    & \leq 
    \widecheck{\mathbf{C}}^{ (\mathcal{J}^I_b)|b }\Big( B^{\Delta^{(n)}} \cup ( (I_\epsilon)^\complement )^{\Delta^{(n)}} \Big)
    \qquad 
    \text{due to }(E\cup F)^\Delta \subseteq E^\Delta \cup F^\Delta
    \\ 
    & = 
    \widecheck{\mathbf{C}}^{ (\mathcal{J}^I_b)|b }\Big( B^{\Delta^{(n)}} \cup ( (I_\epsilon)^\complement )^{\Delta^{(n)}} \cap I^\complement \Big)
    +
    \widecheck{\mathbf{C}}^{ (\mathcal{J}^I_b)|b }\Big( B^{\Delta^{(n)}} \cup ( (I_\epsilon)^\complement )^{\Delta^{(n)}} \cap I \Big)
    \\
    & \leq 
    \widecheck{\mathbf{C}}^{ (\mathcal{J}^I_b)|b }\Big( B^{\Delta^{(n)}}\setminus I\Big)
    +
    \widecheck{\mathbf{C}}^{ (\mathcal{J}^I_b)|b }\Big( ( (I_\epsilon)^\complement )^{\Delta^{(n)}} \cap I \Big)
    \\ 
    & \leq 
    \widecheck{\mathbf{C}}^{ (\mathcal{J}^I_b)|b }\Big( B^{\Delta^{(n)}}\setminus I\Big)
    +
    \widecheck{\mathbf{C}}^{ (\mathcal{J}^I_b)|b }\Big( (\partial I)^{ \epsilon + \Delta^{(n)}} \Big)
    \\
    & \leq 
    \widecheck{\mathbf{C}}^{ (\mathcal{J}^I_b)|b }\Big( B^-\setminus I\Big)
    +
    \widecheck{\mathbf{C}}^{ (\mathcal{J}^I_b)|b }\Big( (B^{\Delta^{(n)}}\cap I^\complement ) \setminus (B^- \cap I^\complement)\Big)
    +
    \widecheck{\mathbf{C}}^{ (\mathcal{J}^I_b)|b }\Big( (\partial I)^{ \epsilon + \Delta^{(n)}} \Big)
\end{align*}
Plugging this bound back into \eqref{proof, ineq for upper bound, lemma: exit prob one cycle, with exit location B, first exit analysis},
the upper bound in the claim~\eqref{proof, goal upper bound, lemma: exit prob one cycle, with exit location B, first exit analysis} then follows from 
the definition of $\widecheck{\delta}_B$ in \eqref{def: bm delta B, first exit analysis}
and our choice of $C(\cdot)$ in \eqref{def: measure C and scale gamma when applying the exit time framework}.
Similarly, by Theorem~\ref{corollary: LDP 2} and the property \eqref{measure C bound, 1, first exit time analysis}, we obtain (for any $\epsilon > 0$ small enough)
\begin{align}
    & \liminf_{\eta \downarrow 0}\inf_{ \bm x:\ \norm{\bm x} \leq \epsilon }
         \frac{ 
    \P\Big(
        \tau^{\eta|b}_\epsilon(\bm x) \leq T/\eta;\ \bm X^{\eta|b}_{\tau^{\eta|b}_\epsilon(\bm x)}(\bm x) \in B
    \Big)   
    }{ \big(\lambda(\eta)\big)^{ \mathcal{J}^I_b }} 
    \label{proof, ineq for lower bound, lemma: exit prob one cycle, with exit location B, first exit analysis}
    \\
    &  \geq 
    \inf_{ \bm x:\ \norm{\bm x} \leq \epsilon }
    \mathbf{C}^{ (\mathcal{J}^I_b)|b}_{[0,T]}\bigg( \Big(\widecheck E\big(\epsilon, \widetilde B(\epsilon) ,T\big)\Big)^\circ;\bm x\bigg)
    % \nonumber
    % \\ 
    % & 
    \geq
    (T - \bar t)\cdot 
    \bigg( 
        \widecheck{\mathbf{C}}^{ (\mathcal{J}^I_b)|b }\Big( (B \setminus I_\epsilon)_{\Delta^{(n)}} \Big) - \widecheck{\bm c}(\epsilon^{(n)})
    \bigg),
    \nonumber
\end{align}
where $n = n_\epsilon$ is the unique positive integer such that $\epsilon \in (\epsilon^{(n+1)},\epsilon^{(n)}]$.
Furthermore, by the preliminary bound $(E\cap F)_\Delta \supseteq E_\Delta \cap F_\Delta$,
\begin{align*}
  \widecheck{\mathbf{C}}^{ (\mathcal{J}^I_b)|b }\Big( (B \setminus I_\epsilon)_{\Delta^{(n)}} \Big)
  & \geq 
  \widecheck{\mathbf{C}}^{ (\mathcal{J}^I_b)|b }\Big( (B \setminus I)_{\Delta^{(n)}} \Big)
  % \\ 
  % & 
  \geq 
  \widecheck{\mathbf{C}}^{ (\mathcal{J}^I_b)|b }\Big( B_{\Delta^{(n)}} \cap (I^\complement)_{\Delta^{(n)}} \Big).
\end{align*}
Together with the fact that
$
B_\Delta \setminus I 
= B_\Delta \cap I^\complement
\subseteq
\Big(B_\Delta \cap (I^\complement)_\Delta\Big) 
\cup 
\Big( I^\complement \setminus (I^\complement)_\Delta \Big),
$
we yield
\begin{align*}
    & \widecheck{\mathbf{C}}^{ (\mathcal{J}^I_b)|b }\Big( (B \setminus I_\epsilon)_{\Delta^{(n)}} \Big)
    \geq 
  \widecheck{\mathbf{C}}^{ (\mathcal{J}^I_b)|b }\Big( B_{\Delta^{(n)}} \cap (I^\complement)_{\Delta^{(n)}} \Big)
    \\ 
    & \geq 
    \widecheck{\mathbf{C}}^{ (\mathcal{J}^I_b)|b }\Big( B_{\Delta^{(n)}} \setminus I \Big)
    -
    \widecheck{\mathbf{C}}^{ (\mathcal{J}^I_b)|b }\Big( I^\complement \setminus (I^\complement)_{\Delta^{(n)}} \Big)
    \\ 
    & \geq 
    \widecheck{\mathbf{C}}^{ (\mathcal{J}^I_b)|b }\Big( B_{\Delta^{(n)}} \setminus I \Big)
    -
    \widecheck{\mathbf{C}}^{ (\mathcal{J}^I_b)|b }\Big( (\partial I)^{\Delta^{(n)}} \Big)
    \\
    & =
    \widecheck{\mathbf{C}}^{ (\mathcal{J}^I_b)|b }\Big( B^\circ \setminus I \Big)
    -
    \widecheck{\mathbf{C}}^{ (\mathcal{J}^I_b)|b }\Big( (B^\circ \cap I^\complement) \setminus (B_{\Delta^{(n)}} \cap I^\complement ) \Big)
    -
    \widecheck{\mathbf{C}}^{ (\mathcal{J}^I_b)|b }\Big( (\partial I)^{\Delta^{(n)}} \Big).
\end{align*}
Plugging this back into \eqref{proof, ineq for lower bound, lemma: exit prob one cycle, with exit location B, first exit analysis},
we conclude the proof for the lower bound \eqref{proof, goal lower bound, lemma: exit prob one cycle, with exit location B, first exit analysis}.
\end{proof}

The next result verifies condition \eqref{eq:E3}.
Under our choice of $A(\epsilon) = \{\bm x \in \R^m:\ \norm{\bm x} < \epsilon\}$
and $I(\epsilon) = I_\epsilon$,
the event $\{\tau^{\eta}_{(I(\epsilon)\setminus A(\epsilon))^\complement}(x) > T/\eta\}$ in condition \eqref{eq:E3}
means that $\bm X^{\eta|b}_j(\bm x) \in I_\epsilon \symbol{92} \{ \bm x:\norm{\bm x} < \epsilon  \}$ for all $j \leq T/\eta$.
Also, recall
% the definition of $\bm t(\cdot)$ in \eqref{def: t epsilon function, first exit analysis},
% as well as 
our choice of 
$
\gamma(\eta)T/\eta =  C^I_b  T \cdot \big(\lambda(\eta)\big)^{\mathcal J^I_b}
$
if $C^I_b > 0$,
or 
$
\gamma(\eta)T/\eta =  T \cdot \big(\lambda(\eta)\big)^{\mathcal J^I_b}
$
if $C^I_b = 0$.
To verify condition \eqref{eq:E3}, it suffices to prove the following result.

\begin{lemma}[Verifying condition \eqref{eq:E3}]
\label{lemma: fixed cycle exit or return taking too long, first exit analysis}
\linksinthm{lemma: fixed cycle exit or return taking too long, first exit analysis}
 % Let Assumptions \ref{assumption gradient noise heavy-tailed}, \ref{assumption: f and sigma, stationary distribution of SGD}, and \ref{assumption: shape of f, first exit analysis} hold.
 Let $\bar\epsilon > 0$ be the constant in \eqref{constant bar epsilon, new, 1, first exit time analysis}--\eqref{constant bar epsilon, new, 3, first exit time analysis},
 and $\bm t(\cdot)$ be defined as in \eqref{def: t epsilon function, first exit analysis}.
 Given $k \geq 1$ and $\epsilon \in (0,\bar\epsilon)$,
 it holds for any $T \geq k \cdot \bm{t}(\epsilon/2)$ that
 \begin{align*}
     \lim_{\eta \downarrow 0}\ \sup_{ \bm x \in I_\epsilon }
     \ \, \frac1{\lambda^{k-1}(\eta)}\P\Big( \bm X^{\eta|b}_j(\bm x) \in I_\epsilon\setminus \{ \bm x:\norm{\bm x} < \epsilon  \}\ \  \forall j \leq T/\eta \Big)
      = 0.
 \end{align*}
\end{lemma}

\begin{proof}
\linksinpf{lemma: fixed cycle exit or return taking too long, first exit analysis}
In this proof, we write $\xi(t) = \xi_t$ for any $\xi \in \D[0,T]$,
and write $B_\epsilon(\bm 0) = \{\bm x \in \R^m:\ \norm{\bm x} < \epsilon\}$.
Note that
$\big\{
\bm X^{\eta|b}_j(\bm x) \in I_\epsilon \setminus B_\epsilon(\bm 0)\ \ \forall j \leq T/\eta
\big\}
= \big\{ \bm{X}^{\eta|b}_{[0,T]}(x) \in E(\epsilon)\big\}
$
where
\begin{align*}
    E(\epsilon) \delequal 
    \Big\{ \xi \in \mathbb{D}[0,T]:\ \xi(t) \in I_\epsilon\setminus B_\epsilon(\bm 0) \ \ \forall t \in[0,T]\Big\}.
\end{align*}
Recall the definition of $\mathbb{D}^{(k)|b}_{A}[0,T](\epsilon)$ in \eqref{def: l * tilde jump number for function g, clipped SGD}.
We claim that $E(\epsilon)$ is bounded away from $\mathbb{D}^{ (k - 1)|b }_{ (I_\epsilon)^- }[0,T](\epsilon)$.
This allows us to apply Theorem \ref{corollary: LDP 2} and conclude that
\begin{align*}
    \sup_{ x \in I_\epsilon }
    \P\Big( \bm{X}^{\eta|b}_{[0,T]}(\bm x) \in E(\epsilon)\Big)
    =
    \bm{O}\big( \lambda^{k}(\eta) \big)
    =
    \bm{o}\big(\lambda^{k - 1}(\eta)\big)\ \ \ \text{as }\eta \downarrow 0.
\end{align*}
Now, it only remains to verify that $E(\epsilon)$ is bounded away from $\mathbb{D}^{ (k - 1)|b }_{ (I_\epsilon)^- }[0,T](\epsilon)$,
which follows if we show
the existence of some $\delta > 0$ such that
\begin{align}
    \dj{[0,T]}(\xi,\xi^\prime) \geq \delta > 0
    \qquad
    \forall \xi \in \mathbb{D}^{ (k - 1)|b }_{ (I_\epsilon)^- }[0,T](\epsilon),\ \xi^\prime \in E(\epsilon).
    \label{proof, ineq, lemma: fixed cycle exit or return taking too long, first exit analysis}
\end{align}
First, 
by definition of $E(\epsilon)$, we have $\xi^\prime(t) \in I_\epsilon\ \forall t \in [0,T]$
for any $\xi^\prime \in E(\epsilon)$.
Also, note that
$
\inf\{ \norm{\bm x - \bm y}:\ \bm x \in I_\epsilon,\ \bm y \notin I_{\epsilon/2} \} \geq \epsilon/2.
$
Therefore, 
for any $\xi \in \mathbb{D}^{ (k - 1)|b }_{ (I_\epsilon)^- }[0,T](\epsilon)$,
if $\xi(t) \notin I_{\epsilon/2}$
for some $t \in [0,T]$,
then
$
\dj{[0,T]}(\xi,\xi^\prime) \geq \epsilon/2 > 0.
$
Next,
given $\xi \in \mathbb{D}^{ (k - 1)|b }_{ (I_\epsilon)^- }[0,T](\epsilon)$,
we consider the case where $\xi(t) \in I_{\epsilon/2}$
for any $t \in [0,T]$.
By definition of
$
\mathbb{D}^{ (k - 1)|b }_{  (I_\epsilon)^- }[0,T](\epsilon),
$
there is some $\bm x \in (I_\epsilon)^-$, $\textbf W \in \R^{ d \times (k - 1)}$, 
$\textbf V \in \big(\bar{B}_{\epsilon}(\bm 0)\big)^{k-1}$,
and $(t_1,\cdots,t_{k - 1}) \in (0,T]^{ k - 1 \uparrow}$ such that
$
\xi = \bar h^{ (k - 1)|b }_{[0,T]}\big(\bm x,\textbf W,\textbf V,(t_1,\cdots,t_{k - 1})\big).
$
Under the convention that $t_0 = 0$ and $t_{k} = T$,
we have (for each $j \in [k]$)
\begin{align}
    \xi(t) = \bm{y}_{ t - t_{j-1} }\big(\xi(t_{j-1})\big)\qquad 
    \forall t \in [t_{j-1},t_j).
    \label{proof, property of path xi, lemma: fixed cycle exit or return taking too long, first exit analysis}
\end{align}
Here, $\bm{y}_\cdot(x)$ is the ODE defined in \eqref{def ODE path y t}.
By the running assumption in this lemma that $T \geq k \cdot \bm{t}(\epsilon/2)$, 
there must be some $j \in [k]$ such that $t_j - t_{j - 1} \geq \bm{t}({\epsilon}/{2})$.
Since $\xi(t) \in I_{\epsilon/2}\ \forall t \in [0,T]$, we have $\xi(t_{j-1}) \in I_{\epsilon/2}$.
Combining \eqref{proof, property of path xi, lemma: fixed cycle exit or return taking too long, first exit analysis} and the property \eqref{property: t epsilon function, first exit analysis}, we get
$
\lim_{t \uparrow t_j}\xi(t) \in \bar B_{\epsilon/2}(\bm 0) \subset B_\epsilon(\bm 0).
$
On the other hand, by definition of $E(\epsilon)$, we have $\xi^\prime(t) \notin B_\epsilon(\bm 0)$
for all $t \in [0,T]$, 
which implies 
$
\dj{[0,T]}(\xi,\xi^\prime) \geq \frac{\epsilon}{2}.
$
This concludes the proof.
\end{proof}

Let
\begin{align}
     \notationdef{notation-R-eta-b-epsilon-return-time}{R^{\eta|b}_\epsilon(\bm x)} & \delequal 
     \min\bigg\{ j \geq 0:\ \norm{\bm X^{\eta|b}_j(\bm x)} < \epsilon \bigg\}
    \label{def: return time R in first cycle, first exit analysis}
\end{align}
be the first time $\bm X^{\eta|b}_j(\bm x)$ returns to the $\epsilon$-neighborhood of the origin.
Lastly, we establish condition \eqref{eq:E4}
under our choice of $A(\epsilon) = \{ \bm x \in \R^m:\ \norm{\bm x} < \epsilon \}$.
Note that the first visit time $\tau^{\eta}_{A(\epsilon)}(x)$ therein coincides with $R^{\eta|b}_\epsilon(x)$.

\begin{lemma}[Verifying condition \eqref{eq:E4}]
\label{lemma: cycle, efficient return}
\linksinthm{lemma: cycle, efficient return}
% Let Assumptions \ref{assumption gradient noise heavy-tailed}, \ref{assumption: f and sigma, stationary distribution of SGD}, and \ref{assumption: shape of f, first exit analysis} hold.
    Let $\bm{t}(\cdot)$ be defined as in \eqref{def: t epsilon function, first exit analysis} and
    \begin{align*}
        E(\eta,\epsilon,\bm x)
        \delequal 
        \Bigg\{ R^{\eta|b}_\epsilon(\bm x) \leq \frac{\bm t(\epsilon/2)}{\eta};\ \bm X^{\eta|b}_j(\bm x) \in I\ \forall j \leq R^{\eta|b}_\epsilon(\bm x) \Bigg\}.
    \end{align*}
    For any $\epsilon \in (0,\bar\epsilon)$,
    $
     \lim_{\eta\downarrow 0}\sup_{\bm x \in (I_\epsilon)^-}
        \P\Big( \big(E(\eta,\epsilon,\bm x)\big)^\complement\Big) = 0.
    $
    % \begin{align*}
    %     \lim_{\eta\downarrow 0}\sup_{x \in [s_\text{left} + \epsilon,s_\text{right}-\epsilon]}
    %     \P\Big( \big(A(\eta,\epsilon,x)\big)^\complement\Big) = 0.
    % \end{align*}
\end{lemma}

\begin{proof}
\linksinpf{lemma: cycle, efficient return}
In this proof, we write $B_\epsilon(\bm 0) = \{\bm x \in \R^m:\ \norm{\bm x} < \epsilon\}$.
    Note that
    $
    \big(E(\eta,\epsilon,\bm x)\big)^\complement
    \subseteq 
    \big\{ \bm{X}^{\eta|b}_{ [0,\bm{t}(\epsilon/2)] }(\bm x) \in E^*_1(\epsilon) \cup E^*_2(\epsilon) \big\},
    $
    where 
\begin{align*}
    E^*_1(\epsilon) & \delequal \bigg\{ \xi \in \mathbb{D}[0,\bm{t}(\epsilon/2)]:\ \xi(t) \notin B_\epsilon(\bm 0)\ \forall t \in [0,\bm{t}(\epsilon/2)] \bigg\},
    \\
    E^*_2(\epsilon) & \delequal \bigg\{ \xi \in \mathbb{D}[0,\bm{t}(\epsilon/2)]:\ \exists 0 \leq s < t \leq \bm{t}(\epsilon/2)\ s.t.\ \xi(t) \in B_\epsilon(\bm 0),\ \xi(s) \notin I \bigg\}.
    % \\
    % E^*_2(\epsilon) & \delequal \big\{ \xi \in \mathbb{D}[0,\bm{t}(\epsilon/2)]:\ \exists 0 \leq s \leq t \leq \bm{t}(\epsilon/2)\ s.t.\ \xi(t) \in (-\epsilon,\epsilon),\ |\xi(s)| > |x| + \frac{\epsilon}{2} \big\}.
\end{align*}
Recall the definition of $\mathbb{D}^{(k)|b}_{A}[0,T](\epsilon)$ in \eqref{def: l * tilde jump number for function g, clipped SGD}.
We claim that both $E^*_1(\epsilon)$ and $E^*_2(\epsilon)$ are bounded away from 
\begin{align*}
    \mathbb{D}^{(0)|b}_{ (I(\epsilon))^- }[0,\bm{t}(\epsilon/2)] = 
    \bigg\{ 
        \Big\{ \bm{y}_t(\bm x):\ t \in [0,\bm{t}(\epsilon/2)] \Big\}:\ \bm x \in  \big(I_\epsilon\big)^-
    \bigg\}.
\end{align*}
To see why, note that
% the property \eqref{property: bounded away, I epsilon, different epsilon} allows us to fix some $\delta > 0$ such that  
$
\inf\{ \norm{\bm x - \bm y}:\ \bm x \in I_\epsilon,\ \bm y \notin I_{\epsilon}) \} \geq \epsilon/2.
$
On the other hand, properties \eqref{constant bar epsilon, new, 2, first exit time analysis} and \eqref{property: t epsilon function, first exit analysis}
imply that
 $\bm{y}_{ \bm{t}(\epsilon/2) }(\bm x) \in \bar B_{\epsilon/2}(\bm 0)$ for any $\bm x \in \big(I_\epsilon\big)^-$.
% from Assumption \ref{assumption: shape of f, first exit analysis} and property \eqref{property: t epsilon function, first exit analysis},
% we get and $\bm{y}_t(x) \in I_\epsilon$, $|\bm y_t(x)| \leq |x|$ for all $t$ and $x$ such that $t \in [0,\bm{t}(\epsilon/2)]$ and $x \in I^-_\epsilon$.
Therefore,
\begin{align}
    \dj{[0,\bm{t}(\epsilon/2)]}
    \bigg( 
        \mathbb{D}^{(0)|b}_{ (I(\epsilon))^- }[0,\bm{t}(\epsilon/2)],\ E^*_1(\epsilon)
    \bigg) & \geq \frac{\epsilon}{2} > 0,
    \label{proof, bound 1, lemma: cycle, efficient return}
     % \\ 
     % \dj{[0,\bm{t}(\epsilon/2)]}\Big( \mathbb{D}^{(0)|b}_{ I_\epsilon }[0,\bm{t}(\epsilon/2)],\ E^*_3(\epsilon)\Big) & \geq \frac{\epsilon}{2}> 0.
     % \label{proof, bound 3, lemma: cycle, efficient return}
\end{align}
Meanwhile, by property \eqref{property: coverage of I_epsilon, bounded away from I complement},
\begin{align}
    \dj{[0,\bm{t}(\epsilon/2)]}
    \bigg( 
        \mathbb{D}^{(0)|b}_{ (I(\epsilon))^- }[0,\bm{t}(\epsilon/2)],\ E^*_2(\epsilon)
    \bigg) & > 0.
     \label{proof, bound 2, lemma: cycle, efficient return}
\end{align}
This allows us to apply Theorem \ref{corollary: LDP 2} and obtain 
$$
\sup_{\bm x \in (I_\epsilon)^-}
\P\bigg( \big(E(\eta,\epsilon,\bm x)\big)^\complement\bigg) \leq 
\sup_{\bm x \in (I_\epsilon)^-}
\P\bigg(\bm{X}^{\eta|b}_{ [0,\bm{t}(\epsilon/2)] }(\bm x) \in E^*_1(\epsilon) \cup E^*_2(\epsilon)\bigg)
= \bm{O}\big(\lambda(\eta)\big)
$$
as $\eta \downarrow 0$.
To conclude the proof, one only needs to note that $\lambda(\eta) \in \RV_{\alpha - 1}(\eta)$ (with $\alpha > 1$) and hence $\lim_{\eta \downarrow 0}\lambda(\eta) = 0$.
\end{proof}

We conclude this section with the proof of Theorem \ref{theorem: first exit time, unclipped}.

\begin{proof}[Proof of Theorem \ref{theorem: first exit time, unclipped}]
\linksinpf{theorem: first exit time, unclipped}
As noted above,
the claim $C^I_b < \infty$ is verified by Lemma~\ref{lemma: exit rate strictly positive, first exit analysis}.
Next, since Lemmas~\ref{lemma: exit prob one cycle, with exit location B, first exit analysis}--\ref{lemma: cycle, efficient return}
verify Condition~\ref{condition E2},
Theorem~\ref{theorem: first exit time, unclipped}
follows immediately from Theorem~\ref{thm: exit time analysis framework}.
\end{proof}

\section{Proof of Corollary~\ref{corollary: first exit time, untruncated case}}
\label{sec: proof for untruncated case}

This section collects the proof of Corollary~\ref{corollary: first exit time, untruncated case}.
We first give a straightforward bound for the law of geometric random variables.
\begin{lemma} \label{lemmaGeomFront}
\linksinthm{lemmaGeomFront}
Let $a:(0,\infty) \to (0,\infty)$, $b:(0,\infty) \to (0,\infty)$ 
be two functions
such that 
$\lim_{\epsilon \downarrow 0} a(\epsilon) = 0, \lim_{\epsilon \downarrow 0} b(\epsilon) = 0$.
Let $\{U(\epsilon): \epsilon > 0\}$ be a family of geometric RVs with success rate $a(\epsilon)$,
i.e.
$\P(U(\epsilon) > k) = (1 - a(\epsilon))^{k}$ for $k \in \mathbb{N}$.
For any $c > 1$, there exists $\epsilon_0 > 0$ such that
$$\exp\Big(-\frac{c\cdot a(\epsilon)}{b(\epsilon)}\Big) \leq \P\Big( U(\epsilon) > \frac{1}{b(\epsilon)} \Big) \leq \exp\Big(-\frac{a(\epsilon)}{c\cdot b(\epsilon)}\Big)
\ \ \ \forall \epsilon \in (0,\epsilon_0)
.$$

\end{lemma}

\begin{proof}
\linksinpf{lemmaGeomFront}%
Note that
$\P( U(\epsilon) > \frac{1}{b(\epsilon)} ) = \big(1 - a(\epsilon)\big)^{ \floor{1/b(\epsilon)} }.$
By taking logarithm on both sides, we have
\begin{align*}
    \ln \P\Big( U(\epsilon) > \frac{1}{b(\epsilon)} \Big) & = \floor{1/b(\epsilon)}\ln\Big(1 - a(\epsilon)\Big) = \frac{\floor{1/b(\epsilon)} }{ 1/b(\epsilon) }\frac{\ln\Big(1 - a(\epsilon)\Big) }{-a(\epsilon) }\frac{-a(\epsilon) }{ b(\epsilon) }.
\end{align*}
Since $\lim_{x \rightarrow 0}\frac{\ln(1 + x)}{x} = 1$, we know that for $\epsilon$ sufficiently small, we will have
$-c \frac{a(\epsilon)}{b(\epsilon)}   \leq \ln \P\Big( U(\epsilon) > \frac{1}{b(\epsilon)} \Big) \leq -\frac{a(\epsilon)}{c\cdot b(\epsilon)}.$
% \begin{align}
%  -c \frac{a(\epsilon)}{b(\epsilon)}   \leq \ln \P\Big( U(\epsilon) > \frac{1}{b(\epsilon)} \Big) \leq -\frac{a(\epsilon)}{c\cdot b(\epsilon)}.
%  \label{proofGeomBound}
% \end{align}
By taking exponential on both sides, we conclude the proof.
% (ii)
% To begin with, from the lower bound of part (i), we have
% $$\P\Big( U(\epsilon) \leq \frac{1}{b(\epsilon)} \Big) = 1 -  \P\Big( U(\epsilon) > \frac{1}{b(\epsilon)} \Big)
% \leq 1 -\exp\Big(-c\cdot \frac{a(\epsilon)}{b(\epsilon)}\Big)
% \leq 
%  \frac{c\cdot a(\epsilon)}{b(\epsilon)}
% $$
% for sufficiently small $\epsilon > 0$.
% % by applying the fact that $1-\exp(-x) \leq x\ \forall x\in \mathbb{R}$ at $x=c\cdot a(\epsilon)/b(\epsilon)$.
% For the lower bound, recall that
% $1 - \exp(-x) \geq \frac{x}{\sqrt{c}}$ holds for $x>0$ sufficiently close to $0$. 
% Since we assume $\lim_{\epsilon \downarrow 0}a(\epsilon)/b(\epsilon) = 0$, applying this bound with $x = \frac{a(\epsilon)}{\sqrt c\cdot b(\epsilon)}$
% along with the upper bound of part (i), we get
% $$\P\Big( U(\epsilon)  \leq \frac{1}{b(\epsilon)} \Big) 
% \geq 1 - \exp\Big(-\frac{1}{\sqrt{c}}\cdot\frac{a(\epsilon)}{b(\epsilon)}\Big) 
% \geq \frac{a(\epsilon)}{c\cdot b(\epsilon)}$$
% for sufficiently small $\epsilon$.
% % Therefore, it holds for any $\epsilon$ small enough that
% % $ \P\Big( U(\epsilon)  \leq \frac{1}{b(\epsilon)} \Big) \geq \frac{1}{c}\cdot\frac{a(\epsilon)}{b(\epsilon)}. $
\end{proof}

\begin{proof}[Proof of Corollary~\ref{corollary: first exit time, untruncated case}]
\linksinpf{corollary: first exit time, untruncated case}
Note that the value of $\bm \sigma(\cdot)$ and $\bm a(\cdot)$ outside of the domain $I$ has no impact on the first exit analysis.
Therefore, by modifying the value of $\bm \sigma(\cdot)$ and $\bm a(\cdot)$ outside of $I$,
we can assume w.l.o.g.\ that
\begin{align}
    \norm{\bm a(\bm x)}\vee \norm{\bm \sigma(\bm x)} \leq C
    \qquad
    \forall \bm x \in \R^m.
    \label{boundedness assumption, corollary: first exit time, untruncated case}
\end{align}
for some $C \in (0,\infty)$.
That is, we can impose the boundedness condition in Assumption~\ref{assumption: boundedness of drift and diffusion coefficients} w.l.o.g.
% there is some
% $C > 0$
% such that $0 \leq \sigma(x) \leq C$ 
% and $|a(x)| \leq C$ for all $x \in \R$.
% Also, from the continuity of $a(\cdot)$ (see Assumption \ref{assumption: lipschitz continuity of drift and diffusion coefficients}),
% we are able to fix some $\widetilde{C} \in [ C,\infty)$ such that 
% $
% |a(x)| \leq \widetilde{C}
% $
% for all $x \in [s_\text{left},s_\text{right}]$.

We start with a few observations.
First, under any $\eta \in (0,\frac{b}{2C})$, 
on the event $\{\eta\norm{\bm Z_j} \leq \frac{b}{2C}\ \forall j \leq t\}$ the norm of the step-size (before truncation) 
$\eta \bm a\big(\bm X^{\eta|b}_{j - 1}(\bm x)\big) + \eta \bm \sigma\big(\bm X^{\eta|b}_{j - 1}(\bm x)\big)\bm Z_j$ of $\bm X_j^{\eta|b}(\bm x)$ is less than $b$ for each $j \leq t$. 
Therefore, $\bm X^{\eta|b}_j(\bm x)$ and $\bm X^\eta_j(\bm x)$ coincide for such $j$'s.
\elaborate{First, note that
given any $\eta \in (0, \frac{b}{2 C })$ and $|w| \leq \frac{b}{2C}$, we have
\begin{align*}
    \Big| \eta a(y) + \sigma(y) w   \Big|
    & \leq 
    \eta |a(y)| + | \sigma(y) | \cdot |w|
    % \\
    % & 
    \leq \eta \cdot C + C \cdot \frac{b}{ 2C }
    < \frac{b}{2} + \frac{b}{2} = b
    \qquad\forall y \in I^-.
\end{align*}
As a result, given any $t > 0$, $x \in \R$,  and $\eta \in (0,1 \wedge \frac{b}{2 C })$,
it holds on event $\{\eta|Z_j| \leq \frac{b}{2C}\ \forall j \leq t\}$ that
\begin{align*}
    X^{\eta|b}_j(x) & = X^{\eta|b}_{j - 1}(x) + \varphi_b\Big( \eta a\big(X^{\eta|b}_{j - 1}(x)\big) + \eta \sigma\big(X^{\eta|b}_{j - 1}(x)\big)Z_j  \Big)
    \\
    & = 
    X^{\eta|b}_{j - 1}(x) + \eta a\big(X^{\eta|b}_{j - 1}(x)\big) + \eta \sigma\big(X^{\eta|b}_{j - 1}(x)\big)Z_j
    \ \ \ \ \forall j \leq \tau^{\eta|b}(x) \wedge t.
\end{align*}
}%
In other words, for any $\eta \in (0,\frac{b}{2C})$, on event $ \big\{ \eta\norm{\bm Z_j} \leq \frac{b}{2C}\ \forall j \leq t\big\}$ 
we have
\begin{equation}
    \bm X^{\eta|b}_j(\bm x) = \bm X^\eta_j(\bm x)\qquad \forall j \leq t.
    \label{proof, observation 1, theorem: first exit time, unclipped}
\end{equation}
Next, recall that $I$ is a bounded under Assumption~\ref{assumption: shape of f, first exit analysis}.
Therefore, under $b > \sup_{\bm x \in I}\norm{\bm x}$,
it holds for $\bm w \in \R^d$ that
\begin{align*}
    \varphi_b\big(\bm\sigma(\bm 0)\bm w\big) \notin I\qquad \Longleftrightarrow \qquad \bm \sigma(\bm 0)\bm w \notin I.
\end{align*}
\elaborate{To see why, first suppose that $|\sigma(0)\cdot w| \geq b$. 
Then $\sigma(0)\cdot w \notin I$. 
In the meantime, we also have $\big|\varphi_b\big(\sigma(0)\cdot w\big)\big| = b$, 
and hence, $\varphi_b\big(\sigma(0)\cdot w\big) \notin I$.
On the other hand, the equivalence is trivial if $|\sigma(0)\cdot w| < b$ since $\varphi_b(\sigma(0)\cdot w)  = \sigma(0)\cdot w$ in such a case.
This establishes claim \eqref{proof, observation 2, theorem: first exit time, unclipped}.
}%
As a result, for all $b$ large enough, we have
\begin{align}
   C^I_b = \widecheck{ \mathbf{C} }^{(1)|b}(I^\complement) 
   & = \int \mathbbm{I}\Big\{ \varphi_b\big(\bm \sigma(\bm 0) \bm w\big) \notin  I \Big\} \big((\nu_\alpha \times \mathbf S)\circ \Phi\big)(d \bm w)
   \nonumber
   \\
   & = \int \mathbbm{I}\Big\{ \bm \sigma(\bm 0) \bm w \notin I \Big\} \big((\nu_\alpha \times \mathbf S)\circ \Phi\big)(d \bm w)
   = \widecheck{ \mathbf{C} }(I^\complement) \delequal C^I_\infty.
   \label{proof, observation 2, theorem: first exit time, unclipped}
\end{align}
Similarly, one can show that for all $b$ large enough,
\begin{align}
     \widecheck{ \mathbf{C} }^{(1)|b}(\partial I)
     =
      \widecheck{ \mathbf{C} }^{(1)}(\partial I).
      \label{proof, observation, mass on the boundary set of I, theorem: first exit time, unclipped}
\end{align}
Moreover, given any measurable $A \subseteq \R$ such that $r_A = \inf\{\norm{\bm x}:\ \bm x \in A\} > 0$,
we claim that
\begin{align}
    \lim_{b \to \infty}\widecheck{\mathbf C}^{(1)|b}(A) = \widecheck{ \mathbf C }(A).
    \label{proof, observation 2 general form, theorem: first exit time, unclipped}
\end{align}
This claim follows from a simple application of the dominated convergence theorem.
Indeed, by definition, we have
$
\widecheck{\mathbf C}^{(1)|b}(A) = \int \mathbbm{I}\big\{ \varphi_b\big(\bm \sigma(\bm 0) \bm w\big) \in A \big\}
 \big((\nu_\alpha \times \mathbf S)\circ \Phi\big)(d \bm w).
$
For $f_b(\bm w) \delequal \mathbbm{I}\big\{ \varphi_b\big(\bm \sigma(\bm 0) \bm w\big) \in A \big\}$,
we first note that given $\bm w \in \R^m$, we have $f_b(\bm w) = f(\bm w)\delequal \mathbbm{I}\big\{ \bm \sigma(\bm 0) \bm w \in A \big\}$ for all $b > \norm{\bm w} \norm{\bm  \sigma(\bm 0)}$.
Therefore, the point-wise convergence $\lim_{b \to \infty}f_b(\bm w) = f(\bm w)$ holds for all $\bm w \in \R^m$.
Next, 
observe that
\begin{align*}
    \big\{ \varphi_b\big(\bm \sigma(\bm 0)\bm w\big) \in A \big\}
    &
    \subseteq 
    \big\{
        \norm{ \bm \sigma(\bm 0)\bm w } \geq r_A
    \big\}
    \subseteq 
    \big\{
        \norm{ \bm \sigma(\bm 0)}\cdot \norm{\bm w } \geq r_A
    \big\}
    =
    \big\{
        \norm{\bm w } \geq r_A/\norm{ \bm \sigma(\bm 0)}
    \big\}.
\end{align*}
This implies
$
f_b(\bm w) \leq \mathbbm{I}\big\{ \norm{\bm w } \geq r_A/\norm{ \bm \sigma(\bm 0)} \big\}
$
for all $b > 0$ and $\bm w \in \R^d$.
Also, by definition of the measure $\nu_\alpha$ in \eqref{def: measure nu alpha},
\begin{align}
    \int \mathbbm{I}\big\{ \norm{\bm w} \geq r_A/ \norm{\bm \sigma(\bm 0)} \big\}  \big((\nu_\alpha \times \mathbf S)\circ \Phi\big)(d \bm w) = ( \norm{\bm \sigma(\bm 0)}/r_A)^\alpha < \infty.
    \label{proof: upper bound for measure hat C 1 b, theorem: first exit time, unclipped}
\end{align}
The last inequality follows from $r_A > 0$.
This allows us to apply dominated convergence theorem and establish \eqref{proof, observation 2 general form, theorem: first exit time, unclipped}.

Moving on, we verify a few regularity conditions.
By repeating the calculations in \eqref{proof: upper bound for measure hat C 1 b, theorem: first exit time, unclipped} with  $A = I^\complement$,
we are able to verify the condition ${C^I_\infty} =\widecheck{ \mathbf{C} }(I^\complement) < \infty$
in Corollary~\ref{corollary: first exit time, untruncated case}.
Next, by the convention in \eqref{def: 0 jump coverage set, first exit times},
we have that $\mathcal{G}^{(0)|b}(\epsilon)$ is bounded away from $I^\complement$ for all $\epsilon > 0$ small enough and all $b > 0$.
In the meantime, recall the definition of 
$
\mathcal{G}^{(1)|b} = \big\{ \varphi_b\big(\bm \sigma(\bm 0)\bm w\big):\ \bm w \in \R^d  \big\}.
$
Due to $\norm{\bm \sigma(\bm 0)} > 0$, there exists $\bm w^*$ such that 
$
\norm{\bm \sigma(\bm 0)\bm w^*} > \sup_{\bm x \in I}\norm{\bm x},
$
and hence $\bm \sigma(\bm 0)\bm w^* \notin I$.
As a result, for all $b > \norm{\bm \sigma(\bm 0)\bm w^*}$ we have
$
\mathcal{G}^{(1)|b} \cap I^\complement \neq \emptyset.
$
That is, we have shown that 
\begin{align*}
    \mathcal J^I_b = 1\qquad\text{for all $b > 0$ large enough};
\end{align*}
see \eqref{def: first exit time, J *} for the definition.
Together with \eqref{proof, observation, mass on the boundary set of I, theorem: first exit time, unclipped} and the running assumption $\widecheck{\mathbf C}(\partial I) = 0$ in Corollary~\ref{corollary: first exit time, untruncated case},
we have 
$
 \widecheck{ \mathbf{C} }^{(1)|b}(\partial I) = 0
$
for all $b$ large enough.
These conditions will allow us to apply Theorem~\ref{theorem: first exit time, unclipped}, with $b > 0$ large enough, in the remainder of this proof.

Now, we fix $t \geq 0$ and $B \subseteq I^c$, and recall that
our goal is to study the probability of the event
$$A(\eta,\bm x) \delequal \big\{ C^I_\infty H(\eta^{-1}) \tau^\eta(\bm x) > t,\ \bm X^\eta_{\tau^\eta(\bm x)}(\bm x) \in B \big\}.$$
Here, note that $\lambda(\eta) = \eta^{-1}  H(\eta^{-1})$ and hence $\eta \cdot \lambda(\eta) = H(\eta^{-1})$.
Also, henceforth in the proof we only consider $b$ large enough
such that
$
C^I_\infty = C^I_b;
$
see \eqref{proof, observation 2, theorem: first exit time, unclipped}.
We focus on the case where $C^I_\infty > 0$, but we stress that the proof for the case with $C^I_\infty = 0$ is almost identical.
First,
we arbitrarily pick some $T > t$ and observe that
\begin{align}
    & A(\eta,\bm x) 
    \nonumber \\ 
    & = 
    \underbrace{\Big\{ C^I_\infty H(\eta^{-1}) \tau^\eta(\bm x) \in (t,T],\ \bm X^\eta_{\tau^\eta(\bm x)}(\bm x) \in B \Big\}}_{ \delequal A_1(\eta,\bm x,T) }
    \cup 
    \underbrace{\Big\{ C^I_\infty H(\eta^{-1}) \tau^\eta(\bm x) > T,\ \bm X^\eta_{\tau^\eta(\bm x)}(\bm x) \in B \Big\}}_{ \delequal A_2(\eta,\bm x,T) }.
    \label{proof, decompose event A, theorem: first exit time, unclipped}
\end{align}
Let $E_b(\eta,T) \delequal \big\{ \eta\norm{\bm Z_j} \leq \frac{b}{2C}\ \forall j \leq \frac{T}{ C^I_\infty H(\eta^{-1}) }   \big\}$ and note that
$$
A_1(\eta,\bm x,T) = \Big(A_1(\eta,\bm x,T) \cap E_b(\eta,T) \Big) \cup \Big(A_1(\eta,\bm x,T) \setminus E_b(\eta,T) \Big).
$$
Moreover, for all $\eta \in (0,\frac{b}{2C})$,
\begin{align*}
    & \P\Big(A_1(\eta,\bm x,T) \cap E_b(\eta,T) \Big) 
    \\ 
    & = 
    \P\bigg( \Big\{  C^I_b \eta \cdot \lambda(\eta) \tau^{\eta|b}(\bm x) \in (t,T],\ \bm X^{\eta|b}_{ \tau^{\eta|b}(\bm x) }(\bm x) \in B  \Big\} \cap E_b(\eta,T) \bigg)
    \qquad 
    \text{due to \eqref{proof, observation 1, theorem: first exit time, unclipped} and \eqref{proof, observation 2, theorem: first exit time, unclipped}}
    \\ 
    & \leq 
    \P\bigg( C^I_b \eta \cdot \lambda(\eta) \tau^{\eta|b}(\bm x) \in (t,T],\ \bm X^{\eta|b}_{ \tau^{\eta|b}(\bm x) }(\bm x) \in B  \bigg)
    \\ & = 
    \P\bigg( C^I_b \eta \cdot \lambda(\eta) \tau^{\eta|b}(\bm x) > t,\ \bm X^{\eta|b}_{ \tau^{\eta|b}(\bm x) }(\bm x) \in B  \bigg)
    % \\
    % &\qquad \qquad 
    -
    \P\bigg( C^I_b \eta \cdot \lambda(\eta) \tau^{\eta|b}(\bm x) > T,\ \bm X^{\eta|b}_{ \tau^{\eta|b}(\bm x) }(\bm x) \in B  \bigg).
\end{align*}
By Theorem~\ref{theorem: first exit time, unclipped} and claim~\eqref{proof, observation 2, theorem: first exit time, unclipped}, we get
\begin{align}
\limsup_{\eta \downarrow 0}\sup_{\bm x \in I_\epsilon }\P\Big(A_1(\eta,\bm x,T) \cap E_b(\eta,T) \Big) 
\leq 
\frac{ \widecheck{\mathbf{C}}^{ (1)|b }(B^-) }{ C^I_\infty }\cdot\exp(-t) - \frac{ \widecheck{\mathbf{C}}^{ (1)|b }(B^\circ) }{ C^I_\infty }\cdot\exp(-T).
\label{proof, ineq 0, theorem: first exit time, unclipped}
\end{align}
Meanwhile,
$$
\sup_{\bm x \in I_\epsilon}\P\big(A_1(\eta,\bm x,T) \setminus E_b(\eta,T)\big) \leq \P\big( ( E_b(\eta,T))^\complement\big)
=
\P\bigg( \eta\norm{\bm Z_j} > \frac{b}{2C}\text{ for some }j \leq \frac{T}{C^I_\infty H(\eta^{-1})} \bigg).
$$
Applying Lemma \ref{lemmaGeomFront}, we get (recall that $H(\cdot) = \P(\norm{\bm Z_j} > \cdot)$ and $H(x)\in \RV_{-\alpha}(x)$ as $x \to \infty$)
\begin{align}
    % &  
    \limsup_{\eta \downarrow 0}\P\bigg( \eta\norm{\bm Z_j} >\frac{b}{2C}\text{ for some }j \leq \frac{T}{C^I_\infty H(\eta^{-1})} \bigg)
    \nonumber
    % \\
    & = 
    1 - 
    \liminf_{\eta \downarrow 0}\P\bigg( \text{Geom}\Big(H\big( \frac{b}{\eta \cdot 2 C} \big)\Big) > \frac{T}{C^I_\infty H(\eta^{-1})} \bigg)
    \nonumber
    \\
    &
    \leq 
    1 - \lim_{\eta \downarrow 0}\exp\bigg( -\frac{T \cdot  H(\eta^{-1} \cdot \frac{b}{2C} )}{ C^I_\infty H(\eta^{-1})  }\bigg)
    \nonumber
    \\ 
    & = 
    1 - \exp\bigg( - \frac{T}{C^I_\infty} \cdot \Big( \frac{2C}{b} \Big)^\alpha  \bigg).
    % \qquad 
    % \text{due to }H(x) \in \RV_{-\alpha}(x)\text{ as }x \to \infty
    \label{proof, ineq 1, theorem: first exit time, unclipped}
\end{align}
Similarly,
\begin{align*}
    A_2(\eta,\bm x,T) & \subseteq \Big\{ C^I_\infty H(\eta^{-1})\tau^\eta(\bm x) > T  \Big\}
    \\ 
    & = \bigg( \Big\{ C^I_\infty H(\eta^{-1})\tau^\eta(\bm x) > T  \Big\} \cap E_b(\eta,T) \bigg) \cup  \bigg( \Big\{ C^I_\infty H(\eta^{-1})\tau^\eta(\bm x) > T  \Big\} \setminus E_b(\eta,T) \bigg).
\end{align*}
On $\{ C^I_\infty  H(\eta^{-1})\tau^\eta(\bm x) > T \} \cap E_b(\eta,T)$, we have $\tau^\eta(\bm x) = \tau^{\eta|b}(\bm x)$
again due to \eqref{proof, observation 1, theorem: first exit time, unclipped}.
By Theorem \ref{theorem: first exit time, unclipped} and \eqref{proof, observation 2, theorem: first exit time, unclipped},
we get
\begin{align}
    & \limsup_{\eta \downarrow 0}\sup_{\bm x \in I_\epsilon}\P\bigg( \Big\{ C^I_\infty H(\eta^{-1})\tau^\eta(\bm x) > T  \Big\} \cap E_b(\eta,T) \bigg)
    \nonumber
    \\ 
    &
    \leq 
    \limsup_{\eta \downarrow 0}\sup_{\bm x \in I_\epsilon}
        \P\Big( C^I_b \eta \cdot \lambda(\eta)\tau^{\eta|b}(\bm x) > T\Big) \leq \exp(-T).
    \label{proof, ineq 2, theorem: first exit time, unclipped}
\end{align}
Meanwhile, the limit of $\sup_{\bm x \in I_\epsilon}\P\big( C^I_\infty H(\eta^{-1})\tau^\eta(\bm x) > T \} \setminus E_b(\eta,T) \big)$ as $\eta \downarrow 0$ is again bounded by \eqref{proof, ineq 1, theorem: first exit time, unclipped}.
Collecting \eqref{proof, ineq 0, theorem: first exit time, unclipped}, \eqref{proof, ineq 1, theorem: first exit time, unclipped}, and \eqref{proof, ineq 2, theorem: first exit time, unclipped},
we yield that for all $b > 0$ large enough and all $T > t$,
\begin{align*}
    & \limsup_{\eta \downarrow 0}\sup_{\bm x \in I_\epsilon}\P\big(A(\eta,\bm x)\big)
    \\
    & \leq 
        \frac{ \widecheck{\mathbf{C}}^{ (1)|b }(B^-) }{ C^I_\infty }\cdot\exp(-t) - \frac{ \widecheck{\mathbf{C}}^{ (1)|b }(B^\circ) }{ C^I_\infty }\cdot\exp(-T) + \exp(-T)
    % \\ 
    % &\qquad
    + 
    2 \cdot \bigg[  1 - \exp\bigg( - \frac{T}{C^I_\infty} \cdot \Big( \frac{2C}{b} \Big)^\alpha  \bigg)\bigg].
\end{align*}
In light of claim \eqref{proof, observation 2 general form, theorem: first exit time, unclipped},
we send $b \to \infty$ and $T \to \infty$
% $
%  \limsup_{\eta \downarrow 0}\sup_{x \in I_\epsilon}\P\big(A(\eta,x)\big)
%  \leq 
%  \frac{ \widecheck{\mathbf{C}}(B^-) }{ C^* }\cdot\exp(-t) - \frac{ \widecheck{\mathbf{C}}(B^\circ) }{ C^* }\cdot\exp(-T) + \exp(-T).
% $
% Letting $T$ tend to $\infty$, we 
to conclude the proof of the upper bound.

The lower bound can be established analogously. 
By the decomposition of events in \eqref{proof, decompose event A, theorem: first exit time, unclipped},
\begin{align*}
    & \inf_{\bm x \in I_\epsilon}\P\big(A(\eta,\bm x)\big)
    \\ 
    & \geq 
    \inf_{\bm x \in I_\epsilon}\P\big(A_1(\eta,\bm x,T)\big)
    \geq 
    \inf_{\bm x \in I_\epsilon}\P\big(A_1(\eta,\bm x,T) \cap E_b(\eta,T)\big)
    \\
    & = 
    \inf_{\bm x \in I_\epsilon}
    \P\bigg( \Big\{  C^I_b \eta \cdot \lambda(\eta) \tau^{\eta|b}(\bm x) \in (t,T],\ \bm X^{\eta|b}_{ \tau^{\eta|b}(\bm x) }(\bm x) \in B  \Big\} \cap E_b(\eta,T) \bigg)
    \qquad 
    \text{due to \eqref{proof, observation 1, theorem: first exit time, unclipped} and \eqref{proof, observation 2, theorem: first exit time, unclipped}}
    \\ 
    & \geq 
    \inf_{\bm x \in I_\epsilon}
    \P\bigg( C^I_b \eta \cdot \lambda(\eta) \tau^{\eta|b}(\bm x) \in (t,T],\ \bm X^{\eta|b}_{ \tau^{\eta|b}(\bm x) }(\bm x) \in B  \bigg)
    -
    \P\Big( \big(E_b(\eta,T)\big)^\complement\Big)
    \\ 
    & \geq 
    \inf_{\bm x \in I_\epsilon}
    \P\bigg( C^I_b \eta \cdot \lambda(\eta) \tau^{\eta|b}(\bm x) > t,\ \bm X^{\eta|b}_{ \tau^{\eta|b}(\bm x) }(\bm x) \in B  \bigg)
    -
    \sup_{\bm x \in I_\epsilon}
    \P\bigg( C^I_b \eta \cdot \lambda(\eta) \tau^{\eta|b}(\bm x) > T,\ \bm X^{\eta|b}_{ \tau^{\eta|b}(\bm x) }(\bm x) \in B  \bigg)
    \\ 
    &\qquad \qquad 
    -
    \P\Big( \big(E_b(\eta,T)\big)^\complement\Big).
\end{align*}
By Theorem \ref{theorem: first exit time, unclipped} and the limit in \eqref{proof, ineq 1, theorem: first exit time, unclipped},
we yield (for all $b > 0$ large enough and all $T  > t$)
\begin{align*}
    \liminf_{\eta \downarrow 0}\inf_{\bm x \in I_\epsilon}\P\Big(A(\eta,\bm x)\Big)
    & \leq 
        \frac{ \widecheck{\mathbf{C}}^{ (1)|b }(B^\circ) }{ C^I_\infty }\cdot\exp(-t) - \frac{ \widecheck{\mathbf{C}}^{ (1)|b }(B^-) }{ C^I_\infty }\cdot\exp(-T)
    % \\ 
    % &\qquad 
    -
    \bigg[  1 - \exp\bigg( - \frac{T}{C^I_\infty} \cdot \Big( \frac{2C}{b} \Big)^\alpha  \bigg)\bigg].
\end{align*}
By claim \eqref{proof, observation 2 general form, theorem: first exit time, unclipped},
we send $b \to \infty$ and $T \to \infty$ to conclude the proof of the lower bound.
\end{proof}

To conclude, we provide the proof of Lemma~\ref{lemma: limiting measure, with exit location B, first exit analysis}.

\section*{Acknowledgements}
The authors gratefully acknowledge the support of the National Science Foundation (NSF) under CMMI-2146530.

\bibliographystyle{abbrv} % outcomment this and next line in Case 1
\bibliography{bib_appendix} % if more than one, comma separated

@misc{wang2024largedeviationsmetastabilityanalysis,
      title={Large Deviations and Metastability Analysis for Heavy-Tailed Dynamical Systems}, 
      author={Xingyu Wang and Chang-Han Rhee},
      year={2024},
      eprint={2307.03479},
      archivePrefix={arXiv},
      primaryClass={math.PR},
      url={https://arxiv.org/abs/2307.03479}, 
}

@article{gaudilliere2014dirichlet,
  title={A Dirichlet principle for non reversible Markov chains and some recurrence theorems},
  author={Gaudilliere, Alexandre and Landim, Claudio},
  journal={Probability Theory and Related Fields},
  volume={158},
  pages={55--89},
  year={2014},
  publisher={Springer}
}

@article{landim2014metastability,
  title={Metastability for a non-reversible dynamics: the evolution of the condensate in totally asymmetric zero range processes},
  author={Landim, C},
  journal={Communications in Mathematical Physics},
  volume={330},
  pages={1--32},
  year={2014},
  publisher={Springer}
}

@article{slowik2012note,
  title={A note on variational representations of capacities for reversible and nonreversible Markov chains},
  author={Slowik, Martin},
  journal={Unpublished, Technische Universit{\"a}t Berlin},
  year={2012}
}

@article{lynch1987large,
  author={Lynch, James and Jayaram Sethuraman},
  title={Large deviations for processes with independent increments},
  journal={The annals of probability},
  volume={15},
  issue={2}, 
  year={1987},
  pages={610--627}
}

@article{mogulskii1993large,
  author={Mogulskii, AA},
  year={1993},
  title={Large deviations for processes with independent increments},
  journal={The annals of probability},
  year={1993},
  volume={12},
  issue={1},
  pages={202--215}
}

@article{imkeller2008levy,
  author={Imkeller, P. and Pavlyukevich, I.}, 
  year = {2008},
  title={Lévy flights: transitions and meta-stability},
  journal={Journal of Physics A: Mathematical and General}, 
  volume={39},
  number={15},
  pages={L237}
}

@book{glasstone1941theory,
  title={The theory of rate processes: the kinetics of chemical reactions, viscosity, diffusion and electrochemical phenomena},
  author={Glasstone, Samuel and Laidler, Keith James and Eyring, Henry},
  year={1941},
  publisher={McGraw-Hill},
  address={New York}
}

@article{freidlin1973some,
  title={Some problems concerning stability under small random perturbations},
  author={M I Freidlin and A D Wentzell},
  journal={Theory of Probability \& Its Applications},
  volume={17},
  number={2},
  pages={269--283},
  year={1973},
  publisher={SIAM}
}

@article{cassandro1984metastable,
  title={Metastable behavior of stochastic dynamics: a pathwise approach},
  author={Cassandro, Marzio and Galves, Antonio and Olivieri, Enzo and Vares, Maria Eul{\'a}lia},
  journal={Journal of statistical physics},
  volume={35},
  pages={603--634},
  year={1984},
  publisher={Springer}
}

@article{lee2022non-b,
  title={Non-reversible metastable diffusions with Gibbs invariant measure II: Markov chain convergence},
  author={Lee, Jungkyoung and Seo, Insuk},
  journal={Journal of Statistical Physics},
  volume={189},
  number={2},
  pages={25},
  year={2022},
  publisher={Springer}
}

@inproceedings{rezakhanlou2023scaling,
  title={Scaling limit of small random perturbation of dynamical systems},
  author={Rezakhanlou, Fraydoun and Seo, Insuk},
  booktitle={Annales de l'Institut Henri Poincare (B) Probabilites et statistiques},
  volume={59},
  number={2},
  pages={867--903},
  year={2023},
  organization={Institut Henri Poincar{\'e}}
}

@article{lee2022non,
  title={Non-reversible metastable diffusions with Gibbs invariant measure I: Eyring--Kramers formula},
  author={Lee, Jungkyoung and Seo, Insuk},
  journal={Probability Theory and Related Fields},
  volume={182},
  number={3},
  pages={849--903},
  year={2022},
  publisher={Springer}
}

@inproceedings{
Engstrom2020Implementation,
title={Implementation Matters in Deep RL: A Case Study on PPO and TRPO},
author={Logan Engstrom and Andrew Ilyas and Shibani Santurkar and Dimitris Tsipras and Firdaus Janoos and Larry Rudolph and Aleksander Madry},
booktitle={International Conference on Learning Representations},
year={2020},
url={https://openreview.net/forum?id=r1etN1rtPB}
}

@misc{wangSGDpaper2,
      title={Global Dynamics of Heavy-Tailed SGDs in Nonconvex Loss Landscape: Characterization and Control}, 
      author={Xingyu Wang and Chang-Han Rhee},
      year={2025},
      eprint={2510.20905},
      archivePrefix={arXiv},
      primaryClass={cs.LG},
      url={https://arxiv.org/abs/2510.20905}, 
}

@article{borovkov2010large,
  title={On large deviation principles in metric spaces},
  author={Borovkov, Aleksandr A and Mogul’skii, Anatolii A},
  journal={Siberian Mathematical Journal},
  volume={51},
  pages={989--1003},
  year={2010},
  publisher={Springer},
  doi = {10.1007/s11202-010-0098-0},
  url = {ttps://doi.org/10.1007/s11202-010-0098-0}
}

@article{imkeller2010first,
  title={First exit times of non-linear dynamical systems in {$\mathbb R^d$} perturbed by multifractal {L}{\'e}vy noise},
  author={Imkeller, Peter and Pavlyukevich, Ilya and Stauch, Michael},
  journal={Journal of Statistical Physics},
  volume={141},
  number={1},
  pages={94--119},
  year={2010},
  publisher={Springer}
}

@article{imkeller2010hierarchy,
  title={The hierarchy of exit times of {L}{\'e}vy-driven {L}angevin equations},
  author={Imkeller, Peter and Pavlyukevich, Ilya and Wetzel, Torsten},
  journal={The European Physical Journal Special Topics},
  volume={191},
  number={1},
  pages={211--222},
  year={2010},
  publisher={Springer}
}

@book{resnick2007heavy,
  title={Heavy-tail phenomena: probabilistic and statistical modeling},
  author={Resnick, Sidney I},
  year={2007},
  publisher={Springer Science \& Business Media}
}

@book{protter2005stochastic,
  author={Protter, Philip E},
  title={Stochastic integration and differential equations},
  year={2005},
  publisher={Springer}
}

@article{bernhard2020heavy,
  title={Heavy-tailed random walks, buffered queues and hidden large deviations},
  author={Bernhard, Harald and Das, Bikramjit and others},
  journal={Bernoulli},
  volume={26},
  number={1},
  pages={61--92},
  year={2020},
  publisher={Bernoulli Society for Mathematical Statistics and Probability}
}

@article{pavlyukevich2008metastable,
  author = {Imkeller, Peter and Pavlyukevich, Ilya},
	title = {Metastable behaviour of small noise {L}\'evy-Driven diffusions},
	DOI= "10.1051/ps:2007051",
	url= "https://doi.org/10.1051/ps:2007051",
	journal = {ESAIM: PS},
	year = 2008,
	volume = 12,
	pages = "412-437",
	month = "",
}

@book{sato1999levy,
  title={{L}{\'e}vy Processes and Infinitely Divisible Distributions},
  author={Sato, Ken-iti and Ken-Iti, Sato and Katok, A},
  year={1999},
  publisher={Cambridge university press}
}

@article{rhee2019sample,
  title={Sample path large deviations for {L}{\'e}vy processes and random walks with regularly varying increments},
  author={Rhee, Chang-Han and Blanchet, Jose and Zwart, Bert and others},
  journal={The Annals of Probability},
  volume={47},
  number={6},
  pages={3551--3605},
  year={2019},
  publisher={Institute of Mathematical Statistics}
}

@article{lindskog2014regularly,
  title={Regularly varying measures on metric spaces: Hidden regular variation and hidden jumps},
  author={Lindskog, Filip and Resnick, Sidney I and Roy, Joyjit and others},
  journal={Probability Surveys},
  volume={11},
  pages={270--314},
  year={2014},
  publisher={The Institute of Mathematical Statistics and the Bernoulli Society}
}

@book{bingham1989regular,
  title={Regular variation},
  author={Bingham, Nicholas H and Goldie, Charles M and Teugels, Jef L},
  number={27},
  year={1989},
  publisher={Cambridge university press}
}

@inproceedings{
zhang2020why,
title={Why Gradient Clipping Accelerates Training: A Theoretical Justification for Adaptivity},
author={Jingzhao Zhang and Tianxing He and Suvrit Sra and Ali Jadbabaie},
booktitle={International Conference on Learning Representations},
year={2020},
url={https://openreview.net/forum?id=BJgnXpVYwS}
}

@article{kramers1940brownian,
  title={Brownian motion in a field of force and the diffusion model of chemical reactions},
  author={Kramers, Hendrik Anthony},
  journal={Physica},
  volume={7},
  number={4},
  pages={284--304},
  year={1940},
  publisher={Elsevier}
}

@book{foss2011introduction,
  title={An introduction to heavy-tailed and subexponential distributions},
  author={Foss, Sergey and Korshunov, Dmitry and Zachary, Stan and others},
  volume={6},
  year={2011},
  publisher={Springer},
  DOI={https://doi.org/10.1007/978-1-4419-9473-8}
}

@article{hult2005functional,
author = {Henrik Hult and Filip Lindskog and Thomas Mikosch and Gennady Samorodnitsky},
title = {{Functional large deviations for multivariate regularly varying random walks}},
volume = {15},
journal = {The Annals of Applied Probability},
number = {4},
publisher = {Institute of Mathematical Statistics},
pages = {2651 -- 2680},
year = {2005},
doi = {10.1214/105051605000000502},
URL = {https://doi.org/10.1214/105051605000000502}
}

@article{doi:10.1287/moor.1120.0539,
author = {Foss, Sergey and Korshunov, Dmitry},
title = {On Large Delays in Multi-Server Queues with Heavy Tails},
journal = {Mathematics of Operations Research},
volume = {37},
number = {2},
pages = {201-218},
year = {2012},
doi = {10.1287/moor.1120.0539},

URL = { 
    
        https://doi.org/10.1287/moor.1120.0539
    
    

},
eprint = { 
    
        https://doi.org/10.1287/moor.1120.0539
    
    

}
}

@book{freidlin1984random ,
  title={Random Perturbations of Dynamical Systems},
  author={M. I. Freidlin and A. D. Wentzell},
  year={1984},
  publisher={Springer},
  address={New York, NY},
  DOI = {https://doi.org/10.1007/978-1-4612-0611-8}
}

@article{imkeller2006first,
    title = {First exit times of SDEs driven by stable Lévy processes},
    journal = {Stochastic Processes and their Applications},
    volume = {116},
    number = {4},
    pages = {611-642},
    year = {2006},
    issn = {0304-4149},
    doi = {https://doi.org/10.1016/j.spa.2005.11.006},
    url = {https://www.sciencedirect.com/science/article/pii/S0304414905001614},
    author = {P. Imkeller and I. Pavlyukevich}
}

@article{imkeller2009exponential,
author = {Peter Imkeller and Ilya Pavlyukevich and Torsten Wetzel},
title = {{First exit times for Lévy-driven diffusions with exponentially light jumps}},
volume = {37},
journal = {The Annals of Probability},
number = {2},
publisher = {Institute of Mathematical Statistics},
pages = {530 -- 564},
year = {2009},
doi = {10.1214/08-AOP412},
URL = {https://doi.org/10.1214/08-AOP412}
}

@article{eyring1935chemical,
author = {Eyring,Henry},
title = {The Activated Complex in Chemical Reactions},
journal = {The Journal of Chemical Physics},
volume = {3},
number = {2},
pages = {107-115},
year = {1935},
doi = {10.1063/1.1749604},
publisher={American Institute of Physics},
URL = { 
        https://doi.org/10.1063/1.1749604
},
eprint = { 
        https://doi.org/10.1063/1.1749604
}
}

@inproceedings{
keskar2017on,
title={On Large-Batch Training for Deep Learning: Generalization Gap and Sharp Minima},
author={Nitish Shirish Keskar and Dheevatsa Mudigere and Jorge Nocedal and Mikhail Smelyanskiy and Ping Tak Peter Tang},
booktitle={International Conference on Learning Representations},
year={2017},
url={https://openreview.net/forum?id=H1oyRlYgg}
}

@article{Albrecher_Chen_Vatamidou_Zwart_2020, title={Finite-time ruin probabilities under large-claim reinsurance treaties for heavy-tailed claim sizes}, volume={57}, DOI={10.1017/jpr.2020.8}, number={2}, journal={Journal of Applied Probability}, author={Albrecher, Hansjörg and Chen, Bohan and Vatamidou, Eleni and Zwart, Bert}, year={2020}, pages={513–530}}

@inproceedings{
wang2022eliminating,
title={Eliminating Sharp Minima from {SGD} with Truncated Heavy-tailed Noise},
author={Xingyu Wang and Sewoong Oh and Chang-Han Rhee},
booktitle={International Conference on Learning Representations},
year={2022},
url={https://openreview.net/forum?id=B3Nde6lvab}
}

@book{tankov2003financial,
  title={Financial modelling with jump processes},
  author={Tankov, Peter},
  year={2003},
  publisher={Chapman and Hall/CRC}
}

@article{foss2006heavy,
  title={Heavy tails in multi-server queue},
  author={Foss, Serguei and Korshunov, Dmitry},
  journal={Queueing Systems},
  volume={52},
  pages={31--48},
  year={2006},
  publisher={Springer}
}

@article{freidlin1970onsmall,
doi = {10.1070/RM1970v025n01ABEH001254},
url = {https://dx.doi.org/10.1070/RM1970v025n01ABEH001254},
year = {1970},
month = {Feb},
publisher = {},
volume = {25},
number = {1},
pages = {1},
author = {M I Freidlin and A D Wentzell},
title = {ON SMALL RANDOM PERTURBATIONS OF DYNAMICAL SYSTEMS},
journal = {Russian Mathematical Surveys}
}

@book{bovier2016metastability,
  title={Metastability: a potential-theoretic approach},
  author={Bovier, Anton and Den Hollander, Frank},
  volume={351},
  year={2016},
  publisher={Springer}
}

@inproceedings{NEURIPS2019_a97da629,
 author = {Nguyen, Thanh Huy and Simsekli, Umut and Gurbuzbalaban, Mert and Richard, Ga\"{e}l},
 booktitle = {Advances in Neural Information Processing Systems},
 editor = {H. Wallach and H. Larochelle and A. Beygelzimer and F. d\textquotesingle Alch\'{e}-Buc and E. Fox and R. Garnett},
 pages = {},
 publisher = {Curran Associates, Inc.},
 title = {First Exit Time Analysis of Stochastic Gradient Descent Under Heavy-Tailed Gradient Noise},
 url = {https://proceedings.neurips.cc/paper_files/paper/2019/file/a97da629b098b75c294dffdc3e463904-Paper.pdf},
 volume = {32},
 year = {2019}
}

@article{Resnick_2004, title={On the foundations of multivariate heavy-tail analysis}, volume={41}, DOI={10.1239/jap/1082552199}, number={A}, journal={Journal of Applied Probability}, author={Resnick, Sidney}, year={2004}, pages={191–212}}

@article{hult2006regular,
  title={Regular variation for measures on metric spaces},
  author={Hult, Henrik and Lindskog, Filip},
  journal={Publications de l'Institut Math{\'e}matique},
  volume={80},
  number={94},
  pages={121--140},
  year={2006}
}

@article{doi:10.1142/S0219493711003413,
author = {Pavlyukevich, Ilya},
title = {FIRST EXIT TIMES OF SOLUTIONS OF STOCHASTIC DIFFERENTIAL EQUATIONS DRIVEN BY MULTIPLICATIVE LÉVY NOISE WITH HEAVY TAILS},
journal = {Stochastics and Dynamics},
volume = {11},
number = {02n03},
pages = {495-519},
year = {2011},
doi = {10.1142/S0219493711003413},

URL = { 
    
        https://doi.org/10.1142/S0219493711003413
    
    

},
eprint = { 
    
        https://doi.org/10.1142/S0219493711003413
    
    

}
}

@article{doi:10.1142/S0219493715500197,
author = {H\"{o}gele, Michael and Pavlyukevich, Ilya},
title = {Metastability in a class of hyperbolic dynamical systems perturbed by heavy-tailed {L}\'evy type noise},
journal = {Stochastics and Dynamics},
volume = {15},
number = {03},
pages = {1550019},
year = {2015},
doi = {10.1142/S0219493715500197},

URL = { 
    
        https://doi.org/10.1142/S0219493715500197
    
    

},
eprint = { 
    
        https://doi.org/10.1142/S0219493715500197
    
    

}
}

@article{hogele2014exit,
author = {Michael Högele and Ilya Pavlyukevich},
title = {The Exit Problem from a Neighborhood of the Global Attractor for Dynamical Systems Perturbed by Heavy-Tailed Lévy Processes},
journal = {Stochastic Analysis and Applications},
volume = {32},
number = {1},
pages = {163--190},
year = {2014},
publisher = {Taylor \& Francis},
doi = {10.1080/07362994.2014.858554},


URL = { 
    
        https://doi.org/10.1080/07362994.2014.858554
    
    

},
eprint = { 
    
        https://doi.org/10.1080/07362994.2014.858554
    
    

}

}

@article{bovier2004metastability,
  title={Metastability in reversible diffusion processes. I. Sharp asymptotics for capacities and exit times},
  author={Bovier, Anton and Eckhoff, Michael and Gayrard, V{\'e}ronique and Klein, Markus},
  journal={J. Eur. Math. Soc.(JEMS)},
  volume={6},
  number={4},
  pages={399--424},
  year={2004}
}

@article{bovier2005metastability,
  title={Metastability in reversible diffusion processes II: Precise asymptotics for small eigenvalues},
  author={Bovier, Anton and Eckhoff, Michael and Gayrard, V{\'e}ronique and Klein, Markus},
  journal={Journal of the European Mathematical Society},
  volume={7},
  number={1},
  pages={69--99},
  year={2005}
}

@article{bovier2001metastability,
  title={Metastability in stochastic dynamics of disordered mean-field models},
  author={Bovier, Anton and Eckhoff, Michael and Gayrard, V{\'e}ronique and Klein, Markus},
  journal={Probability Theory and Related Fields},
  volume={119},
  pages={99--161},
  year={2001},
  publisher={Springer}
}

@article{pascanu2013difficulty,
  title={On the difficulty of training recurrent neural networks},
  author={Pascanu, R},
  journal={arXiv preprint arXiv:1211.5063},
  year={2013}
}

@InProceedings{pmlr-v202-koloskova23a,
  title = 	 {Revisiting Gradient Clipping: Stochastic bias and tight convergence guarantees},
  author =       {Koloskova, Anastasia and Hendrikx, Hadrien and Stich, Sebastian U},
  booktitle = 	 {Proceedings of the 40th International Conference on Machine Learning},
  pages = 	 {17343--17363},
  year = 	 {2023},
  editor = 	 {Krause, Andreas and Brunskill, Emma and Cho, Kyunghyun and Engelhardt, Barbara and Sabato, Sivan and Scarlett, Jonathan},
  volume = 	 {202},
  series = 	 {Proceedings of Machine Learning Research},
  month = 	 {23--29 Jul},
  publisher =    {PMLR},
  url = 	 {https://proceedings.mlr.press/v202/koloskova23a.html}
}

@book{Olivieri_Vares_2005, 
  place={Cambridge}, 
  series={Encyclopedia of Mathematics and its Applications}, 
  title={Large Deviations and Metastability}, 
  publisher={Cambridge University Press}, 
  author={Olivieri, Enzo and Vares, Maria Eulália}, 
  year={2005}, 
  collection={Encyclopedia of Mathematics and its Applications}
}

@article{pavlyukevich2011first,
  title={First exit times of solutions of stochastic differential equations driven by multiplicative {L}{\'e}vy noise with heavy tails},
  author={Pavlyukevich, Ilya},
  journal={Stochastics and Dynamics},
  volume={11},
  number={02n03},
  pages={495--519},
  year={2011},
  publisher={World Scientific}
}

@book{debussche2013dynamics,
  title={The dynamics of nonlinear reaction-diffusion equations with small {L}{\'e}vy noise},
  author={Debussche, Arnaud and H{\"o}gele, Michael and Imkeller, Peter},
  volume={2085},
  year={2013},
  publisher={Springer}
}

@article{JMLR:v25:21-1343,
  author  = {Amrit Singh Bedi and Anjaly Parayil and Junyu Zhang and Mengdi Wang and Alec Koppel},
  title   = {On the Sample Complexity and Metastability of Heavy-tailed Policy Search in Continuous Control},
  journal = {Journal of Machine Learning Research},
  year    = {2024},
  volume  = {25},
  number  = {39},
  pages   = {1--58},
  url     = {http://jmlr.org/papers/v25/21-1343.html}
}

@inproceedings{simsekli2019tail,
  title={A tail-index analysis of stochastic gradient noise in deep neural networks},
  author={Simsekli, Umut and Sagun, Levent and Gurbuzbalaban, Mert},
  booktitle={International Conference on Machine Learning},
  pages={5827--5837},
  year={2019},
  organization={PMLR}
}

\newpage
\appendix

\section{Results under General Scaling}
\label{sec: appendix, SGD, general scaling, results}

Below, we present results analogous to those in Section~\ref{sec: main results} under a general scaling.
Specifically, throughout this section we define 
% $(\bm X^\eta_j(\bm x))_{j \geq 0}$ 
% and 
% $(\bm X^{\eta|b}_j(\bm x))_{j \geq 0}$ by the recursions
\begin{equation}
\begin{aligned}
    &
    \bm X^\eta_t(\bm x) = \bm X^\eta_{t - 1}(\bm x) +  \eta \bm a\big(\bm X^\eta_{t - 1}(\bm x)\big) + \eta^\gamma\bm \sigma\big(\bm X^\eta_{t - 1}(\bm x)\big)\bm Z_t,\quad \forall t \geq 1,
    \\
    &{\bm X^{\eta|b}_t(\bm x)} = \bm X^{\eta|b}_{t - 1}(\bm x) +  \varphi_b\Big(\eta \bm a\big(\bm X^{\eta|b}_{t - 1}(\bm x)\big) + \eta^\gamma \bm \sigma\big(\bm X^{\eta|b}_{t - 1}(\bm x)\big)\bm Z_t\Big),\quad \forall t \geq 1,
\end{aligned}
\label{eq:gneral-scaling}
\end{equation}
with $\gamma \in (\frac{1}{2 \wedge \alpha}, \infty)$,
where $\alpha > 1$ is the heavy-tailed index in Assumption~\ref{assumption gradient noise heavy-tailed}.
Let 
\begin{align*}
    \lambda(\eta;\gamma) = \eta^{-1}H(\eta^{-\gamma}).
\end{align*}
We adopt the notations
$
\mathbf{C}^{(k)|b}_{[0,T]},
$
$
\mathbb{D}^{(k)|b}_A[0,T](\epsilon),
$
$
{\bm{X}^{\eta|b}_{[0,T]}(\bm x)},
$
etc., introduced
in the main paper.
Let $I \subseteq \R^m$ be an open set such that $\bm 0 \in \bm I$ and Assumption~\ref{assumption: shape of f, first exit analysis} holds.
Let the first exit times $\tau^\eta(\bm x)$ and $\tau^{\eta|b}(\bm x)$ be defined as in \eqref{def: first exit time for heavy tailed SGD}.
We adopt the notations $\mathcal J^I_b$, $\mathcal G^{(k)|b}(\epsilon)$, $\widecheck{\mathbf C}^{k|b}$, etc.,
introduced
in Section~\ref{sec: first exit time simple version}.

\begin{theorem}
    Let Assumptions \ref{assumption gradient noise heavy-tailed}, \ref{assumption: lipschitz continuity of drift and diffusion coefficients}, and \ref{assumption: shape of f, first exit analysis} hold.
    Let $\gamma \in (\frac{1}{2 \wedge \alpha}, \infty)$.
\begin{enumerate}[(a)]
    \item 
         Let $b > 0$.
         Suppose that $\mathcal J^I_b < \infty$,
         $I^c$ is bounded away from $\mathcal G^{(\mathcal J^I_b - 1)|b}(\epsilon)$ for some (and hence all) $\epsilon > 0$ small enough,
        and
        $
        \widecheck{\mathbf C}^{( \mathcal J^I_b )|b}(\partial I) = 0.
        $
        Then 
        ${C^I_b} \delequal  \widecheck{ \mathbf{C} }^{ (\mathcal{J}^I_b)|b }(I^\complement) < \infty$.
        Furthermore, if $C^I_b \in (0,\infty)$,
        then
        for any $\epsilon > 0$, $t \geq 0$, and measurable set $B \subseteq I^c$,
        \begin{align*}
            \limsup_{\eta\downarrow 0}\sup_{\bm x \in I_\epsilon}
            \P\bigg(
                C^I_b \eta\cdot \lambda^{ \mathcal{J}^I_b }(\eta;\gamma)\tau^{\eta|b}(\bm x) > t
             ;\ \bm X^{\eta|b}_{ \tau^{\eta|b}(\bm x)}(\bm x) \in B
             \bigg)
             & \leq \frac{ \widecheck{\mathbf{C}}^{ (\mathcal{J}^I_b)|b }(B^-) }{ C^I_b }\cdot\exp(-t),
             \\
             \liminf_{\eta\downarrow 0}\inf_{\bm x \in I_\epsilon}
            \P\bigg(
                C^I_b \eta\cdot \lambda^{ \mathcal{J}^I_b }(\eta;\gamma)\tau^{\eta|b}(\bm x) > t
             ;\ \bm X^{\eta|b}_{ \tau^{\eta|b}(\bm x)}(\bm x) \in B
             \bigg)
             & \geq \frac{ \widecheck{\mathbf{C}}^{ (\mathcal{J}^I_b)|b }(B^\circ) }{ C^I_b }\cdot\exp(-t).
        \end{align*}
        Otherwise, we have $C^I_b = 0$, and
        \begin{align*}
            \limsup_{\eta\downarrow 0}\sup_{\bm x \in I_\epsilon}
        \P\bigg(
            \eta\cdot \lambda^{ \mathcal J^I_b }(\eta;\gamma)\tau^{\eta|b}(\bm x) \leq t
         \bigg) = 0
         \qquad
         \forall \epsilon > 0,\ t \geq 0.
        \end{align*}

    \item 
        Suppose that $\widecheck{\mathbf C}(\partial I) = 0$.
        Then ${C^I_\infty} \delequal  \widecheck{ \mathbf{C} }(I^\complement) < \infty$.
        Furthermore, if $C^I_\infty > 0$,
        then
        for any $t \geq 0$ and measurable set $B \subseteq I^c$,
    \begin{align*}
        \limsup_{\eta\downarrow 0}\sup_{\bm x \in I_\epsilon}
        \P\bigg(
            C^I_\infty \eta\cdot \lambda(\eta;\gamma)\tau^\eta(\bm x) > t;\ \bm X^\eta_{ \tau^\eta(\bm x)}(\bm x) \in B
        \bigg)
        & \leq
        \frac{ \widecheck{\mathbf{C}}(B^-) }{ C^I_\infty }\cdot\exp(-t),
        \\ 
        \liminf_{\eta\downarrow 0}\inf_{\bm x \in I_\epsilon}
        \P\bigg(
            C^I_\infty \eta\cdot \lambda(\eta;\gamma)\tau^\eta(\bm x) > t;\ \bm X^\eta_{ \tau^\eta(\bm x)}(\bm x) \in B
        \bigg)
        & \geq
        \frac{ \widecheck{\mathbf{C}}(B^\circ) }{ C^I_\infty  }\cdot\exp(-t).
    \end{align*}
     Otherwise, we have $C^I_\infty = 0$, and 
    \begin{align*}
        \limsup_{\eta\downarrow 0}\sup_{\bm x \in I_\epsilon}
        \P\bigg(
            \eta\cdot \lambda(\eta;\gamma)\tau^\eta(\bm x) \leq t
        \bigg) = 0
        \qquad
        \forall \epsilon > 0,\ t \geq 0.
    \end{align*}
    
\end{enumerate}
\end{theorem}
The proofs for results in this section will be almost identical to those presented in the main paper.
We omit the details to avoid repetition.

\section{Results for L\'evy-Driven Stochastic Differential Equations}
\label{sec: appendix, SDE reults}

In this section, we collect the results for stochastic differential equations driven by L\'evy processes with regularly varying increments. Specifically,
any multidimensional L\'evy process $\notationdef{notation-levy-process}{\bm{L}} = \{ \bm L_t:\ t \geq 0\}$
can be characterized by its generating triplet $(\bm c_{\bm L},\bm \Sigma_{\bm L},\nu)$ where $\bm c_{\bm L} \in \mathbb{R}^m$ is the drift parameter, the positive semi-definite matrix $\bm \Sigma_{\bm L} \in \R^{m \times m}$ is the magnitude of the Brownian motion term in $\bm L_t$, 
and $\nu$ is the L\'evy measure characterizing the intensity of jumps in $\bm L_t$. 
More precisely, 
we have the following L\'evy–Itô decomposition
\begin{align}
    \bm L_t \distequal \bm c_{\bm L}t + \bm \Sigma_{\bm L}^{1/2} 
    \bm B_t+ 
    \int_{\norm{\bm x} \leq 1}\bm x\big[ N([0,t]\times d\bm x) - t\nu(dx) \big] 
    + 
    \int_{\norm{\bm x} > 1}\bm xN( [0,t]\times d\bm x ) 
    \label{prelim: levy ito decomp}
\end{align}
where $\bm B_t$ is a standard Brownian motion in $\R^m$, 
the measure $\nu$ satisfies $\int (\norm{\bm x}^2\wedge 1) \nu(d\bm x) < \infty$,
and $N$ is a Poisson random measure independent of $\bm B_t$ with intensity measure $\mathcal L_\infty\times \nu$.
See Chapter~4 of \cite{sato1999levy} for details.
We impose the following heavy-tailed assumption on the increments of $\bm L_t$.

\begin{assumption}\label{assumption: heavy-tailed levy process}
$\E \bm L_1 = \bm 0$. Besides, there exist $\alpha > 1$ 
and 
a probability measure $\mathbf S(\cdot)$ on the unit sphere of $\R^d$ such that
\begin{itemize}
    \item 
         $\notationdef{notation-H-L}{H_L(x)} \delequal \nu\Big( \big\{\bm y \in \R^d:\ \norm{\bm y} > x  \big\} \Big) \in \RV_{-\alpha}(x)$ as $x \to \infty$,

    \item As $r \to \infty$,
        \begin{align*}
            \frac{ \big(\nu \circ \Phi^{-1}_r\big)(\cdot) }{ H_L(r)} 
            \xrightarrow{v}
            \nu_\alpha \times \mathbf S,
            % \qquad 
            % \text{in $\mathbb M\Big( 
            % \big([0,\infty) \times \mathfrak N_d \big)
            % \setminus
            % \big( \{0\} \times \mathfrak N_d \big)
            % \Big)$},
        \end{align*}
        where 
        $\mathfrak N_d$ is the unit sphere of $\R^d$,
        the measure $\nu_\alpha$ is defined in \eqref{def: measure nu alpha},
        and
        % the measure $\big(\nu \circ \Phi^{-1}_r\big)$ is defined by
        \begin{align*}
            \big(\nu \circ \Phi^{-1}_r\big)(\cdot)
            \delequal
            \nu\Big(
                \Phi^{-1}(r\cdot\ ,\ \cdot)
            \Big),
        \end{align*}
        i.e. $ \big(\nu \circ \Phi^{-1}_r\big)(A\times B) = 
        \nu\big(
        \Phi^{-1}(rA,B)
        \big)
        $
        for all Borel sets $A \subseteq (0,\infty)$ and $B \subseteq \mathfrak N_d$.
\end{itemize}
\end{assumption}
Consider a filtered probability space
$\big(\Omega,\mathcal{F},\mathbb{F} = (\mathcal{F}_{t})_{t \geq 0},\P\big)$
satisfying the usual hypotheses stated in Chapter I, \cite{protter2005stochastic} and supporting the L\'evy process $\bm L$, where
$\mathcal{F}_{0} = \{\emptyset,\Omega\}$ and $\mathcal{F}_{t}$ is the $\sigma$-algebra generated by $\{\bm L_s:s\in[0,t]\}$.
% Let $\beta > 0$.
For $\eta \in (0,1]$
and $\beta \geq 0$, define the scaled process
\begin{align}
    {\bar{\bm L}^\eta} \delequal \big\{ \bar{\bm L}^\eta_t = \eta \bm L_{t/\eta^\beta}:\ t \geq 0\big\},
    \label{def: scaled levy process}
\end{align}
and
let $\bm Y^\eta_t(\bm x)$ be the solution to SDE
\begin{align}
\bm Y^\eta_0(\bm x) = \bm x,\qquad 
    d\notationdef{notation-Y-eta-SDE}{\bm Y^\eta_t(\bm x)} 
    % = a\big(Y^\eta_{t-}(x)\big)dt + \eta \sigma\big(Y^\eta_{t-}(x)\big)dL_{t/\eta^\beta}
    =
    \bm a\big(\bm Y^\eta_{t-}(\bm x)\big)dt
        +
    \bm \sigma\big(\bm Y^\eta_{t-}(\bm x)\big) d\bar{\bm L}^\eta_t.
    \label{defSDE, initial condition x}
\end{align}
Henceforth in Section~\ref{sec: appendix, SDE reults}, we consider $\beta \in [0,2\wedge \alpha)$
where $\alpha > 1$ is the tail index in Assumption~\ref{assumption: heavy-tailed levy process}.
% Below, we present the sample-path large deviations and metastability analysis of $\bm Y^\eta_t(\bm x)$.
% \subsection{Sample Path Large Deviations}
% \label{subsec: LD, SDE}
% Recall the definitions of the mapping $h^{(k)}_{[0,T]} = h^{(k)|\infty}_{[0,T]}$ in \eqref{def: perturb ode mapping h k b, 1}--\eqref{def: perturb ode mapping h k b, 3}
% as well as the measure $\mathbf{C}^{(k)}_{[0,T]} = \mathbf{C}^{(k)|\infty}_{[0,T]}$ in \eqref{def: measure mu k b t}.
% Also, recall uniform $\mathbb{M}$-convergence introduced in Definition \ref{def: uniform M convergence}.
% Define
% $\notationdef{notation-bm-Y-0T-eta-SDE}{\bm Y^\eta_{[0,T]}(\bm x)} = \{\bm Y^\eta_t(\bm x):\ t\in[0,T]\}$ as a random element in $\D[0,T]$.
% In case that $T = 1$, we suppress $[0,1]$ and write $\notationdef{notation-bm-Y-eta-SDE}{\bm Y^\eta(\bm x)}$.
% The next result characterizes the sample-path large deviations for $\bm Y^\eta_{[0,T]}(x)$  by establishing $\mathbb{M}$-convergence that is uniform in the initial condition $x$.
% The proofs are almost identical to those of $\bm X^\eta_j(\bm x)$ and hence omitted to avoid repetition.
% Recall that ${H_L(x)} = \nu\big( \{\bm y \in \R^d:\ \norm{\bm y} > x  \} \big)$.
Let
\begin{align*}
   \notationdef{notation-scale-function-lambda-L}{\lambda_L(\eta;\beta)} \delequal \eta^{-\beta} H_L(\eta^{-1})
\end{align*}
and
$
\lambda^k_L(\eta;\beta) = \big(\lambda_L(\eta;\beta)\big)^k,
$
where $\beta \in [0,2\wedge \alpha)$ determines the time scaling in \eqref{def: scaled levy process}.

Analogous to the truncated dynamics $\bm X^{\eta|b}_j(\bm x)$,
we introduce a truncated variation ${\bm Y}^{\eta|b}_t(\bm x)$
where all jumps are truncated under the threshold value $b$.
More generally, we consider
a sequence of stochastic processes $\big(\bm{Y}^{\eta|b;(k)}_t(\bm x)\big)_{k \geq 0}$.
% given any $f: \R \to \R$ and $g: \R \to \R$ that are Lipschitz continuous.
First, for any $\bm x \in \R^m$ and $t \geq 0$, let 
\begin{align}
    d{\bm Y}^{\eta|b;(0)}_t(\bm x) = \bm a\big({\bm Y}^{\eta|b;(0)}_{t-}(\bm x)\big)dt + 
    % \eta
    \bm \sigma\big(\bm{Y}^{\eta|b;(0)}_{t-}(\bm x)\big)d\bar{\bm L}_t
    \label{def: Y eta b 0 f g, SDE clipped}
\end{align}
% and set $\bm{Y}^{\eta|b;(0)}(x) \delequal \big\{{Y}^{\eta|b;(0)}_{t}(x):\ t \in [0,1]\big\}$ for any $b > 0$.
% As an immediate result of this construction, we have 
% $
% {Y}^{\eta|b;(0)}_t(x;a,\sigma) = Y^\eta_t(x)
% $
% and
% $\bm{Y}^{\eta|b;(0)}(x;a,\sigma) = \bm{Y}^\eta(x)$.
under initial condition ${\bm Y}^{\eta|b;(0)}_0(\bm x) = \bm x$.
Next, we define
\begin{align}
    \tau^{\eta|b;(1)}_{Y}(\bm x) & \delequal 
    \inf\Big\{ t > 0: \norm{\bm\sigma\big(\bm {Y}^{\eta|b;(0)}_{t-}(\bm x)\big) \Delta \bar{\bm L}^\eta_t } 
    % = \norm{\Delta \bm{Y}^{\eta|b;(0)}_t(\bm x) } 
    > b \Big\},
    \label{def, tau and W, discont in Y, eta b, 1}
    \\
    \bm W^{\eta|b;(1)}_{Y}(\bm x) & \delequal \Delta \bm{Y}^{\eta|b;(0)}_{ \tau^{\eta|b;(1)}_{Y}(\bm x)}(\bm x)
    \label{def, tau and W, discont in Y, eta b, 2}
\end{align}
as the arrival time and size of the first jump in $\bm{Y}^{\eta|b;(0)}_t(\bm x)$ with $L_2$ norm larger than $b$.
Furthermore, we define (for any $k \geq 1$)
\begin{align}
    {
    \bm{Y}^{\eta|b;(k)}_{ \tau^{\eta|b;(k)}_{Y}(\bm x) }(\bm x)
    } & \delequal 
    \bm{Y}^{\eta|b;(k)}_{ \tau^{\eta|b;(k)}_{Y}(\bm x)- }(\bm x)
    +
    \varphi_b\Big(
    \bm W^{\eta|b;(k)}_{Y}(\bm x)
    \Big),
    \label{def: objects for defining Y eta b, clipped SDE, 1}
    \\
    d{\bm Y^{\eta|b;(k)}_{ t}(\bm x)}& \delequal 
    \bm a\big(\bm{Y}^{\eta|b;(k)}_{ t-}(\bm x)\big) dt 
    + 
    \bm\sigma\big(\bm{Y}^{\eta|b;(k)}_{ t-}(\bm x)\big)d \bar{\bm L}^\eta_t
    \qquad \forall t > \tau^{\eta|b;(k)}_{Y}(\bm x),
    \label{def: objects for defining Y eta b, clipped SDE, 2}
    \\
    {
    \tau^{\eta|b;(k+1)}_{Y}(\bm x)
    } & \delequal 
    \min\Big\{ t > \tau^{\eta|b;(k)}_{Y}(\bm x): \norm{\sigma\big(\bm{Y}^{\eta|b;(k)}_{t-}(\bm x)\big) \Delta{\bar{\bm L}}^\eta_t } > b \Big\},
    \label{def: objects for defining Y eta b, clipped SDE, 3}
    \\
    {\bm W^{\eta|b;(k+1)}_{Y}(\bm x)} & \delequal 
    \Delta\bm{Y}^{\eta|b;(k)}_{ \tau^{\eta|b;(k + 1)}_{Y}(\bm x)}(\bm x)
    \label{def: objects for defining Y eta b, clipped SDE, 4}
\end{align}
Lastly, for any $t \geq 0,\ b > 0,\ k \in \mathbb N$ and $\bm x \in \R^m$, we define
(under convention $\tau^{\eta|b;(0)}_{Y}(\bm x) = 0$)
\begin{align}
    \notationdef{notation-Y-eta-b-SDE}{\bm Y^{\eta|b}_t(\bm x)}
    \delequal 
    \sum_{k \geq 0}
    \bm Y^{\eta|b;(k)}_{t}(\bm x)
    \cdot 
    \mathbbm{I}
    \Big\{
    t \in \Big[  
    \tau^{\eta|b;(k)}_{Y}(\bm x), 
     \tau^{\eta|b;(k+1)}_{Y}(\bm x) 
    \Big)
    \Big\}
    \label{defSDE, o.g., initial condition x, clipped}
\end{align}
and let $\notationdef{notation-bm-Y-0T-eta-b-SDE}{\bm{Y}^{\eta|b}_{[0,T]}(x)} \delequal \big\{\bm Y^{\eta|b}_t(\bm x):\ t \in [0,T]\big\}$.
By definition, for any $t \geq 0,\ b > 0,\ k \in \mathbb N$ and $\bm x \in \R^m$,
\begin{align}
    {\bm Y^{\eta|b}_t(\bm x)}  = \bm Y^{\eta|b;(k)}_{t}(\bm x)
    \qquad \Longleftrightarrow \qquad 
    t \in \Big[  
    \tau^{\eta|b;(k)}_{Y}(\bm x), 
     \tau^{\eta|b;(k+1)}_{Y}(\bm x) 
    \Big).
    \label{defSDE, initial condition x, clipped}
\end{align}
% In case that $T = 1$, we suppress $[0,1]$ and write $\notationdef{notation-bm-Y-eta-b-SDE}{\bm Y^{\eta|b}(\bm x)}$.
% The next theorem presents the sample path large deviations for $\bm Y^{\eta|b}_t(\bm x)$.
% Once again, the proof is omitted as it closely resembles that of $\bm X^{\eta|b}_j(\bm x)$.

Consider an
open set $I \subseteq \R^m$ such that $\bm 0 \in \bm I$ and Assumption~\ref{assumption: shape of f, first exit analysis} holds.
% We impose Assumption~\ref{assumption: shape of f, first exit analysis} on $a(\cdot)$.
Define stopping times
\begin{align*}
    {\tau^\eta_Y(\bm x)} \delequal \inf\big\{t \geq 0:\ \bm Y^\eta_t(\bm x) \notin I\big\},\qquad
    {\tau^{\eta|b}_Y(\bm x)} \delequal \inf\big\{t \geq 0:\ \bm Y^{\eta|b}_t(\bm x) \notin I \big\}
\end{align*}
as the first exit times of $\bm Y^\eta_t(\bm x)$ and $\bm Y^{\eta|b}_t(\bm x)$ from $I$, respectively.
The following result characterizes the asymptotic law of the first exit times and exit locations,
using the measures ${\widecheck{ \mathbf C }^{(k)|b}(\cdot )}$ defined in \eqref{def: measure check C k b}
and
${\widecheck{\mathbf C}(\cdot)}$ defined in \eqref{def: measure check C}.
We omit the proof
due to its similarity to that of Theorem \ref{theorem: first exit time, unclipped}.

\begin{theorem}
\label{theorem: first exit time, unclipped, SDE}
\linksinthm{theorem: first exit time, unclipped, SDE}
    Let Assumptions \ref{assumption: lipschitz continuity of drift and diffusion coefficients},
    \ref{assumption: shape of f, first exit analysis}, and \ref{assumption: heavy-tailed levy process} hold.
    Let $\beta \in [0,2\wedge \alpha)$.

    \begin{enumerate}[(a)]

\item 
    Let $b > 0$.
         Suppose that $\mathcal J^I_b < \infty$, $I^c$ is bounded away from $\mathcal G^{(\mathcal J^I_b - 1)|b}(\epsilon)$ for some (and hence all) $\epsilon > 0$ small enough,
        and
        $
        \widecheck{\mathbf C}^{( \mathcal J^I_b )|b}(\partial I) = 0.
        $
        Then 
        ${C^I_b} \delequal  \widecheck{ \mathbf{C} }^{ (\mathcal{J}^I_b)|b }(I^\complement) < \infty$.
        Furthermore, if $C^I_b \in (0,\infty)$,
        then
        for any $\epsilon > 0$, $t \geq 0$, and measurable set $B \subseteq I^c$,
\begin{align*}
    \limsup_{\eta\downarrow 0}\sup_{\bm x \in I_\epsilon}
    \P\bigg(C^I_b \lambda^{ \mathcal{J}^I_b }_L(\eta;\beta)\tau^{\eta|b}_Y(\bm x) > t
     ;\ \bm Y^{\eta|b}_{ \tau^{\eta|b}_Y(\bm x)}(\bm x) \in B\bigg)
     & \leq \frac{ \widecheck{\mathbf{C}}^{ (\mathcal{J}^I_b)|b }(B^-) }{ C^I_b }\cdot\exp(-t),
     \\
     \liminf_{\eta\downarrow 0}\inf_{\bm x \in I_\epsilon}
    \P\bigg(C^I_b \lambda^{ \mathcal{J}^I_b }_L(\eta;\beta)\tau^{\eta|b}_Y(\bm x) > t
     ;\ \bm Y^{\eta|b}_{ \tau^{\eta|b}_Y(\bm x)}(\bm x) \in B\bigg)
     & \geq \frac{ \widecheck{\mathbf{C}}^{ (\mathcal{J}^I_b)|b }(B^\circ) }{ C^I_b }\cdot\exp(-t).
\end{align*}
    Otherwise, we have $C^I_b = 0$, and
        \begin{align*}
            \limsup_{\eta\downarrow 0}\sup_{\bm x \in I_\epsilon}
        \P\bigg(
            \lambda^{ \mathcal J^I_b }_L(\eta;\gamma)\tau^{\eta|b}_Y(\bm x) \leq t
         \bigg) = 0
         \qquad
         \forall \epsilon > 0,\ t \geq 0.
        \end{align*}

    \item
        Suppose that $\widecheck{\mathbf C}(\partial I) = 0$.
        Then ${C^I_\infty} \delequal  \widecheck{ \mathbf{C} }(I^\complement) < \infty$.
        Furthermore, if $C^I_\infty > 0$,
        then
        for any $t \geq 0$ and measurable set $B \subseteq I^c$,
    \begin{align*}
        \limsup_{\eta\downarrow 0}\sup_{\bm x \in I_\epsilon}
        \P\bigg(
            C^I_\infty \lambda_L(\eta;\beta)\tau^\eta_Y(\bm x) > t;\ \bm Y^\eta_{ \tau^\eta_Y(\bm x)}(\bm x) \in B
        \bigg)
        & \leq
        \frac{ \widecheck{\mathbf{C}}(B^-) }{ C^I_\infty }\cdot\exp(-t),
        \\ 
        \liminf_{\eta\downarrow 0}\inf_{\bm x \in I_\epsilon}
        \P\bigg(
            C^I_\infty \lambda_L(\eta;\beta)\tau^\eta_Y(\bm x) > t;\ \bm Y^\eta_{ \tau^\eta_Y(\bm x)}(\bm x) \in B
        \bigg)
        & \geq
        \frac{ \widecheck{\mathbf{C}}(B^\circ) }{ C^I_\infty }\cdot\exp(-t).
    \end{align*}
     Otherwise, we have $C^I_\infty = 0$, and 
    \begin{align*}
        \limsup_{\eta\downarrow 0}\sup_{\bm x \in I_\epsilon}
        \P\bigg(
            \lambda(_L\eta;\gamma)\tau^\eta_Y(\bm x) \leq t
        \bigg) = 0
        \qquad
        \forall \epsilon > 0,\ t \geq 0.
    \end{align*}

    \end{enumerate}
\end{theorem}

\ifshownotationindex
\newpage
\footnotesize
% \newgeometry{left=1cm,right=1cm,top=0.5cm,bottom=1.5cm}
\section{Notation Index}
\label{subsec: Notation Index}
% \section*{\linkdest{location of notation index}Notation Index}
\begin{itemize}[leftmargin=*]

\item 
    \notationidx{asymptotic-equivalence}{Asymptotic Equivalence}:
    $X_n$ is asymptotically equivalent to $Y^\delta_n$ when bounded away from $\mathbb{C}$
    w.r.t.\ $\epsilon_n$ as $\delta \downarrow 0$ if the following holds:
    for each $\Delta > 0$
    and each $B \in \mathscr{S}_\mathbb{S}$ that is bounded away from $\mathbb{C}$,
    \begin{align*}
    \lim_{\delta \downarrow 0}\lim_{n \rightarrow \infty} \frac{\P\big( \bm{d}(X_n, Y^\delta_n)\mathbbm{I}( X_n\in B\text{ or }Y^\delta_n \in B) > \Delta \big)}{ \epsilon_n } = 0.
\end{align*}

\item 
    \notationidx{order-k-time-on-[0,t]}{$A^{k\uparrow}$}:
    Given $A \subseteq \R$,
    ${A^{k \uparrow}} \delequal 
\{
(t_1,\cdots,t_k) \in A^k:\ t_1 < t_2 < \cdots < t_k
\}
    $
    
% \item 
%     \notationidx{order-k-time-on-(0,t]}{$(0,t]^{k\uparrow}$}:
%     $(0,t]^{k\uparrow}
%     \delequal
%     \big\{
%     (t_1,\cdots,t_k) \in \mathbb{R}^k:\ 0 \leq t_1 < t_2 < \cdots < t_k < t
%     \big\}.
%     $

\item 
    \notationidx{set-for-integers-below-n}{$[n]$}: $[n] = \{1,2,\cdots,n\}$ for any positive integer $n$. For $n = 0$ we set $[n] = \emptyset$.

\item 
    \notationidx{floor-operator}{$\floor{x}$}:
    $\floor{x} = \max\{n \in \mathbb{Z}:\ n \leq x\}$
    
\item
    \notationidx{ceil-operator}{$\ceil{x}$}:
    ${\ceil{x}} = \min\{n \in \Z:\ n \geq x\}$
    % $
    % {\ceil{x}} \delequal \min\{n \in \Z:\ n \geq x\}$.
    % $

\item 
    \notationidx{notation-closure-of-set-E}{$E^-$}: closure of set $E$ 

\item 
    \notationidx{notation-interior-of-set-E}{$E^\circ$}: interior of set $E$

\item 
    \notationidx{notation-epsilon-enlargement-of-set-E}{$E^\epsilon$}: 
    $E^\epsilon \delequal 
\{ y \in \mathbb{S}:\ \bm{d}(E,y)\leq\epsilon \}$
    ($\epsilon$-enlargement)

\item 
    \notationidx{notation-epsilon-shrinkage-of-set-E}{$E_{\epsilon}$}:
    $E_{\epsilon} \delequal
    ((E^\complement)^\epsilon)^\complement
    $
    ($\epsilon$-shrinkage)

\linkdest{location, notation index A}

\item 
    \notationidx{a}{$\bm a$}: drift coefficient $\bm a: \mathbb{R}^m \to \mathbb{R}^m$

\item
    \notationidx{a-M}{$\bm a_M$}: drift coefficient with the argument $\bm x$ projected onto the ball $\{ \bm x \in \R^m:\ \norm{\bm x} \leq M \}$

\item 
    \notationidx{alpha-noise-tail-index-LDP}{$\alpha$}:
    $\alpha > 1$; the heavy tail index for $(\bm Z_j)_{j \geq 1}$ in Assumption \ref{assumption gradient noise heavy-tailed}

\item
    \notationidx{notation-event-A-i-concentration-of-small-jumps}{$A_i(\eta,b,\epsilon,\delta,\bm x)$}:
    $
    A_i(\eta,b,\epsilon,\delta,\bm x)
    \delequal 
    % \left\{\rule{0cm}{0.9cm}
    \Big\{
    \underset{j \in I_i(\eta,\delta) }{\max}\  
        \eta\norm{\sum_{n = \tau^{>\delta}_{i-1}(\eta) + 1}^j \bm \sigma\big( \bm X^{\eta|b}_{n-1}(\bm x) \big)\bm Z_n} \leq \epsilon 
        \Big\}
    % \right\}
    $

\linkdest{location, notation index B}

\item   
    $\notationidx{notation-ball-r-x}{\bar B_r(\bm x)}:\ {\bar B_r(\bm x)} \delequal \{ \bm{y}\in\R^m:\ \norm{\bm y - \bm x} \leq r \}$

\item 
    $
    \notationidx{notation-set-B-0}{B_0}:\ {B_0}\delequal 
    \Big\{ \bm{X}^{\eta|b}(\bm x)\in B\text{ or }\hat{\bm X}^{\eta|b;(k)}(\bm x) \in B;\   
    \bm{d}_{J_1}\big(\bm{X}^{\eta|b}(\bm x),\hat{\bm X}^{\eta|b;(k)}(\bm x)\big) > \Delta
    \Big\}
    $

\item 
    $
    \notationidx{notation-B1}
    {B_1}:\ {B_1} \delequal 
    \{ \tau^{>\delta}_{k+1}(\eta) > \floor{1/\eta}\}
    $

\item 
    $
    \notationidx{notation-B2}
    {B_2}:\ {B_2} \delequal 
    \{ \tau^{>\delta}_{k}(\eta) \leq \floor{1/\eta} \}
    $

\item 
    $
    \notationidx{notation-B3}
    {B_3}:\ {B_3} \delequal 
    \Big\{\eta \norm{ \bm W^{>\delta}_{i}(\eta) } > \bar{\delta}\ \text{for all }i \in [k] \Big\}
    $

\item
    $
    \notationidx{notation-B4}
    {B_4}:\ {B_4} \delequal 
    \Big\{\eta \norm{ \bm W^{>\delta}_{i}(\eta) } \leq 1/\epsilon^{ \frac{1}{2k}  } \ \text{for all }i \in [k] \Big\}
    $

\linkdest{location, notation index C}

\item 
    \notationidx{notation-constant-C-boundedness-assumption}{$C$}:
    $C \in [1,\infty)$ is the constant in Assumption \ref{assumption: boundedness of drift and diffusion coefficients} with $\norm{\bm a(\bm x)} \vee \norm{\bm \sigma(\bm x)} \leq  C\ \ \forall \bm x \in \mathbb{R}^m$.

\item 
    \notationidx{notation-C-*-first-exit-time}{$C^I_\infty$}: 
    $
    C^I_\infty \delequal  \widecheck{ \mathbf{C} }\big( I^\complement\big)
    $

\item 
    \notationidx{notation-C-b-*}{$C_b^I$}: 
    $
    {C^I_b} \delequal  \widecheck{ \mathbf{C} }^{ (\mathcal{J}^I_b)|b }(I^\complement)
    $

\item 
    \notationidx{notation-measure-C-k-t-mu-LDP}{$\mathbf{C}^{(k)}_{[0,T]}(\ \cdot\ ;\bm x)$}:
    $
    {\mathbf{C}^{(k)}_{ {[0,T]} }(\ \cdot\ ;\bm x)} 
    \delequal
    {\mathbf{C}^{(k)|\infty}_{ {[0,T]} }(\ \cdot\ ;\bm x)}
    =
   \int \mathbbm{I}\Big\{ h^{(k)}_{ {[0,T]}}\big( \bm x, \bm W, \bm t  \big) \in\ \cdot\  \Big\}
    \big((\nu_\alpha \times \mathbf S)\circ \Phi\big)^k(d \bm W) \times\mathcal{L}^{k\uparrow}_{{T}}(d\bm t).
    $

\item 
    \notationidx{notation-measure-C-k-t-mu-LDP-T-1}{$\mathbf C^{(k)}$}:
    $\mathbf C^{(k)} = \mathbf C^{(k)}_{[0,1]}$
    % =
    % \int \mathbbm{I}\Big\{ h^{(k)}\big( x,(w_1,\cdots,w_k),(t_1,\cdots,t_k)   \big) \in\ \cdot\  \Big\} \nu^k_\alpha(d w_1,\cdots,dw_k) \times\mathcal{L}^{k\uparrow}_1(dt_1, dt_2,\cdots,dt_k)

\item 
    \notationidx{notation-measure-C-k-t-truncation-b-LDP}{${\mathbf{C}}^{(k)|b}_{[0,T]}(\ \cdot\ ;\bm x)$}:
    $
    {\mathbf{C}}^{(k)|b}_{[0,T]}(\ \cdot\ ;\bm x) \delequal 
    \int \mathbbm{I}\Big\{ h^{(k)|b}_{{[0,T]}}\big( \bm x,\textbf W,\bm t  \big) \in\ \cdot\  \Big\} 
   \big((\nu_\alpha \times \mathbf S)\circ \Phi\big)^k(d \textbf W) \times\mathcal{L}^{k\uparrow}_{ {T} }(d\bm t)
   $

\item 
    \notationidx{notation-measure-C-k-t-truncation-b-LDP-T=1}{${\mathbf{C}}^{(k)|b}$}:
    $
    {\mathbf{C}}^{(k)|b} = {\mathbf{C}}^{(k)|b}_{[0,1]}
    $

\item  
    \notationidx{notation-check-C}{$\widecheck{\mathbf C}(\ \cdot\ )$}:
    $
    {\widecheck{\mathbf C}(\ \cdot\ )} \delequal \int \mathbbm{I}\Big\{ \bm \sigma(\bm 0) \bm w \in\ \cdot\ \Big\}
    \big((\nu_\alpha \times \mathbf S)\circ \Phi\big)(d \bm w)
    $
    
% \item
%     \notationidx{notation-check-C-x=0}{$\widecheck{\mathbf C}(\cdot)$}:
%     $\widecheck{\mathbf C}(\cdot)=\widecheck{\mathbf C}(\ \cdot\ ;0)$

\item 
    \notationidx{notation-check-C-k-b}{$\widecheck{ \mathbf C }^{(k)|b}(\ \cdot\ )$}:
    $
    {\widecheck{ \mathbf C }^{(k)|b}(\ \cdot\ )}
    \delequal 
    \int \mathbbm{I}\bigg\{ \widecheck{g}^{(k-1)|b}\Big( \varphi_b\big(\bm\sigma(\bm x)\bm w_1\big),(\bm w_2,\cdots,\bm w_k),\bm t \Big) \in \ \cdot \  \bigg\}
     \big((\nu_\alpha \times \mathbf S)\circ \Phi\big)^k(d \textbf W) \times \mathcal{L}^{k-1\uparrow}_\infty(d\bm t)
    $

% \item 
%     \notationidx{notation-check-C-k-b-x=0}{$\widecheck{ \mathbf C }^{(k)|b}(\cdot)$}:
%     $
%     \widecheck{ \mathbf C }^{(k)|b}(\cdot) =
%     \widecheck{ \mathbf C }^{(k)|b}(\ \cdot\ ;0 )
%     $

\item 
    \notationidx{notation-mathcal-C-S-exclude-C}{$\mathcal{C}({ \mathbb{S}\setminus \mathbb{C} })$}:
    % $
% \mathcal{C}({ \mathbb{S}\setminus \mathbb{C} })
% $
% is 
the set of all real-valued, non-negative, bounded and continuous functions with support bounded away from $\mathbb{C}$

\item   
    $
    \notationidx{notation-error-function-check-c-epsilon}{\widecheck{\bm c}(\epsilon)}:\ {\widecheck{\bm c}(\epsilon)} 
    \delequal 
    \mathcal J^I_b \cdot (\bar t)^{ \mathcal J^I_b - 1 } \cdot (\bar\delta)^{ -\alpha \cdot (\mathcal J^I_b - 1) }
    \cdot 
    \epsilon^{ \frac{\alpha}{2\mathcal J^I_b}  }.
    $

\linkdest{location, notation index D}
\item
    \notationidx{notation-Lipschitz-constant-L-LDP}{$D$}:
    The Lipschitz $D \in[1,\infty)$ in Assumption \ref{assumption: lipschitz continuity of drift and diffusion coefficients}:
    $\norm{\bm \sigma(\bm x) - \bm \sigma(\bm y)} \vee \norm{\bm a(\bm x)-\bm a(\bm y)} \leq D\norm{\bm x - \bm y}\ \ \ \forall \bm x,\ \bm y \in \mathbb{R}^m$

\item 
    \notationidx{notation-D-0T-cadlag-space}{$\mathbb{D}{[0,T]}$}:
     $\mathbb{D}[0,{T}] = \D\big([0,T],\R^m\big)$ is
    the space of all càdlàg functions with domain $[0,{T}]$ and codomain $\R^m$
    % we suppress the subscript $m$ in the notation $\D_m[0,T]$ since the ambient space $\R^m$ is fixed.

\item 
    \notationidx{notation-D-0T-cadlag-space-T=1}{$\mathbb{D}$}:
    $
    \mathbb D \delequal \mathbb D[0,1]
    $

\item 
    \notationidx{notation-D-A-k-t-LDP}{$\mathbb{D}^{(k)}_A[0,T](\epsilon)$}:
    $\mathbb{D}^{(k)}_A[0,T](\epsilon) \delequal h^{(k)}_{ {[0,T]} }\Big( A \times \mathbb{R}^{m \times k} \times \big( \bar B_\epsilon(\bm 0)\big)^k \times (0,{T}]^{k\uparrow} \Big)$ with convention that $\mathbb{D}_A^{(-1)}[0,T](\epsilon) = \emptyset$

\item
    \notationidx{notation-D-A-k-t-LDP-T=1}{$\mathbb{D}^{(k)}_A(\epsilon)$}:
    $
    \mathbb{D}^{(k)}_A(\epsilon) \delequal \mathbb{D}^{(k)}_A[0,1](\epsilon) = \bar h^{(k)}\Big( A \times \mathbb{R}^{m \times k} \times \big(\bar B_\epsilon(\bm 0)\big)^k \times (0,1]^{k\uparrow} \Big)
    $

\item 
    \notationidx{notation-D-A-k-t-truncation-b-LDP}{$\mathbb{D}_{A}^{(k)|b} [0,T](\epsilon)$}:
    ${ \mathbb{D}}_{A}^{(k)|b}[0,T](\epsilon) \delequal h^{(k)|b}_{[0,T]} \Big( A \times \mathbb{R}^{m\times k} \times \big(\bar B_\epsilon(\bm 0)\big)^k\times(0,T]^{k\uparrow} \Big)$ with convention that $\mathbb{D}_{A}^{(-1)|b}[0,T](\epsilon)  = \emptyset$

\item 
    \notationidx{notation-D-A-k-t-truncation-b-LDP-T=1}{${\mathbb{D}}_{A}^{(k)|b}(\epsilon)$}:
    ${\mathbb{D}}_{A}^{(k)|b}(\epsilon) \delequal {\mathbb{D}}_{A}^{(k)|b}[0,1](\epsilon)
    =
    \bar h^{(k)|b} \Big( A \times \mathbb{R}^{m \times k} \times \big(\bar B_\epsilon(\bm 0)\big)^k\times(0,1]^{k\uparrow} \Big)$

\item 
    \notationidx{notation-D-A-k-t-truncation-b-M-LDP}{$\mathbb{D}_{A;M\downarrow}^{(k)|b}(\epsilon)$}: 
    $\mathbb{D}_{A;M\downarrow}^{(k)|b}(\epsilon) \delequal 
\bar h^{(k)|b}_{M\downarrow}\Big( A \times \mathbb{R}^{m \times k} \times \big(\bar B_\epsilon(\bm 0)\big)^k \times (0,1]^{k\uparrow} \Big)
$

% \item
%     \notationidx{notation-D-0-A-T}{$\mathbb{D}_A^{(0)}[0,T]$}:
%     $\mathbb{D}_A^{(0)}[0,T] \delequal \Big\{ \{\bm{x}(t;x):\ t \in [0,T]\} :\ x \in A \Big\}$
    
% \item
%     \notationidx{notation-D-1-A-T}{$\mathbb{D}_A^{(1)}[0,T]$}:
%     $\mathbb{D}_A^{(1)}[0,T] \delequal h^{(1)}_{[0,T]}\big(A \times \R \times [0,T]\big).$

% \item 
%     \notationidx{notation-set-D-k-A-b-T}{$\mathbb{D}^{(k)|b}_{A}[0,T]$}:
%     ${\mathbb{D}^{(k)|b}_{A}[0,T]} \delequal h^{(k)|b}_{[0,T]}\big( A \times \R^k \times (0,T)^{k\uparrow} \big)$

\item 
    \notationidx{notation-D-J1}{$\dj{[0,T]}$}:
    Skorokhod $J_1$ metric on $\mathbb{D}[0,T]$

\item 
    \notationidx{notation-D-J1-T=1}{$\bm{d}_{J_1}$}:
    $\bm{d}_{J_1} = \dj{[0,1]}$ is the Skorodhod metric on $\mathbb{D} = \mathbb{D}[0,1]$

\linkdest{location, notation index E}

\item 
    \notationidx{notation-set-E-delta-LDP}{$E^\delta_{c,k}(\eta)$}:
    $E^\delta_{c,k}(\eta) \delequal \{ \tau^{>\delta}_{k}(\eta) \leq \floor{1/\eta} < \tau^{>\delta}_{k+1}(\eta);\ \eta\norm{\bm W^{>\delta}_j(\eta)} > c\ \ \forall j \in [k] \}$
    \quad
    $(c > \delta)$
    \quad
    (event that there are exactly $k$ ``big'' jumps by $\floor{1/\eta}$)

\item 
    $
    \notationidx{notation-set-check-E-epsilon-B-T}{\widecheck{E}(\epsilon,B,T)}:\ {\widecheck{E}(\epsilon,B,T)} 
    \delequal 
    \Big\{ \xi \in \mathbb{D}[0,T]:\ \exists t \leq T\ s.t.\ \xi_t \in B\text{ and }\xi_s \in I(\epsilon)\ \forall s \in [0,t) \Big\}
    $

\item 
    \notationidx{notation-eta}{$\eta$}: step length

\linkdest{location, notation index F}

\item
    $\notationidx{notation-sigma-algebra-F}{\mathcal{F}}$:
    the $\sigma-$algebra generated by iid copies $(\bm Z_j)_{j \geq 1}$
    
\item 
    \notationidx{notation-F}{$\mathbb F$}: 
    the filtration
    $\mathbb{F} = (\mathcal{F}_j)_{j \geq 0}$ where  $\mathcal{F}_0 \delequal \{\Omega, \emptyset\}$
    and $\mathcal{F}_j$ is the $\sigma$-algebra generated by $\bm Z_1,\cdots,\bm Z_j$

\linkdest{location, notation index G}

\item 
    $
    \notationidx{notation-mapping-bar-g-k-b}{\bar g^{(k)|b}}:\ 
    \bar g^{(k)|b}\big( \bm x, \textbf W, \textbf V, (t_1,\cdots,t_k)\big)
    \delequal 
    \bar h^{(k)|b}_{ [0,t_k + 1] }
    \Big(
        \bm x,
        \textbf W,
        \textbf V,
        (t_1,\cdots,t_k)
    \Big)(t_k)
    $

\item
    \notationidx{notation-check-g-k-b}{$\widecheck{g}^{(k)|b}$}:
    $
    {\widecheck{g}^{(k)|b}(\bm x,\textbf W,\bm t)}
    \delequal 
    \bar g^{(k)|b}\big(\bm x, \textbf W, (\bm 0,\cdots,\bm 0), \bm t\big)
    =
    h^{(k)|b}_{[0,t_k+1]}(\bm x,\textbf W,\bm t)(t_{k})
    $

\item 
    $
    \notationidx{notation-set-G-k-b-epsilon}{\mathcal G^{(k)|b}(\epsilon)}:\ 
    {\mathcal G^{(k)|b}(\epsilon)}
    \delequal 
    \bigg\{
    \bar g^{(k - 1)|b}
    \Big( \bm v_1 + \varphi_b\big(\bm \sigma(\bm v_1)\bm w_1\big),
    (\bm w_2,\cdots, \bm w_k), (\bm v_2,\cdots,\bm v_k), \bm t 
    \Big):
    % \\ 
    % & \qquad\qquad
    \textbf W = (\bm w_1,\cdots, \bm w_k) \in \R^{d\times k},
    \textbf V = (\bm v_1,\cdots, \bm v_k) \in \Big(\bar B_\epsilon(\bm 0)\Big)^k,
    \bm t \in (0,\infty)^{k \uparrow}
    \bigg\}.
    $
    $
    \mathcal{G}^{(0)|b}(\epsilon) \delequal \bar B_{\epsilon}(\bm 0).
    $

\item 
    $
    \notationidx{notation-set-G-k-b}{\mathcal G^{(k)|b}}:\ 
    {\mathcal G^{(k)|b}}
    \delequal 
    \mathcal G^{(k)|b}(0)
    = 
    \bigg\{
    \widecheck{g}^{(k - 1)|b}
    \Big( \varphi_b\big(\bm \sigma(\bm 0)\bm w_1\big),
    (\bm w_2,\cdots, \bm w_k), \bm t 
    \Big):\ 
    \textbf W = (\bm w_1,\cdots,\bm w_k) \in \R^{d \times k},
    \bm t \in (0,\infty)^{k \uparrow}
    \bigg\}.
    $

\item 
    $
    \notationidx{notation-extended-coverage-set-bar-G-k-b-epsilon}{\bar{\mathcal G}^{(k)|b}(\epsilon)}:\ 
    {\bar{\mathcal G}^{(k)|b}(\epsilon)}
    \delequal
    \Big\{
        \bm y_t(\bm x):\ \bm x \in \mathcal G^{(k)|b}(\epsilon),\ t \geq 0
    \Big\},
    $

% \item 
%     \notationidx{notation-F-mu}{$\mathbb F_\mu$}: 
%     $\mathbb{F}_\mu = (\mathcal{F}_{j,\mu})_{j \geq 0}$

% \item 
%     \notationidx{notation-fixed-function-g-LDP}{$g$}:
%     a bounded, continuous function $g: \mathbb{D}[0,T]\to \mathbb{R}$ that is fixed at the beginning of Section \ref{subsec: LDP unclipped, proof of main results} and \ref{subsec: LDP clipped, proof of main results}

% \item
%     $\notationidx{notation-G-b-typical-transition-graph}{\mathcal{G}_b}$: Directed graph ${\mathcal{G}_b} = (V,E_b)$,
%     where $V = \{m_1,\cdots,m_{ n_\text{min} }\}$,
%     and
%     an edge $(m_i\rightarrow m_j)$ is in $E_b$ iff $\mathcal J^I_b(i,j) = \mathcal J^I_b(i)$.

\item 
    \notationidx{notation-Gamma-M-adapted-process-bounded-by-M-LDP}{$\bm{\Gamma}_M$}:
    $\bm{\Gamma}_M \delequal \big\{ (\bm V_j)_{j \geq 0}\text{ is adapted to }\mathbb{F}:\ \norm{\bm V_j} \leq M\ \forall j \geq 0\text{ almost surely} \big\};$
    see \eqref{def: Gamma M, set of bounded adapted process}
    
% \item
%     \notationidx{notation-Gamma-M-cont-adapted-process-bounded-by-M-LDP}{$\bm{\Gamma}_M^{\text{cont}}$}:
%     $\bm{\Gamma}^{\text{cont}}_M \delequal \Big\{ V\text{ takes value in }\mathbb{D}\text{ and is adapted to }\mathbb{F}:\ \sup_{t \in [0,1]}|V(t)| \leq M\text{ almost surely} \Big\}.$

\linkdest{location, notation index H}
% \item 
%     \notationidx{notation-H-plus}{$H^{(+)}$}:
%     $H^{(+)}(x)  \delequal \P(Z > x)$ 

% \item 
%     \notationidx{notation-H-minus}{$H^{(-)}$:} $H^{(-)}(x) \delequal \P(Z < -x) $

\item 
    \notationidx{notation-H}{$H$:} $H(x)  \delequal \P(\norm{\bm Z} > x) \in \RV_{-\alpha}(x)$

\item
    \notationidx{notation-H-L}{$H_L$}:
    $H_L(x)\delequal  \nu\big( \{ \bm y \in \R^m:\ \norm{\bm y} > x\} \big) \in \RV_{-\alpha}(x)$

% \item 
%     \notationidx{notation-h-0-T-stationary-dist-LDP}{$h^{(1)}_{[0,T]}$}:
%     $h^{(1)}_{[0,T]}: \R \times \R \times (0,T) \to \mathbb{D}[0,T]$ is such that $\xi = h^{(1)}_{[0,T]}(x_0,w,t)$ solves \eqref{def: mapping h, 1, LD of stationary distribution of heavy tailed SGD}-\eqref{def: mapping h, 3, LD of stationary distribution of heavy tailed SGD}.

% \item 
%     \notationidx{notation-h-k-t-mapping-LDP}{$h^{(k)}_{[0,T]}$}: 
%     ${h^{(k)}_{[0,T]}}(\bm x,\bm{W},\bm{t}) \delequal
%     h^{(k)}_{[0,T];\bm 1}(\bm x,\bm{W},\bm{t})$;
%     That is, with all $\sigma_{i,j}(\cdot) \equiv 1$.

\item
    $\notationidx{notation-h-k-t-bar-mapping-LDP}{\bar h^{(k)}_{[0,T]}}$: 
    ${\bar h^{(k)}_{[0,T]}} = {\bar h^{(k)|\infty}_{[0,T]}}$.
    An operator for perturbed gradient flow under $\bm a(\cdot)$ with initial value $\bm x$, jump sizes $\bm w_j$'s (modulated by $\bm \sigma(\cdot)$) with perturbations $\bm v_j$'s, and jump times $t_j$'s

\item 
    $\notationidx{notation-h-k-t-mapping-LDP}{h^{(k)}_{[0,T]}}:\ 
    h^{(k)}_{[0,T]}(\bm x, \bm W, \bm t)
    \delequal 
    h^{(k)|\infty}_{[0,T]}(\bm x, \bm W, \bm t)
    =
    \bar h^{(k)}_{[0,T]}\big(\bm x,\bm W, (\bm 0,\cdots,\bm 0), \bm t\big)
    $

% \item 
%     \notationidx{notation-h-k-t-mapping-LDP-T=1}{$h^{(k)}$}:
%     $h^{(k)}\delequal h^{(k)}_{[0,1]}$

\item 
    $\notationidx{notation-h-k-t-sigma-mapping-T=1}{h^{(k)}}:\ 
    {h^{(k)}} \delequal {h^{(k)}_{[0,1]}}
    $

\item 
    $\notationidx{notation-h-k-t-bar-mapping-truncation-level-b-LDP}{\bar h^{(k)|b}_{[0,T]}}$:
    an operator for perturbed gradient flow under $\bm a(\cdot)$ with initial value $\bm x$, jump sizes $\bm w_j$'s (modulated by $\bm \sigma(\cdot)$ and truncated under $b > 0$) with perturbations $\bm v_j$'s, and jump times $t_j$'s; see \eqref{def: perturb ode mapping h k b, 1}--\eqref{def: perturb ode mapping h k b, 3}

\item 
    $\notationidx{notation-h-k-b-t-mapping-LDP}{h^{(k)|b}_{[0,T]}}:\ h^{(k)|b}_{[0,T]}(\bm x, \bm W, \bm t)
    \delequal 
    \bar h^{(k)|b}_{[0,T]}\big(\bm x,\bm W, (\bm 0,\cdots,\bm 0), \bm t\big)$

% \item 
%     \notationidx{notation-h-k-b-t-mapping-LDP}{$h^{(k)|b}_{[0,T]}$}:
%     $
%     {h^{(k)|b}_{[0,T]}} \delequal
%     h^{(k)|b}_{[0,T];\bm 1}
%     $

% \item 
%     \notationidx{notation-h-k-t-mapping-truncation-level-b-LDP-T=1}{$h^{(k)|b}$}:
%     $h^{(k)|b} \delequal h^{(k)|b}_{[0,1]}$

\item 
    $\notationidx{notation-h-k-t-sigma-mapping-truncation-level-b-LDP-T=1}{h^{(k)|b}}:\ {h^{(k)|b}}  \delequal h^{(k)|b}_{[0,1]}$

\item
    \notationidx{notation-mapping-bar-h-k-t-b-M-LDP}{$\bar h^{(k)|b}_{M\downarrow}$}: a modified version of {$\bar h^{(k)|b}$}
    where the truncated drift and diffusion coefficients $\bm a_{M}$, $\bm \sigma_M$ are applied instead of $\bm a$, $\bm\sigma$;
    see \eqref{def: perturb ode mapping h k b, truncated at M, 1}--\eqref{def: perturb ode mapping h k b, truncated at M, 3}

\item
    $\notationidx{mapping-h-k-b-M-down}{h^{(k)|b}_{M\downarrow}}:\ 
    {h^{(k)|b}_{M\downarrow}}\big(\bm x, \textbf W,\bm t\big)
    \delequal 
    \bar h^{(k)|b}_{M\downarrow}\big(\bm x, \textbf W,(\bm 0,\cdots,\bm 0),\bm t\big).$

\linkdest{location, notation index I}

\item 
    $\notationidx{notation-exit-domain-I}{I}$:
    the open, bounded domain $\bm 0 \in {I} \subset \R^m$ that belongs to the attraction field for $\bm 0$ satisfying Assumption~\ref{assumption: shape of f, first exit analysis}.

\item
    $
    \notationidx{notation-I-epsilon-shrinkage}{I_\epsilon}:\ {I_\epsilon} = \{ \bm y:\ \norm{\bm x - \bm y} < \epsilon\ \Longrightarrow\ \bm x \in I \}
    $
    % $\notationidx{notation-exit-domain-I-epsilon}{I_\epsilon}$:
    % ${I_\epsilon} = \{ \bm y:\ \bm x \in I\ \forall \bm x \in \R^m\ s.t.\ \norm{\bm x - \bm y} < \epsilon \}$

\item 
    $\notationidx{notation-domain-check-I-epsilon}{\widecheck I(\epsilon)}:\ 
    {\widecheck I(\epsilon)}\delequal 
    \Big\{
        \bm x \in I:\  \norm{ \bm y_{1/\epsilon}(\bm x) } < \widecheck \epsilon
    \Big\}$
    with $\widecheck{\epsilon} > 0$ defined in \eqref{def: covering sets I epsilon, first exit time}

\item

    \notationidx{notation-A-i-concentration-of-small-jumps-2}{$I_i(\eta,\delta)$}:
    $I_i(\eta,\delta) 
    \delequal 
    \big\{j \in \mathbb{N}:\  \tau^{>\delta}_{i-1}(\eta) + 1 \leq j \leq \big(\tau^{>\delta}_{i}(\eta) - 1 \big) \wedge \floor{1/\eta}\big\}.$

% \item 
%     $\notationidx{notation-identity-matrix}{\textbf I_m}$: identity matrix in $\R^{m \times m}$

\linkdest{location, notation index J}

\item
    \notationidx{J-Z-c-n}{$\mathcal{J}_Z(c,n):$}
    $\mathcal{J}_{\bm Z}(c,n) \delequal \#\{i \in [n]:\ \norm{\bm Z_i} \geq c \}$

% \item 
%     \notationidx{notation-J-L-c-T}{$\mathcal{J}_L(c,T)$}:
%     $
%      \mathcal{J}_L(c,T) \delequal \#\big\{ i \in \mathbb{N}:\ \widetilde{\tau}^L_i \leq T,\ |Z^L_i| > c \big\}
%     $

% \item 
%     \notationidx{notation-J-eta-leq-k-L-i}{$J^{\eta;(k)}_L(i)$}:
%     $
%     J^{\eta;(k)}_L(i) \delequal \min\big\{ j > J^{\eta;(k)}_L(i-1):\ |Z^{L}_j| \geq \bm{Z}^{(k)}_L(\eta)  \big\}
%     $

% \item 
%     \notationidx{notation-J-B-x-jump-number}{$\mathcal{J}(B;x)$}:
%     $
% \mathcal{J}(B;x)\delequal \min\{ k \geq 0: B\cap \mathbb{D}^{(k)}_{\{x\}}\neq \emptyset \}.
%     $

% \item 
%     \notationidx{notation-jump-number-J-b-B-x}{$\mathcal{J}_b(B;x)$}:
%     $\mathcal{J}_b(B;x)\delequal \min\{ k \geq 0: B\cap \mathbb{D}^{(k)|b}_{\{x\}}\neq \emptyset \}.$

\item  
    \notationidx{notation-J-*-first-exit-analysis}{$\mathcal{J}^I_b$}:
    $
    \mathcal{J}^I_b \delequal \min\big\{ k \geq 1:\ \mathcal G^{(k)|b} \cap I^\complement \neq \emptyset \big\}.
    $
    The ``discretized width'' metric for $I$ w.r.t.\ truncation threshold $b$.

\linkdest{location, notation index K}

% \item 
%     $\notationidx{notation-compacta-of-S-minus-C}{\mathcal K(\mathbb S\setminus \mathbb C)}$: 
%     the collection of all compact sets of $\mathbb S\setminus \mathbb C$

\linkdest{location, notation index L}

% \item 
%     $\notationidx{notation-r-radius-of-exit-domain}{l}$:
%     ${l} \delequal \inf_{x \in I^\complement}|x| = |s_\text{left}| \wedge s_\text{right}$

% \item 
%     $\notationidx{notation-r-i-radius-of-I-i}{l_i}: {l_i} \delequal{} \inf_{x \in I_i^\complement}|x - m_i|
%     = 
%     |m_i - s_{i-1}| \wedge |s_i - m_i|$

% \item 
%     $\notationidx{notation-l-i-j}{l_{i,j}}: {l_{i,j}} \delequal \inf_{x \in I_j}|x - m_i| =
%     \begin{cases}
%         s_{j-1} - m_i & \text{if }\ j > i \\
%         m_i - s_j & \text{if }\ j < i
%     \end{cases}$

\item
    \notationidx{notation-levy-process}{$\bm L$}:
    $\bm{L} = \{\bm L_t: t \geq 0\}$ is the L\'evy process taking values in $\R^m$ with the generating triplet $(c_{\bm L},\bm \Sigma_{\bm L},\nu)$ where $c_{\bm L} \in \mathbb{R}^m$ is the drift parameter, 
    $\bm \Sigma_{\bm L}$ is the positive semi-definite matrix that dictates the magnitude of the Brownian motion term in $\bm L_t$, and $\nu$ is the L\'evy measure.

% \item
%     \notationidx{notation-scaled-levy-process}{$\bar{\bm L}^\eta$}:
%     $\bar{\bm L}^\eta \delequal \big\{ \bar{L}^\eta_t = \eta L_{t/\eta}:\ t \in [0,1]\big\}$

% \item
%     \notationidx{notation-L-larger-than-delta-eta}{$L^{>\frac{\delta}{\eta}}_t$}:
%      $L^{>\frac{\delta}{\eta}}_t \delequal{} \sum_{0 \leq s \leq t}\Delta L_s\mathbbm{I}\{ |\Delta L_s| > \delta/\eta\}$
     
%  \item
%     \notationidx{notation-L-larger-than-delta-eta-scaled}{$\bar{L}^{\eta,>\frac{\delta}{\eta}}_t$}:
%      $\bar{L}^{\eta,>\frac{\delta}{\eta}}_t= \eta \cdot L^{>\frac{\delta}{\eta}}_{t/\eta} = \eta\sum_{0 \leq s \leq \frac{t}{\eta}}\Delta L_s\mathbbm{I}\{ |\Delta L_s| > \delta/\eta  \}$

% \item
%     \notationidx{notation-L-smaller-than-delta-eta-N}{$L^{(N,\frac{\delta}{\eta}]}_t$}:
%     $L^{(N,\frac{\delta}{\eta}]}_t \delequal{} -\mu_L(N)t + \sum_{0 \leq s \leq t}\Delta L_s\mathbbm{I}\{ |\Delta L_s| \in (N,\delta/\eta]  \}$ where $\mu_L(z) \delequal{} \int_{|x| > z}x\nu(dx)$
    
% \item 
%     \notationidx{notation-L-smaller-than-delta-eta-N-scaled}{$\bar{L}^{\eta,(N,\frac{\delta}{\eta}] }_t$}:
%     $\bar{L}^{\eta,(N,\frac{\delta}{\eta}] }_t = \eta \cdot L^{(N,\frac{\delta}{\eta}]}_{t/\eta}$

\item 
    \notationidx{notation-lebesgue-measure-restricted}{$\mathcal{L}_t$}: 
    Lebesgue measure restricted on $(0,t)$

\item 
    \notationidx{notation-lebesgue-measure-on-ordered-[0,t]}{$\mathcal{L}^{k\uparrow}_t$}:
    Lebesgue measure restricted on $(0,t)^{k \uparrow}$

\item
    \notationidx{notation-measure-L-k-up-infty}{$\mathcal{L}^{k\uparrow}_\infty$}:
    Lebesgue measure restricted on $\{ (t_1,\cdots,t_k) \in (0,\infty)^k:\ 0 < t_1 < t_2 < \cdots < t_k \}$

\item 
    \notationidx{notation-law-of-X}{$\mathscr{L}(X)$}:
    law of the random element $X$

\item 
    \notationidx{notation-law-of-X-cond-on-A}{$\mathscr{L}(X|A)$}:
    conditional law of $X$ on event $A$

% \item
%     \notationidx{notation-L-P-X-LDP}{$\mathscr L_\P(X)$}:
%     The distribution of random variable $X$ under $\P$
    
% \item
%     \notationidx{notation-L-P-X-A-LDP}{$\mathscr L_\P(X|A)$}:
%     The distribution of random variable $X$ conditioning on event $A$ under $\P$

% \item 
%     \notationidx{notation-scale-function-lambda}{$\lambda^k(\eta)$}:
%     $\lambda^k(\eta)\delequal \Big(\frac{H(1/\eta)}{\eta}\Big)^{k} \in \RV_{k(\alpha - 1)}(\eta)$
\item 
    $\notationidx{notation-lambda-scale-function}{\lambda(\eta)}$:
    $
    {\lambda(\eta)} \delequal \eta^{-1}H(\eta^{-1}) \in \RV_{\alpha -1}(\eta)
    $
    as $\eta \downarrow 0$. 
    
\item 
    \notationidx{notation-scale-function-lambda-L}{$\lambda_L(\eta;\beta)$}:
    $
    {\lambda_L(\eta;\beta)} \delequal \eta^{-\beta}H_L(\eta^{-1}) \in \RV_{\alpha-1}(\eta)
    $ as $\eta \downarrow 0$

% \item 
%     $\notationidx{notation-scale-function-lambda-*-b}{\lambda^*_b(\eta)}: 
%     {\lambda^*_b(\eta)} \delequal \eta \cdot \lambda^{ \mathcal J^I_b(V) }(\eta) \in \RV_{ \mathcal J^I_b(V)\cdot (\alpha-1)  + 1 }(\eta).$

% \item
%     $\notationidx{notation-scale-function-metastability-SDE}{\lambda^*_{b;L}(\eta)}: {\lambda^*_{b;L}(\eta)}\delequal \big(\lambda_L(\eta)\big)^{ \mathcal J^I_b(V) } \in \RV_{ \mathcal J^I_b(V)\cdot (\alpha - 1) }(\eta).$
    
\linkdest{location, notation index M}
    
\item
    \notationidx{notation-M-S-exclude-C}{$\mathbb{M}(\mathbb{S}\setminus \mathbb{C})$}:
    $\mathbb{M}(\mathbb{S}\setminus \mathbb{C})
    \delequal 
    \{
    \nu(\cdot)\text{ is a Borel measure on }\mathbb{S}\setminus \mathbb{C} :\ \nu(\mathbb{S}\setminus \mathbb{C}^r) < \infty\ \forall r > 0
    \}.$

% \item 
%     $\notationidx{notation-M-+-S-minus-C}{\mathbb M_+(\mathbb S \setminus \mathbb C)}:\ 
%     {\mathbb M_+(\mathbb S \setminus \mathbb C)}
%     \delequal 
%     \{
%          \nu(\cdot)\text{ is a Borel measure on }\mathbb{S}\setminus \mathbb{C}:\ 
%          \nu(K) < \infty\ \forall K \in \mathcal K(\mathbb S\setminus \mathbb C)
%     \}$
    
% \item
%     \notationidx{notation-M-convergence}{$\mathbb{M}(\mathbb{S}\setminus \mathbb{C})$-convergence}:
%     $\mu_n(f) \rightarrow \mu(f)$ for any $f \in \mathcal{C}({ \mathbb{S}\setminus \mathbb{C} })$

% \item
%     \notationidx{notation-uniform-M-convergence}{$\mathbb{M}(\mathbb{S}\setminus \mathbb{C})$-convergence uniformly over $\Theta$}:
%     $ \lim_{\eta \downarrow 0}\sup_{\theta \in \Theta}|\mu^\eta_\theta(f) - \mu_\theta(f)| = 0$ for any $f \in \mathcal{C}({ \mathbb{S}\setminus \mathbb{C} })$

\linkdest{location, notation index N}

\item
    \notationidx{notation-non-negative-numbers}{$\mathbb N$}:
    $\mathbb{N} = \{0,1,2,\cdots\}$

\item 
    $\notationidx{notation-R-d-unit-sphere}{\mathfrak N_d}:\ {\mathfrak N_d} \delequal \{\bm x \in \R^d:\ \norm{\bm x} = 1\}$,
    unit sphere in $\R^d$

\item 
    \notationidx{notation-measure-nu-alpha}{$\nu_\alpha$}: $\nu_\alpha[x,\infty) = x^{-\alpha}$

% \item
%     $\notationidx{notation-nu-k-alpha}{\nu_\alpha^k(\cdot)}$:
%     $k$-fold product measure of $\nu_\alpha$.

\linkdest{location, notation index O}
\linkdest{location, notation index P}

% \item 
%     \notationidx{notation-p-plus-and-minus}{$p^{(+)},\ p^{(-)}$}: 
%     $\lim_{x \rightarrow \infty}H^{(+)}(x)\big/H(x) = p^{(+)},\ \lim_{x \rightarrow \infty}H^{(-)}(x)\big/H(x) = p^{(-)}$.

\item 
    \notationidx{notation-truncation-operator-level-b}{$\varphi_c$}:
    ${\varphi_c}(\bm w) 
    \delequal{} 
    \Big(\frac{c}{\norm{\bm w}} \wedge 1\Big)\cdot \bm w$; truncation operator at level $c > 0$

\item 
    $\notationidx{notation-Phi-polar-transform}{\Phi(\bm x)}:\ {\Phi(\bm x)} \delequal \big(\norm{\bm x},\frac{\bm x}{\norm{\bm x}}\big)$ for all $\bm x \neq 0$; polar transform

\linkdest{location, notation index Q}

% \item 
%     $\notationidx{notation-q-b-i}{q_b(i)}: {q_b(i)} \delequal \widecheck{\mathbf C}^{ ( \mathcal J^I_b(i) )|b }( I_i^c;m_i)$

% \item 
%     $\notationidx{notation-q-b-i,j}{q_{b}(i,j)}: {q_{b}(i,j)} \delequal \widecheck{\mathbf C}^{ ( \mathcal J^I_b(i) )|b }( I_j;m_i)$

% % \item 
% %     $\notationidx{notation-q-i}{q(i)}: q(i) \delequal \widecheck{\mathbf C}( I_i^c;m_i)$

% \item 
%     $\notationidx{notation-q-i-j}{q(i,j)}: {q(i,j)} \delequal \widecheck{\mathbf C}( I_j;m_i)$

\linkdest{location, notation index R}

\item 
    \notationidx{notation-rho-LDP}{$\rho$}:
    $\rho \delequal \exp(D)$;  $D$ is the constant in Assumption \ref{assumption: lipschitz continuity of drift and diffusion coefficients}

% \item 
%     \notationidx{notation-rho-t-LDP}{$\rho(t)$}:
%     $\rho(t) \delequal \exp(Dt)$; $D$ is the constant in Assumption \ref{assumption: lipschitz continuity of drift and diffusion coefficients}
    
% \item 
%     \notationidx{notation-rho-j}{$\rho^{(j)}$}:
%     $\rho^{(j)} = (6\rho D)^{j+1}$

% \item 
%     \notationidx{notation-rho-*}{$\rho^*$}:
%     $\rho^* \delequal 4\cdot\big[ 6 \rho(T) \cdot (L \vee 1) \big]^{k}$
    
% \item 
%     \notationidx{notation-rho-*-b}{$\rho^*_b$}:
%     $\rho^*_b \delequal 4\cdot\big[ 6 \rho(T) \cdot (L \vee 1)(b \vee 1) \big]^{k}$

\item 
    \notationidx{notation-RV-LDP}{$\RV_\beta$}:
    $\phi \in \RV_\beta$ (as $x \rightarrow \infty$) if $\lim_{x \rightarrow \infty}\phi(tx)/\phi(x) = t^\beta$ for any $t>0$;
    $\phi \in \RV_\beta(\eta)$ (as $\eta \downarrow 0$) if $\lim_{\eta \downarrow 0}\phi(t\eta)/\phi(\eta) = t^\beta$ for any $t>0$

\item  
    \notationidx{notation-R-eta-b-epsilon-return-time}{$R^{\eta|b}_\epsilon(\bm x)$}:
    $
    {R^{\eta|b}_\epsilon(\bm x)}  \delequal \min\big\{ j \geq 0:\ \norm{\bm X^{\eta|b}_j(\bm x)} < \epsilon \big\}
    $ 

% \item 
%     $\notationidx{notation-return-time-to-I-i}{R_{i;\epsilon}^{\eta|b}(x)}: 
%     {R_{i;\epsilon}^{\eta|b}(x)} \delequal \min\{ j \geq 0:\ X^{\eta|b}_j(x) \in (m_i - \epsilon,m_i + \epsilon) \}$

\linkdest{location, notation index S}

% \item 
%     $\notationidx{notation-S-delta-boundary-set}{S(\delta)}: {S(\delta)} \delequal \bigcup_{i \in [n_\text{min} - 1]}[s_i - \delta,s_i + \delta]$

\item 
    \notationidx{sigma}{$\bm \sigma$}: diffusion coefficient $\bm \sigma: \mathbb{R}^m \to \mathbb{R}^{m\times d}$
    
\item
    \notationidx{sigma-M}{$\bm \sigma_M$}: diffusion coefficient with the argument $\bm x$ projected onto the ball $\{ \bm x \in \R^m:\ \norm{\bm x} \leq M \}$

% \item 
%      $\notationidx{notation-sigma-eta-b-i-epsilon}{\sigma^{\eta|b}_{i;\epsilon}(x)}:
%      {\sigma^{\eta|b}_{i;\epsilon}(x)} \delequal \min\{j \geq 0:\ X^{\eta|b}_j(x) \in \bigcup_{ l \neq i }(m_l - \epsilon,m_l + \epsilon)\}$
    
% \item 
%     \notationidx{notation-bar-sigma}{$\sigma(0)$}:
%     $
%     \sigma(0) \delequal \sigma(0)
%     $

\item 
    \notationidx{notation-support-of-function-g}{$\text{supp} (g)$}:
$\text{supp} (g) \delequal \big(\{ x \in \mathbb S:\ g(x) \neq 0 \}\big)^-$;
support of $g: \mathbb S \to \mathbb{R}$

\item 
    \notationidx{notation-support-of-mu}{$\text{supp}(\mu)$}: the smallest closed set $C$ such that $\mu(\S \setminus C)= 0$

\item 
    \notationidx{notation-borel-sigma-algebra}{$\mathscr{S}_\mathbb{S}$}:
    Borel $\sigma$-algebra of the metric space $(\mathbb{S},\bm{d})$

\linkdest{location, notation index T}
% \item 
%     \notationidx{notation-constant-T-LDP}{$T$}:
%     A constant $T \in (0,\infty)$ that is fixed at the beginning of Section \ref{subsec: LDP unclipped, proof of main results} and \ref{subsec: LDP clipped, proof of main results}

\item
    $
    \notationidx{notation-hitting-time-t-x-epsilon}{\bm t_{\bm x}(\epsilon)}:\ {\bm t_{\bm x}(\epsilon)} \delequal \inf\Big\{ t \geq 0:\ \bm y_t(\bm x) \in \bar B_{\epsilon}(\bm 0) \Big\}
    $

\item 
    $\notationidx{notation-t-epsilon-ode-return-time}{ \bm{t}(\epsilon) }:
        { \bm{t}(\epsilon) } \delequal \sup\Big\{ \bm t_{\bm x}(\epsilon):\ \bm x \in (I_\epsilon)^-   \Big\}$

\item 
    \notationidx{notation-large-jump-time}{$\tau^{>\delta}_i(\eta)$}:
    $\tau^{>\delta}_i(\eta) \delequal{} \min\{ n > \tau^{>\delta}_{i-1}(\eta):\ \eta\norm{\bm Z_j} > \delta  \},\ \tau^{>\delta}_0(\eta) = 0$; arrival time of $j$\textsuperscript{th} large jump

\item   
    \notationidx{notation-tau-eta-x-first-exit-time}{$\tau^\eta(\bm x)$}:
    $
    \tau^\eta(\bm x) \delequal \min\big\{j \geq 0:\ \bm X^\eta_j(\bm x) \notin I\big\}
    $

\item  
    \notationidx{notation-tau-eta-b-x-first-exit-time}{$\tau^{\eta|b}(\bm x)$}:
    $
    \tau^{\eta|b}(\bm x) \delequal \min\big\{j \geq 0:\ \bm X^{\eta|b}_j(\bm x) \notin I \big\}
    $

\item  
    \notationidx{notation-tau-eta-b-epsilon-exit-time}{$\tau^{\eta|b}_\epsilon(\bm x)$}:
    $
    {\tau^{\eta|b}_\epsilon(\bm x)} \delequal \min\big\{ j \geq 0:\ \bm X^{\eta|b}_j(\bm x) \notin I_\epsilon \big\}
    $

% \item 
%     $\notationidx{notation-transition-time-metastability}{\hat \tau^{\eta,\epsilon|b}_k(x)}: {\hat \tau^{\eta,\epsilon|b}_k(x)} \delequal \min\Big\{ j \geq \hat \tau^{\eta,\epsilon|b}_{k-1}(x):\ X^{\eta|b}_j(x) \in \bigcup_{i \neq \hat{\mathcal I}^{\eta,\epsilon|b}_{k-1}(x)}(m_i - \epsilon,m_i+\epsilon) \Big\}$,
%     the $k$-th transition time to a different $\epsilon$-neighborhood of local minima $m_i$

% \item
%     \notationidx{notation-tilde-tau-L-i}{$\widetilde{\tau}^L_i$}:
%     $
%     \widetilde{\tau}^L_i \delequal \min\{ t > \widetilde{\tau}^L_{i-1}:\ |\Delta L_t| > 1 \}
%     $

% \item 
%     \notationidx{notation-tau-eta-leq-k-L-i}{$ \tau^{\eta;(k)}_L(i)$}:
%     $
%     \tau^{\eta;(k)}_L(i) \delequal \widetilde{\tau}^{L}_{J^{\eta;(k)}_L(i)}
%     $

\linkdest{location, notation index U}

\item 
    \notationidx{notation-U-j-t}{$U_j$}:
    iid copies of Unif$(0,1)$
    
\item  
    \notationidx{notation-U-j-k-LDP}{$U_{(j;k)}$}:
    $0 \leq U_{(1;k)} \leq U_{(2;k)} \leq \cdots \leq U_{(k;k)}$;
    the order statistics of iid $\big(U_j\big)_{j = 1}^k$

\linkdest{location, notation index V}

% \item 
%     $\notationidx{notation-V-of-G-b}{V}: V = \{m_1,\cdots,m_{ n_\text{min} }\}$

% \item
%     $\notationidx{notation-V-*-b}{V^*_b}: V^*_b \delequal \{m_i:\ \mathcal J^I_b(i) = \mathcal{J}^I_b(V)\}$

\linkdest{location, notation index W}

% \item 
%     \notationidx{notation-w-plus-limit-for-LDP-of-stationary-dist}{$w^{(+)}(t;\gamma)$},\notationidx{notation-w-minus-limit-for-LDP-of-stationary-dist}{$\ w^{(-)}(t;\gamma)$}:
%     $w^{(+)}(t;\gamma) = \inf\big\{ y > 0:\ \bm{y}_t(y) \geq \gamma \big\};\ \ \ 
%     w^{(-)}(t;\gamma) = \inf\big\{ y > 0:\ \bm{y}_t(-y) \leq -\gamma \big\}.$
   
\item 
    \notationidx{notation-large-jump-size}{$\bm W^{>\delta}_i(\eta)$}: 
    $\bm W^{>\delta}_i(\eta) \delequal{} \bm Z_{\tau^{>\delta}_i(\eta)}$; size of $j$\textsuperscript{th} large jump, i.e., with size above threshold $\delta/\eta$
   
\item 
    \notationidx{notation-W-*_j}{$\bm W^*_j(\cdot)$}:
    iid copies of $\bm W^*(c)$ defined in \eqref{def: prob measure Q, LDP}
    
% \item
%     \notationidx{W-leq-j-i-n}{$W^{(j)}_i(\eta)$}:
%     $W^{(j)}_i(\eta) \delequal Z_{ \tau^{(j)}_i(\eta) }$

\linkdest{location, notation index X}

\item 
    \notationidx{notation-discrete-gradient-descent}{$\bm{x}^\eta_{j}(x)$}: (deterministic) difference equation
    $\bm{x}^\eta_{j}(x) = \bm{x}^\eta_{j-1}(x) + \eta \bm a\big(\bm{x}^\eta_{j-1}(x) \big)$ for any $j \geq 1$ with initial condition $\bm{x}^\eta_{0}(x) = x$.

\item 
    \notationidx{notation-breve-X-eta-delta-t}{$\breve{\bm X}^{\eta,\delta}_t(\bm x)$}: ODE that coincides with $\bm X^{\eta}_{\floor{t/\eta} }(\bm x)$ at times $t = \eta \tau^{>\delta}_i(\eta)$, $i=1,2,\ldots$.
    
\item  
    \notationidx{notation-breve-X-eta-b-delta-t}{$\breve{\bm X}^{\eta|b;\delta}_t(\bm x)$}:
    ODE that coincides with $\bm X^{\eta|b}_{\floor{t/\eta} }(\bm x)$ at times $t = \eta \tau^{>\delta}_i(\eta)$, $i=1,2,\ldots$.

\item
    \notationidx{notation-hat-X-clip-b-top-j-jumps}{$\hat{\bm{X}}^{\eta|b; >\delta }(\bm x)$}:
    ODE perturbed by $\bm W^{>\delta}_i(\eta)$'s, with sizes modulated by $\bm \sigma(\cdot)$ and truncated under $b$.
    
% \item 
%     $\notationidx{notation-marker-process-metastability-hat-X}{\hat X^{\eta,\epsilon|b}_t(x)}$:
%      ${\hat X^{\eta,\epsilon|b}_t(x)}$ is the $\Big( \Big( \big({\hat \tau^{\eta,\epsilon|b}_k(x) - \hat \tau^{\eta,\epsilon|b}_{k-1}(x)}\big) \cdot { \lambda^*_b(\eta) } \Big)_{k \geq 1},\big( m_{ \hat{\mathcal I}^{\eta,\epsilon|b}_k(x) } \big)_{k \geq 1} \Big)$ jump process.
    
\item 
    \notationidx{notation-X-j-eta-x}{$\bm X^\eta_j(x)$}: $\bm X^\eta_0(\bm x) = \bm x;\ 
    \bm X^\eta_j(\bm x) = \bm X^\eta_{j - 1}(\bm x) +  \eta \bm a\big(\bm X^\eta_{j - 1}(\bm x)\big) + \eta\bm \sigma\big(\bm X^\eta_{j - 1}(\bm x)\big)\bm Z_j,\ \ \forall j \geq 1$

% \item 
%     \notationidx{notation-X-eta-x-cont-time-version}{$X^\eta(s;x)$}:
%     $X^\eta(s;x) = X^\eta_{\floor{s}}(x)$

\item 
    \notationidx{notation-scaled-X-0T-eta-LDP}{$\bm{X}^\eta_{[0,T]}(x)$}:
    $\bm{X}_{[0,T]}^\eta(x) \delequal \big\{ \bm X^\eta_{ \floor{ t/\eta } }(x):\ t \in [0,T] \big\}$

\item 
    \notationidx{notation-scaled-X-eta-LDP}{$\bm{X}^\eta(x)$}: $\bm{X}^\eta(x) = \bm{X}_{[0,1]}^\eta(x) \delequal \big\{ \bm X^\eta_{ \floor{ t/\eta } }(x):\ t \in [0,1] \big\}$

% \item 
%     \notationidx{notation-scaled-X-eta-mu-LDP}{$\bm{X}^\eta(x)$}:
%     $\bm{X}^\eta(x) = \bm{X}^\eta(\bm{\delta}_x)$ for any $x \in \mathbb{R}$   

\item 
    \notationidx{notation-X-eta-j-truncation-b-LDP}{$\bm X^{\eta|b}_j(\bm x)$}:
    $\bm X^{\eta|b}_j(\bm x)= \bm X^{\eta|b}_{j - 1}(\bm x)+  \varphi_b\Big(\eta \big[\bm a\big(\bm X^{\eta|b}_{j - 1}(\bm x)\big) + \bm \sigma\big(\bm X^{\eta|b}_{j - 1}(\bm x)\big)\bm Z_j\big]\Big)\ \ \forall j \geq 1$

% \item 
%     \notationidx{notation-X-j-eta-b-x}{$X^{\eta|b}_j(x)$}:
%     $X^{\eta|b}_0(x) = x;\ X^{\eta|b}_j(x) = X^{\eta|b}_{j-1}(x) + \varphi_b\Big(\eta a\big(X^{\eta|b}_{j-1}(x)\big) + \eta \sigma\big(X^{\eta|b}_{j-1}(x)\big)Z_j\Big)$

% \item 
%     \notationidx{notation-X-eta-j-mu-truncation-b-LDP}{$X^{\eta|b}_j(\mu)$}:
%     $X^{\eta|b}_0(\mu) \stackrel{d}{\sim}\mu;\ \ X^{\eta|b}_j(\mu) = X^{\eta|b}_{j - 1}(\mu) +  \varphi_b\Big(\eta \big[a\big(X^{\eta|b}_{j - 1}(\mu)\big) + \sigma\big(X^{\eta|b}_{j - 1}(\mu)\big)Z_j\big]\Big)\ \ \forall j \geq 1.$
    
% \item 
%     \notationidx{notation-X-eta-b-M-LDP}{$X^{\eta|b}_{M\downarrow}(j)$}:
%     $ X^{\eta|b}_{M\downarrow}(j) = X^{\eta|b}_{M\downarrow}(j - 1) + \varphi_b\Big( \eta a_M\big( X^{\eta|b}_{M\downarrow}(j- 1) \big) + \eta \sigma_M\big( X^{\eta|b}_{M\downarrow}(j- 1) \big)Z_j \Big)\ \ \forall j \geq 1$;
%     The main difference between $X^{\eta|b}_{M\downarrow}$ and $X^{\eta|b}$ is that $X^{\eta|b}_{M\downarrow}$ is constructed under the truncated $a_M(\cdot),\sigma_M(\cdot)$

\item 
    \notationidx{notation-scaled-X-eta-mu-truncation-b-LDP}{${\bm{X}}^{\eta|b}_{[0,T]}(\bm x)$}:
    $ {\bm{X}}^{\eta|b}_{[0,T]}(\bm x) \delequal \big\{ \bm X^{\eta|b}_{ \floor{ t/\eta } }(\bm x):\ t \in [0,T] \big\}$

\item 
    \notationidx{notation-scaled-X-eta-mu-truncation-b-LDP-T=1}{${\bm{X}}^{\eta|b}(\bm x)$}:
    ${\bm{X}}^{\eta|b}(\bm x) = {\bm{X}}^{\eta|b}_{[0,1]}(\bm x) \delequal \big\{ \bm X^{\eta|b}_{ \floor{ t/\eta } }(\bm x):\ t \in [0,1] \big\}$

% \item 
%     \notationidx{notation-scaled-X-eta-mu-truncation-b-LDP}{${\bm{X}}^{\eta|b}(x)$}:
%     ${\bm{X}}^{\eta|b}(x) \delequal \big\{ X^{\eta|b}_{ \floor{ s/\eta } }(x):\ s \geq 0 \big\}$ for any $x \in \mathbb{R}$

% \item
%     \notationidx{notation-scaled-X-eta-b-M}{$\bm{X}^{\eta|b}_{M\downarrow}(x)$}:
% $\bm{X}^{\eta|b}_{M\downarrow}(x) \delequal \big\{ \bm{X}^{\eta|b}_{M\downarrow}(s;x) = X^{\eta|b}_{M\downarrow}\big(\floor{s/\eta};x\big):\ s \in [0,t] \big\}$

\linkdest{location, notation index Y}

\item 
    \notationidx{notation-continuous-gradient-descent}{$\bm{y}_t(\bm x)$}:  ODE path
    $\frac{d\bm{y}_t(\bm x)}{dt} = \bm a\big(\bm{y}_t(\bm x)\big)$ for any $t > 0$
    with initial condition $\bm{y}_0(\bm x) = \bm x$.

\item
    \notationidx{notation-Y-eta-SDE}{$\bm Y^\eta_t(\bm x)$}:
    $d\bm Y^\eta_t(\bm x) 
    =
    \bm a\big(\bm Y^\eta_{t-}(\bm x)\big)dt
        +
    \bm \sigma\big(\bm Y^\eta_{t-}(\bm x)\big) d\bar{\bm L}^\eta_t$
    
\item 
    \notationidx{notation-bm-Y-0T-eta-SDE}{$\bm Y^\eta_{[0,T]}(\bm x)$}:
    $
    {\bm Y^\eta_{[0,T]}(\bm x)} = \{\bm Y^\eta_t(\bm x):\ t\in[0,T]\}
    $

\item 
    $\notationidx{notation-bm-Y-eta-SDE}{\bm Y^\eta(\bm x)} = \{\bm Y^\eta_t(\bm x):\ t\in[0,1]\}$

\item
    \notationidx{notation-Y-eta-b-SDE}{$\bm Y^{\eta|b}_t(\bm x)$}:
    A modified version of the SDE $\bm Y^\eta_t(\bm x)$ with each discontinuity truncated under $b$
    
\item
    \notationidx{notation-bm-Y-0T-eta-b-SDE}{$\bm{Y}^{\eta|b}_{[0,T]}(\bm x)$}:
    $
    {\bm{Y}^{\eta|b}_{[0,T]}(\bm x)} \delequal \big\{\bm Y^{\eta|b}_t(\bm x):\ t \in [0,T]\big\}
    $

\item 
    $\notationidx{notation-bm-Y-eta-b-SDE}{\bm Y^{\eta|b}(\bm x)}= \{\bm Y^{\eta|b}_t(\bm x):\ t\in[0,1]\}$

\linkdest{location, notation index Z}

\item 
    \notationidx{notation-Z-iid-noise-LDP}{$\bm Z_j$}:
    $(\bm Z_j)_{j \geq 1}$ is a sequence of iid copies of a random vector $\bm Z$ such that
    $\E \bm Z = \bm 0$ and 
    the multivariate regular variation assumption (i.e., Assumption~\ref{assumption gradient noise heavy-tailed}) holds for the law of $\bm Z$.

%   \item \hyperlink{{notation-Y-GD-perturbed-by-large-jump-LDP}}{$Y^{\eta,\delta}_j(x)$}:
%   $ Y^{\eta,\delta}_j(x) = Y^{\eta,\delta}_{j-1}(x) +  \eta a\big(Y^{\eta,\delta}_{j-1}(x)\big) + \eta\sigma\big(Y^{\eta,\delta}_{j-1}(x)\big)\sum_{k \geq 1}W^{>\delta}_{k}(\eta)\mathbbm{I}\{j = \tau^{>\delta}_k(\eta)\}$

\end{itemize}
\fi

\ifshowtheoremtree
\newpage
\footnotesize
\newgeometry{left=1cm,right=1cm,top=0.5cm,bottom=1.5cm}

\section*{\linkdest{location of theorem tree}Theorem Tree}
\begin{thmdependence}[leftmargin=*]

\thmtreenode{-}
    {Theorem}{theorem: portmanteau, uniform M convergence}
    {0.8}{Portmanteau Theorem for uniform $\M(\S\setminus\C)$-convergence}
\bigskip

\thmtreenodewopf{}
    {Lemma}{lemma: asymptotic equivalence when bounded away, equivalence of M convergence}
    {0.8}{
        asympt.\ equiv.\ of $X_n$ and $Y_n$ w.r.t.\\ $\epsilon_n^{-1}$ in $\M(\mathbb S\setminus\mathbb{C})$ implies the same $\M$-convergence of $\epsilon_n^{-1}\P(X_n\in \cdot)$ and $\epsilon_n^{-1}\P(Y_n\in \cdot)$.%
    }
    % \begin{thmdependence}
    % \thmtreenodewopf{}
    %     {Theorem}{thm: portmanteau, M convergence}
    %     {0.8}{Portmanteau Theorem for $\M(\S\setminus\C)$-convergence}
    % \end{thmdependence}
\bigskip

% \thmtreenode{\complete}
%     {Theorem}{theorem: sample path LDP, unclipped}
%     {0.8}{
%         [A\ref{assumption gradient noise heavy-tailed},\ref{assumption: lipschitz continuity of drift and diffusion coefficients},\ref{assumption: boundedness of drift and diffusion coefficients}]
%         Sample Path Large Deviations for SGD $\bm X^\eta$
%     }
%     \begin{thmdependence}
%     \thmtreeref
%         {Theorem}{thm: portmanteau, M convergence}
        
%     \thmtreeref
%         {Theorem}{theorem: LDP 1, unclipped}
%     \end{thmdependence}
% \bigskip

\thmtreenode{\complete}
    {Theorem}{theorem: LDP 1, unclipped}
    {0.8}{
    [A\ref{assumption gradient noise heavy-tailed},\ref{assumption: lipschitz continuity of drift and diffusion coefficients},\ref{assumption: nondegeneracy of diffusion coefficients},\ref{assumption: boundedness of drift and diffusion coefficients}]
    Sample Path Large Deviations for SGD $\bm X^\eta$.
    $\P ({\bm X}^{\eta}(x)\in \cdot)/\lambda^k(\eta) \to \mathbf{C}^{(k)}(\cdot;x)$ in $\mathbb{M}(\D \setminus \D_A^{(k-1)})$ uniformly in $x$ on any compact set $A$ 
    }
    \begin{thmdependence}
    \thmtreenode{-}
            {Lemma}
            {lemma: continuity of h k b mapping}
            {0.8}{
            [A\ref{assumption: lipschitz continuity of drift and diffusion coefficients},\ref{assumption: boundedness of drift and diffusion coefficients}]
            $h^{(k)}_{[0,T]}$ is continuous on $\R\times \R^k \times (0,T)^{k\uparrow}$.
            }
        \begin{thmdependence}
        \thmtreenode{\issue}
            {Lemma}
            {lemma: continuity of h k b mapping clipped}
            {0.8}{
                [A\ref{assumption: lipschitz continuity of drift and diffusion coefficients},\ref{assumption: nondegeneracy of diffusion coefficients}] 
                $h^{(k)|b}$ is continuous.
            }
            \begin{thmdependence}
                 \thmtreenode{-}
            {Corollary}{corollary: existence of M 0 bar delta bar epsilon, clipped case, LDP}
            {0.8}{
                [A\ref{assumption: lipschitz continuity of drift and diffusion coefficients},\ref{assumption: nondegeneracy of diffusion coefficients}] 
                If $d(B, \D_{A|b}^{(k-1)}) >0$, the shocks of the paths in $B \cap \D_{A}^{(k)|b}$ are bounded away from 0. 
            }
        
            \begin{thmdependence}
            \thmtreenode{-}
                {Lemma}{lemma: boundedness of k jump set under truncation, LDP clipped}
                {0.8}{
                    $\sup_{\xi\in \D_{A}^{(k)|b}} \|\xi\| < \infty$ for any $b\in(0,\infty)$, $k\in \mathbb N$, and compact set $A$.
                }
            \end{thmdependence}
            \end{thmdependence}
        \end{thmdependence}

    \thmtreenode{-}
            {Lemma}{lemma: LDP, bar epsilon and delta}
            {0.8}{
                [A\ref{assumption: lipschitz continuity of drift and diffusion coefficients},\ref{assumption: boundedness of drift and diffusion coefficients}] 
                If $d(B, \D_{A}^{(k-1)}) >0$, the shocks of the paths in $B \cap \D_{A}^{(k)}$ are bounded away from 0. 
            }

    \thmtreenode{-}
        {Lemma}{lemma: sequential compactness for limiting measures, LD of SGD}{0.8}
        {
        Verify $\lim_{n \to \infty}\mathbf C^{(k)}(f;x_{n})
            =
            \mathbf C^{(k)}(f;x^*)$
            and
            $\lim_{n \to \infty}\mathbf C^{(k)|b}(f;x_{n})
            =
            \mathbf C^{(k)|b}(f;x^*)$
        }
        \begin{thmdependence}
            \thmtreeref
                {Lemma}{lemma: continuity of h k b mapping}
            \thmtreeref
                {Lemma}{lemma: continuity of h k b mapping clipped}
            \thmtreeref
                {Lemma}{lemma: LDP, bar epsilon and delta}
            \thmtreenode{-}
                {Lemma}
                {lemma: LDP, bar epsilon and delta, clipped version}
                {0.8}{
                [A\ref{assumption: lipschitz continuity of drift and diffusion coefficients}, A\ref{assumption: boundedness of drift and diffusion coefficients}]
                If $d(B, \D_{A|b}^{(k-1)}) >0$, the shocks of the paths in $B \cap \D_{A}^{(k)|b}$ are bounded away from 0.
                }
        \end{thmdependence}

    \thmtreeref
        {Theorem}{theorem: portmanteau, uniform M convergence}

    \thmtreeref
        {Proposition}{proposition: standard M convergence, LDP unclipped}

    \end{thmdependence}

\bigskip

\thmtreenode{\complete}
    {Theorem}{corollary: LDP 2}
    {0.8}{
    [A\ref{assumption gradient noise heavy-tailed},\ref{assumption: lipschitz continuity of drift and diffusion coefficients},\ref{assumption: nondegeneracy of diffusion coefficients}]
    Sample path large deviations for $\bm X^{\eta|b}$.
    $\P ({\bm X}^{\eta|b}(x)\in \cdot)/\lambda^k(\eta) \to \mathbf{C}_b^{(k)}(\cdot;x)$ in $\mathbb{M}(\D \setminus \D_{A|b}^{(k-1)})$ uniformly in $x$ on any compact set $A$ 
    }
    
    \begin{thmdependence}
        \thmtreeref
            {Proposition}{proposition: standard M convergence, LDP clipped}
        \thmtreeref
            {Lemma}{lemma: sequential compactness for limiting measures, LD of SGD}
        \thmtreeref
            {Lemma}{lemma: LDP, bar epsilon and delta, clipped version}
        \thmtreeref
            {Theorem}{theorem: portmanteau, uniform M convergence}
    \end{thmdependence}

\bigskip

\thmtreenode{-}
    {Proposition}{proposition: standard M convergence, LDP unclipped}
    {0.8}{
    [A\ref{assumption gradient noise heavy-tailed},\ref{assumption: lipschitz continuity of drift and diffusion coefficients},\ref{assumption: boundedness of drift and diffusion coefficients}]
        $\P({\bm X}^{\eta_n}(x_n)\in \cdot)/\lambda^k(\eta_n) \to \mathbf{C}^{(k)}(\cdot; x^*)$ in $\M(\D\setminus \D_{A}^{(k-1)})$ if $\eta_n\to 0$, $x_n \to x^*$, and $x_n, x^* \in A$: cpt
    }
    \begin{thmdependence}
    
    \thmtreenode{-}
    {Proposition}{proposition: standard M convergence, LDP clipped}
    {0.8}{ [A\ref{assumption gradient noise heavy-tailed},\ref{assumption: lipschitz continuity of drift and diffusion coefficients},\ref{assumption: nondegeneracy of diffusion coefficients}]
        $\P({\bm X}^{\eta_n}(x_n)\in \cdot)/\lambda^k(\eta_n) \to \mathbf{C}^{(k)}(\cdot; x^*)$ in $\M(\D\setminus \D_{A}^{(k-1)})$ if $\eta_n\to 0$, $x_n \to x^*$, and $x_n, x^* \in A$: cpt
    }
    
    \begin{thmdependence}
    \thmtreeref{Proposition}{proposition: standard M convergence, LDP clipped, stronger boundedness assumption}
    
    \thmtreeref
            {Corollary}{corollary: existence of M 0 bar delta bar epsilon, clipped case, LDP}
            % {0.8}{
            %     [A\ref{assumption: lipschitz continuity of drift and diffusion coefficients},\ref{assumption: nondegeneracy of diffusion coefficients}] 
            %     If $d(B, \D_{A|b}^{(k-1)}) >0$, the shocks of the paths in $B \cap \D_{A}^{(k)|b}$ are bounded away from 0. 
            % }
        
            % \begin{thmdependence}
            % \thmtreenode{-}
            %     {Lemma}{lemma: boundedness of k jump set under truncation, LDP clipped}
            %     {0.8}{
            %         $\sup_{\xi\in \D_{A}^{(k)|b}} \|\xi\| < \infty$ for any $b\in(0,\infty)$, $k\in \mathbb N$, and compact set $A$.
            %     }
            % \end{thmdependence}
    \end{thmdependence}

    % \thmtreeref
    %     {Theorem}{thm: portmanteau, M convergence}

    \thmtreeref{Lemma}{lemma LDP, small jump perturbation}
    
    \thmtreeref
            {Lemma}{lemma: LDP, bar epsilon and delta}
            % {0.8}{
            %     [A\ref{assumption: lipschitz continuity of drift and diffusion coefficients},\ref{assumption: boundedness of drift and diffusion coefficients}] 
            %     If $d(B, \D_{A}^{(k-1)}) >0$, the shocks of the paths in $B \cap \D_{A}^{(k)}$ are bounded away from 0. 
            % }

    \end{thmdependence}

\bigskip

\thmtreenode{-}
    {Proposition}{proposition: standard M convergence, LDP clipped, stronger boundedness assumption}
    {0.8}{ [A\ref{assumption gradient noise heavy-tailed},\ref{assumption: lipschitz continuity of drift and diffusion coefficients},\ref{assumption: boundedness of drift and diffusion coefficients}]
        $\P({\bm X}^{\eta_n|b}(x_n)\in \cdot)/\lambda^k(\eta_n) \to \mathbf{C}^{(j)}_b(\cdot; x^*)$ in $\M(\D\setminus \D_{A|b}^{(k-1)})$ 
        if $\eta_n\to 0$, $x_n \to x^*$, and $x_n, x^* \in A$: cpt
    }
    \begin{thmdependence}
        \thmtreeref
            {Lemma}{lemma: asymptotic equivalence when bounded away, equivalence of M convergence}
        
        \thmtreenode{-}
            {Proposition}{proposition: asymptotic equivalence, clipped}
            {0.8}{
                $\bm{X}^{\eta_n|b}(x_n)$ and $\hat{\bm X}^{\eta_n|b;(k)}(x_n)$ are asympt.\ equiv.\ w.r.t.\\ $\lambda^k(\eta)$ in $\M(\D\setminus \D_{A|b}^{(k-1)})$
            }
    
    \begin{thmdependence}
    
        \thmtreeref
                {Lemma}
                {lemma: LDP, bar epsilon and delta, clipped version}

        \thmtreenode{-}
            {Lemma}{lemma: SGD close to approximation x circ, LDP}
            {0.8}
            {
                (time-scaled) SGD $X^{\eta|b}_{\lfloor t/\eta \rfloor}$ is close to (slow) ODE until first large jump
            }
            \begin{thmdependence}
            \thmtreenode{-} 
                {Lemma}{lemmaBasicGronwall}
                {0.8}{
                SGD and GD are close to each other if the noises are small
                }

            \thmtreenode{-}
                {Lemma}{lemma Ode Gd Gap} 
                {0.8}{
                GD ($\bm y^\eta_{\floor{s}}(y)$) and GF ($\bm x^\eta(\cdot;x)$) are close to each other
                }

            \end{thmdependence}
        
        \thmtreenode{-}
            {Lemma}{lemma LDP, small jump perturbation} 
            {0.8}{
                \begin{minipage}[t]{\linewidth}
                (a) 
                $
                \sup_{ (W_i)_{i \geq 0} \in  \bm{\Gamma}_M}\P\Big( \max_{ j \leq \floor{t/\eta} \wedge \big(\tau^{>\delta}_{1}(\eta) - 1\big) }\ \eta\big|\sum_{i = 1}^j W_{i-1}Z_i \big| > \epsilon \Big)   
                = \bm o({\eta^N}) 
                $\\
                
                (b) 
                $
                \sup_{x \in \R} \P\Big( \big(\bigcap_{i = 1}^k A_i(\eta,\epsilon,\delta,t,x)\big)^c \Big) 
                = \bm o(\eta^N)
                $
                \end{minipage}
            }
        
        \thmtreenode{-}
            {Lemma}{lemma: SGD close to approximation x breve, LDP clipped}
            {0.8}{
                $\hat{X}^{\eta|b;>\delta}_{t}(x)$ and $X_{\lfloor t\rfloor}^{\eta|b}(x)$ are close to each other on $\cap_{i=1}^{k+1} A_i(\eta,\epsilon,\delta, x)$
            }
    \end{thmdependence}
            
    \thmtreenode{-}
        {Proposition}{proposition: uniform weak convergence, clipped}
        {0.8}{ [A\ref{assumption: boundedness of drift and diffusion coefficients}]
        % $\mathbb{M}$-convergence from $\hat{\bm X}^{\eta|b;(j)}$ to $\mathbf{C}^{(j)}_b$
        $\P(\hat{\bm X}^{\eta|b;(k)}\in \cdot)/\lambda^k(\eta) \to \mathbf{C}^{(k)|b}   $ in $\M(\D\setminus \D_{A|b}^{(k-1)})$
        }
    
        \begin{thmdependence}
        \thmtreeref{Lemma}{lemma: LDP, bar epsilon and delta, clipped version}
        \thmtreeref
            {Lemma}
            {lemma: continuity of h k b mapping clipped}
        \thmtreenode{-}
            {Lemma}
            {lemma: weak convergence, expectation wrt approximation, LDP, preparation}
            {0.8}{
            On \hyperlink{index, notation-set-E-delta-LDP}{$E^\delta_{c,k}(\eta)$}, continuous and bounded functionals of scaled jump times and jump sizes converge to that of uniform and Pareto distributions.
            }
            \begin{thmdependence}
            \thmtreenode{-}
                {Lemma}{lemma: weak convergence of cond law of large jump, LDP}
                {0.8}{
                Conditional on \hyperlink{index, notation-set-E-delta-LDP}{$E^\delta_{c,k}(\eta)$}, the scaled jump times and jump sizes converge to uniform and Pareto distributions. 
                }
                % \vspace{-5pt}
            \end{thmdependence}
        
        \end{thmdependence}
    \end{thmdependence}

\bigskip

\bigskip

\thmtreenode{-}
    {Theorem}{theorem: first exit time, unclipped}{0.8}
    {
    (a)
    $
        \lim_{\eta \downarrow 0}\P\Big(C^I_b \eta\cdot \lambda^{ \mathcal{J}^I_b }(\eta)\tau^{\eta|b}(x) > t
        % ;\ X^{\eta|b}_{ \tau^{\eta|b}(x)}(x) \in B
        \Big)
        =
        % \big[ 
        \exp(-t)
        % \big]\cdot\frac{ \widecheck{\mathbf{C}}^{ (\mathcal{J}^I_b)|b }(B) }{ C^I_b }.
        $\\
        % where 
        % $\mathcal{J}^I_b \delequal \ceil{|s_\text{left}|\wedge s_\text{right}/b}$,
        % $C^I_b \delequal  \widecheck{ \mathbf{C} }^{ (\mathcal{J}^I_b)|b }\big( (-\infty,s_\text{left}]\cup[s_\text{right},\infty)\big)$,
        % and $\lambda(\eta)\delequal \eta^{-1}H(\eta^{-1})$.
    
    (b)
    $
    \lim_{\eta \downarrow 0}\P\Big(C^* \eta \lambda(\eta)\tau^\eta(x) > t
    % ;\ X^\eta_{ \tau^\eta(x)}(x) \in B
    \Big)
    =
    % \big[
    \exp(-t)
    % \big
    % ]\cdot\frac{ \widecheck{\mathbf{C}}(B) }{ C^* }.
    $
    % where 
    % $C^* \delequal  \widecheck{ \mathbf{C} }\big( (-\infty,s_\text{left}]\cup[s_\text{right},\infty)\big)$
    }
    \begin{thmdependence}
        \thmtreenode{-}
            {Lemma}{lemmaGeomFront}{0.8}{} 
        \thmtreenode{-}
            {Theorem}{thm: exit time analysis framework} {0.8}
            {
            First Exit Time Analysis Framework: $\P\big(
        \gamma(\eta)\tau_{I^\complement}^{\eta}(x)>t,\,V_{\tau}^\eta(x)\in B  
    \big) \approx C(B)e^{-t}$
            }
        \begin{thmdependence}
            \thmtreenode{-}
                {Proposition}{prop: exit time analysis main proposition}{0.8}{
                First Exit Time Analysis Framework ($\epsilon$-relaxed version):
                $
                \P\big(\gamma(\eta) \tau_{I(\epsilon)^\complement}^\eta(x) > t;\; V^\eta_{\tau_\epsilon}(x) \in B\big)
                \approx C(B)e^{-t} + \delta_{t,B}(\epsilon)
                $
                }
        \end{thmdependence}
        \thmtreenodewopf{}
            {Lemma}{lemma: exit prob one cycle, with exit location B, first exit analysis}{0.8}{Apply general framework: Verifying conditions \eqref{eq: exit time condition lower bound} and \eqref{eq: exit time condition upper bound}}
        \thmtreenodewopf{}
            {Lemma}{lemma: fixed cycle exit or return taking too long, first exit analysis}{0.8}{Apply general framework: Verifying condition \eqref{eq:E3}}
        \thmtreenodewopf{}
            {Lemma}{lemma: cycle, efficient return}{0.8}{Apply general framework: Verifying condition \eqref{eq:E4}}
    \end{thmdependence}

\bigskip

\thmtreenode{-}
    {Lemma}{lemma: exit prob one cycle, with exit location B, first exit analysis}{0.8}{Apply general framework: Verifying conditions \eqref{eq: exit time condition lower bound} and \eqref{eq: exit time condition upper bound}}
    
    \begin{thmdependence}
        \thmtreeref
            {Theorem}{corollary: LDP 2}
        \thmtreenode{-}
            {Lemma}{lemma: choose key parameters, first exit time analysis}{0.8}{}
        \thmtreenode{-}
            {Lemma}{lemma: measure check C J * b, continuity, first exit analysis}{0.8}{
            Verifying  $C(\partial I) = 0$ for the measure $C$ defined in \eqref{def: measure C and scale gamma when applying the exit time framework}
            }
            \begin{thmdependence}
                \thmtreeref{Lemma}{lemma: choose key parameters, first exit time analysis}
            \end{thmdependence}

        \thmtreenode{-}
            {Lemma}{lemma: exit rate strictly positive, first exit analysis}{0.8}{
            Verifying $
                    C^I_b = \widecheck{\mathbf C}^{(\mathcal{J}^I_b)|b}\big( I^\complement\big)\in(0,\infty)
                        $
            for the measure $C$ defined in \eqref{def: measure C and scale gamma when applying the exit time framework}
            }
            \begin{thmdependence}
                \thmtreeref{Lemma}{lemma: choose key parameters, first exit time analysis}
                \thmtreeref{Lemma}{lemma: continuity of h k b mapping clipped}
            \end{thmdependence}
        
        \thmtreenode{-}
            {Lemma}{lemma: limiting measure, with exit location B, first exit analysis}{0.8}{}
            \begin{thmdependence}
                \thmtreeref{Lemma}{lemma: choose key parameters, first exit time analysis}
            \end{thmdependence}
        
    \end{thmdependence}

\thmtreenode{-}
    {Lemma}{lemma: fixed cycle exit or return taking too long, first exit analysis}{0.8}{Apply general framework: Verifying condition \eqref{eq:E3}}
    \begin{thmdependence}
        \thmtreeref
            {Theorem}{corollary: LDP 2}
    \end{thmdependence}

\thmtreenode{-}
    {Lemma}{lemma: cycle, efficient return}{0.8}{Apply general framework: Verifying condition \eqref{eq:E4}}
    \begin{thmdependence}
        \thmtreeref
            {Theorem}{corollary: LDP 2}
    \end{thmdependence}

\bigskip

\bigskip

\end{thmdependence}
\fi

\ifshowtheoremlist
\newpage
\footnotesize
\newgeometry{left=1cm,right=1cm,top=0.5cm,bottom=1.5cm}
\linkdest{location of theorem list}
%\section*{List of Theorems}
\listoftheorems
\fi

\ifshowequationlist
\newpage
\linkdest{location of equation number list}
\section*{Numbered Equations}
%\label{def: set of ODE with k jumps}

\eqref{def: X eta b j x, unclipped SGD}
\eqref{def: H, law of Z_j}
% \eqref{assumption gradient noise heavy-tailed}
% \eqref{assumption: lipschitz continuity of drift and diffusion coefficients}
% \eqref{assumption: boundedness of drift and diffusion coefficients}
% \eqref{def: perturb ode mapping h k, 1}
% \eqref{def: perturb ode mapping h k, 2}
% \eqref{def: perturb ode mapping h k, 3}
\eqref{def: measure nu alpha}
% \eqref{def: measure C k t}
\eqref{def: scaled SGD, LDP}
% \eqref{def: set of ODE with k jumps}
% \eqref{theorem: LDP 1, unclipped}
% \eqref{theorem: sample path LDP, unclipped}
% \eqref{corollary, LDP 1 till time T, unclipped}
\eqref{def: X eta b j x, clipped SGD}
% \eqref{defTruncationClippingOperator}
\eqref{def: perturb ode mapping h k b, 1}
\eqref{def: perturb ode mapping h k b, 2}
\eqref{def: perturb ode mapping h k b, 3}
\eqref{def: l * tilde jump number for function g, clipped SGD}
\eqref{def: measure mu k b t}
% \eqref{assumption: nondegeneracy of diffusion coefficients}
% \eqref{theorem: LDP 1}
% \eqref{corollary: LDP 2}
% \eqref{theorem: sample path LDP, clipped}
% \eqref{corollary, LDP 1 till time T}
% \eqref{proof, def recursion tilde X, corollary, LDP 1 till time T}
% \eqref{proof, goal 1, corollary, LDP 1 till time T}
% \eqref{proof, observation 1, corollary, LDP 1 till time T}
% \eqref{lemmaGeomFront}
% \eqref{lemmaBasicGronwall}
\eqref{def: gradient descent process y}
% \eqref{lemma Ode Gd Gap}
\eqref{defArrivalTime large jump}
\eqref{defSize large jump}
\eqref{property: large jump time probability}
% \eqref{property: large jump time probability, exact jump time at the endpoint}
\eqref{def: Gamma M, set of bounded adapted process}
% \eqref{lemma LDP, small jump perturbation}
\eqref{def: event A i concentration of small jumps, 1}
\eqref{def: event A i concentration of small jumps, 2}
\eqref{proof: LDP, small jump perturbation, asymptotics part 1}
\eqref{proof: LDP, small jump perturbation, asymptotics part 2}
\eqref{proof: LDP, small jump perturbation, asymptotics part 3}
\eqref{term E Z 1, lemma LDP, small jump perturbation}
% \eqref{proof: LDP, small jump perturbation, ineq 0}
\eqref{proof: LDP, small jump perturbation, choose p}
\eqref{proof: LDP, small jump perturbation, ineq 1}
\eqref{proof: applying berstein ineq, lemma LDP, small jump perturbation}
\eqref{term second order moment hat Z 1, lemma LDP, small jump perturbation}
\eqref{def: E eta delta set, LDP}
\eqref{def: prob measure Q, LDP}
% \eqref{lemma: weak convergence of cond law of large jump, LDP}
\eqref{proof: lemma weak convergence of cond law of large jump, 1, LDP}
% \eqref{lemma: LDP, bar epsilon and delta}
\eqref{proof: choose bar delta, lemma LDP, bar epsilon and delta}
% \eqref{definition of xi prime1}
\eqref{bound of difference between xi and xi prime between t_J and t_J+1}
% \eqref{}
\eqref{goal, lemma: LDP, bar epsilon and delta, clipped version}
% \eqref{}
\eqref{choice of delta, proof, lemma: continuity of h k b mapping}
\eqref{choice of delta, 2, proof, lemma: continuity of h k b mapping}
\eqref{ineq 1, proof, lemma: continuity of h k b mapping}
% \eqref{}
\eqref{condition, x prime and x, lemma: continuity of h k b mapping}
% \eqref{lemma: boundedness of k jump set under truncation, LDP clipped}
\eqref{def: a sigma truncated at M, LDP}
\eqref{def: perturb ode mapping h k b, truncated at M, 1}
\eqref{def: perturb ode mapping h k b, truncated at M, 2}
\eqref{def: perturb ode mapping h k b, truncated at M, 3}
% \eqref{corollary: existence of M 0 bar delta bar epsilon, clipped case, LDP}
% \eqref{}
\eqref{ineq, no jump time, a, lemma: SGD close to approximation x circ, LDP}
% \eqref{ineq, with jump time, a, lemma: SGD close to approximation x circ, LDP}
\eqref{ineq, no jump time, b, lemma: SGD close to approximation x circ, LDP}
\eqref{ineq, with jump time, b, lemma: SGD close to approximation x circ, LDP}
\eqref{proof, ineq gap between X and y, SGD close to approximation x circ, LDP}
\eqref{proof, ineq gap between y and xi, SGD close to approximation x circ, LDP}
\eqref{proof, up to t strictly less than eta tau_1^eta, SGD close to approximation x circ, LDP}
% \eqref{def: x breve approximation, clipped, LDP, 1}
% \eqref{def: x breve approximation, clipped, LDP, 2}
% \eqref{def: x breve approximation, clipped, LDP, 3}
% \eqref{lemma: SGD close to approximation x breve, LDP clipped}
% \eqref{subsec: LDP clipped, proof of main results}
% \eqref{proposition: standard M convergence, LDP clipped}
% \eqref{proposition: standard M convergence, LDP unclipped}
\eqref{goal 1, proposition: standard M convergence, LDP unclipped}
% \eqref{goal 2, proposition: standard M convergence, LDP unclipped}
\eqref{goal 3, proposition: standard M convergence, LDP unclipped}
\eqref{goal 4, proposition: standard M convergence, LDP unclipped}
\eqref{goal 5, proposition: standard M convergence, LDP unclipped}
\eqref{subgoal for goal 2, proposition: standard M convergence, LDP unclipped}
% \eqref{proof, ineq for mathring x, proposition: asymptotic equivalence, unclipped}
\eqref{subgoal, goal 4, proposition: standard M convergence, LDP unclipped}
% \eqref{pick bar delta, goal 4, proposition: standard M convergence, LDP unclipped}
% \eqref{proposition: standard M convergence, LDP clipped, stronger boundedness assumption}
\eqref{property: choice of M 0, new, proposition: standard M convergence, LDP clipped}
\eqref{property: Y and tilde Y, 1, proposition: standard M convergence, LDP clipped}
\eqref{property: Y and tilde Y, 2, proposition: standard M convergence, LDP clipped}
\eqref{goal new 1, proposition: standard M convergence, LDP clipped}
\eqref{goal new 2, proposition: standard M convergence, LDP clipped}
\eqref{def: objects for definition of hat X leq k}
\eqref{def: time and size of top j jumps before n steps}
\eqref{def: hat X truncated b, j top jumps, 1}
\eqref{def: hat X truncated b, j top jumps, 2}
\eqref{property: equivalence between breve X and hat X, clipped at b}
% \eqref{proposition: asymptotic equivalence, clipped}
% \eqref{proposition: uniform weak convergence, clipped}
\eqref{goal: asymptotic equivalence claim, proposition: asymptotic equivalence, clipped}
\eqref{choice of bar delta, proof, proposition: asymptotic equivalence, clipped}
\eqref{choice of bar epsilon, proof, proposition: asymptotic equivalence, clipped}
\eqref{goal, event B 1, clipped, proposition: asymptotic equivalence, clipped}
\eqref{goal, event B 2, clipped, proposition: asymptotic equivalence, clipped}
\eqref{goal, event B 3, clipped, proposition: asymptotic equivalence, clipped}
\eqref{goal, event B 4, clipped, proposition: asymptotic equivalence, clipped}
% \eqref{goal, event B 5, clipped, proposition: asymptotic equivalence, clipped}
% \eqref{def: x eta b M circ approximation, LDP, 1}
\eqref{def: x eta b M circ approximation, LDP, 2}
\eqref{def: x eta b M circ approximation, LDP, 3}
\eqref{proof, ineq for mathring x, proposition: asymptotic equivalence, clipped}
\eqref{proof: bounded jump size for breve X, goal, event B 2, clipped, proposition: asymptotic equivalence, clipped}
\eqref{proof: goal 1, lemma: atypical 3, large jumps being too small, LDP clipped}
\eqref{proof: ineq 1, lemma: atypical 3, large jumps being too small, LDP, clipped}
% \eqref{}
\eqref{choice of bar delta, proposition: uniform weak convergence, clipped}
% \eqref{def: Phi mapping, LDP, unclipped}
\smallskip

\noindent
% \eqref{def: first exit time for heavy tailed SGD}
\eqref{def: mapping check g k b, endpoint of path after the last jump, first exit analysis}
% \eqref{theorem: first exit time, unclipped}
% \eqref{theorem: first exit time, clipped}
\eqref{def: first exit time, J *}
% \eqref{lemma: choose key parameters, first exit time analysis}
\eqref{proof, observation on xi, lemma: choose key parameters, first exit time analysis}
% \eqref{property: fixed constant bar epsilon delta t, first exit analysis, 4}
% \eqref{property: fixed constant bar epsilon delta t, first exit analysis, 1}
% \eqref{property: fixed constant bar epsilon delta t, first exit analysis, 2}
% \eqref{property: fixed constant bar epsilon delta t, first exit analysis, 3}
% \eqref{property: between g k b and h k b, 1}
% \eqref{property: between g k b and h k b, 2}
\eqref{def: t epsilon function, first exit analysis}
\eqref{property: t epsilon function, first exit analysis}
% \eqref{property: bound return time log scale, first exit analysis}
% \eqref{lemma: measure check C J * b, continuity, first exit analysis}
% \eqref{lemma: exit rate strictly positive, first exit analysis}
\eqref{proof, goal, lemma: exit rate strictly positive, first exit analysis}
% \eqref{lemma: limiting measure, first exit analysis}
% \eqref{lemma: key lemma constant cycle bound}
\eqref{def: epsilon relaxed first exit time}
\eqref{def: return time R in first cycle, first exit analysis}
% % \eqref{lemma: cycle, efficient return}
% \eqref{lemma: cycle, exit when close to boundary}
% \eqref{lemma: fixed cycle exit or return taking too long, first exit analysis}
% \eqref{lemma: exit prob one cycle, first exit analysis}
% \eqref{proposition: first exit time}
% \eqref{claim: proposition: first exit time, upper bound}
% \eqref{claim: proposition: first exit time, lower bound}
% \eqref{proof: goal 1, proposition: first exit time}
% \eqref{proof: goal 2, proposition: first exit time}
% \eqref{proof, lower bound 1, proposition: first exit time}
% \eqref{proof, lower bound 2, proposition: first exit time}
% \eqref{proof: decomp of main event, theorem: first exit time, unclipped}
\eqref{proof, ineq 1, theorem: first exit time, unclipped}
\smallskip

\noindent
% \eqref{bounded away condition, theorem: lifiting from sample path LDP to LD of stationary distribuion of SGD}
% \eqref{sec: SDE}
% \eqref{assumption: heavy-tailed levy process}
\eqref{defSDE, initial condition x}
% \eqref{theorem: LDP, SDE, uniform M convergence}
% \eqref{theorem: LDP, SDE}
\eqref{def: Y eta b 0 f g, SDE clipped}
\eqref{def, tau and W, discont in Y, eta b, 1}
\eqref{def, tau and W, discont in Y, eta b, 2}
\eqref{def: objects for defining Y eta b, clipped SDE, 1}
\eqref{def: objects for defining Y eta b, clipped SDE, 2}
\eqref{def: objects for defining Y eta b, clipped SDE, 3}
\eqref{def: objects for defining Y eta b, clipped SDE, 4}
\eqref{defSDE, initial condition x, clipped}
\fi

\ifshownavigationpage
\newpage
\normalsize
\tableofcontents

\section*{Navigation Links}

\ifshowtheoremlist
    \noindent
    \hyperlink{location of theorem list}{List of Theorems}
    \bigskip
\fi

\ifshowtheoremtree
    \noindent
    \hyperlink{location of theorem tree}{Theorem Tree}
    \begin{itemize}
    \thmtreeref
        {Proposition}{proposition: standard M convergence, LDP clipped}
    
    \thmtreeref
        {Proposition}{proposition: standard M convergence, LDP unclipped}
        
    \thmtreeref
        {Theorem}{theorem: LDP 1, unclipped}
        
    \thmtreeref
        {Theorem}{theorem: sample path LDP, unclipped}
    
    \thmtreeref
        {Proposition}{proposition: standard M convergence, LDP clipped}
    
    \thmtreeref
        {Theorem}{theorem: LDP 1}
    
    \thmtreeref
        {Theorem}{theorem: first exit time, unclipped}
    
    % \thmtreeref
    %     {Theorem}{theorem: lifting from sample path LDP to LD of stationary distribution of SGD}
        
    \end{itemize}
\fi

\ifshowreminders
    \noindent
    \hyperlink{location of reminders}{Assumptions, etc}
    \bigskip
\fi

\ifshowtheoremlist
    \noindent
    \hyperlink{location of equation number list}{Numbered Equations}
    \bigskip
\fi

\ifshownotationindex
    \noindent
    \hyperlink{location of notation index}{Notation Index}
    \begin{itemize}
    \item[] 
        \hyperlink{location, notation index A}{A},
        \hyperlink{location, notation index B}{B},
        \hyperlink{location, notation index C}{C},
        \hyperlink{location, notation index D}{D},
        \hyperlink{location, notation index E}{E},
        \hyperlink{location, notation index F}{F},
        \hyperlink{location, notation index G}{G},
        \hyperlink{location, notation index H}{H},
        \hyperlink{location, notation index I}{I},
        \hyperlink{location, notation index J}{J},
        \hyperlink{location, notation index K}{K},
        \hyperlink{location, notation index L}{L},
        \hyperlink{location, notation index M}{M},
        \hyperlink{location, notation index N}{N},
        \hyperlink{location, notation index O}{O},
        \hyperlink{location, notation index P}{P},
        \hyperlink{location, notation index Q}{Q},
        \hyperlink{location, notation index R}{R},
        \hyperlink{location, notation index S}{S},
        \hyperlink{location, notation index T}{T},
        \hyperlink{location, notation index U}{U},
        \hyperlink{location, notation index V}{V},
        \hyperlink{location, notation index W}{W},
        \hyperlink{location, notation index X}{X},
        \hyperlink{location, notation index Y}{Y},
        \hyperlink{location, notation index Z}{Z}
    \end{itemize}
\fi

\fi

%\end{APPENDIX}
%%%%%%%%%%%%%%%%%
\end{document}